\documentclass{article}
\usepackage{graphicx} 
\usepackage{tikz}
\usetikzlibrary{patterns}
\usepackage{tikz}
\usetikzlibrary{calc} 
\usepackage[dvipsnames, svgnames, x11names]{xcolor}
\usepackage[authoryear,round]{natbib}
\usepackage{color}
\usepackage{xcolor} 

\usepackage[unicode=true,
 bookmarks=false,
 breaklinks=false,pdfborder={0 0 1},colorlinks=false]
 {hyperref}
\hypersetup{
colorlinks,citecolor=blue,filecolor=blue,linkcolor=blue,urlcolor=blue}
\usepackage{geometry}
\usepackage{booktabs}
\usepackage{footmisc}
\usepackage{amsmath}
\usepackage{amssymb}
\usepackage{bbm}
\usepackage{amsthm}
\usepackage{graphicx}
\usepackage{algorithmic}
\usepackage{courier}
\usepackage{mathtools}
\usepackage[ruled,vlined]{algorithm2e}
\usepackage{subcaption}

\SetCommentSty{algocmtstyle}
\usepackage{appendix}
\newtheorem{theorem}{Theorem}[section]
\newtheorem{lemma}{Lemma}[section]
\newtheorem{proposition}{Proposition}[section]
\newtheorem{fact}{Fact}[section]

\newtheorem{assumption}{Assumption}[section]
\newtheorem{definition}{Definition}[section]
\newtheorem{remark}{Remark}[section]
\newtheorem{claim}{Claim}[section]

\definecolor{yxc}{RGB}{255,0,0}

\usepackage{amsmath,amsfonts,bm,cleveref,comment}

\newcommand{\driftocp}{\textsc{DriftOCP}\xspace}

\newcommand{\driftocpfull}{\textsc{DriftOCP}\text{-}\textsc{full}\xspace}

\newcommand{\gap}{\mathsf{cvg}\text{-}\mathsf{gap}}

\def\rd{{\textnormal{d}}}

\DeclareMathAlphabet{\mathsfit}{\encodingdefault}{\sfdefault}{m}{sl}
\SetMathAlphabet{\mathsfit}{bold}{\encodingdefault}{\sfdefault}{bx}{n}

\def\gA{{\mathcal{A}}}
\def\gB{{\mathcal{B}}}
\def\gC{{\mathcal{C}}}
\def\gD{{\mathcal{D}}}
\def\gE{{\mathcal{E}}}
\def\gF{{\mathcal{F}}}
\def\gG{{\mathcal{G}}}
\def\gH{{\mathcal{H}}}
\def\gI{{\mathcal{I}}}

\def\gK{{\mathcal{K}}}
\def\gL{{\mathcal{L}}}
\def\gM{{\mathcal{M}}}
\def\gN{{\mathcal{N}}}

\def\gP{{\mathcal{P}}}

\def\gS{{\mathcal{S}}}
\def\gT{{\mathcal{T}}}
\def\gU{{\mathcal{U}}}

\def\gX{{\mathcal{X}}}

\def\gZ{{\mathcal{Z}}}

\def\0{{\bf 0}}
\def\1{{\bf 1}}

\def\EB{{\mathbb E}}

\def\NB{{\mathbb N}}

\def\PB{{\mathbb P}}

\def\RB{{\mathbb R}}

\def\ZB{{\mathbb Z}}

\def\argmin{\mathop{\rm argmin}}

\def\Q{\mathsf{Q}}

\newcommand{\R}{\mathbb{R}}
\newcommand{\E}{\mathbb{E}}

\newcommand{\ssum}[3]{\sum\limits_{{#1}={#2}}^{#3}}

\newcommand{\TV}{\mathsf{TV}}

\def\inner#1#2{\left\langle #1, #2 \right\rangle}
\newcommand{\norm}[1]{\left\|{#1}\right\|}
\newcommand{\abs}[1]{\left|{#1}\right|}
\allowdisplaybreaks

\usepackage[utf8]{inputenc} 
\usepackage[T1]{fontenc}    
\usepackage{hyperref}       
\usepackage{url}            
\usepackage{booktabs}       
\usepackage{amsfonts}       
\usepackage{nicefrac}       
\usepackage{microtype}      
\usepackage{xcolor}         

\title{Optimal training-conditional regret for online conformal prediction}

\author{%
Jiadong Liang\thanks{Department of Statistics and Data Science, the Wharton School, University of Pennsylvania; email: \texttt{\{jdl97,zren,yuxinc\}@wharton.upenn.edu}.} 
\and
Zhimei Ren\footnotemark[1] 
\and
 Yuxin Chen\footnotemark[1]
}

\begin{document}

\maketitle

\begin{abstract}

We study online conformal prediction for non-stationary data streams subject to unknown distribution drift. While most prior work studied this problem under adversarial settings and/or assessed performance in terms of gaps of time-averaged marginal coverage, we instead evaluate performance through training-conditional cumulative regret. We specifically focus on independently generated data with two types of distribution shift: abrupt change points and smooth drift.

When non-conformity score functions are pretrained on an independent dataset, we propose a split-conformal–style algorithm that leverages drift detection to adaptively update calibration sets, which provably achieves minimax-optimal regret. When non-conformity scores are instead trained online, we develop a full-conformal–style algorithm that again incorporates drift detection to handle non-stationarity; this approach relies on stability—rather than permutation symmetry—of the model-fitting algorithm, which is often better suited to online learning under evolving environments. We establish non-asymptotic regret guarantees for our online full conformal algorithm, which match the minimax lower bound under appropriate restrictions on the prediction sets. Numerical experiments corroborate our theoretical findings.


\end{abstract}


\setcounter{tocdepth}{2}
\tableofcontents

\section{Introduction}

Conformal prediction, also known as conformal inference,  has emerged as a versatile, distribution-free framework for quantifying uncertainty in modern data science   \citep{vovk1999machine,papadopoulos2002inductive,vovk2005algorithmic,angelopoulos2023conformal,angelopoulos2024theoretical}. 
What sets it apart is its ability to offer rigorous, finite-sample coverage guarantees under minimal distribution assumptions, allowing practitioners to treat complex machine learning models as black boxes while still producing reliable measures of uncertainty. 
In its classical formulation, we observe $n$ training data taking the form of $n$ feature-response pairs $\{(X_i,Y_i)\}_{1\leq i\leq n} \subset \mathcal{X}\times \mathbb{R}$, and are given a test point $X_{n+1}\in \mathcal{X}$ for which the corresponding response $Y_{n+1}$ is unknown.  The aim is to construct a prediction set $\widehat{\mathcal{C}} (X_{n+1})$ that is likely to cover $Y_{n+1}$. 
Conformal prediction achieves this objective in a distribution-free fashion, 
provided the data $\{(X_i,Y_i)\}_{1\leq i\leq n+1}$ are {\em exchangeable} \citep{angelopoulos2023conformal,angelopoulos2024theoretical}.

%
%
%

While the ability to accommodate exchangeable data applies to wide-ranging practical scenarios, there is no shortage of scenarios that naturally violate exchangeability. One notable example arises when the data distributions drift over time, 
as is often the case with sequential or online data \citep{zhou2025conformal,fannjiang2022conformal}. 
This motivates a flurry of recent studies exploring online conformal prediction, with the objective to extend the conformal prediction framework to accommodate sequentially arriving data streams (e.g., \citet{vovk2009line,weinstein2020online,gibbs2021adaptive,bastani2022practical,zaffran2022adaptive,bhatnagar2023improved,lin2022conformal,feldman2022achieving,auer2023conformal,sun2023copula,xu2023sequential,xu2023conformal,xu2024conformal,gibbs2024conformal,han2024distribution,angelopoulos2023conformal,angelopoulos2024online,angelopoulos2025gradient,bao2024cap,lee2024single,yang2024bellman,podkopaev2024adaptive,su2024adaptive,zhang2024benefit,ramalingam2024relationship,sale2025online,humbert2025online}).

\subsection{Online conformal prediction}

Setting the stage, consider a sequential data stream 
$\{(X_t, Y_t)\}_{1 \le t \le T}$ generated by a dynamic process, where $X_t \in \mathcal{X}$ denotes the feature (or covariate) at time $t$ and $Y_t\in \mathbb{R}$ the corresponding response. The data-generating distribution is allowed to drift over time; namely, the distribution of $(X_t, Y_t)$, denoted by $\gD_t$, 
may vary with $t$.  
At each time $t$, the task is to use the previously observed data 
$\{(X_s, Y_s)\}_{s < t}$, together with the newly observed feature $X_{t}$, 
to construct a prediction set $\mathcal{C}_t(X_{t})$ that is likely to contain the as-yet-unobserved response $Y_{t}$. More precisely,  for a prescribed miscoverage level $\alpha\in(0,1)$, a desirable prediction set ${\gC}_t(X_{t})$ would satisfy
\begin{align}\label{eq:obj}
\mathbb{P}\left\{\, Y_{t}\in \gC_t\big(X_{t}\big)\ \big|\ \{(X_s,Y_s)\}_{s<t} \right\}
\geq 1-\alpha .
\end{align}
%

Central to conformal prediction is the non-conformity score function $s_{t}(\cdot,\cdot)$, which is computed at time $t$ and may sometimes depend on past observations $\{(X_{\tau},Y_{\tau})\}_{\tau:\tau < t}$.
For the most part, the score $s_t(x,y)$ measures the extent to which a data point $(x,y) \in \mathcal{X}\times \mathbb{R}$ deviates from the prediction of a fitted model. A canonical example is the absolute residual score $s_t(x,y)=|y-\widehat{\mu}_t(x)|$, where $\widehat{\mu}_t(\cdot)$ denotes a predictive model trained by an arbitrary machine learning algorithm (for instance, a neural network, or a nonparametric estimator).  
A widely studied class of prediction intervals takes the form
\begin{align}
\label{eq:def:C_tx}
\mathcal{C}_t(x) \coloneqq  \left\{ y : s_t(x, y) \leq q_t \right\}
\end{align}
for some adaptively chosen threshold $q_t$, in which case prediction interval construction amounts to dynamically adjusting $\{q_t\}$ given the non-conformity scores.

 The online nature of the above problem has motivated a recent line of work to reframe \eqref{eq:obj} as an online decision-making
task and leverage techniques from online learning to address it. 
A prominent example is {\em Adaptive Conformal Inference (ACI)},    proposed by  \citet{gibbs2021adaptive}. 
In a nutshell, the ACI algorithm sequentially calibrates the quantile estimates via the  iterative update rule: 
\begin{align}
\label{eq:aci}
q_{t+1}
=
q_t + \eta_t \big(\mathbbm{1}\{s_t(X_t,Y_t) > q_t \} - \alpha\big),
\end{align}
which can be interpreted as an instance of the online subgradient method applied to optimize the quantile loss  (or pinball loss). 
%
%
%
%


\subsection{Prior coverage guarantees and their inadequacy}

To establish theoretical validity, a substantial body of prior work developed coverage guarantees for online conformal prediction methods. For instance,  \citet{gibbs2021adaptive} demonstrated that the ACI algorithm achieves some sort of {\em time-averaged coverage} without imposing any assumption on the data-generating mechanism;  
more formally, they proved that, with a suitable constant learning rate schedule, ACI satisfies
\begin{equation}\label{eq:long_term_cvg_def-candes}
\underset{\text{empirical long-term coverage frequency}}{\underbrace{\frac{1}{T}\sum_{t=1}^T 
 \mathbbm{1} \big(Y_t \in \mathcal{C}_t(X_t)\big)}}
\hspace{-1.7em}\rightarrow\, 1-\alpha
\end{equation}
as $T$ grows, which holds even when the data stream is generated adversarially. Building on this result, subsequent work has extended time-averaged coverage results to a broader family of algorithms (e.g., \citet{zaffran2022adaptive,angelopoulos2024online,bhatnagar2023improved,zhang2024discounted}).

 Note, however, that controlling the empirical long-term coverage frequency in (\ref{eq:long_term_cvg_def-candes}) does not, by itself, preclude vacuous solutions. As noted in prior studies (e.g., \citet{bastani2022practical,bhatnagar2023improved}) and further elaborated in \Cref{sec:why-regret}, one can easily construct prediction sets that fulfill property (\ref{eq:long_term_cvg_def-candes}) while failing to incorporate any information of the underlying data distributions. In other words, achieving convergence of empirical long-term coverage frequency does not ensure reliable coverage at any individual time, nor does it guarantee that the prediction sets are informative and efficient.

To remedy the above issue, a line of subsequent work  (e.g., \citet{bhatnagar2023improved,gibbs2024conformal,hajihashemi2024multi,ramalingam2024relationship,zhang2024discounted,zhang2024benefit}) shifted focus towards {\em regret-based analysis}, drawing heavily from the online learning literature  \citep{shalev2012online,hazan2016introduction}. 
While multiple notions of regret have been explored in this strand of work,  they are primarily formulated for adversarial online settings, where the underlying data-generating distributions are left completely unspecified. Consequently, these regret metrics often lack a direct correspondence with standard conformal validity targets, such as training-conditional coverage. Furthermore, several prior works (e.g., \citet{bhatnagar2023improved,hajihashemi2024multi,ramalingam2024relationship}) evaluated the cumulative performance gap relative to global quantile optimized in hindsight (i.e., the quantile computed based on all data), which is, however, not well-suited to non-stationary environments with drifting data distributions.

It is important to emphasize again that the adoption of empirical long-term coverage frequency and adversarial regret largely stems from the objective to dispense with distributional assumptions, thereby maximizing the ``distribution-free'' nature of online predictive inference. 
However, if one is willing to impose more structure on the data-generating mechanism, it may become possible to derive coverage guarantees that align more closely with classical validity notions. While several prior work \citep{gibbs2021adaptive,han2024distribution,zaffran2022adaptive,xu2023conformal,angelopoulos2024online,humbert2025online} had investigated more specialized settings—such as independent data with drifting distributions, hidden Markov models—the optimality of the resulting theoretical guarantees remain largely unexplored.

\subsection{This paper}

In this work, we make progress by focusing on the following non-adversarial scenario:  
\begin{itemize} 
    \item {\em A non-adversarial setting with independent data:} The data $\{(X_t,Y_t)\}_{1\leq t\leq T}$ are {\em independently} generated but otherwise distribution-free. The distribution of $(X_t,Y_t)$, denoted by $\mathcal{D}_t$, is allowed to drift over time, but the predictive inference algorithm has no prior knowledge of the distributional drift. 
\end{itemize}
The independence assumption enables us to move beyond performance metrics like time-averaged marginal coverage and adversarial regret, and instead adopt a regret metric that aligns more closely with classical statistical validity. Informally, we focus on the following {\em training-conditional} cumulative regret metric
\begin{equation}
\mathsf{regret}_T\coloneqq \ssum{t}{1}{T}\EB
\Big[\abs{\PB\big(Y_t \in {\gC}_t(X_t) ~\big|~ \text{past data, internal randomness}\big) - (1-\alpha)}\,\Big],
\end{equation}
whose precise definition is given in \Cref{sec:tc_regret_def}. This metric measures, at each time $t$, the deviation of the conditional coverage probability---conditioned on past observations and, in the case of a randomized procedure, any internal randomness used by the procedure---from the target level, and then aggregates these deviations over time. The emphasis on training-conditional (sample-conditional) validity is standard in conformal prediction
\citep[e.g.,][]{vovk2012conditional,bian2023training,amann2023assumption,liang2025algorithmic}. 

Within this framework, we pay particular attention to two forms of distribution drift: (i) {\em the change-point setting}, where the data distribution is piecewise stationary with several abrupt change points; (ii) {\em the smooth drift setting}, where the distributions evolve continuously and smoothly over time, subject to an upper bound on its aggregate variation.  Note that the predictive inference algorithm operates without prior knowledge of the drift structure. Our main contributions are summarized as follows. 

\paragraph{Online conformal prediction with pretrained scores.}
Consider first the scenario in which the non-conformity score functions are pretrained on a separate, independent dataset—a common setting in online conformal prediction where split-conformal-style methods are naturally applicable. We propose an online conformal prediction algorithm, dubbed \driftocp (see Algorithm~\ref{alg:OCID}), which leverages drift detection subroutines to adaptively update calibration sets---the set of data used for calibrating $q_t$---over time.  Our algorithm is computationally lightweight, horizon-independent,  and adapts efficiently to the distribution drift. We provide non-asymptotic theoretical guarantees by establishing regret upper bounds for \driftocp that match the minimax lower bounds (up to a logarithmic factor) in both the change-point and smooth drift settings. 
Numerical experiments across a range of distribution-shift scenarios further demonstrate the efficacy of \driftocp, showing that it adapts effectively to diverse data-generating mechanisms.

\paragraph{Online conformal prediction with adaptively trained scores.} Next, consider a more challenging scenario in which both the predictive models and the non-conformity score functions are trained online, potentially depending on past observations. To enhance data efficiency without data splitting, we adopt the full conformal paradigm, and put forward an online full conformal prediction algorithm called \driftocpfull (see Algorithm~\ref{alg:FOCID}), which integrates drift detection subroutines to tackle non-stationarity.  Rather than assuming permutation symmetry of the model fitting algorithm---which is often violated in online learning---we focus instead on stable learning algorithms, and establish non-asymptotic upper bounds on the training-conditional cumulative regret of \driftocpfull. We further demonstrate the optimality of our approach by deriving matching minimax lower bounds (up to a log factor) under appropriate restrictions on the prediction sets. Notably, our training-conditional lower bound applies universally to \emph{all} prediction methods regardless of their specific construction—a result that was previously out of reach.  
Empirically, we benchmark several conformal prediction methods and validate the plausibility of constructing prediction sets using sequentially fitted models.

\subsection{Notation}
We now gather a set of notations used throughout the paper. For any $a,b\in \mathbb{R}$, denote $a\vee b=\max\{a,b\}$, $a \land b = \min\{a,b\}$, $(a)_+ = a\vee 0$, and $(a)_- = a\wedge 0$. For any integer $n$, let $[n]\coloneqq \{1,\dots,n\}$.
For $x\in\RB$, we use $\lceil x\rceil$ to denote the smallest integer greater than or equal to $x$, and $\lfloor x\rfloor$ the largest integer less than or equal to $x$.
For two nonnegative functions \(f\) and \(g\), we write \(f\lesssim g\) (equivalently, \(f=O(g)\) and \(g=\Omega(f)\)) if there
exists a universal constant \(C>0\) such that \(f\le Cg\).  We write \(f\gtrsim g\) if \(g\lesssim f\), and
\(f\asymp g\) if both \(f\lesssim g\) and \(g\lesssim f\) hold.  The notation \(f=\widetilde{O}(g)\) and \(g=\widetilde{\Omega}(f)\) is defined
analogously, up to additional logarithmic factors. 
For the set $\RB$ of real numbers, we denote by $\gB(\RB)$ the Borel sets on it.
We denote by \(\norm{v}_2\) the Euclidean norm of a vector \(v\in\R^d\).
For a matrix \(A\in\R^{m\times d}\), we use \(\norm{A} \coloneqq \sup_{\norm{x}_2=1}\norm{Ax}_2\) for its spectral norm.
For two probability distributions \(P\) and \(Q\) defined on a measure space $(\Xi, \gF)$, we denote by \(\mathsf{TV}(P,Q)\) their total-variation (TV)
distance, i.e.,
\[
\mathsf{TV}(P,Q) \coloneqq \sup\limits_{A \in \gF}\big\{\abs{P(A) - Q(A)}\big\}.
\]
Suppose $P$ and $Q$ admit probability density functions $p$ and $q$ on $\Xi$, respectively.
We define the Kullback--Leibler (KL) divergence from $Q$ to $P$ by
\[
\mathsf{KL}(P\,\|\,Q)\coloneqq \int_{\Xi} p(x)\log\!\frac{p(x)}{q(x)}\,\rd x,
\]
whenever the integral is well-defined. 
%
Furthermore, if \(P\) and \(Q\) are two probability distributions defined on $\big(\RB, \gB(\RB)\big)$, 
we denote by \(\mathsf{KS}(P,Q)\) their Kolmogorov--Smirnov (KS) distance, i.e., 
\begin{align}
\label{eq:defn-KS-dist}
\mathsf{KS}(P,Q) \coloneqq \sup\limits_{x\in \RB}\Big\{\abs{P\big((-\infty, x]\big) - Q\big((-\infty, x]\big)}\Big\}.
\end{align}
For random objects $Z\sim P$ and $\widetilde{Z}\sim Q$, we overload the notation by letting
$\mathsf{TV}(Z,\widetilde{Z})$, $\mathsf{KL}(Z\,\|\,\widetilde{Z})$ and $\mathsf{KS}(Z,\widetilde{Z})$   denote
$\mathsf{TV}(P,Q)$, $\mathsf{KL}(P\,\|\,Q)$ and $\mathsf{KS}(P,Q)$, respectively. Also, for any sequence of objects $\{Z_i\}_{i\geq 1}$, we adopt the notation $Z_{k:m}=\{Z_k,\dots,Z_m\}$ for any $m\geq k\geq 1$. 

\section{Problem formulation and key metrics}\label{sec:prob_form}
\subsection{Settings}\label{sec:online_conf_prediction_pro}
%


%
Consider a sequence of $T$ 
{\em independent} data points arriving sequentially, denoted by $Z_t = (X_t, Y_t) \in \mathcal{X} \times \mathbb{R}$, $t=1,\dots,T$,  where the set $\mathcal{X}$ represents the feature domain. At each time $t$, the feature $X_t$ is revealed first, and the response $Y_t$ becomes available after a prediction set has been formed.  Throughout this paper, we use $Z_{1:t}=\{(X_s,Y_s)\}_{s\leq t}$ to denote the set of all data up to time $t$.

\paragraph{Procedure.}
An online conformal prediction  procedure $\pi$ operates as follows. 
Let $U$ denote---when $\pi$ is a randomized procedure---the internal randomness of $\pi$, generated via a random seed and assumed to be independent of the data. 
Given the samples $\{(X_s,Y_s)\}_{s<t}$ observed prior to time $t$, the newly arrived feature $X_t$, and the internal randomness $U$,  algorithm $\pi$ seeks to construct a prediction set  
\[
\gC_t \;= \; \gC_t^\pi\!\left(X_t;\{(X_s,Y_s)\}_{s<t}, U\right)\subseteq \R,
\]
designed to contain \(Y_t\) with probability exceeding---and ideally close to---the target level \(1-\alpha\). Here, we often write $\mathcal{C}_t$ for brevity if it is clear from the context.  The set $\mathcal{C}_t$ is typically built with the aid of a non-conformity score function $s_t(\cdot,\cdot)$ along with a fitted predictive model $\widehat{\mu}_t(\cdot)$. In this paper, we consider two practically important scenarios, distinguished by how the predictive models and non-conformity scores are trained. 
\begin{itemize}
    \item {\em Online conformal prediction with pretrained scores.} In this scenario, both the score functions and the predicted models are pretrained on a separate, independent dataset or data stream. As a result, the $s_t(\cdot,\cdot)$'s are independent of the online data stream on which the prediction sets are constructed, while still being allowed to evolve over time. 

    \item {\em Online conformal prediction with adaptively trained scores.} In this scenario, we allow both the score functions and the predictive models to be trained online, possibly depending on the past observations of the data stream. Therefore, the $s_t(\cdot,\cdot)$'s may be statistically dependent on $\{(X_s,Y_s)\}_{s<t}$. 
\end{itemize}
For both scenarios, an ideal online conformal prediction algorithm would adapt efficiently to the dynamic environment while allowing for tractable computation.



\paragraph{Distribution shift over time.}

Denote by $\mathcal{D}_t$ (resp.~$\mathcal{D}_{1:t}$) the distribution of $Z_t=(X_t, Y_t)$ (resp.~$Z_{1:t}$). We allow $\mathcal{D}_t$ to vary over time, which generally violates the exchangeability assumption. In this work, we pay particular attention to the following two distribution drift scenarios.


\bigskip
\noindent {\em (i) The change-point setting.}
This concerns the scenario where the data stream is, in some sense, piecewise stationary. 
Formally,  assume the existence of $N^{\mathsf{cp}}$ (\textit{a priori} unknown) change points, denote by 
\begin{equation}
1 = \tau_0 < \tau_1 < \cdots < \tau_{N^{\mathsf{cp}}} < \tau_{N^{\mathsf{cp}}+1} = T+1,
\end{equation}
such that for each $k=0,\ldots,N^{\mathsf{cp}}$, 
\begin{align}
\begin{cases}
s_t(X_t,Y_t) \sim \mathcal{D}_{k,\mathsf{seg}}^{\mathsf{score}},
\qquad &\tau_{k} \le t < \tau_{k+1}, \qquad \text{when scores are pretrained},\\
(X_t,Y_t) \sim \mathcal{D}_{k,\mathsf{seg}},
\qquad &\tau_{k} \le t < \tau_{k+1}, \qquad \text{when scores are trained online},
\end{cases}
\end{align}
where $\mathcal{D}_{k,\mathsf{seg}}^{\mathsf{score}}$ (resp.~$\mathcal{D}_{k,\mathsf{seg}}$) represents the score (resp.~data) distribution over the $k$-th time segment $[\tau_k,\tau_{k+1})$. 
In words, the distribution of interest remains fixed within each time segment, but may change abruptly at the change points $\tau_1,\ldots,\tau_{N^{\mathsf{cp}}}$.  
It is assumed that the number and locations of the change points, as well as the associated data distributions, are arbitrary and unknown to the online conformal prediction algorithm.

%
%


\bigskip
\noindent  {\em (ii) The smooth drift setting.}
In contrast to the above change-point setting that is well suited to modeling infrequent but potentially abrupt distributional jumps, the second setting targets the scenario in which $\mathcal{D}_t$ evolves continuously and smoothly over time. To quantify the overall extent of such distributional variation, we rely on the following two metrics.
\begin{itemize}
\item \emph{Cumulative data variation:} this metric measures the aggregate total-variation distance between consecutive data distributions: 
\begin{equation}\label{eq:data_cumu_var}
\mathsf{TV}_T
\;\coloneqq\;
\sum_{t=1}^{T-1}
\mathsf{TV}( \mathcal{D}_t,\,\mathcal{D}_{t+1} ).
\end{equation}

\item \emph{Cumulative score variation:} in contrast to $\mathsf{TV}_T$, which is defined based on data distributions, this metric is score-based and tracks the cumulative Kolmogorov-Smirnov distance
 of consecutive score distributions: 
\begin{equation}\label{eq:score_cumu_var}
\mathsf{KS}_T \coloneqq \ssum{t}{1}{T-1}\mathsf{KS}\big(\gD_t^{\mathsf{score}} , \gD_{t+1}^{\mathsf{score}}\big),
\end{equation}
where 
\(\gD_t^{\mathsf{score}}\) denotes the  distribution of $s_t(X_t,Y_t)$ under data distribution $(X_t,Y_t)\sim \gD_t$.

\end{itemize}

\noindent 
Notably, the score-based metric
$\mathsf{KS}_T$ can be viewed as a particular instance of the
more general cumulative data variation $\mathsf{TV}_T$.
In fact, similar quantities have been adopted in prior studies on online learning under data distribution shift~\citep{besbes2014stochastic,besbes2019optimal,cheung2019learning,zhao2020simple}. 
In this smooth drift setting, our aim is to design online conformal prediction algorithms whose performance can adapt gracefully to such cumulative variations.

%


%

\subsection{Key metrics: training-conditional coverage and cumulative regret}\label{sec:tc_regret_def}

To assess the performance of an online conformal prediction procedure $\pi$, a natural  metric is the 
\emph{training-conditional coverage rate}, defined as
\begin{align}\label{eq:def_Cvg}
\mathsf{cvg}_t = 
\mathsf{cvg}^\pi_t(Z_{1:t-1}, U)
    \coloneqq 
    \mathbb{P}\big( Y_t \in \mathcal{C}^\pi_t(X_t ; Z_{1:t-1}, U) \,\big|\, Z_{1:t-1}, U \big),
\end{align}
where we often write $\mathsf{cvg}_t$ for brevity. 
This metric quantifies the probability of successful coverage conditional on all past data (note that the prediction set $\mathcal{C}_t$ is often constructed based on past observations) as well as any internal randomness.
Compared with marginal coverage, training-conditional coverage is a stronger notion that ensures most of the test points are covered given the constructed prediction set $\mathcal{C}_t$.



Ideally, one would anticipate $\mathsf{cvg}_t$ to match the target level $1 - \alpha$. The deviation between the nominal and actual coverage at time $t$---which may be interpreted as the ``regret'' incurred at time $t$---is quantified by the \emph{training-conditional coverage gap} metric defined as
\begin{align}
\label{eq:def-gap-cvg}
\gap_t = 
\gap^\pi_t(Z_{1:t-1},U)
    \coloneqq
    \bigl|
        \mathsf{cvg}^{\pi}_t(Z_{1:t-1},U) - (1-\alpha)
    \bigr|.
\end{align}
The \emph{training-conditional cumulative regret}---hereafter often abbreviated as cumulative regret, or simply regret---of algorithm $\pi$ is then defined as
\begin{align}\label{eq:def_regret}
\mathsf{regret}_T=\mathsf{regret}_\pi\left(\gD_{1:T}, T\right)
    \coloneqq
    \sum_{t=1}^{T}
       \mathop{{\EB}}\limits_{Z_{1:t-1}\sim \gD_{1:t-1}, U}\big[ \gap^\pi_t(Z_{1:t-1}, U)\big],
\end{align}
which aggregates the training-conditional coverage gaps over time and captures the deviation from the target coverage rate. Here and throughout, we often suppress the explicit dependence on past data and distributions and write $\gap_t$ and $\mathsf{regret}_T$ when it is clear from the context. 
Importantly, this cumulative regret notion bridges predictive inference (through coverage guarantees) and online learning (through regret analysis).

Another metric is the long-term coverage rate  defined as
\begin{equation}\label{eq:long_term_cvg_def}
\begin{aligned}
\mathsf{lt\text{-}cvg}_T=
    \mathsf{lt\text{-}cvg}_\pi\big(
    \gD_{1:T}, 
    T\big) &\coloneqq 
\frac{1}{T}\sum_{t=1}^T 
\PB \big(Y_t \in \mathcal{C}_t^\pi(X_t; Z_{1:t-1},U)\big),
\end{aligned}
\end{equation}
which is often abbreviated by $\mathsf{lt\text{-}cvg}_T$ and can be viewed as an expected version of the empirical long-term coverage frequency in (\ref{eq:long_term_cvg_def-candes}). 
This metric reflects the {\em time-averaged} coverage probability of a procedure over a horizon \(T\). A large body of prior work (e.g., \citet{gibbs2021adaptive,bastani2022practical,angelopoulos2024online}) studied how far the long-term coverage of a procedure deviates from the target level by looking at the quantity $\mathsf{lt\text{-}cvg}_T-(1-\alpha)$. 
Note, however, that this quantity captures only the gap between the average coverage probability and the nominal level, rather than the average of the coverage gaps over time (i.e., {\em gap of average} versus {\em average of gaps}); as a result, it does not necessarily reflect variations across individual times.




\subsection{Why training-conditional cumulative regret?}
\label{sec:why-regret}

The training-conditional cumulative regret  defined above offers a meaningful criterion for evaluating online conformal 
prediction algorithms.  
Unlike long-term coverage metrics like \eqref{eq:long_term_cvg_def}, cumulative regret remains informative under distributional drift by aggregating coverage gaps over the entire horizon. The fact below summarizes some basic connections between long-term coverage and regret; the proof can be found in \Cref{sec:proof-tc_regret_def}. 
\begin{fact}\label{lem:long-term_notge_regret}
The following connections between long-term coverage rate and cumulative regret hold.  
\begin{itemize}
    \item[(i)] The long-term coverage rate of any online conformal prediction algorithm satisfies
\[
\big| \mathsf{lt\text{-}cvg}_T - (1-\alpha) \big|
\;\le\;
\frac{\mathsf{regret}_T%
}{T}.
\]

  \item[(ii)] Consider any $0<\alpha\le 1/2$. 
There exists an online conformal prediction algorithm such that: 
\begin{itemize}
\item For every $t = 1,\dots,T$, its prediction set $\mathcal{C}_t$
is either $\varnothing$ or $\mathbb{R}$, and satisfies
$
%
\mathsf{lt\text{-}cvg}_T
= 1 - \alpha; 
$

\item The regret is lower bounded by
$
\mathsf{regret}_T
\ge \alpha T.$
\end{itemize}
\end{itemize}

\end{fact}
On the one hand, Fact~\ref{lem:long-term_notge_regret} asserts that sublinear regret (i.e., $\mathsf{regret}_T=o(T)$)  guarantees faithful calibration of the long-term coverage rate. On the other hand, Fact~\ref{lem:long-term_notge_regret} indicates that the converse fails to hold---as already observed previously (e.g., \citet{bastani2022practical,bhatnagar2023improved,gibbs2024conformal})---an algorithm can achieve perfectly calibrated long-term coverage while still incurring training-conditional regret that grows linearly with $T$. 
More specifically, long-term coverage does not distinguish an algorithm that consistently achieves coverage close to $1-\alpha$ from one whose average coverage only coincidentally approaches $1-\alpha$.
For this reason, $\mathsf{regret}_T$ serves as a more informative performance measure for 
online conformal prediction.

\section{Online conformal prediction with pretrained scores}\label{sec:split_conf}

In this section, we study online conformal prediction with pretrained score functions, and put forward an algorithm that achieves minimax-optimal regret (up to logarithmic factors) for both the change-point and the smooth drift settings. To be precise, we impose the following assumption throughout this section. 
\begin{assumption}[Pretrained scores]
\label{assump:split}
Assume that the nonconformity score functions $\{s_t(\cdot,\cdot)\}_{t=1}^T$ are trained on an separate dataset independent of $\{(X_t,Y_t)\}_{t=1}^T$. 
Conditionally on $\{s_t(\cdot,\cdot)\}_{t=1}^T$, the observations $\{(X_t,Y_t)\}_{t=1}^T$ are independent. 
Moreover, for each $t\in[T]$, the random variable $s_t(X_t,Y_t)$ admits a continuous cumulative distribution function.
\end{assumption}
In words, the data used to pretrain the scores---such as an offline dataset or a different data stream---are separate from, and independently generated of, the data stream for which we construct conformal prediction sets. Consequently, the procedures studied in this section have the flavor of split conformal methods \citep{vovk2005algorithmic}. It is also noteworthy that the score functions are allowed to be time-varying.

\subsection{Algorithm}
\label{sec:online-split-algorithm-theory}

Let us motivate our algorithmic ideas and describe the proposed procedure for handling distribution shifts over time. 
Intuitively, when the data distributions drift significantly while the online conformal prediction algorithm continues to rely on stale quantile estimates, 
 miscoverage can occur frequently, resulting in loss of regret optimality. To remedy this issue, a natural strategy is to continuously monitor the empirical coverage and promptly reset the quantile estimates once they become statistically unreliable. 
 This idea underlies our algorithm design.


\subsubsection{Motivating examples}

To formalize the above intuition, we begin by examining two simplified cases.  A metric that we shall pay particular attention to is the following block coverage error over the time interval $[s, t]$: 
\begin{align}
\mathsf{cvg}\text{-}\mathsf{err}^{\star}_q(s,t) 
\coloneqq 
\sum_{l=s}^{t} 
\Big(\mathbb{P}\big(s_l(X_l,Y_l) \le q\big)
-(1- \alpha)\Big)
\end{align}
with $q$ a given threshold; we elucidate how this metric allows us to detect distribution shift below.


\paragraph{A simple case with 1 change point.} 
%
Before time $t$, there is a unique change point $t_0 < t$:
\begin{itemize}
    \item  for every $1\leq l\leq t_0$, the score $s_l(X_l,Y_l)$ is independently drawn from the distribution $\mathcal{P}_1^{\mathsf{seg}}$; 

    \item for every $t_0< l\leq t$,  the score $s_l(X_l,Y_l)$ is independently drawn from the distribution $\mathcal{P}_2^{\mathsf{seg}}$. 
\end{itemize}
\noindent 
The threshold $q$ is taken to be the $(1-\alpha)$-quantile of $\mathcal{P}_1^{\mathsf{seg}}$. Below, we  write $s_l=s_l(X_l,Y_l)$ for brevity. 


Figure~\ref{fig:block-coverage} provides a schematic illustration of this simple
scenario.  
The left panel plots $\mathbb{P}(s_t\le q)-(1-\alpha)$ as $t$ varies.  
By construction,  
$\mathbb{P}(s_t\le q)-(1-\alpha)=0$ before the
change point $t_0$.  
At time $t_0$, the score distribution shifts from $\mathcal{P}_1^{\mathsf{seg}}$ to $\mathcal{P}_2^{\mathsf{seg}}$, causing $\mathbb{P}(s_t\le q)-(1-\alpha)$ to jump to a nonzero value. This jump reflects the miscalibration induced by applying the pre-change cutoff $q$ to the post-change distribution.
The right panel plots $\mathsf{cvg}\text{-}\mathsf{err}^{\star}_q(1,t) $ versus $t$,  illustrating the cumulative
effect of these pointwise deviations.  
We have $\mathsf{cvg}\text{-}\mathsf{err}^{\star}_q(1,t) =0$ prior to
$t_0$, after which the bias  accumulates over time and $\mathsf{cvg}\text{-}\mathsf{err}^{\star}_q(1,t) $ grows linearly.  
If we fix a threshold $\sigma>0$ and declare a distributional change once
$\mathsf{cvg}\text{-}\mathsf{err}^{\star}_q(1,t) >\sigma$, then for $\sigma$ sufficiently small the detection time $t_1$ will
occur shortly after $t_0$.  
This illustrates how a simple block-coverage statistic can enable timely detection of distributional drift.

\begin{figure}[t]
\centering
\begin{tikzpicture}[>=stealth, scale=1]

\begin{scope}
  \draw[->] (0,0) -- (5,0) node[below] {$t$};
  \draw[->] (0,0) -- (0,3) node[left] {$\mathbb{P}(s_t\le q)$};

  \draw[dashed] (0,0.2) -- (5,0.2);
  \node[left] at (0,0.2) {$1 - \alpha$};

  \draw[thick] (0,0.2) -- (2.5,0.2);

  \draw[densely dotted] (2.5,-0.1) -- (2.5,2.2);
  \node[below] at (2.5,-0.1) {$t_0$};

  \draw[thick] (2.5,1.7) -- (4.5,1.7);

  \draw[dashed] (4.0,-0.1) -- (4.0,2.2);
  \node[below] at (4.0,-0.1) {$t_1$};

\end{scope}

\begin{scope}[xshift=7cm]
  \draw[->] (0,0) -- (5,0) node[below] {$t$};
  \draw[->] (0,0) -- (0,3) node[left] {$\mathsf{cvg}\text{-}\mathsf{err}^{\star}_q(1,t)$};

  \draw[dashed] (0,2.0) -- (5,2.0);
  \node[right] at (5,2.0) {$\sigma$};

  \draw[dashed] (0,0.2) -- (5,0.2);
  \node[left] at (0,0.2) {$0$};

  \draw[thick] (0,0.2) -- (2.5,0.2);

  \draw[densely dotted] (2.5,-0.1) -- (2.5,2.6);
  \node[below] at (2.5,-0.1) {$t_0$};

  \draw[thick] (2.5,0.2) -- (4.0,2.0);

  \draw[densely dotted] (4.0,-0.1) -- (4.0,2.6);
  \node[below] at (4.0,-0.1) {$t_1$};

\end{scope}

\end{tikzpicture}
\caption{The case with a single change point at $t_0$.  
(Left) pointwise coverage $\mathbb{P}(s_t\le q)$ vs.~$t$;  
(right) block coverage error $\mathsf{cvg}\text{-}\mathsf{err}^{\star}_q(1,t)$ vs.~$t$ along with a detection
threshold $\sigma$.}
\label{fig:block-coverage}
\end{figure}
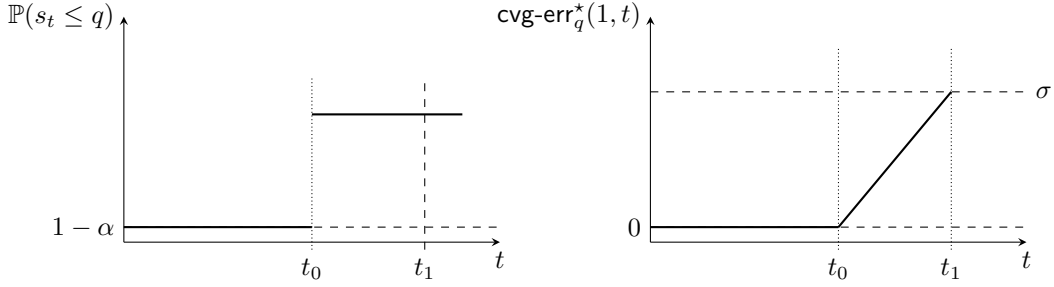

\paragraph{A case with smooth and oscillating distribution shifts.}
Consider another simple example, where the score distributions evolve smoothly over time and
the instantaneous deviations \(\PB(s_t\le q)-(1- \alpha)\) oscillate in sign.
The variation of \(\PB(s_t\le q)\)  vs.~$t$ is displayed in Figure~\ref{fig:block-coverage-general}(left), where   \(\PB(s_t\le q)\) crosses the
reference level \(1-\alpha\) multiple times, with the signed deviation being positive on some sub-intervals and negative on
others. 


Such oscillations cause a cancellation effect. 
As illustrated in Figure~\ref{fig:block-coverage-general}(right), the block coverage error \( \mathsf{cvg}\text{-}\mathsf{err}^{\star}_q(t_0,t)\) may initially increase but subsequently return to 0 as positive and negative contributions offset one another. Consequently, monitoring deviations from a
single starting point \(t_0\) can fail to detect distribution drift. 
Motivated by this, our proposed solution is to scan over different starting times within a time window
and track the maximum deviation.
 As shown in Figure~\ref{fig:block-coverage-general}(right), the block \([t_0,t_2]\) exhibits a large  deviation
\(\mathsf{cvg}\text{-}\mathsf{err}^{\star}_q(t_1,t_2)\), even though deviations measured from $t_0$ cancel out.  
This maximum-deviation statistic is therefore capable of detecting 
smooth and oscillating distribution shifts.



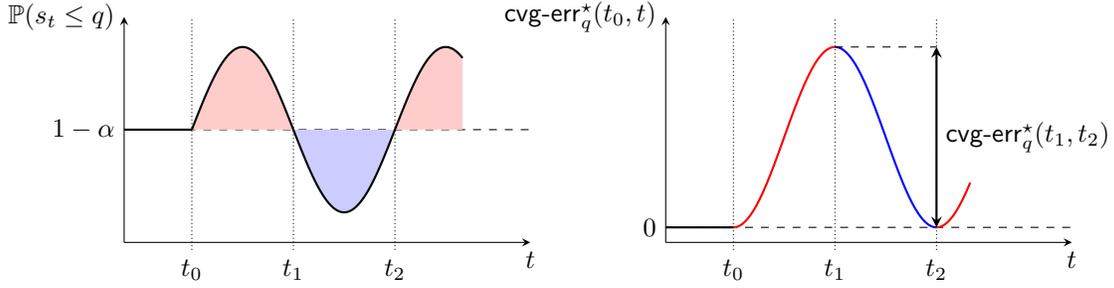
\begin{figure}[t]
\centering
\begin{tikzpicture}[>=stealth, xscale=0.9, yscale=1]

\def\tzero{1.0}      
\def\ttwo{4.0}       
\def\Tlen{3.0}       
\def\tone{2.5}       
\def\tthree{5.0}     
\def\tfour{4.5}
\def\Alpha{1.5}      
\def\Amp{1.1}        

\begin{scope}
  \draw[->] (0,0) -- (6,0) node[below] {$t$};
  \draw[->] (0,0) -- (0,3) node[left] {$\mathbb{P}(s_t\le q)$};

  \draw[dashed] (0,\Alpha) -- (6,\Alpha);
  \node[left] at (0,\Alpha) {$1 - \alpha$};

  \draw[densely dotted] (\tzero,-0.1) -- (\tzero,3);
  \node[below] at (\tzero,-0.1) {$t_0$};

  \draw[densely dotted] (\tone,-0.1) -- (\tone,3);
  \node[below] at (\tone,-0.1) {$t_1$};

  \draw[densely dotted] (\ttwo,-0.1) -- (\ttwo,3);
  \node[below] at (\ttwo,-0.1) {$t_2$};

  \draw[thick] (0,\Alpha) -- (\tzero,\Alpha);

  \begin{scope}
    \clip (\tzero,\Alpha) rectangle (\tone,3);
    \fill[red!20]
      (\tzero,\Alpha)
      -- plot[domain=\tzero:\tone, smooth, samples=200]
         ({\x},{\Alpha + \Amp*sin(deg(2*pi*(\x-\tzero)/\Tlen))})
      -- (\tone,\Alpha) -- cycle;
  \end{scope}

  \begin{scope}
    \clip (\tone,0) rectangle (\ttwo,\Alpha);
    \fill[blue!20]
      (\tone,\Alpha)
      -- plot[domain=\tone:\ttwo, smooth, samples=200]
         ({\x},{\Alpha + \Amp*sin(deg(2*pi*(\x-\tzero)/\Tlen))})
      -- (\ttwo,\Alpha) -- cycle;
  \end{scope}

  \begin{scope}
    \clip (\ttwo,\Alpha) rectangle (\tthree,3);
    \fill[red!20]
      (\ttwo,\Alpha)
      -- plot[domain=\ttwo:\tthree, smooth, samples=200]
         ({\x},{\Alpha + \Amp*sin(deg(2*pi*(\x-\tzero)/\Tlen))})
      -- (\tthree,\Alpha) -- cycle;
  \end{scope}

  \draw[thick,domain=\tzero:\tthree, smooth, samples=200]
    plot ({\x},{\Alpha + \Amp*sin(deg(2*pi*(\x-\tzero)/\Tlen))});

\end{scope}

\begin{scope}[xshift=8.0cm]  

  \draw[->] (0,0) -- (6,0) node[below] {$t$};
  \draw[->] (0,0) -- (0,3)
    node[left] {$\mathsf{cvg}\text{-}\mathsf{err}^{\star}_q(t_0,t)$};

  \draw[dashed] (0,0.2) -- (6,0.2);
  \node[left] at (0,0.2) {$0$};

  \draw[densely dotted] (\tzero,-0.1) -- (\tzero,3);
  \node[below] at (\tzero,-0.1) {$t_0$};

  \draw[densely dotted] (\tone,-0.1) -- (\tone,3);
  \node[below] at (\tone,-0.1) {$t_1$};

  \draw[densely dotted] (\ttwo,-0.1) -- (\ttwo,3);
  \node[below] at (\ttwo,-0.1) {$t_2$};

  \draw[thick] (0,0.2) -- (\tzero,0.2);

  \def\Base{0.2}
  \def\Scale{1.2} 

  \draw[thick,red,  domain=\tzero:\tone,  smooth, samples=200]
    plot ({\x},{\Base + \Scale*(1 - cos(deg(2*pi*(\x-\tzero)/\Tlen)))});
  \draw[thick,blue, domain=\tone:\ttwo,  smooth, samples=200]
    plot ({\x},{\Base + \Scale*(1 - cos(deg(2*pi*(\x-\tzero)/\Tlen)))});
  \draw[thick,red,  domain=\ttwo:\tfour, smooth, samples=200]
    plot ({\x},{\Base + \Scale*(1 - cos(deg(2*pi*(\x-\tzero)/\Tlen)))});

  \pgfmathsetmacro{\Fone}{\Base + \Scale*(1 - cos(deg(2*pi*(\tone-\tzero)/\Tlen)))}
  \pgfmathsetmacro{\Ftwo}{\Base + \Scale*(1 - cos(deg(2*pi*(\ttwo-\tzero)/\Tlen)))}

  \draw[dashed] (\tone,\Fone) -- (\ttwo,\Fone);

  \draw[<->,thick] (\ttwo,\Ftwo) -- (\ttwo,\Fone);
  \node[right] at (\ttwo,0.5*\Fone+0.5*\Ftwo)
    {$\mathsf{cvg}\text{-}\mathsf{err}^{\star}_q(t_1,t_2)$};

\end{scope}

\end{tikzpicture}
\caption{Schematic illustration of a case with smooth, oscillating distribution shifts. (Left) pointwise coverage $\mathbb{P}(s_t\le q)$ vs.~$t$;  
(right) block coverage error $\mathsf{cvg}\text{-}\mathsf{err}^{\star}_q(1,t)$ vs.~$t$.} 
\label{fig:block-coverage-general}
\end{figure}

\subsubsection{The proposed procedure: \driftocp}

We are now positioned to present the proposed online conformal prediction procedure in the presence of pretrained scores, beginning with a distribution drift detection subroutine. 

\paragraph{Subroutine: detection of distribution drift (\textsc{DriftDetect}).} 
Thus far, we have illustrated the potential utility of $\mathsf{cvg}\text{-}\mathsf{err}^{\star}_q(s,t)$ in the face of distribution shift. 
Given that this quantity is not accessible in practice, we propose to approximate it via the following empirical block coverage error:
\begin{equation}\label{eq:emp_blk_cvg_err}
\mathsf{cvg}\text{-}\mathsf{err}_q(s,t) 
\coloneqq 
\sum_{l=s}^{t} 
\big(\mathbbm{1}\{s_l(X_l,Y_l) \le q\}
-(1- \alpha)\big).
\end{equation}
Assuming statistical independence, the central limit theorem implies that 
$\mathsf{cvg}\text{-}\mathsf{err}_q(s,t)$ fluctuates around 
$\mathsf{cvg}\text{-}\mathsf{err}^{\star}_q(s,t)$ with uncertainty on the order of $(t-s+1)^{1/2}$. Moreover, under stationarity of the scores within $[s,t]$, one has $\mathsf{cvg}\text{-}\mathsf{err}^{\star}_q(s,t)=0$.  
Consequently, testing whether the normalized empirical fluctuation 
\(
{\bigl|\mathsf{cvg}\text{-}\mathsf{err}_q(s,t)\bigr|}\big/{\sqrt{t-s+1}}
\)
exceeds a suitably chosen threshold $\sigma$ provides a natural criterion for detecting distribution shifts within the time interval $[s,t]$. 
The intuition is formalized in the subroutine described in~\Cref{alg:drift-detection}, 
denoted by
\textsc{DriftDetect}$(q; t_0, t_1; \sigma)$, which scans the window $[t_0,t_1]$
for statistically significant departure from stationarity. 
The subroutine plays a pivotal role in our main
procedure.

\begin{algorithm}[t]
\DontPrintSemicolon
\caption{\textsc{DriftDetect}$(q; t_0,t_1; \sigma)$}\label{alg:drift-detection}
\textbf{input:} quantile $q$; time window $[t_0,t_1]\subseteq[T]$; detection threshold $\sigma$.\\
\For(\tcp*[f]{scan the entire time window.}){$j= t_0,\cdots,t_1$}{
    compute $Z_{j,t_1}\leftarrow \dfrac{\bigl|\mathsf{cvg}\text{-}\mathsf{err}_q(j,t_1)\bigr|}{\sqrt{t_1-j+1}}$ (cf.~\eqref{eq:emp_blk_cvg_err}).\tcp*{construct detection statistics.} 
    \If{$Z_{j,t_1}>\sigma$}{\Return \textsf{true}.\tcp*{declare detection of drift once this statistic exceeds the threshold.}}
}
\Return \textsf{false}
\tcp*{no distribution drift has been detected.}
\end{algorithm}

\paragraph{Full procedure.} 
We now describe several key components of our main procedure.  The complete procedure, called \driftocp (short for \textit{online conformal prediction with drift detection}), is summarized in \Cref{alg:OCID}. 
\begin{itemize}
    \item {\em Stage-wise decomposition.} The entire time horizon  is divided into a sequence of {\em stages} in a data-driven manner, where we use $n$ to index stages. 
  A new stage is initiated whenever the subroutine \textsc{DriftDetect} signals a substantial distribution drift. Within each stage, the score distributions are treated as
{\em approximately stationary}.  

  \item {\em Decomposition into rounds within each stage.} 
  Provided no distribution drift is detected, each stage is
further partitioned into a sequence of \emph{rounds}. Following the standard doubling trick (e.g., \citet[Chapter~2.3]{cesa-bianchi2006prediction}, \citet[Chapter~6]{lattimore2020bandit}), we let the round lengths grow
geometrically, which eliminates the need for prior knowledge of the horizon length. We use $r$ to index rounds. For round $r$ of stage $n$, all data from the preceding
round are used to update the quantile estimate $q_{n,r}$, which in turn determines the prediction set at any time $\tau$ within the current round: 
\[
    \mathcal{C}_{\tau} = \{ y: s_{\tau}(X_{\tau},y) \leq q_{n,r} \}.
\]

    \item {\em Drift detection within each round.} Let 
$\tau_{n,r}$ represent the time at which round $r$ of stage $n$
begins. During this round, incoming samples are monitored via
\(
\textsc{DriftDetect}\bigl(q_{n,r};\,\tau_{n,r},\,\tau_{n,r}+t;\,\sigma_{n,r}\bigr)
\),
so that each call to the subroutine $\textsc{DriftDetect}$  operates on a block beginning at the onset of the current round.

\end{itemize}

\begin{algorithm}[t]
\DontPrintSemicolon
\caption{\textsc{Online Conformal Prediction with Drift Detection (\driftocp)}}\label{alg:OCID}
\textbf{input:}  target coverage level \(1-\alpha\); detection thresholds \(\{\sigma_{n,r}\}_{n,r=1}^\infty\).\\
\textbf{initialize:} \(n\leftarrow 1\), \(r\leftarrow 1\), \(\tau\leftarrow 0\), \(\tau_{1,1} \leftarrow 1\), \(q_{1,1} \leftarrow 0\).

\While{\textsf{true}}{
  \For(\tcp*[f]{round $r$ contains at most $T_{r}=3^r$ time points.}){\(t= 1,\cdots,T_{r}\;(\coloneqq 3^r)\)}{
    \(\tau\leftarrow \tau+1\).\tcp*{update global time index.}
    observe feature $X_{\tau}$ and score function $s_{\tau}(\cdot,\cdot)$; set \(X_{n,r,t}\leftarrow X_{\tau}\), \(s_{n,r,t}(\cdot,\cdot)\leftarrow s_{\tau}(\cdot,\cdot)\).\;
    construct prediction set \(\gC_{\tau}
    \leftarrow \{y: s_{\tau}(X_{\tau},y)\le q_{\tau}\}\) with $q_{\tau}=q_{n,r}$.\;
    observe $Y_{\tau}$; set \(Y_{n,r,t}\leftarrow Y_{\tau}\). \tcp*{response is observed after the prediction set is formed.}
    $\textsf{drift} \leftarrow $   \textsc{DriftDetect}\((q_{n,r};\,\tau_{n,r},\,\tau_{n,r}+t-1;\,\sigma_{n,r}).\) \tcp*{call \Cref{alg:drift-detection}.}

    \If{\textsf{drift}\text{ is }\textsf{true}}{
    \(\widetilde{q} \leftarrow q_{n,r}\), 
      \(n\leftarrow n+1\), \(r\leftarrow 1\).\tcp*{update stage index and round index.}
      \(q_{n,1}\leftarrow \widetilde{q}\), \(\tau_{n,1}\leftarrow \tau+1\).\tcp*{initialization for new stage.}
      \textbf{break}.\tcp*{enter next stage.}
    }
  }
  \If{\textsf{drift} is \textsf{false}}{
  \(q_{n,r+1}\leftarrow \argmin_{q}\Bigl|\sum_{j=1}^{T_{r}}\bigl(\mathbbm{1}\{s_{n,r,j}(X_{n,r,j},Y_{n,r,j})>q\}-\alpha\bigr)\Bigr|\).\tcp*{update quantile estimate.}
  \(r\leftarrow r+1\), \(\tau_{n,r}\leftarrow \tau+1\).\tcp*{start time of next round.} 
  }
}
\end{algorithm}

We would also like to highlight several appealing features of 
Algorithm~\ref{alg:OCID}.
First, it is horizon-free, meaning that the procedure does not require any prior knowledge of the horizon length \(T\); as we shall see later, our algorithm achieves the desirable {\em anytime} regret---a terminology commonly adopted in the online learning literature \citep{lattimore2020bandit} to emphasize its horizon-free nature. Second, it is computationally lightweight. Each new observation triggers at most one drift detection subroutine over the current window and each round only updates the quantile estimate once.  In particular, the computational cost at each time $t$ scales linearly with the length of the current scanning window, instead of recalibrating over many candidate lookback windows as in some prior work (see \Cref{sec:related-split}). Moreover, the drift detection subroutine is inexpensive in practice, since the underlying detection statistics can be maintained incrementally. Finally, we emphasize that Algorithm~\ref{alg:OCID} operates without any prior knowledge of the underlying distributional drift—such as the number and locations of change points, or the degree of cumulative variation—highlighting its adaptation to unknown and evolving data-generating mechanisms.


\subsection{Theoretical guarantees}

Next, we establish non-asymptotic upper bounds on the training-conditional regret for the proposed Algorithm~\ref{alg:OCID}, encompassing both the change-point and smooth drift settings introduced in \Cref{sec:online_conf_prediction_pro}.
 
\begin{theorem}\label{thm:split_regret}
Suppose that Assumption~\ref{assump:split} holds. 
If we set the detection thresholds as
\(\sigma_{n,r} \coloneqq 24\sqrt{\log(4\tau_{n,r})}\) for every  stage-round index pair $(n,r)$, then 
Algorithm~\ref{alg:OCID} achieves
\begin{align}
    \mathsf{regret}_T \leq  
    \begin{cases}
        \widetilde{O}(\sqrt{(N^{\mathsf{cp}}+1)T}) 
        & \text{for the change-point setting}; \\
\widetilde{O}\left(\sqrt{T} + (\mathsf{KS}_T)^{\frac{1}{3}}T^{\frac{2}{3}}\right)
& \text{for the smooth drift setting}.
    \end{cases}
\end{align}
\end{theorem}
The proof of this theorem is provided in \Cref{sec:proof:thm:split_regret}. 
For the change-point setting, \Cref{thm:split_regret} reveals that the regret scales proportionally to the square root of the number of change points $N^{\mathsf{cp}}$, in addition to the $\sqrt{T}$ dependence on the time horizon---a scaling in $T$ that arises commonly in online learning \citep{shalev2012online,lattimore2020bandit}. In contrast, for the smooth drift setting, our regret bound contains a term 
$(\mathsf{KS}_T)^{1/3}T^{2/3}$, whose dependence on $T$ is worse than the $\sqrt{T}$ scaling. 
This suggests that the dominant source of regret may stem from the temporal evolution of the underlying score distribution, underscoring the important role of real-time adaptation. 
We also emphasize that the coverage gap depends on the KS distance between the 
{\em score} distributions rather than those of the {\em raw data}. 
As discussed in~\citet{barber2023conformal}, this distinction can lead to much tighter guarantees, 
since the scores may be far closer in distribution than the underlying data.
Encouragingly, these regret bounds match the minimax lower bound (up to logarithmic factors), as we shall demonstrate next.


%

\subsection{Minimax lower bound}\label{sec:split_lower_bd}

To examine the optimality of Algorithm~\ref{alg:OCID}, this subsection develops minimax lower bounds on the cumulative regret, tailored to the class of online algorithms with pretrained scores. We begin by specifying the admissible algorithms and distribution classes of interest, which are necessary for the lower-bound analysis.

\begin{itemize}
    \item {\em Admissible algorithms.} 
    For any online conformal prediction algorithm $\pi$ with pretrained scores, let $\pi_t$ denote its rule for selecting the quantile threshold $q_t$ (cf.~\eqref{eq:def:C_tx}) at time $t$. 
Let \(U\sim \mathsf{Unif}(0,1)\) be an auxiliary random seed, independent of the data stream.  
We consider a family $\mathcal{Q}$ of non-anticipating algorithms $\pi=\{\pi_t\}_{t\ge 1}$, where
$\pi_1:[0,1]\rightarrow\mathbb{R}$ and 
$\pi_t:\mathbb{R}^{t-1}\times[0,1]\rightarrow\mathbb{R}$ ($t\ge2$) are measurable mappings.  
For each $t$, $\pi$ specifies the quantile threshold
\[
q_t=
\begin{cases}
\pi_1(U), & \text{if }t=1,\\[4pt]
\pi_t(s_{t-1},\ldots,s_1,U), & \text{if }t\ge 2,
\end{cases}
\]
where we remind the reader that $s_t=s_t(X_t,Y_t)$.
Each $\pi_t$ depends only on the past scores and the random seed $U$, hence algorithms in $\mathcal{Q}$ are score-based, non-anticipatory, and possibly randomized. 

\item {\em Distribution classes.}
We introduce two \emph{score-based} distribution classes, corresponding to the two settings in \Cref{sec:online_conf_prediction_pro}: for given budgets \(N^{\mathsf{cp}}\in\ZB^+\) and \(\mathsf{KS}_T>0\), define
%
\begin{subequations}
\begin{align}
\mathcal{L}_1(N^{\mathsf{cp}})
&\coloneqq 
\Bigl\{
        (\mathcal{D}_1,\ldots,\mathcal{D}_T):\;
        (\gD_1^{\mathsf{score}},\ldots,\gD_T^{\mathsf{score}})\text{ change at most $N^{\mathsf{cp}}$ times.}
\Bigr\};\label{eq:N_change_dist}\\
\mathcal{L}_2(\mathsf{KS}_T)
&\coloneqq 
\Bigl\{
        (\mathcal{D}_1,\ldots,\mathcal{D}_T):\;
        \sum_{t=1}^{T-1}
        \mathsf{KS}(\mathcal{D}_t^{\mathsf{score}} , \mathcal{D}_{t+1}^{\mathsf{score}})
        \le
        \mathsf{KS}_T
\Bigr\}.\label{eq:smooth_vary_dist}
\end{align}
\end{subequations}

\end{itemize}

For a distribution class $\mathcal{L}$, the worst-case regret of algorithm $\pi \in \mathcal{Q}$ is defined as
\begin{align}
\label{eq:def-worstcase-regret}
\mathsf{regret}_\pi(\gL,T)
&\coloneqq
\sup_{(\mathcal{D}_1,\ldots,\mathcal{D}_T)\in\mathcal{L}}
\mathsf{regret}_\pi(\mathcal{D}_{1:T},T)\notag\\
&= \sup_{(\mathcal{D}_1,\ldots,\mathcal{D}_T)\in\mathcal{L}}\sum_{t=1}^T\EB\Big[\bigl|\,\PB_{\gD_t^{\mathsf{score}}}(s_t \le q_t\mid q_t) - (1-\alpha)\,\bigr|\Big].
\end{align}
Armed with these definitions and notation, we are ready to present our minimax lower bounds.



\begin{theorem}
\label{thm:lower_bd_split_regret}
Consider any fixed $\alpha \in (0,1)$. 
Suppose that Assumption~\ref{assump:split} holds. 
For any admissible algorithm $\pi\in\mathcal{Q}$, its worst-case regret (cf.~(\ref{eq:def-worstcase-regret}))  satisfies
\begin{align*}
\mathsf{regret}_\pi\big(\gL_1(N^{\mathsf{cp}}),T\big)
&=
{\Omega}
\!\left(
    \sqrt{(N^{\mathsf{cp}}+1)T}
\right);\\
\mathsf{regret}_\pi\big(\gL_2(\mathsf{KS}_T),T\big)
&=
{\Omega}
\!\left(
    \sqrt{T}
    +
    (\mathsf{KS}_T)^{1/3}T^{2/3}
\right).
\end{align*}
\end{theorem}

Evidently, the minimax regret lower bounds in \Cref{thm:lower_bd_split_regret} match the achievable regret of Algorithm~\ref{alg:OCID} in \Cref{thm:split_regret} (modulo some logarithmic factors), thereby confirming the regret optimality of our proposed procedure in a minimax sense. The proof is postponed to \Cref{sec:proof:thm:lower_bd_split_regret}.

\subsection{Comparisons with prior art}
\label{sec:related-split}

\citet{gibbs2021adaptive} introduced ACI with a time-invariant stepsize schedule,  and established guarantees in terms of the time-averaged long-run coverage frequency (cf.~\eqref{eq:long_term_cvg_def-candes}), which hold irrespective of the data generating mechanism but do not imply valid coverage at individual time points. 
Stronger (asymptotic) guarantees were also established under stationary hidden Markov models.  Building on this work, \citet{angelopoulos2024online} studied ACI with decaying stepsizes and proved asymptotically exact (pointwise) coverage under i.i.d.~data;  these results, however, are asymptotic in nature and do not readily extend to settings with distribution drift. The analysis for the empirical long-term coverage frequency has further motivated the studies of a new perspective in online learning called ``gradient equilibrium,'' which yields a useful framework for several other statistical applications \citep{angelopoulos2025gradient}.  
Relatedly,  \citet{bastani2022practical} generalized the notion of long-term coverage frequency by proposing an approach that
achieves \emph{multi-valid coverage} guarantees even in adversarial settings. Their guarantees, however, are stated in terms of empirical frequencies along the realized sequence and do not yield training-(and-calibration)-conditional coverage guarantees.

Several recent works \citet{pournaderi2024training,humbert2025online} began to investigate training-conditional guarantees for online conformal prediction. Nevertheless, these results either do not account for distribution shift or rely on fairly strong assumptions (e.g., a uniform upper bound on the pre-/post-drift density ratio). Assuming independently trained scores (as in Assumption~\ref{assump:split}), \citet{han2024distribution} derived training-conditional guarantees for online conformal prediction under distribution drift via adaptive lookback-window selection for quantile calibration. Their results focus on \emph{last-step} (terminal-time) validity, whereas we study training-conditional \emph{cumulative} regret. Our method is also computationally more efficient: at time \(t\), their procedure requires \(t\) quantile estimates and \(t^2\) empirical CDF evaluations, while Algorithm~\ref{alg:OCID} needs at most one quantile estimate and \(t\) empirical coverage computations. The same authors also studied model assessment and selection under distribution drift \citep{han2024model}.

When it comes to lower bound analysis,
\citet{areces2024two,duchi2025few} discussed minimax lower bounds for the training-conditional coverage error in the presence of independently trained score
functions. Compared to their results, \Cref{thm:lower_bd_split_regret} moves beyond coverage guarantees under worst-case covariate shift and 
explicitly accounts for the effect of distribution drift over time.

\section{Online conformal prediction with adaptively  trained scores}\label{sec:full_conf}

We now turn our attention to the scenario in which the non-conformity scores and the predictive models are allowed to be trained online based on past observations. More precisely, we make the following assumptions throughout this section.
\begin{assumption}[Online-trained scores]
\label{assump:full}
Suppose that the non-conformity scores are constructed online. At each time $t$, the score functions may depend on the past data $\{(X_t,Y_t)\}_{s<t}$, but not on any data observed at or after time $t$. The data $\{(X_t,Y_t)\}_{t=1}^T$ are independently generated. 
\end{assumption}
%

Given the flexibility to adaptively update the score functions, we adopt the 
full conformal paradigm~\citep{vovk2005algorithmic}, which leverages all available data for 
both score construction and quantile estimation, without resorting 
to data splitting.
While this full conformal approach enables more efficient use of the data, 
it also introduces intricate statistical dependence across time, 
making it challenging to detect distributional drift and 
to establish training-conditional coverage. We develop several technical 
innovations to address these challenges.



\subsection{Algorithm}\label{sec:full_conf_alg}
We first review the standard full conformal prediction method, and then describe how it can be adapted to streaming data with distribution drift.

\paragraph{Review: (batch) full conformal prediction.}

Imagine we are given two datasets,
\[
\mathcal{Z}^{\mathsf{train}} \coloneqq \big\{(X_i^{\mathsf{train}},Y_i^{\mathsf{train}})\big\}_{i=1}^n
\qquad \text{and}\qquad 
\mathcal{Z}^{\mathsf{cal}} \coloneqq \big\{(X_i^{\mathsf{cal}},Y_i^{\mathsf{cal}})\big\}_{i=1}^m,
\]
which may overlap. While it is common to take $\mathcal{Z}^{\mathsf{train}}=\mathcal{Z}^{\mathsf{cal}}$  to maximize data efficiency, we allow $\mathcal{Z}^{\mathsf{train}}$ and $\mathcal{Z}^{\mathsf{cal}}$ to differ, a flexibility that will be useful for subsequent algorithmic development. 
The test point contains the feature $X^{\mathsf{test}}$.

The training dataset $\mathcal{Z}^{\mathsf{train}}$ is used to train predictive models via a learning algorithm $\gA$, yielding\footnote{The learning algorithm $\mathcal{A}$ may additionally depend on an internal random seed $U$, taken to be independent of the data. Since all subsequent arguments are carried out conditional on $U$, we suppress this dependence throughout for notational convenience.}
\begin{align}
\widehat{\mu}^{(X^{\mathsf{test}},y)}(\cdot
) \coloneqq \gA\left(\mathcal{Z}^{\mathsf{train}}; 
(X^{\mathsf{test}}, y)\right)
\label{eq:mu-hat-Xy-algorithm}
\end{align}
for every candidate response $y\in \mathbb{R}$,  whereas the calibration dataset $\mathcal{Z}^{\mathsf{cal}}$ is used to construct the prediction set with the aid of the fitted models $\widehat{\mu}$. Importantly, the fitted model $\widehat{\mu}^{(X^{\mathsf{test}},y)}$ depends on the hypothesized response $y$ (cf.~(\ref{eq:mu-hat-Xy-algorithm})), and may need to be refitted for each $y$ under consideration.  For each $(X^{\mathsf{test}},y)$, we define a set of non-conformity scores (or residual scores) as:
\begin{subequations}
\label{eq:res_score}
\begin{align}
s_{i}^{(X^{\mathsf{test}},y)} & \coloneqq\big|Y_{i}^{\mathsf{cal}}-\widehat{\mu}^{(X^{\mathsf{test}},y)}(X_{i}^{\mathsf{cal}})\big|,\qquad i=1,\ldots,m,\\
s_{\mathsf{test}}^{(X^{\mathsf{test}},y)} & \coloneqq\big|y-\widehat{\mu}^{(X^{\mathsf{test}},y)}(X^{\mathsf{test}})\big|.
\end{align}
\end{subequations}
%
The full conformal prediction set is then taken to be:
\begin{equation}\label{eq:full_conformal_set}
\gC\big(X^{\mathsf{test}} 
\big) {\coloneqq} \Bigg\{y: s_{\mathsf{test}}^{(X^{\mathsf{test}},y )}  \le  
\mathsf{Quantile}_{1-\alpha}\Bigg(\frac{1}{m+1}\bigg[\delta\big(s_{\mathsf{test}}^{(X^{\mathsf{test}},y )}\big)+\ssum{i}{1}{m}\delta\big(s_{i}^{(X^{\mathsf{test}},y)}\big)\bigg]\Bigg)
\Bigg\},
\end{equation}
where $\delta(a)$ denotes a point mass (i.e., the Dirac measure) at $a$, and $\mathsf{Quantile}_{1-\alpha}(P)$ denotes the $(1-\alpha)$-quantile of distribution $P$. 
In words, this prediction set contains all candidate values whose residual scores do not exceed the $(1-\alpha)$-quantile of the empirical distribution formed by the $m$ calibration residuals together with the candidate's own residual.  
When $\mathcal{Z}^{\mathsf{train}} = \mathcal{Z}^{\mathsf{cal}}$,
under exchangeability of the data and permutation symmetry of the model fitting algorithm $\mathcal{A}$,  classical results (e.g., \citet{vovk2005algorithmic,lei2018distribution,barber2023conformal,liang2025algorithmic}) guarantee finite-sample validity of this full conformal procedure.


\paragraph{Our algorithm: online full conformal prediction with drift detection (\driftocpfull).}

When distributional drift occurs over time, the assumption of exchangeability breaks down, invalidating the coverage guarantees of the full conformal algorithm described above. Building on the key algorithmic ideas introduced in \Cref{sec:online-split-algorithm-theory}, we extend full conformal methods to online settings with temporal distribution drift.  

We refer to the proposed algorithm as \driftocpfull (short for \textit{online full conformal prediction with drift detection}), and present the full procedure in Algorithm~\ref{alg:FOCID}. We first isolate several key features of \driftocpfull that parallel those of \driftocp.
\begin{itemize}

\item {\em Drift detection subroutine \textsc{DriftDetect+}.}
    We continue to employ a drift detection subroutine to identify the occurrence of a distribution drift. We introduce a slightly extended version of Algorithm~\ref{alg:drift-detection}, formalized in 
Algorithm~\ref{alg:full_drift-detection} and referred to as \textsc{DriftDetect+}. In essence, \textsc{DriftDetect+} differs from \textsc{DriftDetect} only in that it replaces  
$\mathsf{cvg}\text{-}\mathsf{err}_{q}(s,t)$---defined in \eqref{eq:emp_blk_cvg_err} based on quantiles of the non-conformity score---with the more general definition of empirical block coverage error
\begin{align}
\mathsf{cvg}\text{-}\mathsf{err}_{\gC}(s,t) \coloneqq
\bigg|\ssum{l}{s}{t}\big(\mathbbm{1}\big\{Y_{l} \in \gC(X_{l}) \big\} - (1-\alpha)\big)\bigg|.
\label{eq:defn-cvg-err-general}
\end{align}
%

    \item {\em Decomposition into stages and rounds.}  Akin to Algorithm~\ref{alg:OCID}, we partition the entire time horizon into stages---with the aid of the subroutine \textsc{DriftDetect+} in Algorithm~\ref{alg:full_drift-detection} in a data-driven manner---and further decompose each stage into rounds. Within each stage, the data distributions are treated as approximately stationary.  As before, we use $  n  $ and $  r  $ to index stages and rounds, respectively. We will repeatedly use the following notation:
    \begin{itemize}
        \item $T_r=3^r$: the number of time points in round $r$ of each stage, chosen to grow geometrically in $r$ so as to avoid requiring prior knowledge of the horizon length $T$. 

        \item $r_n$: the last round of stage $n$. We adopt the convention that \((n,0)=(n-1,r_{n-1})\). 
    
        \item $\tau_{n,r}$: the time index—measured in the original horizon $\{1,\ldots,T\}$—corresponding to the first time point of round $r$ in stage $n$. We adopt the convention that \(\tau_{n,r_n+1} \coloneqq \tau_{n+1,1}\) and $\tau_{n,0}\coloneqq \tau_{n-1,r_{n-1}}$. 

        \item $X_{n,r,t}$ and $Y_{n,r,t}$: the feature and the response arriving at the $t$-th time point of round $r$ in stage $n$. 
    \end{itemize}
\end{itemize}
Next, we highlight several full conformal components of \driftocpfull that extend batch full conformal methods to online settings.
\begin{itemize}
\item {\em Training and calibration sets for round $r$ of stage $n$.} When constructing the prediction set at any time within round $r$ of stage $n$, we choose the training and  calibration sets as
\begin{align}
\mathcal{Z}^{\mathsf{train}}_{n,r} \coloneqq \hspace{-1.5em} \underset{\text{all data before current round}}{\underbrace{ \big\{(X_i,Y_i)\big\}_{i=1}^{\tau_{n,r}-1} }}
\quad \text{and}\qquad 
\mathcal{Z}^{\mathsf{cal}}_{n,r} \coloneqq \hspace{-0.8em}  \underset{\text{all data in preceding round}}{\underbrace{\big\{(X_i,Y_i)\big\}_{i=\tau_{n,r-1}}^{\tau_{n,r}-1}}}. 
\end{align}
%
In words, the training set $\mathcal{Z}^{\mathsf{train}}_{n,r}$ comprises all samples observed prior to the current round, while the calibration set $\mathcal{Z}^{\mathsf{cal}}_{n,r}$ consists of all samples collected during the immediately preceding round (i.e., round $r-1$ of stage $n$). Intuitively, the data in preceding round are treated as stationary in distribution and are therefore well suited for calibration, whereas all earlier data---regardless of whether distribution shifts have occurred---can be leveraged for model training.

\item {\em Fitted models, scores, and prediction sets.} The construction of the prediction set follows the standard full conformal method described in (\ref{eq:mu-hat-Xy-algorithm})-(\ref{eq:full_conformal_set}). 
Consider any time point within round $r$ of stage $n$. 
For a feature $X$ observed in this round and an imputed response $y$, we invoke a learning algorithm $\mathcal{A}$ to fit a predictive model
\begin{align}
\widehat{\mu}^{(X,y)}(\cdot
) \coloneqq \gA\left(\mathcal{Z}_{n,r}^{\mathsf{train}} ; 
(X, y)\right).
\label{eq:mu-hat-Xy-algorithm-online}
\end{align}
The non-conformity (or residual) scores are then computed for all $T_{r-1}$ points observed in the immediately preceding round (i.e., $\mathcal{Z}_{n,r}^{\mathsf{cal}}$) as well as the hypothesized test point $(X,y)$, yielding
\begin{subequations}
\label{eq:res_score-online}
\begin{align}
s_{i}^{(X,y)} & \coloneqq\big|Y_{n,r-1,i}^{\mathsf{cal}}-\widehat{\mu}^{(X,y)}(X_{n,r-1,i}^{\mathsf{cal}})\big|,\qquad i=1,\ldots,T_{r-1},\label{eq:res_score-online-cal}\\
s_{\mathsf{test}}^{(X,y)} & \coloneqq\big|y-\widehat{\mu}^{(X,y)}(X)\big|.
\label{eq:res_score-online-test}
\end{align}
\end{subequations}
Given these scores, we form the prediction set based on feature $X$ as
\begin{equation}\label{eq:full_conformal_inf}
\gC_{n,r}(X 
) \coloneqq \Bigg\{y: s_{\mathsf{test}}^{(X,y )}  \le  
\mathsf{Quantile}_{1-\alpha}\Bigg(\frac{1}{T_{r-1}+1}\bigg[\delta\big(s_{\mathsf{test}}^{(X,y )}\big)+\ssum{i}{1}{T_{r-1}}\delta\big(s_{i}^{(X,y)}\big)\bigg]\Bigg)
\Bigg\},
\end{equation}
which collects all candidate responses $y$ for which the test residual $s_{\mathsf{test}}^{(X,y)}$ does not exceed the target quantile of the combined calibration and test scores, in direct analogy to the standard full conformal framework.   
The prediction-set construction strategy $\mathcal{C}_{n,r}(\cdot)$ remains fixed throughout the current round, since the same training and calibration sets are used for all time points within this round.
\end{itemize}

\begin{algorithm}[t]
\DontPrintSemicolon
\caption{\textsc{DriftDetect+}$\big(\gC; t_0,t_1; \sigma\big)$}\label{alg:full_drift-detection}
\textbf{input:} set-valued function $\gC(\cdot)$; time window $[t_0,t_1]\subseteq[T]$; detection threshold $\sigma$.\\
\For(\tcp*[f]{scan the entire time window.}){$j= t_0,\cdots,t_1$}{
    compute $Z_{j,t_1}\leftarrow \dfrac{\bigl|\mathsf{cvg}\text{-}\mathsf{err}_{\gC}(j,t_1)\bigr|}{\sqrt{t_1-j+1}}$ (cf.~(\ref{eq:defn-cvg-err-general})).\tcp*{construct detection statistics.} 
    \If{$Z_{j,t_1}>\sigma$}{\Return \textsf{true}.\tcp*{declare detection of drift once this statistic exceeds the threshold.}}
}
\Return \textsf{false}
\tcp*{no distribution drift has been detected.}
\end{algorithm}

\begin{algorithm}[t]
\DontPrintSemicolon
\caption{\textsc{ Online Full Conformal Prediction with Drift Detection} (\driftocpfull)}\label{alg:FOCID}
\textbf{input:} target coverage level \(1-\alpha\); detection thresholds \(\{\sigma_{n,r}\}_{n,r=1}^\infty\). \\
\textbf{initialize:} \(n\leftarrow 1\), \(r\leftarrow 1\), \(\tau\leftarrow 0\), \(\tau_{1,1}\leftarrow 1\), \(\gC_{1,1}(x) \leftarrow \RB;~ \forall x\in \gX\).
\;

\While{\textsf{true}}{
construct set-valued function $ \gC_{n,r}(\cdot)$ by \eqref{eq:full_conformal_inf}. \tcp*{prediction-set forming strategy for this round.}
  \For(\tcp*[f]{round $r$ contains at most $T_{r}=3^r$ time points.}){\(t=1,\cdots,T_{r}\;(\coloneqq 3^r)\)}{
    \(\tau\leftarrow \tau+1\).\tcp*{update global time index.}
    observe feature $X_{\tau}$; set \(X_{n,r,t}\leftarrow X_{\tau}\). \\ 
    construct  prediction set $\gC_{\tau} \leftarrow \gC_{n,r}(X_{\tau})$; take $\gC_{n,r,t}\leftarrow \gC_{\tau}$. \tcp*{form prediction set} 
    observe \(Y_{\tau}\); set $Y_{n,r,t}\leftarrow Y_{\tau}$. \tcp*{response is observed after the prediction set is formed.}
    \textsf{drift} $\leftarrow$ \textsc{DriftDetect+}$(\gC_{n,r};\tau_{n,r}, \tau_{n,r}+t-1;\sigma_{n,r})$. \tcp*{call \Cref{alg:full_drift-detection}.}
    \If{\textsf{drift} is \textsf{true}}{
      \(n\leftarrow n+1\), \(r\leftarrow 1\). \tcp*{update stage index and round index.}
      \textbf{break}.\tcp*[f]{enter next stage.}
    }
  }
  \If{\textsf{drift} is \textsf{false}}{
  \(r\leftarrow r+1,~ \tau_{n,r} \leftarrow \tau + 1\). \tcp*{start time of next round.} }
}
\end{algorithm}


\subsection{Theoretical guarantees under stability assumptions}\label{sec:full_conf_bound}


We now turn to the regret performance of the proposed Algorithm~\ref{alg:FOCID}. A dominant fraction of prior full conformal theory relies on a permutation symmetry assumption of the
model fitting algorithm, namely, that the fitted predictor $\widehat{\mu}$ remains invariant under arbitrary reordering of the training samples. However, many online learning algorithms, such as online gradient descent with time-varying learning rates, do not produce predictors that are exactly permutation invariant. Enforcing permutation symmetry in these cases would oftentimes require, at each time step, retraining the model from scratch on all previously observed data, thereby incurring a substantial computational burden. To better accommodate online model fitting algorithms, we instead rely on two different assumptions---one concerning the Lipschitz continuity of the conditional response distribution, and the other pertaining to the stability of the learning algorithm---replacing the permutation symmetry requirement. 
%
%
\begin{assumption}[Lipschitz continuity of conditional response distribution]\label{ass:lip_cond_distr}
    There exists a quantity $L_1 > 0$ such that, for every time $t \geq 1$ and every $x_t\in \mathcal{X}$,   the  function $g_{x_t}(z)\coloneqq \PB(Y_t \le z\mid X_t=x_t)$  is $L_1$-Lipschitz continuous w.r.t.~$z$.  
\end{assumption}

\begin{assumption}[Stability of learning algorithm]\label{ass:fair_alg}
%
%
Let $\mathcal{Z}=\{z_1,\cdots,z_m\}$ be a training set of size $m$, and let $\widehat{\mu}(\cdot\,|\, \mathcal{Z})$ represent the predictive model returned by algorithm $\mathcal{A}$ when trained on $\mathcal{Z}$. We assume that $\widehat{\mu}(\cdot\,|\, \cdot)$ is a measurable function. 
For any $i\in[m]$ and any replacement sample $w$, define $\mathcal{Z}_{i,w}=\{z_1,\cdots,z_{i-1},w,z_{i+1},\cdots,z_m\}$, which differs from $\mathcal{Z}$ only in its $i$-th element. We assume that there exists a constant $L_2>0$ such that, for an arbitrary $m$, one has
        \begin{align}
        \abs{\widehat{\mu}\big(x\mid \gZ\big) - \widehat{\mu}\big(x\mid \gZ_{i,w}\big)} \le \frac{L_2}{m}\qquad \text{for all } x, w, \mathcal{Z},  \text{ and }i\in [m].
        \end{align}
%
\end{assumption}
Assumptions~\ref{ass:lip_cond_distr} and \ref{ass:fair_alg} are commonly used in full conformal prediction literature (e.g., \citet{barber2021predictive,ndiaye2022stable,steinberger2023conditional,liang2025algorithmic,leeleave}).
In fact, Assumption~\ref{ass:lip_cond_distr} is fairly standard in statistical modeling; a common example concerns the setting $Y_t=m_t(X_t)+\varepsilon_t$, where $\varepsilon_t$ is generated independently of $X_t$ and admits a density uniformly bounded above by $L_1$.
In addition, Assumption~\ref{ass:fair_alg} formalizes a sort of stability requirement of $\widehat{\mu}$: perturbing a single training
example alters the predictive output by at most $O(1/m)$ (assuming a constant $L_2$). To help illustrate the practical relevance of Assumption~\ref{ass:fair_alg}, we single out a few canonical parametric examples that can be readily analyzed within our framework: 
\begin{itemize}
    \item {\em constrained M-estimation}: see \Cref{sec:constrained-M-estimation};
    
    \item {\em linear stochastic approximation}: see \Cref{sec:linear-stoc-approx}; 
    
    \item {\em stochastic strongly convex optimization}: see \Cref{sec:stochastic_scvx_opt}.
\end{itemize}
The interested reader is referred to \Cref{sec:exp_for_full_conf} for detailed verification of Assumption~\ref{ass:fair_alg} in these examples.



Armed with the above assumptions, we establish regret upper bounds for the proposed 
online full conformal algorithm.
\begin{theorem}\label{thm:full_regret}
Suppose that Assumption~\ref{assump:full} holds, and that Assumptions~\ref{ass:lip_cond_distr} and \ref{ass:fair_alg} hold with quantities $L_1$ and $L_2$, respectively.  
Let $L=L_1L_2$. 
If we set the drift detection thresholds as $\sigma_{n,r} \coloneqq 10\log^3(40\tau_{n,r})$ for every stage-round index pair $(n,r)$,  then Algorithm~\ref{alg:FOCID} achieves
\begin{align}
\label{eq:regret-full-ub}
    \mathsf{regret}_T \leq  
    \begin{cases}
        \widetilde{O}\big(\sqrt{(N^{\mathsf{cp}}+L+1)T}\,\big) 
        & \text{for the change-point setting}; \\
\widetilde{O}\big(\sqrt{(L+1)T} + (\mathsf{TV}_T)^{\frac{1}{3}}T^{\frac{2}{3}}\big)
& \text{for the smooth drift setting}.
    \end{cases}
\end{align}
\end{theorem}
Despite the adaptive, online training of the non-conformity score functions, the training-conditional regret attained by our online full conformal prediction algorithm takes a form similar to that achieved with pretrained scores, provided that $L=O(1)$. A main difference is that the score-based Kolmogorov–Smirnov distance appearing in the pretrained-score scenario (see \Cref{thm:split_regret}) is replaced here by the total-variation distance w.r.t.~data distributions, since the scores are now trained based on the observed data.  Our result is fully non-asymptotic, which stands in stark contrast to several prior works (e.g., \citet{angelopoulos2024online}) that focused on asymptotic coverage guarantees (i.e., $T\rightarrow \infty$ with other parameters held fixed).



%

\paragraph{Byproduct:  training-conditional coverage for batch full conformal methods.}
En route to establishing the regret upper bound of \driftocpfull, 
we need to address the challenge of achieving training-conditional 
coverage when scores are trained in-sample using a possibly non-symmetric 
learning algorithm. Our analysis leads to new  
training-conditional coverage results for batch full conformal methods, 
which is stated below and may be of independent interest.
The proof is deferred to 
\Cref{sec:prf:prop:training_conditioned_cov}.

\begin{proposition}\label{prop:training_conditioned_cov}
Consider any integers \(n\ge m\). Let \(Z_{1:m}^{\mathsf{cal}}\) be a calibration dataset and \(Z_{1:n}^{\mathsf{train}}\) a dataset used for model fitting. We assume that the calibration dataset is a subset of the training dataset, and in particular, \(Z_{1:m}^{\mathsf{cal}}=Z_{1:m}^{\mathsf{train}}\).
The samples in $\{Z_{1:m}^{\mathsf{cal}}\}\cup \{Z_{m+1\,:\,n}^{\mathsf{train}}\}$ are independently generated.
Construct the full conformal prediction set \(\gC(\cdot) = \gC(\,\cdot\,\mid Z_{1:m}^{\mathsf{cal}};Z_{1:n}^{\mathsf{train}})\) as in Eqn.~\eqref{eq:full_conformal_set}.
Consider a target pair \(Z = (X,Y)\sim \gD\). Suppose the distribution \(\gD\) and the fitted model 
satisfy Assumptions~\ref{ass:lip_cond_distr} and~\ref{ass:fair_alg} with coefficients $L_1$ and $L_2$, respectively, and denote $L_1L_2$ as ${L}$. Then for any $\delta \in (0,1)$, conditional on any realization \(Z_{m+1:n}^{\mathsf{train}}=z_{m+1:n}^{\mathsf{train}}\), we have
\begin{equation}\label{eq:prop_fc_tc_cov_bound}
\begin{aligned}
\abs{\PB_{\gD}
\left(Y\in \gC\big(X\big) \,\big|\, Z_{1:m}^{\mathsf{cal}}\right) -(1 - \alpha)
} &\le \frac{52 {L}\sqrt{m\log (45n/\delta)}}{n} + 25\sqrt{\frac{\log (40/\delta)}{m}} + \frac{2}{m}\ssum{l}{1}{m}\mathsf{TV}(Z, Z_l^{\mathsf{cal}})
\end{aligned}
\end{equation}
with probability exceeding \(1-\delta\)  (with respect to the randomness only in
\(Z_{1:m}^{\mathsf{cal}}\)). 
\end{proposition}
\begin{remark}
In particular, in the most common case where the training and calibration sets coincide (so that $m=n$ and $Z_{1:n}^{\mathsf{cal}}=Z_{1:n}^{\mathsf{train}}$), this result asserts that the standard full conformal method (cf.~\eqref{eq:full_conformal_set}) achieves  
\begin{equation}\label{eq:prop_fc_tc_cov_bound-special}
\begin{aligned}
\abs{\PB_{\gD}
\left(Y\in \gC\big(X\big) \,\big|\, Z_{1:n}^{\mathsf{train}}\right) -(1 - \alpha)
} &\lesssim \max\{L,1\}\sqrt{\frac{\log (n/\delta)}{n}} 
+ \frac{1}{n}\ssum{l}{1}{n}\mathsf{TV}(Z, Z_l^{\mathsf{cal}})
\end{aligned}
\end{equation}
with probability greater than $1-\delta$. 
\end{remark}

Proposition~\ref{prop:training_conditioned_cov} establishes a
training–conditional concentration bound for full conformal residuals that
holds for a fixed batch of data and a stable learner (no online structure is used).  This result captures the effect of using a
data-dependent predictor inside the full conformal construction, and will play a crucial role in establishing our training-conditional regret bound (when combined with the stage/round decomposition and drift–detection analysis).
In addition, Proposition~\ref{prop:training_conditioned_cov}
generalizes existing results on training-conditional 
coverage for full-conformal-type approach; more detailed comparisons with prior results are provided in~\Cref{sec:full-conf-lit}.


\subsection{Minimax lower bound}\label{sec:model_free_lower_bd}

We now complement Theorem~\ref{thm:full_regret} with a lower bound, which serves to better evaluate the optimality of our proposed procedure. Before proceeding, it is important to note that, while Theorem~\ref{thm:lower_bd_split_regret} already establishes a regret lower bound, that result hinges upon a specific way of constructing the prediction set—namely, one based on quantile estimation of pretrained non-conformity scores. 
In practice, however, a broader class of methods is available, including the online full conformal approach, which can induce substantially more complex and structurally different prediction sets. 
As a result, Theorem~\ref{thm:lower_bd_split_regret} does not provide an appropriate lower bound for the settings considered in this section. 
We develop a new lower bound for this broader class of algorithms below.

\paragraph{Lower bound.} 
%
We start by specifying the scope of the problem.
\begin{itemize}
\item {\em Admissible algorithms.}
Since the prediction-set construction considered in this section no longer relies on a given set of non-conformity score functions, the first step is to redefine the class of admissible algorithms accordingly.
Denote by $\mathsf{Map}(\gX,\gB(\RB))$ the set of mappings from $\gX$ to $\gB(\RB)$ (i.e., this forms the set of prediction-set construction functions).   
Let \(U\sim \mathrm{Unif}(0,1)\) be a random variable independent of the data stream. 
Let \(\pi_{1}:[0,1]\to \mathsf{Map}(\gX,\gB(\RB))\) and, for \(t\geq 2\), let
\(\pi_{t}:(\gX, \RB)^{t-1}\times[0,1]\to \mathsf{Map}(\gX,\gB(\RB))\).
Given a sequence of  data $\{Z_i\}_{i=1}^{t-1} = \{(X_i,Y_i)\}_{i=1}^{t-1}$ prior to time $t$, we define  \(\gC_t\in \mathsf{Map}(\gX,\gB(\RB))\) to be the set-valued mapping induced by algorithm $\pi_t$ at time \(t\), namely, 
\begin{equation}\label{eq:plcy_construct_C}
\gC_t(\cdot) =
\begin{cases}
\pi_1(U), & \text{if }t=1,\\[4pt]
\pi_t\!\left(Z_{t-1},\ldots,Z_{1},\,U\right), & \text{if } t\geq 2.
\end{cases}
\end{equation}
The collection of mappings $\pi=\{\pi_t\}_{t\geq 1}$ that generate such set-valued functions \(\{\gC_t(\cdot): t=1,2,\cdots\}\) constitutes a class of algorithms denoted by  \(\mathcal{P}\). 
Moreover, we restrict attention to a structured subclass of $\mathcal{P}$ in which each prediction set is expressible as a finite union of intervals.


\begin{definition}[$K$-interval procedure]\label{def:Pk}
For every integer \(K\ge 1\), define the $K$-interval algorithm class \(\gP_K\subseteq \gP\) as
\begin{align}
\label{eq:defn-PK-interval}
\gP_K \coloneqq \left\{\pi\in \gP\mid \text{for all } t\in [T] \text{ and } x\in \gX:~ \gC_t(x)\text{ is the union of at most \(K\) intervals}\right\}.
\end{align}
\end{definition}
We shall discuss the practical relevance of this algorithm subclass momentarily. 
%

\item {\em Distribution class.} Analogous to the score-based distribution class $\mathcal{L}_1(N^{\mathsf{cp}})$ (see~(\ref{eq:N_change_dist})) that pertains to the change-point setting, we introduce a closely related distribution class---defined directly in terms of the data distributions---that permits at most $N^{\mathsf{cp}}$ change points:
\begin{subequations}
\begin{align}
\mathcal{L}_3(N^{\mathsf{cp}})
&\coloneqq 
\Bigl\{
        (\mathcal{D}_1,\ldots,\mathcal{D}_T):\;
        (\gD_1,\ldots,\gD_T)\text{ change at most $N^{\mathsf{cp}}$ times.}
\Bigr\};\label{eq:N_change_dist_new}
\end{align}
Additionally, we define another TV-based distribution class concerning the smooth drift setting: for a given budget \(\mathsf{TV}_T>0\), define
\begin{align}
\gL_4(\mathsf{TV}_T)
\;\coloneqq\;
\left\{
(\gD_1,\ldots,\gD_T):\;
\sum_{t=1}^{T-1}\mathsf{TV}(\gD_t,\gD_{t+1}) \le \mathsf{TV}_T
\right\}.
\label{eq:TV_dist}
\end{align}

\end{subequations}
\end{itemize}

\noindent 

Moreover, for a distribution class $\mathcal{L}$, the worst-case regret of algorithm $\pi \in \mathcal{P}_K$ is defined as
\begin{align}
\label{eq:def-worstcase-regret-K}
\mathsf{regret}_\pi(\gL,T,K)
&\coloneqq
\sup_{(\mathcal{D}_1,\ldots,\mathcal{D}_T)\in\mathcal{L}}
\mathsf{regret}_\pi(\mathcal{D}_{1:T},T)\notag\\
&=
\sup_{(\mathcal{D}_1,\ldots,\mathcal{D}_T)\in\mathcal{L}}
\sum_{t=1}^T \EB\Big[\bigl|\, 
\PB_{\gD_t}\big(Y_t \in \gC_t(X_t)\mid \gC_t(\cdot)\big) - (1-\alpha)
\,\bigr|\Big],
\end{align}
where we make explicit the dependency on $K$. 
%
We can now present our minimax lower bound that accommodates online conformal prediction with adaptive training.

%



\begin{theorem}\label{thm:lower_bd_gnr_regret}
Consider any fixed constant $\alpha\in (0,1/2]$. Suppose that Assumption~\ref{assump:full} holds.  
    For any admissible algorithm $\pi\in\gP_K$, the worst-case regret under $\pi$ has the following lower bound:
    \begin{align*}
    \mathsf{regret}_{\pi}\big(\gL_3(N^{\mathsf{cp}}),T,K\big) &= \widetilde{\Omega}\left(\min\left\{\sqrt{(N^{\mathsf{cp}}+1)T}, {\frac{T}{\sqrt{K}}}\right\}\right);\\
    \mathsf{regret}_{\pi}\big(\gL_4(\mathsf{TV}_T),T,K\big) &= \widetilde{\Omega}\left(\min\left\{\sqrt{T} +  (\mathsf{TV}_T)^{\frac{1}{3}}T^{\frac{2}{3}}K^{-\frac{1}{6}}, {\frac{T}{\sqrt{K}}}\right\} \right).
    \end{align*}
\end{theorem}
The proof of Theorem~\ref{thm:lower_bd_gnr_regret} is deferred to \Cref{sec:prf:thm:lower_bd_gnr_regret}. Clearly, when $K$ is a finite constant,  the regret bound in (\ref{eq:regret-full-ub}) matches this minimax lower bound up to a logarithmic factor, provided  $L=O(1)$ (meaning that the learning algorithm is stable and the conditional response distribution is smooth).

\paragraph{Why restricted to \(\gP_K\)?}
We now elucidate the rationale for restricting attention to the algorithm subclass $\gP_K$. In brief, imposing structural constraints on the algorithm class is necessary to formulate a meaningful minimax problem. 
Without such restrictions, one can design ``irregular'' procedures
that achieve asymptotically perfect  marginal coverage while using essentially no information about
the data-generating process. 
To illustrate this point, suppose \(Y\in[0,1]\).  For each \(n\ge 1\), define 
\begin{equation}\label{eq:pathological_Cn}
\gC_n \coloneqq \bigcup_{i=0}^{n-1}\left[\frac{i}{n},\frac{i + (1-\alpha)}{n}\right],
\end{equation}
obtained by partitioning \([0,1]\) into \(n\) equal subintervals and retaining the same \((1-\alpha)\)-fraction of
each subinterval.  For sufficiently regular distributions,  \(\gC_n\) captures roughly a \((1-\alpha)\)-fraction of the total probability mass, largely independent of the actual shape of the density.
The following proposition formalizes this observation.

\begin{proposition}\label{prop:pathological_cover}
Let \(\{\gC_n\}_{n\ge 1}\) be defined by Eqn.~\eqref{eq:pathological_Cn}.  If the distribution \(\gD\) of \(Y\) on \([0,1]\) admits
a Riemann-integrable density, then
\[
\lim_{n\to\infty}\Bigl|\PB(Y\in \gC_n) - (1-\alpha)\Bigr| = 0.
\]
\end{proposition}

This example shows that, in the absence of geometric constraints, marginal coverage alone does not preclude vacuous  
procedures.  Restricting attention to \(\gP_K\), where each prediction interview is a union of at most \(K\) intervals, excludes such uninformative construction and yields a more meaningful lower bound.

\paragraph{Implications beyond the online setting.} 
Although Theorem~\ref{thm:lower_bd_gnr_regret} is stated for the online setting, its proof proceeds by first establishing a
lower bound on the per-round contribution to the cumulative regret, and then constructing a distribution
sequence that allocates the available distribution drift budget in a way that realizes these per-round bottlenecks. 
As a byproduct, our arguments readily yield an \emph{offline} lower bound for training-conditional coverage error
over the algorithm class \(\gP_K\) (see Definition~\ref{def:Pk}).  We record this consequence below, which may be of independent interest.

\begin{proposition}\label{cor:offline_tc_lb}
Fix any $\alpha\in(0,1/2]$. Let $\gS$ be the collection of distributions on $\gX\times\R$ that admit a density. Let $\{(X_i,Y_i)\}_{i=1}^n$ be i.i.d.\ draws from some $\gD\in\gS$, and let $U\sim\mathrm{Unif}(0,1)$ be independent of the data.
Consider any algorithm $\pi$ that maps $\{(X_i,Y_i)\}_{i=1}^n$ and $U$ to a set-valued function
$\widehat{\gC}(\cdot):\gX\to\gB(\RB)$ such that, for each $x\in\gX$, the set $\widehat{\gC}(x)$ is a union of at most $K$ intervals. Then we have
\[
\sup_{\gD\in\gS}
\E\!\left[
\bigg|
\mathop{\PB}\limits_{(X,Y)\sim \gD}\!\left(
Y \in \widehat{\gC}(X)\,\middle|\,\{(X_i,Y_i)\}_{i=1}^n;U
\right) - (1-\alpha)
\bigg|
\right]
\;=\;
\widetilde{\Omega}\!\left(
\min\left\{\frac{1}{\sqrt{K}},\,\frac{1}{\sqrt{n}}\right\}
\right),
\]
where the outer expectation is taken over the training sample \(\{(X_i,Y_i)\}_{i=1}^n\) and the internal
randomization \(U\).
\end{proposition}

Proposition~\ref{cor:offline_tc_lb}---which is a direct consequence of Lemma~\ref{lem:coverage_gap_lower_bd} given in \Cref{sec:prf:thm:lower_bd_gnr_regret}---is independent of the online setting and provides a lower
bound for training-conditioned validity of full conformal prediction in the offline regime.
The result places no parametric restriction on the prediction set \(\widehat{\gC}(\cdot)\), 
and is therefore fundamentally different from those information-theoretic lower bounds in classical parametric estimation problems. The bound in Proposition~\ref{cor:offline_tc_lb} holds for a fixed algorithm 
and considers the worst case over a class of distributions, complementing the result of~\citet{bian2023training},
which instead fixes the distribution and takes the worst case over a class of algorithms.
Moreover, relative to prior work, our bound explicitly characterizes the learning limit
in terms of the structural complexity \(K\) of the prediction sets (i.e., the number of intervals). 
Determining the optimal $K$-dependence in training-conditional lower
bounds remains an interesting open direction, which we leave for future work.
We view this lower bound as a baseline that may
be useful more broadly in the study of conformal inference beyond the online setting.


\subsection{Comparisons with prior art}
\label{sec:full-conf-lit}
\paragraph{Existing training-conditional guarantees.}
Prior literature has established 
training-conditional coverage guarantees 
for split conformal methods~\citep{vovk2012conditional}. 
More recently,~\citet{bian2023training} showed that such 
guarantees can be achieved in a distribution-free manner 
by $K$-fold CV+ when the sample size is sufficiently large 
relative to the number of folds, but not achievable by full conformal 
methods or jackknife+. Training-conditional coverage 
guarantees for full conformal  methods and jackknife+ 
have instead largely been obtained under stability-type
assumptions; see, e.g., \citet{liang2025algorithmic,amann2023assumption} and \citet{pournaderi2024training}. Proposition~\ref{prop:training_conditioned_cov}
strengthens these results in three complementary ways.

First, \citet{liang2025algorithmic} expressed their bounds through an in-sample $m$-stability metric \(\beta_{m,n}^{\mathsf{in}}\) (see  Definition 3.4 therein), which measures the sensitivity of prediction when a training set of size $n$ is augmented by $m$ additional samples. For full conformal methods, \citet[Theorem~3.5]{liang2025algorithmic} showed that the training-conditional coverage error is bounded by
\(
O\Big(\sqrt{\frac{\log(1/\delta)}{\min\{m,n\}}}+\big(\frac{\beta_{m-1,n+1}^{\mathsf{in}}}{\gamma}\big)^{1/3}\Big),
\)
with $\gamma$ some inflation parameter. Achieving vanishing coverage error then requires both \(m\to\infty\) and
\(\beta^{\mathsf{in}}_{m-1,n+1}\to 0\), effectively demanding the fitted predictor \(\widehat\mu(\cdot)\) to stabilize as more
data arrive. As noted by \citet{amann2023assumption},  stabilization of this type can fail under distribution shift,
where \(\widehat\mu(\cdot)\) typically adapts to evolving data distributions. In contrast,
Proposition~\ref{prop:training_conditioned_cov} allows $\widehat{\mu}(\cdot)$ to evolve and remains applicable in drifting scenarios.

Second, \citet{amann2023assumption} derived training-conditional \emph{under}-coverage bounds for cross validation (CV), Jackknife, CV+ and Jackknife+ under stability assumptions, and further investigated the necessity of stability conditions. However, their results do not accommodate full conformal methods.


Third, \citet[Theorem~6]{pournaderi2024training} provided a training-conditional coverage guarantee for full conformal prediction under covariate shift. Their result, however, does not address shifts in the conditional distribution \(Y\mid X\). Moreover, relative to our \Cref{prop:training_conditioned_cov}, their analysis relies on additional structural assumptions,
including a uniform upper bound on the train--test density ratio and a parametric model for the fitted predictor with
a bi-Lipschitz dependence on its parameters.


\paragraph{Verification of stability conditions.}
Stability analyses for both empirical loss minimizers and stochastic optimization methods have been developed in, e.g., \citet{barber2021predictive}, 
\citet{ndiaye2022stable} and \citet{leeleave}. Compared to our work (mainly our results in \Cref{sec:exp_for_full_conf}), these prior results typically differ in several
important respects. First, they did not accommodate constrained optimization problems. Second, they often assume that,
for every data realization \(z\), the loss \(\ell(\cdot,z)\) is uniformly strongly convex, which most naturally holds for
explicitly regularized ERM objectives; in contrast, our verification only requires strong convexity of the
population risk \(L\) on \(\gC\). Third, while \citet{leeleave} (see their Example~2) examined stochastic
optimization methods, their analysis is essentially offline---the algorithm is rerun multiple times on a fixed dataset
under different permutations---and is derived for fixed stepsizes, which does not yield a stability coefficient that
vanishes with the sample size. Such vanishing stability is crucial in our online full conformal analysis in order to obtain
training-conditional coverage guarantees.

\section{Numerical experiments}\label{sec:experi}
\subsection{Experiments: online conformal prediction with pretrained scores}
In this subsection, we evaluate the performance of the proposed \driftocp algorithm against the ACI method under various distribution shift scenarios, assuming the presence of pretrained score functions. The experimental setup is described below.

\paragraph{Data generation.}
Consider a regression setting with a data stream $\{(X_t,Y_t)\}_{t\ge 1}$. The feature vector
$X_t$ is in $\mathbb{R}^d$ with $d = 5$,  and each component is generated by $X_{t,j} \overset{\text{i.i.d.}}{\sim} \mathcal{N}(0, 1)$. The response variable satisfies
\begin{equation*}
    Y_t = 2X_{t,1} + X_{t,2} + \mu_t + \sigma_t \cdot \varepsilon_t, \quad \varepsilon_t \overset{\text{i.i.d.}}{\sim} \mathcal{N}(0, 1),
\end{equation*}
where $\mu_t$ and $\sigma_t$ are varying parameters.
We examine four distribution shift cases as follows. 

\begin{itemize}
    \item \textbf{Setting 1 (piecewise variance shift):} $\mu_t = 0$ and
    \begin{equation*}
        \sigma_t = \begin{cases}
            0.5, & \text{if }t < 4000, \\
            2.0, & \text{if }4000 \leq t < 7000, \\
            3.5, & \text{if }t \geq 7000.
        \end{cases}
    \end{equation*}
    This setting simulates abrupt changes in noise level, representing sudden regime shifts.
    
    \item \textbf{Setting 2 (linear bias drift):} $\sigma_t = 0.5$ and $\mu_t = \kappa \cdot t$ with $\kappa = 0.002$, so that $\mu_T = 20$ at $T = 10000$. This represents smooth temporal drift in the conditional mean.
    
    \item \textbf{Setting 3 (smooth variance growth):} $\mu_t = 0$ and
    $\sigma_t = \sqrt{1 + 0.008t}$. This model continuously increases variability over time.
    
    \item \textbf{Setting 4 (no distribution drift):} $\mu_t = 0$ and $\sigma_t = 0.5$ for all $t$. This serves as a baseline where no distribution shift occurs.
\end{itemize}

\paragraph{Experimental protocol.}
For each setting, we use a training set of size $n = 500$ drawn from the initial distribution ($t = 0$) to fit a random forest regressor \citep{breiman2001random} with 100 trees, implemented in \textsf{scikit-learn} \citep{pedregosa2011scikit}, as the pre-trained prediction model.
The pretrained predictive model $\widehat{\mu}$ remains fixed throughout the online prediction phase and is not updated.

The non-conformity score at each time step $t$ is taken to be the absolute residual between the observed and predictive responses, namely, $s_t \coloneqq |Y_t - \widehat{\mu}(X_t)|$.
The initial quantile $\widehat{q}_0$ is set to be the $(1-\alpha)$-th empirical quantile of the training residuals $\{|Y_i - \widehat{\mu}(X_i)|\}_{i=1}^{n}$.

The test horizon is $T = 10{,}000$ time steps. We set the target miscoverage level to be $\alpha = 0.1$. All experiments are repeated 40 times with different random seeds, and we report the mean and standard deviation of cumulative regret.

\paragraph{Numerical evaluation of cumulative regret.}
We measure performance using the cumulative regret defined in Eqn.~\eqref{eq:def_regret}. However, since the true miscoverage probability $\mathbb{P}(s_t > {q}_t \mid {q}_t)$ is intractable, we estimate it via Monte Carlo simulation. Specifically, for each time step $t$, we pre-generate a fixed evaluation set of $M = 500$ independent samples $\{(X_t^{(m)}, Y_t^{(m)})\}_{m=1}^{M}$ from the true distribution $\gD_t$ at that time step. The instantaneous coverage rate (defined in Eqn.~\eqref{eq:def_Cvg}) is then estimated as
\begin{equation*}
    \widehat{\mathsf{cvg}}_t = \frac{1}{M} \sum_{m=1}^{M} \mathbbm{1}\left\{ |Y_t^{(m)} - \widehat{\mu}(X_t^{(m)})| \le q_t \right\}.
\end{equation*}
The cumulative regret up to time $T$ is then calculated as
\begin{equation*}
    \widehat{\mathsf{Regret}}_T = \sum_{t=1}^{T} \big| \widehat{\mathsf{cvg}}_t - (1-\alpha) \big|.
\end{equation*}

\paragraph{Methods for comparison.}
We compare the following algorithms numerically. 
\begin{itemize}
    \item \textbf{\driftocp (\Cref{alg:OCID}):} We use a drift detection threshold in \Cref{alg:OCID} of $\sigma_{n,r} = 4$. To avoid false positives from high-variance estimates, we require a minimum window size of $t \geq 10$ before any drift detection can be declared.
    \item \textbf{ACI with decaying stepsizes \citep{angelopoulos2024online}:} $\eta_t = (t + 1)^{-\gamma}$ for $\gamma \in \{0.5, 0.6\}$.
    \item \textbf{ACI with fixed stepsizes \citep{gibbs2021adaptive}:} $\eta_t\equiv \eta \in \{0.01, 0.1, 0.5\}$.
\end{itemize}

\paragraph{Results.}
Figure~\ref{fig:driftocp_vs_aci} summarizes both the regret and calibration dynamics across the four data-generating settings. The top row plots the cumulative regret over time, while the bottom row tracks the corresponding evolution of the calibration quantiles; the black dashed curve in the bottom row is an approximation of the ground-truth quantile, obtained via repeated simulations at each time point. Taken together, the two rows highlight the tuning trade-off of ACI: a large constant stepsize reacts quickly to distributional changes, but yields highly variable quantile trajectories even under stationarity, leading to substantial cumulative regret; conversely, smaller or decaying stepsizes stabilize the quantile updates in stationary periods, yet may adapt too slowly after distribution shifts and consequently lag behind the moving target (most notably in Setting~1). 
As a result, the optimal stepsize for ACI differs across various settings, 
making it difficult to select a priori in practice. 
In contrast, \driftocp adapts to different regimes in a data-driven manner, 
achieving stable tracking during stationary segments and rapid re-alignment following change points. 
This behavior leads to uniformly controlled regret across regimes, comparable to that attained by hindsight-optimal tuning of ACI.
%


\begin{figure}
    \centering
    \includegraphics[width=\linewidth]{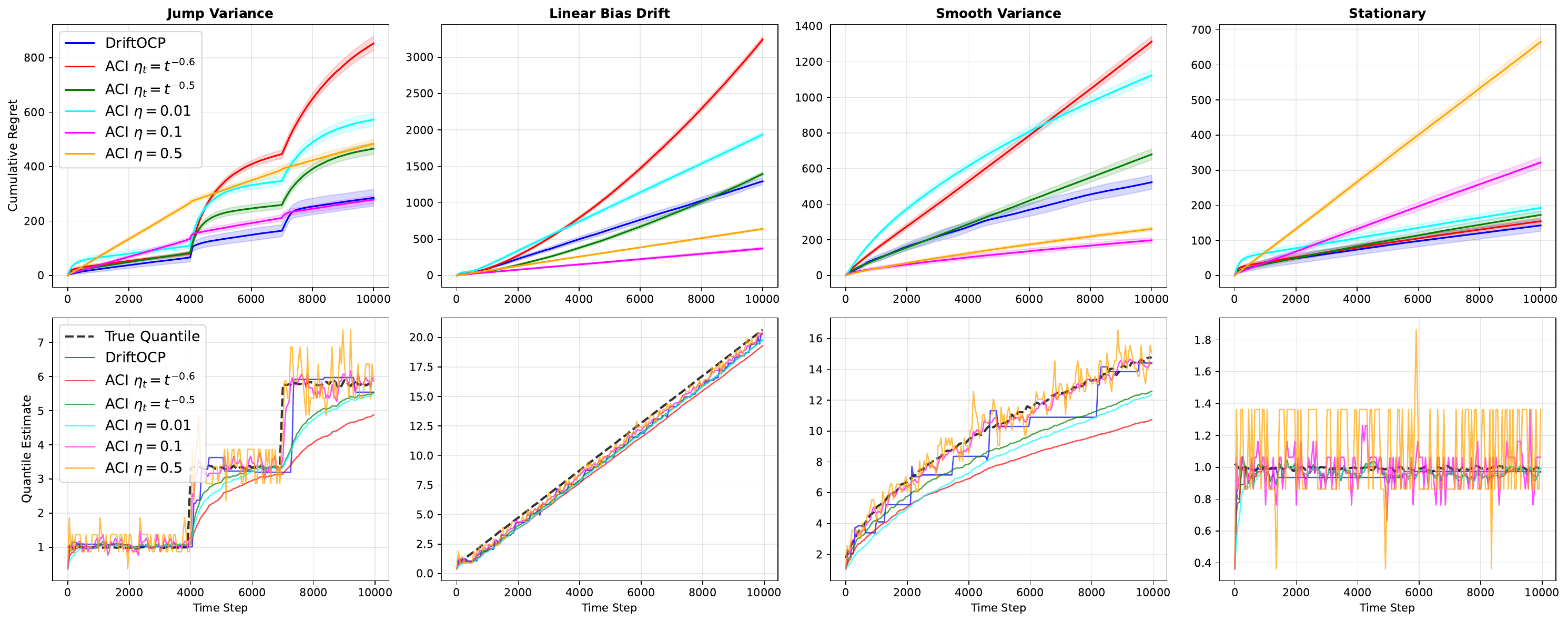}
    \caption{\textbf{Cumulative regret and calibration quantiles under four data-generating settings.}
    Top row: cumulative regret trajectories. Bottom row: calibration-quantile evolution; the black dashed curve indicates an approximation to the ground-truth quantile obtained via repeated simulations at each time point. Across settings, ACI exhibits a clear stepsize trade-off: ACI with large constant stepsizes adapts quickly but produces volatile quantile updates and suboptimal performance under stationarity, whereas ACI with smaller or decaying stepsizes yields more stable updates at the cost of slower adaptation to distributional changes. In comparison, \driftocp is stable within stationary time segments and adapts rapidly to distribution shifts, yielding consistently controlled regret. Curves are averaged over 20 runs; shaded bands indicate \(\pm 1\) standard deviation.
    }
    \label{fig:driftocp_vs_aci}
\end{figure}

\subsection{Experiments: online conformal prediction with adaptively trained scores}

This section examines how our methods interact with different ways of generating non-conformity scores in the presence of distribution drift. 
In particular, we pair our drift-aware recalibration mechanism with (i) a
covariate-agnostic score, (ii) a fixed pretrained model, and (iii) an adaptively updated fitted model. 
The numerical comparisons illustrate tangible gains from adaptively updated models
in terms of predictive efficiency and validity.

\paragraph{Data generation.}
We consider an online regression stream \(\{(X_t,Y_t)\}_{t\ge 1}\) 
with feature dimension \(d=10\).  
Throughout these experiments, we construct prediction sets using a \emph{linear} predictor, and the non-conformity score function is computed from linear regression residuals. We then consider two data-generating models for \((X_t,Y_t)\) to distinguish well-specified learning from misspecification:
\begin{itemize}
\item \textbf{Well-specified case:} \(Y_t=X_t^\top\beta^*+\varepsilon_t\), with \(\varepsilon_t\sim\mathcal N(0,1)\), so the linear predictor can be correctly specified;
\item \textbf{Misspecified case:} \(Y_t=X_t^\top\beta^*+\frac{1}{100}\|X_t\|_2^2+\varepsilon_t\), with
\(\varepsilon_t\sim\mathcal N(0,1)\), where the additional quadratic term introduces a mild deviation from linearity.
\end{itemize}
In both cases, the true coefficient \(\beta^*\in\R^d\) is sampled
from \(\mathcal N(0,I_d)\) and is subsequently held fixed across simulation repetitions.
For each model specification, we introduce piecewise-stationary covariate shifts with change points at
\(t\in\{3333,6667\}\). Specifically, we consider:
\begin{itemize}
\item \textbf{Mean shifts:} \(X_t\sim\mathcal N(\mu_t\mathbf 1_d, I_d)\), where
\[
\mu_t=\begin{cases}
0, & \text{if }t\le 3333;\\
3, & \text{if }3333<t\le 6667;\\
-2,& \text{if }t>6667;
\end{cases}
\]
\item \textbf{Variance shifts:} \(X_t\sim\mathcal N(0,\sigma_t^2 I_d)\), where
\[
\sigma_t=\begin{cases}
1, & \text{if }t\le 3333;\\
5, & \text{if }3333<t\le 6667;\\
10,& \text{if }t>6667.
\end{cases}
\]
\end{itemize}
In each run of the experiments, we first draw \(n_{\text{pretrain}}=100\) independent observations to fit an initial 
ridge regression model, and then draw an additional \(n_{\text{train}}=500\) observations 
to initialize the calibration quantile for
the non-conformity scores. We subsequently generate an online data stream of length \(T=10{,}000\).
\paragraph{Score construction strategies.}
We compare three score construction strategies that are paired with 
drift-aware recalibration. The first uses our full-conformal
variant tailored to online optimization, while the latter two 
use \driftocp with pretrained score functions:
\begin{itemize}
\item \textbf{\driftocpfull + online SGD:} the score is formed using a sequentially updated fitted model,
\(s_t=|Y_t-X_t^\top\widehat\beta_t|\), where \(\widehat\beta_t\) is updated by online SGD with stepsize
\(\eta_t=0.01/\sqrt{t}\);
\item \textbf{\driftocp + pretrained ridge:} the score uses a fixed ridge-regression predictor \citep{hoerl1970ridge} with regularization parameter $\lambda = 1.0$, implemented in \textsf{scikit-learn} \citep{pedregosa2011scikit}, as the pretrained model;
\item \textbf{\driftocp + absolute response:} a covariate-agnostic baseline \(s_t=\abs{Y_t}\).
\end{itemize}
All methods employ the same drift-detection threshold \(\sigma=4\) and the same doubling-round structure.
The target miscoverage level is $\alpha = 0.1$ for all settings.

\paragraph{Results.}
We report the prediction interval width and the local coverage rate computed over a sliding window of 100 time
steps. As shown in~\Cref{fig:adaptive_vs_pretrain},
the results are averaged over 20 independent runs, 
with shaded regions indicating \(\pm 1\) standard deviation. 
The top row plots the prediction interval width over time, 
while the bottom row shows the local coverage rate computed using a rolling window of 100 steps. The horizontal dashed line marks the target level 
\(1-\alpha = 0.9\) and the vertical dashed lines correspond to 
the change points.

Across all four settings---well-specified or misspecified models, and under either mean or variance drift---the adaptive-score method with online SGD consistently achieves the most favorable tradeoff, producing the narrowest intervals while maintaining stable coverage around the target level. 
In contrast, the pretrained-score baseline tends to be sensitive to mismatches between the pretraining and test covariate distributions, 
resulting in wider intervals and a higher degree of coverage fluctuations after the change points. The model-free baseline (\(s_t=\abs{Y_t}\)) is in general conservative and 
produces substantially wider intervals than the other two methods throughout.
It is also sensitive to distribution shift, exhibiting undercoverage at change points.


We also observe a transient effect at the beginning of the data stream: the adaptive method exhibits slightly biased local
coverage and inflated widths early on, which is as expected since the fitted model is still in its initialization phase and the online updates are relatively volatile. As more data arrive, the adaptively fitted model stabilizes, and the
resulting score becomes better calibrated, after which the method tracks the target coverage tightly even after
distributional shifts.

Finally, note that the adaptive method relies on an online SGD-trained predictor, whose trajectory depends on the data
order and thus does not satisfy permutation symmetry. The strong empirical performance of this non-symmetric learning
pipeline provides additional evidence supporting our training-conditional guarantees, which do not require symmetry of
the underlying fitted model.
\begin{figure}
    \centering
    \includegraphics[width=\linewidth]{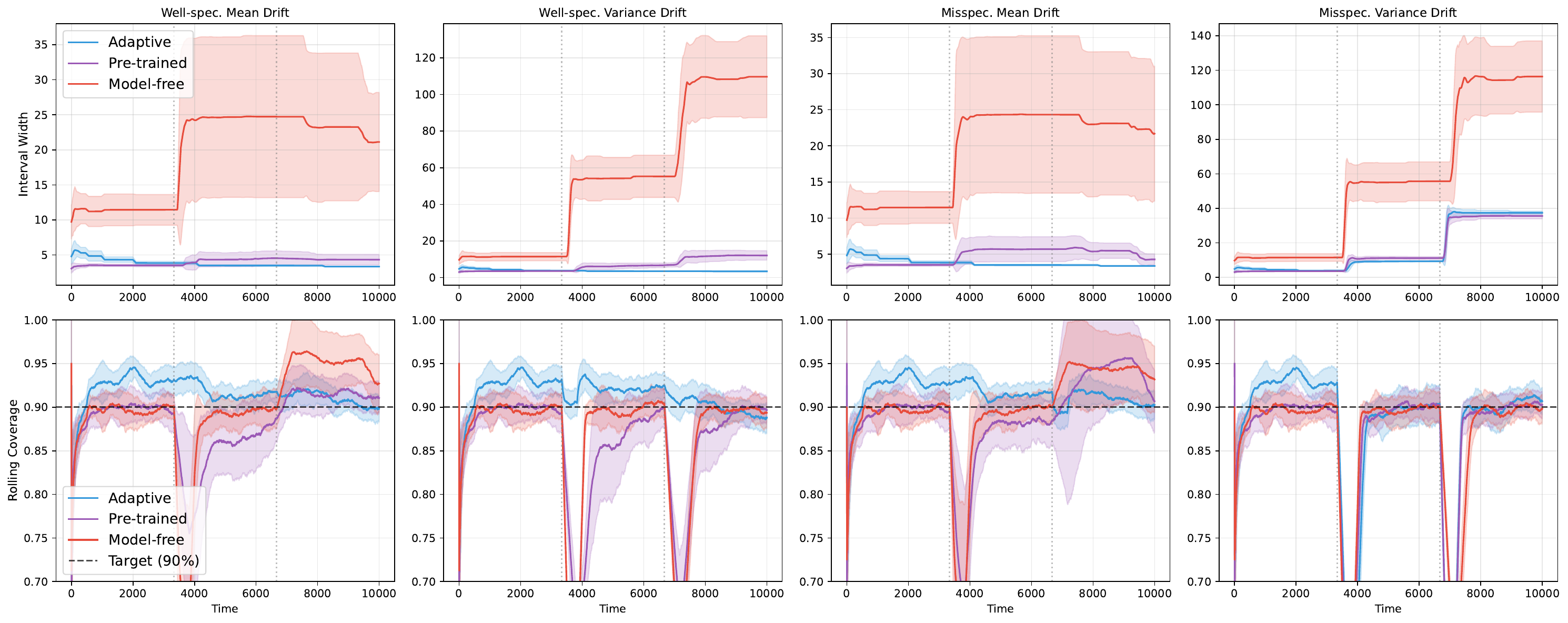}
    \caption{\textbf{Online conformal prediction with different score constructions.}
Top row: prediction-interval width over time. Bottom row: local coverage rate computed with a rolling window of 100
steps; the horizontal line marks the target level \(1-\alpha=0.9\). Vertical dashed lines indicate change points at
\(t=3333\) and \(t=6667\). Columns correspond to the four settings (well-specified vs.\ misspecified model, each under
mean vs.\ variance drift). The adaptive-score method (online SGD) yields noticeably shorter intervals and more stable
coverage under variance drift, whereas the pretrained-score method is sensitive to a mismatch between the pretraining
and test covariate distributions. The model-free baseline (\(s_t=\abs{Y_t}\)) is in general conservative and produces wide prediction intervals; it is also 
sensitive to distribution shift, exhibiting undercoverage at change points.
Curves are averaged over 20 runs, with shaded bands indicating \(\pm1\) standard deviation.}
    \label{fig:adaptive_vs_pretrain}
\end{figure}


\section{Additional related work}

In this section, we briefly discuss a small sample of other related papers.  
To start with, a substantial body of work has established theoretical coverage guarantees for
conformal prediction~\citep{angelopoulos2021gentle,angelopoulos2024theoretical}.  The majority of these results, however, rely on the
assumption that the data are exchangeable (most notably, i.i.d.~observations)
(e.g.,  \citet{vovk2005algorithmic,vovk2015cross,lei2018distribution,barber2021predictive}).  
A growing literature has investigated how conformal prediction procedures can be
modified to retain validity when exchangeability is violated, aiming to preserve meaningful coverage guarantees under various forms of distributional shift. For instance, 
\citet{tibshirani2019conformal,barber2023conformal} developed weighted split conformal methods
that restore \emph{marginal} validity under a weighted-exchangeability condition. Their methods rely on importance weights tied to the data distribution; in a distribution-free setting, one must either use a non-data-dependent weight (as in \citet{barber2023conformal}) or impose stronger structural assumptions (e.g., invariance of \(Y\mid X\) as in \citep{tibshirani2019conformal}). 
In parallel,~\citet{podkopaev2021distribution,si2024pac} developed split conformal  
methods under label shift, when the marginal distribution of $Y$ differs across environments 
while the conditional distribution $P(X\mid Y)$ remains invariant. 
In another line of work~\citep{chernozhukov2018exact,cauchois2024robust,ai2024not,gui2024distributionally,aolaritei2025conformal}, distribution shifts between training and test environments are tackled from a distributionally robust optimization perspective.
Meanwhile, a growing literature developed conformal prediction methods for time-series data \citep{zaffran2022adaptive,xu2021conformal,xu2023sequential,xu2023conformal,xu2024conformal,chen2024conformalized,cleaveland2024conformal,stocker2025gentle}.

For temporally dependent, non-exchangeable sequences, 
\citet{oliveira2024split} showed that split conformal retains \emph{approximate marginal} validity up to an explicit penalty term controlled by decoupling/mixing
conditions (including \(\beta\)-mixing), with sharper results subsequently obtained in~\citet{barber2025predictive}.
These results, however, do not yield sharp training-(and-calibration)-conditional concentration
bounds and typically require stationarity-type dependence assumptions.
In contrast, training-(and-calibration)-conditional coverage guarantees beyond exchangeability are comparatively rarer.

In addition, a line of work connected ACI-style calibration with ideas from online learning and studied regret bounds
under various performance criteria \citep{bhatnagar2023improved,gibbs2024conformal,zhang2024discounted,ramalingam2024relationship,liu2026online}, which, however, fell short of ensuring training-conditional coverage. 
Moreover, the idea of ACI has been extended for broader settings, such as risk control \citep{feldman2022achieving,farinhas2024non}, stochastic control for time series \citep{yang2024bellman}, and parametric quantile calibration \citep{arecesonline}, among other things.

\section{Discussion}

In this work, we have developed two online conformal prediction methods that adapt efficiently to temporal distribution drift, 
producing prediction sets that are both valid and informative. For scenarios involving pretrained score functions, our \driftocp algorithm leverages an efficient drift detection subroutine to update the calibration set sequentially, achieving regret bounds that are minimax optimal (up to logarithmic factors) across both change-point and smooth drift regimes. For scenarios where the predictive models (and hence the score functions) are trained adaptively from prior observations, we propose \driftocpfull, a full-conformal-style online algorithm that enjoys strong regret guarantees under stability assumptions; for this setting,  we further establish matching minimax lower bounds under suitable restrictions on the prediction sets.  Unlike much of the prior work that focused on metrics like empirical long-term coverage frequency or adversarial regret,  our analysis exploits additional independence  assumption across data while remaining otherwise distribution-free, which has enabled training-conditional regret guarantees. The proposed algorithms are horizon-free, computationally efficient,  and supported by fully non-asymptotic, minimax-optimal theoretical guarantees.

Our results naturally suggest several directions for future investigation. To begin with, while our current theoretical development 
hinges upon independence across data samples, many online predictive inference problems involve temporally dependent observations, as commonly encountered in, say, time-series settings \citep{xu2021conformal,zaffran2022adaptive,oliveira2024split}. A natural and
challenging direction is therefore to extend our online conformal methods to tackle temporally dependent, non-stationary Markovian environments, while still delivering rigorous training-conditional coverage guarantees. 
On another front, 
the training-conditional guarantees in \Cref{sec:full_conf}
rely on stability
assumptions for the underlying model-fitting algorithms. When the desirable stability conditions fail to hold or are difficult to verify---e.g., in certain high-dimensional, nonparametric, or deep learning models  \citep{gibbs2025characterizing,lei2011efficient,romano2019conformalized}---it is fundamentally important to develop new online full conformal methods that ensure both validity and efficiency without stability. Finally, our algorithmic and analysis frameworks might shed light on other online statistical problems, such as online multicalibration \citep{collina2026optimal}.  

\section*{Acknowledgments}

Y.~Chen is supported in part by the Alfred P.~Sloan Research Fellowship,  the ONR grant N00014-25-1-2344,  the NSF grants 2221009 and 2218773, 
the Wharton AI \& Analytics Initiative's AI Research Fund, 
and the Amazon Research Award. 
Z.~Ren is supported by the NSF grant DMS-2413135
and Wharton Analytics. 
Y.~Chen would like to thank Jiahao Ai for extensive discussion about adaptive conformal inference.


\appendix

\section{Proof of Fact~\ref{lem:long-term_notge_regret}}
\label{sec:proof-tc_regret_def}

\paragraph{Proof of (i).}
By the triangle inequality, we obtain
\begin{align*}
    \left|
\frac{1}{T}\sum_{t=1}^T 
\PB(Y_t \in \mathcal{C}_t)
- (1-\alpha)
\right|&\le \frac{1}{T}\ssum{t}{1}{T}\big|\PB(Y_t \in \mathcal{C}_t)
- (1-\alpha)\big|
=
\frac{1}{T}\ssum{t}{1}{T}\abs{\EB_{\gC_t(\cdot)}\left[\PB\big(Y_t \in \mathcal{C}_t\mid \gC_t(\cdot)\big)\right]
- (1-\alpha)}
\\
&\le 
\frac{1}{T}\ssum{t}{1}{T}\EB_{\gC_t(\cdot)}\left[\abs{\PB\big(Y_t \in \mathcal{C}_t\mid \gC_t(\cdot)\big)
- (1-\alpha)}\right]
\le 
\frac{\mathsf{regret}_T
}{T}.
\end{align*}

\paragraph{Proof of (ii).}
For any time point $t$, let us construct $\gC_t$ as follows:
\begin{equation*}
    \gC_t \coloneqq\begin{cases}
        \varnothing, & \text{with probability }\alpha,\\
        \RB, & \text{with probability }1-\alpha.
    \end{cases}
\end{equation*}
On the one hand, it can be easily derived that, for any $1\leq t\leq T$,
\[
\PB(Y_t \in \gC_t) = \alpha \mathbbm{1}\{Y_t \in \varnothing\} + (1-\alpha)\mathbbm{1}\{Y_t \in \RB\} = 1-\alpha,
\]
thus implying that
\[
\mathsf{lt\text - cvg}_T = 
\frac{1}{T}\ssum{t}{1}{T}\PB(Y_t \in \gC_t) = 1-\alpha.
\]
On the other hand, it is seen that
\begin{align*}
    &\abs{\PB(Y_t \in \varnothing) - (1-\alpha)} = 1-\alpha;\\
    &\abs{\PB(Y_t \in \RB) - (1-\alpha)} =\abs{1 -(1- \alpha)} = \alpha.
\end{align*}
Therefore, the following holds naturally:
\[
\abs{\PB\big(Y_t \in \gC_t \mid \gC_t(\cdot) \big) - (1-\alpha)} \ge \min\{1-\alpha, \alpha\}= \alpha.
\]
Summing over all $t=1,\dots,T$ and recalling the definition of $\mathsf{regret}_T$, we complete the proof.

\section{Detailed proofs in \Cref{sec:split_conf}}
Before proceeding, let us introduce some convenient notation.

\subsection{Proof of Theorem~\ref{thm:split_regret}}
\label{sec:proof:thm:split_regret}

We now turn to the proof of \Cref{thm:split_regret}. At a high level, we shall first establish a per-round regret bound, then aggregate these
bounds within each stage, and finally sum over all stages to obtain regret bounds over the entire time horizon $[T]$.

\subsubsection{Additional notation}
\label{sec:additional-notation-split-thm-UB}
To facilitate presentation for the analysis of Algorithm~\ref{alg:OCID}, we introduce some additional notation. First,  write $$s_{n,r,l}=s_{n,r,l}(X_{n,r,l}, Y_{n,r,l})\qquad \text{and}\qquad s_t = s_t(X_t,Y_t)$$ as long as it is clear from the context. 
In addition, for any $ 1\le i < j\le T$, we define the ``typical'' event:
\begin{subequations}
\label{eq:typical_set_gA_ij_nr}
\begin{equation}\label{eq:typical_set_gA_ij}
    \gA(i,j) \coloneqq \left\{\sup\limits_{x\in \RB}\left\{\abs{\ssum{t}{i}{j}
    \big(\mathbbm{1}\{s_t \le x\} - \PB(s_t \le x)\big)
    }\right\} \le 6\sqrt{(j-i+1)\log j}\right\}.
\end{equation}
To see that this is a high-probability event, invoking \Cref{lem:DKW_ineq} with $\delta = j^{-6}$ and using the fact
\[
\frac{4}{\sqrt{j-i+1}} + \sqrt{\frac{6\log j}{2(j-i+1)}} \le 4\sqrt{\frac{\log j}{j-i+1}} + \sqrt{\frac{3\log j}{j-i+1}} < 6\sqrt{\frac{\log j}{j-i+1}},
\]
we can establish that
\begin{equation}\label{eq:up_bd_A_ij_c}
\begin{aligned}
    \PB\big(\gA(i,j)^{\mathrm{c}}\big) \le j^{-6}.
\end{aligned}
\end{equation}
Moreover, for any stage-round pair $(n,r)$, we define
\begin{equation}\label{eq:typical_set_gA_nr}
    \gA_{n,r} \coloneqq \bigcap\limits_{i=\tau_{n,r}}^{\tau_{n,r+1}-1}\bigcap\limits_{j = i+1}^{\tau_{n,r+1}-1} \gA(i,j),
\end{equation}
\end{subequations}
where we recall that $\tau_{n,r}$ indicates the time at which round $r$ of stage $n$ begins and $\tau_{n,r_n + 1} = \tau_{n+1, 1}$. 
For convenience, we also write %
\begin{align}
\label{eq:defn-taun-proof}\tau_n \coloneqq \tau_{n,1}.
\end{align}

Moreover, we introduce several notations for the change-point setting. Recall that under the change-point model, 
the entire horizon $[T]$ is partitioned into $N^{\mathsf{cp}}+1$ time segments, 
within each of which the score distributions remain fixed. 
\begin{definition}\label{def:notation_of_cp}
We define the following notation:
\begin{itemize}
    \item $\mathcal{I}_1,\cdots,\mathcal{I}_{N^{\mathsf{cp}}+1}$: the $N^{\mathsf{cp}}+1$ time segments over the entire horizon;

    \item $K_{n,r}$: the total number of time segments in round $r$ of stage $n$; 

    \item $S_{n,r}$ and $t_n$: the total number of iterations in round $r$ of stage $n$;\\ in particular, let $t_n=S_{n,r_n}$, i.e., the number of iterations in the last round of this stage;

    \item $\mathcal{I}_{n,r,k}$ $(k=1,\cdots,K_{n,r})$: the $k$-th time segment in round $r$ of stage $n$; 

    \item $S_n$: the total number of iterations in stage $n$;

    \item $J_n$: the total number of time segments in stage $n$;

    \item $\mathcal{I}_{n,j}$ $ (j=1,\cdots,J_n)$: the $j$-th time segment in stage $n$.
\end{itemize}
\end{definition}
\noindent Also, for any time segment $\mathcal{I}$, we let $|\mathcal{I}|$ represent the length of this time segment. In addition, while the last round $r_n$ of stage $n$ contains $t_n\leq T_{r_n}$ iterations, we generate---for convenience of presentation---a set of random variables $\{s_{n,r,l}\}$ for $l>t_n$ in a way that obeys
\begin{align}
\PB(s_{n,r_n,l} > q_{n,r_n}\mid q_{n,r_n}) = \alpha
\qquad \text{for all }l> t_n. 
\label{eq:auxiliary-sl-exceed}
\end{align}

Moreover, for round $r$ in stage $n$, we define the cumulative KS distance within this round as
\begin{subequations}
\label{eq:defn-KS-tau-n-r}
\begin{align}
\mathsf{KS}_{n,r}^{\mathsf{round}}
    \coloneqq \ssum{l}{1}{S_{n,r}-1} \mathsf{KS}(s_{n,r,l}, s_{n,r,l+1}).
    \end{align}
We also define the cumulative KS distance within stage $n$ (which contains $r_n$ rounds) as
%
%
\begin{align}
    \label{eq:defn-KS-tau-n}
    \mathsf{KS}_{n}^{\mathsf{stage}}
    \coloneqq 
    \sum_{r=1}^{r_n} \mathsf{KS}_{n,r}^{\mathsf{round}}. 
    \end{align}
\end{subequations}

\subsubsection{Decomposing and bounding the cumulative regret}

In order to bound the cumulative regret, 
we first decompose it based on stages and rounds of \driftocp as well as the typical events $\{\mathcal{A}_{n,r}\}$ introduced in \Cref{sec:additional-notation-split-thm-UB}: 
\begin{equation}\label{eq:typical_and_rare_decomp}
\begin{aligned}
\ssum{t}{1}{T}\big|\PB(s_t > q_t \mid q_{t}) - \alpha\big| &= \ssum{n}{1}{N}\ssum{r}{1}{r_n}\ssum{l}{1}{T_r}\big|\PB(s_{n,r,l} > q_{n,r}\mid q_{n,r}) - \alpha\big|\mathbbm{1}\{\gA_{n,r}\}\\
&\quad + \ssum{n}{1}{N}\ssum{r}{1}{r_n}\ssum{l}{1}{T_r}\big|\PB(s_{n,r,l} > q_{n,r}\mid q_{n,r}) - \alpha\big|\mathbbm{1}\{\gA_{n,r}^{\mathrm{c}}\}.
\end{aligned}
\end{equation}
As it turns out, the first term on the right-hand side of \eqref{eq:typical_and_rare_decomp} serves as the dominant term, as argued below.

Consider any time point $t$ belonging to round $r$ of stage $n$. According to the procedure of \driftocp---particularly the fact that the rounds length grow geometrically with $T_r=3^r$---it follows that neither $t/4$ nor $4t$ lies within the same round $(n,r)$, and as a result,
$$\gE_t \coloneqq \bigcap\limits_{j=t/4}^{4t}\bigcap\limits_{i=1}^j \gA(i,j)  \subseteq \gA_{n,r}.$$ 
It can thus be seen from \eqref{eq:up_bd_A_ij_c} that
\begin{align*}
    \PB(\gE_t^{\mathrm{c}}) &\le \ssum{j}{t/4}{4t}\ssum{i}{1}{j}\PB\big(\gA(i,j)^{\mathrm{c}}\big) \le \ssum{j}{t/4}{4t}\ssum{i}{1}{j}\frac{1}{j^4}
    \le \ssum{j}{t/4}{4t}\frac{1}{j^3} \le \frac{32}{t^2},
\end{align*}
which helps us control the second term on the right-hand side of \eqref{eq:typical_and_rare_decomp} as
\begin{equation}\label{eq:rare_eve_bd}
\begin{aligned}
    \EB\left[\ssum{n}{1}{N}\ssum{r}{1}{r_n}\ssum{l}{1}{T_r}\big|\PB(s_{n,r,l} > q_{n,r}\mid q_{n,r}) - \alpha\big|\,\mathbbm{1}\{\gA_{n,r}^{\mathrm{c}}\}\right]&\le \EB\left[\ssum{t}{1}{T} \big|\PB(s_t > q_t\mid q_t) - \alpha\big|\,\mathbbm{1}\{\gE_t^{\mathrm{c}}\}\right]\\
    &\le \ssum{t}{1}{T}\PB\big(\gE_t^{\mathrm{c}}\big) 
    \le \ssum{t}{1}{T}\frac{32}{t^2} = O(1).
\end{aligned}
\end{equation}
%


As a consequence, the remainder of this proof is devoted primarily to bounding the first term of \eqref{eq:typical_and_rare_decomp}. Towards this end, we begin by looking at the cumulative coverage gaps in round $r$ of stage $n$. Informally,  in the change-point setting, the cumulative coverage gap over this round on the typical events defined in \Cref{sec:additional-notation-split-thm-UB} is upper bounded by a sum of square-root terms in the lengths of the time segments. In contrast, under smooth drift, the bound contains a term that scales with a suitable KS distance raised to the 
$1/3$ power, in a manner that resembles the final regret bound in \Cref{thm:split_regret}. This is stated in the following lemma, with the proof deferred to \Cref{sec:prf:lem:one_round_regret}.


\begin{lemma}\label{lem:one_round_regret}
    Consider any stage-round pair $(n,r)$ in Algorithm~\ref{alg:OCID}. 
    If no distribution shift has been detected by  the subroutine $\textsc{DriftDetect}$ by the end of this round, then it holds that
    \begin{align*}
    \bigg( \ssum{l}{1}{T_{r}}\Big|\PB(s_{n,r,l} > q_{n,r}\mid q_{n,r}) - \alpha\Big| \bigg) \mathbbm{1}\{\gA_{n,r}\} =
    \begin{cases}
    \widetilde{O}\bigg(\ssum{k}{1}{K_{n,r}}\sqrt{\abs{\gI_{n,r,k}}}\bigg), & \text{for the change-point setting;}\\
    \widetilde{O}\left(\sqrt{T_{r}} + (\mathsf{KS}_{n,r}^{\mathsf{round}})^{\frac{1}{3}}T_{r}^\frac{2}{3}\right), & \text{for the smooth drift setting.}
    \end{cases}
    \end{align*}
\end{lemma}

Owing to the doubling trick employed in \driftocp (i.e., the round lengths grow geometrically), we can lift the per-round cumulative gap bound in \Cref{lem:one_round_regret} to a per-stage cumulative gap bound, again on the typical events defined in \Cref{sec:additional-notation-split-thm-UB}. This result is formalized in the following lemma, whose proof is postponed to \Cref{sec:prf:lem:one_stage_regret}. 
\begin{lemma}\label{lem:one_stage_regret}
    Consider any stage $n$ in Algorithm~\ref{alg:OCID}, which comprises $r_n$ rounds. 
    Then we have 
    \[
    \ssum{r}{1}{r_n}\bigg(\ssum{l}{1}{T_{r}}\Big|\PB(s_{n,r,l} > q_{n,r}\mid q_{n,r}) - \alpha\Big| \bigg)\mathbbm{1}\{\gA_{n,r}\} = 
    \begin{cases}
        \widetilde{O}\bigg(\ssum{j}{1}{J_n}\sqrt{\abs{\gI_{n,j}}}\bigg),~ &\text{for the change-point setting};\\
        \widetilde{O}\left(\sqrt{S_n}  + (\mathsf{KS}_{n}^{\mathsf{stage}})^{\frac{1}{3}}S_n^{\frac{2}{3}}\right),~ &\text{for the smooth drift setting}.
    \end{cases}
    \]
\end{lemma}

It remains to see how the per-stage cumulative regret bounds derived above can be leveraged to establish Theorem~\ref{thm:split_regret}. To this end, we cope with the change-point and smooth drift settings separately in what follows.


\subsubsection{Controlling the dominant term in the change-point setting}

Recall the definition of $\{\mathcal{I}_k\}_{k=1}^{N^{\textsf{cp}}+1}$ and $\{\mathcal{I}_{n,j}\}_{j=1}^{J_n} $ in \Cref{sec:additional-notation-split-thm-UB}, and 
%
%
denote 
\begin{align}
\label{eq:defn-event-Bn-12345}
\gB_{n}=\mathcal{A}_{n,r_n-1}\cap \mathcal{A}_{n,r_n}
.
\end{align}
Suppose that \driftocp contains $N$ stages. 
Applying \Cref{lem:one_stage_regret} tells us that
%
\begin{align}
\sum_{n=1}^N \sum_{r=1}^{r_n}\sum_{l=1}^{T_r}
&\big|\mathbb{P}(s_{n,r,l} > q_{n,r}\mid q_{n,r}) - \alpha\big|\,
\mathbbm{1}\{\mathcal{A}_{n,r}\}
\le  \widetilde{O}\Bigg(
\sum_{n=1}^N\bigg(\sum_{j=1}^{J_n}\sqrt{|\mathcal{I}_{n,j}|}\bigg)
\Bigg)\notag\\
&\le \widetilde{O}\Bigg(
\sum_{n=1}^N\bigg(\sum_{j=1}^{J_n}\sqrt{|\mathcal{I}_{n,j}|}\bigg)\,
\mathbbm{1}\{\gB_{n}\}
+\sum_{n=1}^N (\tau_{n+1}-\tau_{n})\,
\mathbbm{1}\{\gB_n^{\mathrm{c}}\}\Bigg),
\label{eq:cp-typical_decomp}
\end{align}
%
where the second relation follows since, for any $n\in[N]$, 
\[
\ssum{j}{1}{J_n}\sqrt{\abs{\gI_{n,j}}} \le \sqrt{J_n\bigg(\ssum{j}{1}{J_n}\abs{\gI_{n,j}}\bigg)} = \sqrt{J_n (\tau_{n+1} - \tau_n)} \le \tau_{n+1} - \tau_n.
\]
In the sequel, we bound the two terms on the right-hand side of \eqref{eq:cp-typical_decomp} separately.

\begin{itemize}
\item To bound the first term on the right-hand side of \eqref{eq:cp-typical_decomp}, we first show that on the typical events, the final two rounds of any stage do not share exactly the same score distributions. This is formally stated in the lemma below, whose proof is provided in \Cref{sec:prf:lem:seg-separation}.  
\begin{lemma}
\label{lem:seg-separation}
Consider any stage $n$ in the change-point setting. 
%
On the event
\(\gA_{n,r_n-1}\cap \gA_{n,r_n}\), the rounds $r_n-1$ and $r_n$ cannot both be entirely contained within the same time segment from $\{\gI_k\}_{k=1}^{N^{\mathsf{cp}}+1}$.
\end{lemma}

In view of \Cref{lem:seg-separation}, each time segment $\mathcal{I}_k$ cannot overlap with more than two consecutive stages; in other words, each $\mathcal{I}_k$ can contain at most two time segments from 
$\{\mathcal{I}_{n,j}:\ n\in [N],\ j\in [J_n]\}$. As a consequence, 
\begin{equation*}
\sum_{n=1}^N\Big(\sum_{j=1}^{J_n}\sqrt{|\mathcal{I}_{n,j}|}\Big)\,
\mathbbm{1}\{\gB_{n}\}
\le 2\sum_{k=1}^{N^{\textsf{cp}}+1}\sqrt{|\mathcal{I}_k|}
\le 2\sqrt{(N^{\textsf{cp}}+1)T},
\end{equation*}
where the last step follows from Cauchy-Schwarz as well as the fact that $\sum_{k=1}^{N^{\mathsf{cp}}+1}|\mathcal{I}_k|=T$. 

\item 
Turning to the second term on the right-hand side of \eqref{eq:cp-typical_decomp}, we make note of the decomposition: 
\begin{equation}\label{eq:cp_setting_rare_eve_decomp}
\begin{aligned}
    &\EB\left[\ssum{n}{1}{N}(\tau_{n+1} - \tau_n)\mathbbm{1}\left\{
    \gB_n^{\mathrm{c}}
    \right\}\right] \le \ssum{n}{1}{\infty}\EB\left[
    (\tau_{n+1} - \tau_n)\mathbbm{1}\left\{\gB_n^{\mathrm{c}}\right\}
    \right]\\
    &\qquad \le \ssum{n}{1}{\infty}\sum\limits_{i< j}(j-i)\PB\left(\tau_{n}=i,\tau_{n+1}=j, \gB_n^{\mathrm{c}}\right)\\
    &\qquad \le \ssum{n}{1}{\infty}\sum\limits_{i<j}(j-i) \PB(\tau_{n}=i)\PB\left(\bigcup\limits_{k=i\vee \frac{j}{16}}^j \bigcup\limits_{l=i\vee \frac{j}{16}}^{k-1} \gA(l,k)^{\mathrm{c}}, \tau_{n+1} = j ~\Big|~ \tau_{n}=i \right),
\end{aligned}
\end{equation}
where the last inequality uses a simple property of \textsc{DriftOCP}, namely, \(16\tau_{n,r_n-1} \ge \tau_{n+1}\) due to the exponential growth of round lengths. Additionally, observe that under the independent data assumption, for any $l,k\in[i,j]$ the event $\gA(l,k)^{\mathrm{c}}$ is independent of what has happened prior to time $i$, which taken together with \eqref{eq:up_bd_A_ij_c} and the union bound gives
\begin{align*}
    \PB\Bigg(\bigcup\limits_{k=i\vee \frac{j}{16}}^j \bigcup\limits_{l=i\vee \frac{j}{16}}^{k-1} \gA(l,k)^{\mathrm{c}}, \tau_{n+1} = j ~\Big|~ \tau_{n}=i \Bigg)
    &\le \ssum{k}{i\vee \frac{j}{16}}{j}\ssum{l}{i\vee \frac{j}{16}}{k-1}\PB\big(\gA(l,k)^{\mathrm{c}} \mid \tau_{n}=i\big)\\
    &= \ssum{k}{i\vee \frac{j}{16}}{j}\ssum{l}{i\vee \frac{j}{16}}{k-1} \PB\big(\gA(l,k)^{\mathrm{c}}\big)
    \le \ssum{k}{i\vee \frac{j}{16}}{j}\ssum{l}{i\vee \frac{j}{16}}{k-1}\frac{1}{k^6}\\
    &\le \ssum{k}{i\vee \frac{j}{16}}{j}\frac{1}{k^5}
    < \frac{1}{4\left(i\vee \frac{j}{16}\right)^4}.
\end{align*}
Substituting this bound into \eqref{eq:cp_setting_rare_eve_decomp} results in
\begin{align*}
    \EB\left[\ssum{n}{1}{N}(\tau_{n+1} - \tau_n)\mathbbm{1}\left\{
    \gB_n^{\mathrm{c}}
    \right\}\right]
    &\le \ssum{n}{1}{\infty} \sum\limits_{i<j}(j-i)\PB(\tau_{n}=i)\frac{1}{4\left(i\vee \frac{j}{16}\right)^4}\\
    &\lesssim \ssum{n}{1}{\infty}\ssum{i}{1}{\infty}\PB(\tau_n = i)\left\{\ssum{j}{i+1}{16i}\frac{j-i}{i^4} + \ssum{j}{16i+1}{\infty}\frac{j-i}{j^4}\right\}\\
    &\lesssim \ssum{n}{1}{\infty}\ssum{i}{1}{\infty}\PB(\tau_n =i)\frac{1}{i^2}
    \overset{\mathrm{(a)}}{\le} \ssum{n}{1}{\infty}\ssum{i}{1}{\infty}\PB(\tau_n =i)\frac{1}{n^2} \notag\\
    &=\ssum{n}{1}{\infty} \frac{1}{n^2}
    = O(1 ),
\end{align*}
where (a) relies on the elementary bound $\tau_n\ge n$.
\end{itemize}

Putting the preceding bounds together, we arrive at
\[
    \EB\left[\ssum{n}{1}{N}\ssum{r}{1}{r_n}\ssum{l}{1}{T_r}
    \big|\PB(s_{n,r,l} > q_{n,r}\mid q_{n,r}) - \alpha\big|
    \mathbbm{1}\{\gA_{n,r}\}\right]
    \le 2\sqrt{(N^{\textsf{cp}}+1)T} + O(1).
\]
The advertised bound in the change-point setting then follows by combining this with  \eqref{eq:rare_eve_bd}.

\subsubsection{Controlling the dominant term in the smooth drift setting}

From Lemma~\ref{lem:one_stage_regret}, we have established how to bound the cumulative regret within a single stage. As before, suppose there are $N$ stages in total. Summing over the stage index $n$ then yields the following bound:
\begin{align}
   \ssum{n}{1}{N}\ssum{r}{1}{r_n}\ssum{l}{1}{T_r}\big| \PB(s_{n,r,l} > q_{n,r}\mid q_{n,r}) - \alpha \big| \mathbbm{1}\{\gA_{n,r}\}
    &\le \ssum{n}{1}{N}\widetilde{O}\left(\sqrt{S_n} + \big(\mathsf{KS}^{\textsf{stage}}_n\big)^{\frac{1}{3}}S_n^{\frac{2}{3}}\right) \notag\\
    &\overset{\mathrm{(a)}}{\le}\widetilde{O}\left(\ssum{n}{1}{N-1}\sqrt{S_n} + \sqrt{T} + \ssum{n}{1}{N}\big(\mathsf{KS}^{\textsf{stage}}_n\big)^{\frac{1}{3}}S_n^{\frac{2}{3}}\right) \notag\\
    &\overset{\mathrm{(b)}}{\le} 
    \widetilde{O}\left(\ssum{n}{1}{N-1}\sqrt{S_n} + \sqrt{T} + \bigg(\ssum{n}{1}{N}\mathsf{KS}^{\textsf{stage}}_n\bigg)^{\frac{1}{3}}\bigg(\ssum{n}{1}{N}S_n\bigg)^{\frac{2}{3}}\right) \notag\\
    &=
    \widetilde{O}\left(\ssum{n}{1}{N-1}\sqrt{S_n} + \sqrt{T} + (\mathsf{KS}_T)^{\frac{1}{3}}T^{\frac{2}{3}}\right).
    \label{eq:UB-quantile-gap-Anr}
\end{align}
Here, (a) holds since $\sqrt{S_{N}}\le \sqrt{T}$, whereas (b) results from H\"older’s inequality.

With the above inequality in mind, a crucial task is to bound  $\sum_{n=1}^{N-1}\sqrt{S_n}$, which can be decomposed into
\begin{align}
    \sum_{n=1}^{N-1}\sqrt{S_n}
    =\sum_{n=1}^{N-1}\sqrt{S_n} \mathbbm{1}\{\gB_n\}
    + \sum_{n=1}^{N-1}\sqrt{S_n} \mathbbm{1}\{\gB_n^{\mathrm{c}}\}
\end{align}
with $\gB_n$ denoting the event $\gB_n=\gA_{n,r_{n-1}}\cap \gA_{n,r_n}$ (see \eqref{eq:defn-event-Bn-12345}). 
By applying an argument analogous to the one used to control the second term in \eqref{eq:cp-typical_decomp}, we can readily obtain
\begin{align}
\EB\left[\ssum{n}{1}{N-1}\sqrt{S_n}\mathbbm{1}\{\gB_n^{\mathrm{c}}\}\right] = O(1),
\label{eq:sum-St-Bc-135}
\end{align}
where we omit the details for brevity. Therefore, everything comes down to controlling 
$\sum_{n=1}^{N-1}\sqrt{S_n} \mathbbm{1}\{\gB_n\}$, 
which forms the main content of the remainder of this subsection.

For $n\in[N-1]$, note that the last round $r_n$ of each of these stages ends with a restart; that is, a drift detection is declared in round $r_n$. Following the proof of Lemma~\ref{lem:one_stage_regret},  let $t_n$ denote the number of iterations in round $r_n$, then according to the drift detection subroutine, there exists some $j_n\in[t_n]$ such that
\begin{align}
\bigg|\ssum{l}{j_n}{t_n}\left(\mathbbm{1}\{s_{n,r_n,l} > q_{n,r_n}\} - \alpha\right)\bigg| > 24\sqrt{ (t_n - j_n + 1)\log (4\tau_{n,r_n})}.
\label{eq:jn-devation-13579}
\end{align}
On the event 
%
$\gB_n$,
%
it is observed that
\begin{equation}\label{eq:split_regret_2}
\begin{aligned}
    &\bigg|\ssum{l}{j_n}{t_n} \left(\PB(s_{n,r_n,l} > q_{n,r_n}\mid q_{n,r_n}) - \alpha\right)\bigg|\\
    &\ge \bigg|\ssum{l}{j_n}{t_n} \big(\mathbbm{1}\{s_{n,r_n,l} > q_{n,r_n}\} - \alpha\big)\bigg| - \bigg|\ssum{l}{j_n}{t_n} \big(\mathbbm{1}\{s_{n,r_n,l} > q_{n,r_n}\} - \PB(s_{n,r_n,l} > q_{n,r_n}\mid q_{n,r_n})\big)\bigg|\\
    &> 24\sqrt{(t_n - j_n+1)\log (4\tau_{n,r_n})} - 6\sqrt{(t_n - j_n+1)\log \tau_{n+1}}
    > 18\sqrt{(t_n - j_n + 1)\log \tau_{n+1}},
\end{aligned}
\end{equation}
where the last line arises from \eqref{eq:jn-devation-13579} and the definition of $\mathcal{A}_{n,r_n}$. 
Further, let us introduce
\[
B_n \coloneqq \frac{1}{t_n - j_n + 1}\ssum{l}{j_n}{t_n}\big(
\PB(s_{n,r_n,l} > q_{n,r_n}\mid q_{n,r_n}) - \alpha
\big),
\]
which, according to Eqn.~\eqref{eq:split_regret_2},  satisfies
\begin{equation}\label{eq:split_regret_3}
    \frac{\sqrt{t_n - j_n + 1}}{18\sqrt{\log \tau_{n+1}}}\abs{B_n}\mathbbm{1}\{\gB_n\} \ge \mathbbm{1}\{\gB_n\}.
\end{equation}
With this inequality in place, we can readily obtain
\begin{align}
    &\ssum{n}{1}{N-1}\sqrt{S_n}\mathbbm{1}\{\gB_n\}
    \overset{\eqref{eq:split_regret_3}}{\le} \ssum{n}{1}{N-1} \sqrt{S_n}\left(\frac{\sqrt{t_n - j_n + 1}}{18\sqrt{\log \tau_{n+1}}}\abs{B_n}\right)^{\frac{1}{3}}\mathbbm{1}\{\gB_n\} \notag\\
    &\qquad \le \ssum{n}{1}{N-1}S_n^{\frac{2}{3}}\abs{B_n}^{\frac{1}{3}}\mathbbm{1}\{\gB_n\}
    \le \bigg(\ssum{n}{1}{N-1}S_n\bigg)^{\frac{2}{3}}\bigg(\ssum{n}{1}{N-1}\abs{B_n}\mathbbm{1}\{\gB_n\}\bigg)^\frac{1}{3}
    \le T^{\frac{2}{3}}\bigg( \ssum{n}{1}{N-1}\abs{B_n}\mathbbm{1}\{\gB_n\}\bigg)^{\frac{1}{3}}, 
    \label{eq:Sn-UB-Bn-479}
\end{align}
where the last line follows from
H\"older’s inequality and Jensen's inequality. Thus, it amounts to bounding $\sum_{n=1}^{N-1}\abs{B_n}\mathbbm{1}\{\gB_n\}$, which we accomplish next. 

For every $B_n$, it follows from the triangle inequality that
\begin{align*}
    \abs{B_n} &= \bigg|\frac{1}{t_n - j_n + 1}\ssum{l}{j_n}{t_n}\big(\PB(s_{n,r_n,l} > q_{n,r_n}\mid q_{n,r_n}) - \alpha\big)\bigg|\\
    &\le \underbrace{\bigg|\frac{1}{t_n - j_n + 1}\ssum{l}{j_n}{t_n}\bigg(\PB(s_{n,r_n,l} > q_{n,r_n}\mid q_{n,r_n}) - \frac{1}{T_{r_n-1}}\ssum{i}{1}{T_{r_n-1}}\PB(s_{n,r_n-1,i}' > q_{n,r_n}\mid q_{n,r_n})\bigg)
    \bigg|}_{\eqqcolon\gT_{n,1}}\\
    &~~~~~+ \underbrace{\bigg|\frac{1}{T_{r_n-1}}\ssum{l}{1}{T_{r_n-1}}\big(\PB(s_{n,r_n-1,l}' > q_{n,r_n}\mid q_{n,r_n}) - \mathbbm{1}\{s_{n,r_n-1,l} > q_{n,r_n}\}
    \big)\bigg|}_{\eqqcolon\gT_{n,2}} \\
    &~~~~~+ \underbrace{\bigg|\frac{1}{T_{r_n-1}}\ssum{l}{1}{T_{r_n-1}}\big(\mathbbm{1}\{s_{n,r_n-1,l} > q_{n,r_n}\} - \alpha
    \big)\bigg|}_{\eqqcolon\gT_{n,3}} ,
\end{align*}
where, for each $i=1,\ldots,T_{r_n-1}$, $s'_{n,r_n-1,i}$ is an independent copy of $s_{n,r_n-1,i}$.
The above inequality reduces the problem to controlling three terms.
\begin{itemize}
\item 
Regarding $\gT_{n,1}$, one can insert $\PB(s_{n,r_n,1}>q_{n,r_n}\mid q_{n,r_n})$ into each summand to derive
\begin{equation}\label{eq:split_regret_4}
\begin{aligned}
    \gT_{n,1} &\le \bigg|\frac{1}{t_n - j_n+ 1}\ssum{l}{j_n}{t_n}\big(\PB(s_{n,r_n,l} > q_{n,r_n}\mid q_{n,r_n}) - \PB(s_{n,r_n,1}> q_{n,r_n}\mid q_{n,r_n})\big)\bigg|\\
    &~~~~~ + \bigg|\frac{1}{T_{r_n-1}}\ssum{l}{1}{T_{r_n-1}}
    \big(\PB(s_{n,r_n,1} > q_{n,r_n}\mid q_{n,r_n}) - \PB(s_{n,r_n-1, l}'> q_{n,r_n}\mid q_{n,r_n})\big)
    \bigg|.
\end{aligned}
\end{equation}
For any $l \in [j_n,t_n]$, we can apply the telescoping technique and the triangle inequality to obtain
\begin{align*}
\big|\PB(s_{n,r_n,l} > q_{n,r_n}\mid  q_{n,r_n}) - \PB(s_{n,r_n,1}>\, &q_{n,r_n} \mid q_{n,r_n})\big| \le
\sup\limits_{q\in \RB}\left\{\big|\PB(s_{n,r_n,l}> q) - \PB(s_{n,r_n,1}> q)\big|\right\}\\
&\leq \ssum{j}{1}{l-1}\sup\limits_{q\in \RB}\left\{\Big|\PB(s_{n,r_n,j+1} > q) - \PB(s_{n,r_n,j}> q)\Big|\right\}\\
&\le \ssum{j}{1}{l-1}\mathsf{KS}(s_{n,r_n,j}, s_{n,r_n,j+1}) \le \ssum{j}{1}{t_n - 1}\mathsf{KS}(s_{n,r_n,j},s_{n,r_n,j+1});
\end{align*}
Repeating the same arguments and adopting the notation $s_{n,r_n-1,T_{r_n-1}+1} \coloneqq s_{n,r_n,1}$ also give
\begin{align*}
    \PB(s_{n,r_n,1} > q_{n,r_n}\mid q_{n,r_n}) - \PB(s_{n,r_n-1, l}'> q_{n,r_n}\mid q_{n,r_n}) \le \ssum{j}{1}{T_{r_n-1}}\mathsf{KS}(s_{n,r_n-1,j},s_{n,r_n - 1, j+1}).
\end{align*}
Substituting these two results into Eqn.~\eqref{eq:split_regret_4} yields
\begin{align*}
    \gT_{n,1} &\le  \frac{1}{t_n - j_n+1}\ssum{l}{j_n}{t_n}\ssum{j}{1}{t_n-1}\mathsf{KS}(s_{n,r_n,j},s_{n,r_n,j+1}) + \frac{1}{T_{r_n-1}}\ssum{l}{1}{T_{r_n-1}}\ssum{j}{1}{T_{r_n-1}}
    \mathsf{KS}(s_{n,r_n-1,j},s_{n,r_n - 1, j+1})
    \\
    &= \ssum{j}{1}{T_{r_n-1}}\mathsf{KS}(s_{n,r_n-1,j},s_{n,r_n - 1, j+1}) + \ssum{j}{1}{t_n-1}\mathsf{KS}(s_{n,r_n,j},s_{n,r_n,j+1}),
\end{align*}
where we let $r_n - 1$ represent round $r_{n-1}$ of stage $n-1$ if $r_n = 1$. Therefore, when summing over $n$, each KS distance term is counted at most twice, and as a result, 
\[
\ssum{n}{1}{N-1}\gT_{n,1} \le 2\ssum{t}{1}{T-1}\mathsf{KS}(s_t, s_{t+1}) = 2\mathsf{KS}_T.
\]

 \item With regards to $\gT_{n,2}$, on the event $\gB_n$ (cf.~\eqref{eq:defn-event-Bn-12345}), we can from the definition of $\mathcal{A}_{n,r_n-1}$ (cf.~\eqref{eq:typical_set_gA_ij_nr}) that
\begin{align*}
\gT_{n,2} = \bigg|\frac{1}{T_{r_n-1}}\ssum{l}{1}{T_{r_n-1}}\big(\PB(s_{n,r_n-1,l}' > q_{n,r_n}\mid q_{n,r_n}) - \mathbbm{1}\{s_{n,r_n-1,l} > q_{n,r_n}\}
    \big)\bigg| \le 6\sqrt{\frac{\log \tau_{n+1}}{T_{r_n-1}}},
\end{align*}
which in turn implies that
\begin{align*}
\ssum{n}{1}{N-1}\gT_{n,2}\mathbbm{1}\{\gB_n\} &\le \ssum{n}{1}{N-1}6\sqrt{\frac{\log \tau_{n+1}}{T_{r_n-1}}}\mathbbm{1}\{\gB_n\}\\
&\overset{\eqref{eq:split_regret_3}}{\le}\ssum{n}{1}{N-1}6\sqrt{\frac{\log \tau_{n+1}}{T_{r_n-1}}}\left(\frac{\sqrt{t_n}}{18\sqrt{\log \tau_{n+1}}}\abs{B_n}\mathbbm{1}\{\gB_n\}\right)
\le \ssum{n}{1}{N-1} \frac{2}{3}
\abs{B_n}\mathbbm{1}\{\gB_n\}.
\end{align*}
Here, the last inequality follows from the fact that $t_n \le T_{r_n} \le 4T_{r_n - 1}$ for $r_n > 1$ and $t_n \le T_{r_n} = 1 \le T_{r_n-1}$ for $r_n = 1$.

\item 
When it comes to $\gT_{n,3}$, it is seen from the definition of $q_{n,r}$---which is chosen to be the $\alpha$-empirical-quantile of $\{s_{n,r_n-1,l}\}_{l=1}^{T_{r_n-1}}$---that
\begin{align*}
\gT_{n,3}\mathbbm{1}\{\gB_n\} &= \bigg|\frac{1}{T_{r_n-1}}\ssum{l}{1}{T_{r_n-1}}\big(\mathbbm{1}\{s_{n,r_n-1,l} > q_{n,r_n}\} - \alpha
    \big)\bigg|\mathbbm{1}\{\gB_n\}\\
    &\overset{\mathrm{(a)}}{\le} \frac{1}{T_{r_n - 1}}\mathbbm{1}\{\gB_n\} \le \frac{3\sqrt{\log \tau_{n+1}}}{2\sqrt{t_n-j_n+1}}\mathbbm{1}\{\gB_n\} \overset{\eqref{eq:split_regret_3}}{\le}\frac{1}{12}|B_n|\mathbbm{1}\{\gB_n\}.
\end{align*}
Here (a) holds because Assumption~\ref{assump:split} ensures that the distributions of scores $\{s_t\}_{t=1}^\infty$ are non-degenerate.
\end{itemize}

\noindent 
Taking together the preceding bounds on $\gT_{n,1}$, $\gT_{n,2}$ and $\gT_{n,3}$ results in
\begin{align*}
\ssum{n}{1}{N-1}\abs{B_n}\mathbbm{1}\{\gB_n\} &\le \ssum{n}{1}{N-1}(\gT_{n,1} + \gT_{n,2} + \gT_{n,3})\mathbbm{1}\{\gB_n\}\\
&\le 2\mathsf{KS}_T + \frac{2}{3}\ssum{n}{1}{N-1}\abs{B_n}\mathbbm{1}\{\gB_n\} + \frac{1}{12}\ssum{n}{1}{N-1}\abs{B_n}\mathbbm{1}\{\gB_n\} = 2\mathsf{KS}_T + \frac{3}{4}\ssum{n}{1}{N-1}|B_n|\mathbbm{1}\{\gB_n\},
\end{align*}
from which it follows that
\(
\ssum{n}{1}{N-1}\abs{B_n}\mathbbm{1}\{\gB_n\} \le 8\mathsf{KS}_T.
\)
Combine this bound with \eqref{eq:Sn-UB-Bn-479} to arrive at
\begin{equation}\label{eq:split_regret_5}
\ssum{n}{1}{N-1}\sqrt{S_n}\mathbbm{1}\{\gB_n\} \le 2(\mathsf{KS}_T)^{\frac{1}{3}}T^{\frac{2}{3}}.
\end{equation}


To finish up, taking \eqref{eq:UB-quantile-gap-Anr}, \eqref{eq:sum-St-Bc-135} and \eqref{eq:split_regret_5} collectively yields
\begin{align*}
\EB\bigg[\ssum{n}{1}{N}\ssum{r}{1}{r_n}\ssum{l}{1}{T_r}&\big|\PB(s_{n,r,l} > q_{n,r}\mid q_{n,r}) - \alpha\big|\mathbbm{1}\{\gA_{n,r}\}\bigg] \le \widetilde{O}\Bigg( \EB\bigg[\ssum{n}{1}{N-1}\sqrt{S_n}\bigg] + \sqrt{T} + \mathsf{KS}_T^{\frac{1}{3}}T^{\frac{2}{3}} \Bigg)\\
    &\le \widetilde{O}\bigg( \EB\bigg[\ssum{n}{1}{N-1}\sqrt{S_n}\mathbbm{1}\{\gB_n\}\bigg] + \EB\bigg[\ssum{n}{1}{N-1}\sqrt{S_n}\mathbbm{1}\{\gB_n^{\mathrm{c}}\}\bigg]
    +\sqrt{T} + (\mathsf{KS}_T)^{\frac{1}{3}}T^{\frac{2}{3}} \bigg)\\
    &\overset{\eqref{eq:split_regret_5}}{\le} 
    \widetilde{O}\big(\sqrt{T}+ (\mathsf{KS}_T)^{\frac{1}{3}}T^{\frac{2}{3}}  \big).
\end{align*}
We can immediately finish the proof for the smooth drift setting by combining the above result with \eqref{eq:rare_eve_bd}.

\subsection{Proof of \Cref{thm:lower_bd_split_regret}}
\label{sec:proof:thm:lower_bd_split_regret}
Consider a given $T$, along with a change-point budget $N^{\mathsf{cp}}$ in the change-point setting and a cumulative variation budget $\mathsf{KS}_T$ in the smooth drift setting. In what follows, we intend to construct a subclass of distributions $\gL'$ and use it to establish the claimed minimax lower bound. 


\paragraph{Step 1: construction of a distribution subclass $\gL'$.}
Partition the horizon $[T]\coloneqq\{1,\dots,T\}$ into $m$ consecutive time segments
$\gI_1,\dots,\gI_m$ of (nearly) equal size, where
\[
\gI_j \;\coloneqq\; \bigl\{(j-1)\lceil T/m\rceil+1,\ \dots,\ \min\{j\lceil T/m\rceil,\,T\}\bigr\},
\qquad j\in[m].
\]
Define $\gL'$ as the collection of distribution sequences whose corresponding score distributions $\{\mathcal{D}^{\mathsf{score}}_t\}_{t=1}^T$ obey
\begin{enumerate}
\item for each $t\in[T]$, $\gD^{\mathsf{score}}_t\in\{\mathrm{Exp}(1),\mathrm{Exp}(1+\varepsilon)\}$, where $\mathrm{Exp}(\beta)$ denotes the exponential distribution with the rate parameter $\beta$, and the parameter $\varepsilon\in(0,1]$ will be specified momentarily;
\item $\gD^{\mathsf{score}}_t$ is blockwise constant, namely, for each $j\in[m]$, one has $\gD^{\mathsf{score}}_t=\gD^{\mathsf{score}}_{t'}$ for all $t,t'\in \gI_j$.
\end{enumerate}

Clearly, if $m=N^{\mathsf{cp}}+1$, then $\mathcal{L}'\subseteq \mathcal{L}_1(N^{\mathsf{cp}})$. 
We also verify that in the smooth drift setting, $\gL'\subseteq \gL_2(\mathsf{KS}_T)$ for sufficiently small $\varepsilon$.  
%
%
To see this, we make the observation that
\begin{align*}
\mathsf{KS}\bigl(\mathrm{Exp}(1),\mathrm{Exp}(1+\varepsilon)\bigr)
= \sup_{x\ge 0}\bigl\{\bigl|e^{-x}-e^{-(1+\varepsilon)x}\bigr|\bigr\}
 &= \sup_{x\ge 0}\bigl\{ e^{-x}\bigl(1 -e^{-\varepsilon x}\bigr)\bigr\}\\
&\le \sup\limits_{x \ge 0}\bigl\{\varepsilon xe^{-x}\bigr\}
\le 2\varepsilon,
\end{align*}
where we have used the elementary inequalities $1-e^{-u} \le u$ for $u\ge 0$ and $xe^{-x}\le \frac{x}{1+x}\le 1$ for $x\ge 0$. 
Consequently, for any $\{\gD^{\mathsf{score}}_t\}_{t=1}^T\in\gL'$, it is easily seen that
\[
\sum_{t=1}^{T-1}\mathsf{KS}(\gD^{\mathsf{score}}_t,\gD^{\mathsf{score}}_{t+1})
\le\sum_{j=1}^{m-1}\mathsf{KS}\bigl(\mathrm{Exp}(1),\mathrm{Exp}(1+\varepsilon)\bigr)
\le 2\varepsilon m.
\]
Choosing $\varepsilon \le \mathsf{KS}_T/(2m)$ ensures that $\sum_{t=1}^{T-1}\mathsf{KS}(\gD^{\mathsf{score}}_t,\gD^{\mathsf{score}}_{t+1})\le \mathsf{KS}_T$,
and as a result, $\gL'\subseteq \gL_2(\mathsf{KS}_T)$.

\paragraph{Step 2: lower bound for a single time segment.}
Consider a fixed time segment and suppress the segment index here for notational simplicity. Within this time segment,
the score random variables $\{s_t\}$ are i.i.d.\ drawn from either $\mathrm{Exp}(1)$ or $\mathrm{Exp}(1+\varepsilon)$. Denote $s_{1:t}=\{s_1,\dots,s_t\}$,  let $q_t=q(s_{1:t})$ be an arbitrary
estimator based on the past observations, and set
\(
q^\star_0 \coloneqq \log(1/\alpha)\) and  \(q^\star_1 \coloneqq \frac{1}{1+\varepsilon}\log(1/\alpha).
\)

Denote by $\gD_{1:t}^{0,\mathsf{score}}$ and $\gD_{1:t}^{1,\mathsf{score}}$ the joint distributions of $(s_1,\ldots,s_t)$ when $s_i\sim \mathrm{Exp}(1)$ and $s_i\sim \mathrm{Exp}(1+\varepsilon)$, respectively.
Let the Bernoulli random variable $H\sim \mathrm{Ber}(0.5)$ indicate which distribution generates the sequence $\{s_t\}$.
Then, letting $s$ be a random variable---independent of $s_{1:t}$ conditioned on $H$---such that $s\mid H=0\, \sim \mathrm{Exp}(1)$ and $s\mid H=1\, \sim \mathrm{Exp}(1+\varepsilon)$, and denoting by $\mathbb{P}_0$ (resp.~$\mathbb{P}_1$) the distribution when $H=0$ (resp.~$H=1$), 
we can derive
\begin{equation}\label{eq:split_lower_bd_1}
\begin{aligned}
      \EB_{H, s_{1:t}\sim \gD_{1:t}^{H,\mathsf{score}}}\left[\abs{\PB_H(s> q_t\,|\, q_t) - \alpha}\right] &= 
    \frac{1}{2}\EB
    [\abs{\PB_0(s> q_t\,|\, q_t) - \alpha}] + \frac{1}{2}\EB_{s_{1:t}}[\abs{\PB_{1}(s> q_t\,|\, q_t) - \alpha}]\\
    &=  \frac{1}{2}\EB\left[
    \left(\abs{e^{-q_t} - \alpha} + \big|e^{-(1+\varepsilon)q_t}-\alpha\big|\right)
    \right] \\
    &\ge \frac{1}{2}\EB\left[
    \left(\abs{e^{-q_t} - \alpha} + \big|e^{-(1+\varepsilon)q_t}-\alpha\big|\right)\mathbbm{1}\left\{q_t \in \gK\right\}
    \right]+ \frac{\alpha}{4}\PB\left(
    q_{t}\notin \gK
    \right)\\
    &= \EB_{H,s_{1:t}\sim \gD_{1:t}^{H,\mathsf{score}}}\left[\abs{\PB_H(s> q_t\,|\, q_t)-\alpha}\mathbbm{1}\{q_t \in \gK\}\right] + \frac{\alpha}{4}\PB(q_t\notin \gK),
\end{aligned}
\end{equation}
where we take $\gK\coloneqq  \big[\frac{\log (2/3\alpha)}{1+\varepsilon},\log \frac{2}{\alpha} \big]$. Here, the  inequality above holds due to the elementary fact
\[
\abs{e^{-q}-\alpha} +\big|e^{-(1+\varepsilon)q}-\alpha\big| \ge \frac{\alpha}{2}
\]
as long as $q > \log \frac{2}{\alpha}$ or $q < \frac{\log (2/3\alpha)}{1+\varepsilon}$.

Furthermore, for any $\lambda \in \{1,1+\varepsilon\}$, the mean value theorem tells us the existence of some $\xi$ between $q_t$ and $q^\star_\lambda$ obeying
\[
|\PB_H(s>q_t\,|\, q_t)-\alpha|
=
|\PB_H(s>q_t\,|\, q_t)-\PB_H(s>q^\star_H)|
=
\lambda e^{-\lambda \xi}\,|q_t-q^\star_H|.
\]
Note that when $\varepsilon\le 1/2$, $1\le\lambda\le 1+\varepsilon$ and $\xi \le \log(2/\alpha)$, we have
$\lambda e^{-\lambda \xi}\ge e^{-\lambda \log \frac{2}{\alpha}}\ge e^{-(1+\varepsilon)\log \frac{2}{\alpha}}=(\alpha/2)^{1+\varepsilon} \ge \alpha^{1+\varepsilon}/3$, and as a consequence,
\begin{equation}\label{eq:lin_miscoverage_quantile_err}
|\PB_H(s> q_t\,|\, q_t)-\alpha|\,\mathbbm{1}\{q_t\in \gK\}
\;\ge\;
\frac{\alpha^{1+\varepsilon}}{3}\,|q_t-q^\star_H|\,\mathbbm{1}\{q_t\in\gK\}.
\end{equation}
Substituting \eqref{eq:lin_miscoverage_quantile_err} into \eqref{eq:split_lower_bd_1} leads to
\begin{equation}\label{eq:split_lower_bd_2}
\begin{aligned}
    \EB_{H, s_{1:t}\sim \gD_{1:t}^{H,\mathsf{score}}}\left[\abs{\PB_H(s> q_t\,|\, q_t) - \alpha}\right] &\ge \frac{\alpha}{4}\PB(q_t \notin \gK) + \EB_{H,s_{1:t}\sim \gD_{1:t}^{H,\mathsf{score}}}\left[\abs{\PB_H(s> q_t\,|\, q_t)-\alpha}\mathbbm{1}\{q_t \in \gK\}\right]\\
    &\ge\frac{\alpha^{1+\varepsilon}}{3}\EB\left[
    \abs{q_t - q_H^\star}
    \mathbbm{1}\{q_t \in \gK\}\right] + \frac{\alpha}{4}\PB(q_t \notin \gK).
\end{aligned}
\end{equation}

To further bound the second term on the right-hand side of \eqref{eq:split_lower_bd_2}, let us look at the following test:
\[
\widehat H
=
0~\text{ if } q_t\ge \frac{q^\star_{0}+q^\star_{1}}{2};
\quad
\widehat H
=
1~\text{ otherwise}.
\]
Then given that $H\in\{0,1\}$, it can be derived that
\[
\mathbbm{1}\{\widehat H\neq H\}
\;\le\;
\mathbbm{1}\bigl\{|q_t-q^\star_{H}|>(q_0^\star - q_1^\star)/2\bigr\},
\]
and hence
\begin{equation}\label{eq:quantile_err_ge_test_err}
\begin{aligned}
\EB\bigl[|q_t-q^\star_{H}|\mathbbm{1}\{q_t \in \gK\}\bigr]
&\ge
\frac{q_0^\star - q_1^\star}{2}\,\PB\Bigl(|q_t-q^\star_{H}|>\frac{q_0^\star - q_1^\star}{2};~ q_t \in \gK\Bigr)\\
&\ge
\frac{q_0^\star - q_1^\star}{2}\,\PB(\widehat H\neq H;~ q_t \in \gK).
\end{aligned}
\end{equation}
In addition, if $q_0^\star -q_1^\star = \log\frac{1}{\alpha} - \frac{\log(1/\alpha)}{1+\alpha} = \frac{\varepsilon \log(1/\alpha)}{1+\varepsilon} < \frac{3}{4}$, then we have
\begin{equation}\label{eq:split_lower_bd_3}
    \frac{\alpha}{4}\PB(q_t\notin \gK) \ge 
    \frac{\alpha^{1+\varepsilon}}{4}\PB(q_t\notin \gK; \widehat H \neq H) \ge
    \frac{\alpha^{1+\varepsilon}}{3}(q_0^\star - q_1^\star)\PB(q_t \notin \gK; \widehat H \neq H).
\end{equation}
Substituting \eqref{eq:quantile_err_ge_test_err} and \eqref{eq:split_lower_bd_3} into \eqref{eq:split_lower_bd_2} yields
\begin{equation}\label{eq:split_lower_bd_4}
\begin{aligned}
    \EB_{H, s_{1:t}\sim \gD_{1:t}^{H,\mathsf{score}}}\left[\abs{\PB_H(s> q_t\,|\, q_t) - \alpha}\right] &\ge \frac{\alpha^{1+\varepsilon}}{6}(q_0^\star - q_1^\star)\Big(\PB(q_t\in \gK;~ \widehat H\neq H) + \PB(q_t\notin \gK;~ \widehat H\neq H)\Big)\\
    &= \frac{\alpha^{1+\varepsilon}\varepsilon\log(1/\alpha)}{6(1+\varepsilon)} \PB(\widehat H \neq H).
\end{aligned}
\end{equation}

It remains to lower bound the probability $\PB(\widehat H \neq H)$.
 Le Cam's two-point method~\citep[Theorem 2.2]{tsybakov2009introduction} and Pinsker's inequality~\citep[Lemma 2.5]{tsybakov2009introduction} imply that:
\begin{align*}
\PB(\widehat H \neq H) &=
\frac{1}{2}\PB_0(\widehat H\neq 0)+\frac{1}{2}\PB_1(\widehat H\neq 1)\\
&\ge
\frac{1}{2}\Bigl(1-\TV\bigl(\gD_{1:t}^{0,\mathsf{score}},\gD_{1:t}^{1,\mathsf{score}}\bigr)\Bigr)
\;\ge\;
\frac12\Bigl(1-\sqrt{\tfrac12\,\mathsf{KL}\bigl(\gD_{1:t}^{0,\mathsf{score}}\,\|\,\gD_{1:t}^{1,\mathsf{score}}\bigr)}\Bigr).
\end{align*}
Since the observations are i.i.d.\ within this time segment, one has
\[
\mathsf{KL}\bigl(\gD_{1:t}^{0,\mathsf{score}}\,\|\,\gD_{1:t}^{1,\mathsf{score}}\bigr)
=
t\cdot \mathsf{KL}(\mathrm{Exp}(1)\,\|\,\mathrm{Exp}(1+\varepsilon)).
\]
Moreover, the KL divergence admits the following closed-form expression:
\[
\mathsf{KL}(\mathrm{Exp}(1)\,\|\,\mathrm{Exp}(1+\varepsilon))
=
\log\frac{1}{1+\varepsilon}+(1+\varepsilon)-1
=
\varepsilon-\log(1+\varepsilon),
\]
and for $\varepsilon\in(0,1]$ we have $\varepsilon-\log(1+\varepsilon)\le \varepsilon^2/2$ since
$\log(1+x)\ge x-x^2/2$ for $x\in[0,1]$. Therefore,
\[
\mathsf{KL}\bigl(\gD_{1:t}^{0,\mathsf{score}}\,\|\,\gD_{1:t}^{1,\mathsf{score}}\bigr)\le t\,\varepsilon^2/2.
\]
Choosing $t\le 1/(4\varepsilon^2)$ yields $\mathsf{KL}(\gD_{1:t}^{0,\mathsf{score}}\|\gD_{1:t}^{1,\mathsf{score}})\le 1/8$, hence
$\TV(\gD_{1:t}^{0,\mathsf{score}},\gD_{1:t}^{1,\mathsf{score}})\le 1/4$. Consequently,
\begin{equation}\label{eq:quantile_err_lb}
\PB(\widehat H \neq H) = \frac12\PB_0(\widehat H\neq 0)+\frac12\PB_1(\widehat H\neq 1)\ \ge\ \frac{3}{8}.
\end{equation}
Substituting \eqref{eq:quantile_err_lb} into \eqref{eq:split_lower_bd_4} then yields, for all $\varepsilon\leq 1/ (2\sqrt{t})$, 
\begin{equation}\label{eq:miscoverage_lb_single_t}
\mathop{\EB}_{H\sim \mathrm{Ber}(0.5),~ s_{1:t}\sim \gD_{1:t}^{H,\mathsf{score}}}\Bigl[\bigl| \PB_{H}(s>q_t\mid q_t)-\alpha \bigr|\Bigr]
\;\ge\;
\frac{\alpha^{1+\varepsilon}\varepsilon\log(1/\alpha)}{6}\cdot \frac{3}{8}= \frac{\alpha^{1+\varepsilon}\varepsilon\log(1/\alpha)}{16}.
\end{equation}

\paragraph{Step 3: extension from one time segment to entire horizon.}
We now demonstrate how the single-segment lower bound in Step 2 can be adapted to establish a lower bound for the entire horizon $[T]$. 
%
Recall that, at Step 1, $[T]$ is partitioned into $m$ consecutive time segments $\gI_1,\dots,\gI_m$ obeying $|\gI_j|\asymp T/m$.
Let $H_1,\dots,H_m$ be i.i.d.\ Bernoulli random variables, and construct a random distribution
sequence $\{\gD^{\mathsf{score}}_t\}_{t=1}^T$ by setting, for each time segment $j$ and each $t\in \gI_j$,
\[
\gD^{\mathsf{score}}_t=
\begin{cases}
\mathrm{Exp}(1), & H_j= 0,\\
\mathrm{Exp}(1+\varepsilon), & H_j=1.
\end{cases}
\]

Consider any online procedure producing $q_t=q(s_{1:t})$. Condition on the history up to the end of time segment $j-1$.
Since $H_j$ is independent of the past, the conditional prior on $H_j$ remains uniform over $\{0,1\}$, and the observations within time segment $j$ are i.i.d.\ from $\mathrm{Exp}(1)$ or $\mathrm{Exp}(1+\varepsilon)$ accordingly. Therefore, the single-segment
lower bound \eqref{eq:miscoverage_lb_single_t} applies for all times within time segment $\gI_j$, with the proviso that $4t\varepsilon^2 \le 1$.

To ensure that \eqref{eq:miscoverage_lb_single_t} holds throughout the entire time segment $\mathcal{I}_j$, we impose the condition
\(
4\varepsilon^2\,|\gI_j|\le 1.
\)
Let \(\gF_{j-1} \) denote the \(\sigma\)-field generated by the samples \({s_t}\) observed prior to time segment $j$. 
We can then take the sum of  \eqref{eq:miscoverage_lb_single_t} over $t\in \gI_j$ to reach, for each $\gI_j$,
\[
\EB\bigg[\sum\limits_{t\in \gI_j} \bigl|\PB(s_t>q_t\mid q_t)-\alpha\bigr|\ \bigg|\ \gF_{j-1}\bigg]
\ \ge\
|\gI_j|\cdot \frac{\varepsilon\alpha^{1+\varepsilon}}{16}\cdot 
\log\frac{1}{\alpha},
\]
and hence, summing over $j=1,\dots,m$ and taking expectation over $\{H_j\}_{j=1}^m$ gives
\begin{equation}\label{eq:block_sum_lb_eps}
\begin{aligned}
\sup\limits_{\{\gD^{\mathsf{score}}_k\}_{k=1}^m}\EB\Bigg[\ssum{t}{1}{T}\bigl|\PB(s_t>q_t\mid q_t)-\alpha\bigr|\Bigg]
&\ge
\mathop{\EB}\limits_{H_{1:m},s_{1:T}}\left[\ssum{t}{1}{T}\Big|\PB(s_t > q_t\mid q_t) - \alpha\Big|\right]\\
&= \ssum{k}{1}{m}\sum\limits_{i\in \gI_k} \mathop{\EB}\limits_{H_k, s_{\gI_k}}\left[\Big|\PB(s_i > q_i\mid q_i) - \alpha\Big|\right]\\
&\geq \bigg(\frac{\alpha^{1+\varepsilon}\varepsilon}{16}\log \frac{1}{\alpha} \bigg)\ssum{k}{1}{m} |\gI_k| 
\geq \frac{T\varepsilon \alpha^{1+\varepsilon} }{16} \log \frac{1}{\alpha}, 
\end{aligned}
\end{equation}
provided that $\varepsilon \le \frac{\sqrt{m}}{2\sqrt{T}}$.
To connect this inequality to the advertised minimax lower bounds, we look at the two distribution-shift settings separately.

\begin{itemize}
\item 
\textit{Change-point setting.}
In this case, 
taking $m=N^{\mathsf{cp}}+1$ and $\varepsilon=\sqrt{(N^{\mathsf{cp}}+1)/(4T)}$ in 
\eqref{eq:block_sum_lb_eps} yields
\[
\sup_{\{\gD^{\mathsf{score}}_k\}_{k=1}^m}
\EB \left[\ssum{t}{1}{T}\bigl|\PB(s_t> q_t\mid q_t)-\alpha\bigr|\right]
=\Omega
\left(T\varepsilon\,\alpha^{1+\varepsilon}\log\frac{1}{\alpha}\right)
=\Omega\left(\alpha^{2}\log(1/\alpha)
\sqrt{(N^{\mathsf{cp}}+1)T}\right).
\]

\item \textit{Smooth drift setting.}
In order for \eqref{eq:block_sum_lb_eps} to be applicable in this setting, it suffices to choose $\varepsilon$ such that
\[
\varepsilon \le \min\Bigl\{\frac{\mathsf{KS}_T}{2m},\ \sqrt{\frac{m}{4T}}\Bigr\}.
\]
Now, if $\mathsf{KS}_T \sqrt{T}\ge 1$, then we can choose
\[
m=\mathsf{KS}_T^{2/3}T^{1/3}\,(\ge 1),
\qquad
\varepsilon=\frac{\mathsf{KS}_T^{1/3}}{2T^{1/3}},
\]
which satisfies the above requirement. Plugging this choice into \eqref{eq:block_sum_lb_eps} gives
\begin{align*}
\sup_{\{\gD^{\mathsf{score}}_k\}_{k=1}^m}
\EB \left[\ssum{t}{1}{T}\bigl|\PB(s_t> q_t\mid q_t)-\alpha\bigr|\right]
&=\Omega\left(
T\varepsilon\,\alpha^{1+\varepsilon}\log\frac{1}{\alpha}\right)
=\Omega\left(
\alpha^2\log(1/\alpha)
\,\mathsf{KS}_T^{1/3}T^{2/3}\right).
\end{align*}
On the other hand, if $\mathsf{KS}_T\sqrt{T}< 1$, then one can apply \eqref{eq:miscoverage_lb_single_t} to the entire horizon $[T]$ to arrive at
\begin{align*}
\EB_{H\sim \mathrm{Ber}(0.5),~ s_{1:T}\sim \gD_{1:T}^{H,\mathsf{score}}}&\bigg[\ssum{t}{1}{T}\bigl| \PB_{H}(s>q_t\mid q_t)-\alpha \bigr|\bigg]\\
&=
\EB_{H\sim \mathrm{Ber}(0.5)}\left[\ssum{t}{1}{T}\EB_{s_{1:t}\sim \gD_{1:t}^{H,\mathsf{score}}}\left[\big|\PB(s > q_t \mid q_t)-\alpha\big|\right]\bigg|~ H\right]\\
&\;\ge\;
\frac{\alpha^{1+\varepsilon}T\varepsilon\log(1/\alpha)}{6}\cdot \frac{3}{8}= \frac{\alpha^{1+\varepsilon}T\varepsilon\log(1/\alpha)}{16},
\end{align*}
for all $\varepsilon\le \frac{1}{2\sqrt{T}}$. 
Thus, taking $\varepsilon = \frac{1}{2\sqrt{T}}$ leads to
\[
\sup_{\{\gD^{\mathsf{score}}_k\}_{k=1}^m}
\EB \left[\ssum{t}{1}{T}\bigl|\PB(s_t> q_t\mid q_t)-\alpha\bigr|\right] = \Omega \left(\alpha^2\sqrt{T}\log(1/\alpha)\right).
\]
These two cases taken collectively conclude the proof for the smooth drift setting. 
\end{itemize}

\subsection{Proof of auxiliary lemmas}

\subsubsection{Proof of \Cref{lem:one_round_regret}}\label{sec:prf:lem:one_round_regret}
    Since round $r$ has not been  terminated due to the detection of distribution shift, it follows that 
    %
    \begin{equation}\label{eq:one_round_regret_1}
        \abs{\ssum{l}{i}{j}\big(\mathbbm{1}\{s_{n,r,l} > q_{n,r}\}  - \alpha\big)} \le \sigma_{n,r} \sqrt{j-i+1}
        \le 24 \sqrt{(j-i+1)\log (4\tau_{n,r})}
    \end{equation}
    for every $1\leq i,j \leq T_{r}$. 
    On the event $\gA_{n,r}$, 
     combine the definition \eqref{eq:typical_set_gA_ij_nr} of $\gA_{n,r}$ and Eqn.~\eqref{eq:one_round_regret_1} to reach
    \begin{equation}\label{eq:one_round_regret_2_5}
    \abs{\ssum{l}{i}{j}\big(\PB(s_{n,r,l} > q_{n,r}\mid q_{n,r}) - \alpha\big)
    } \le 30\sqrt{(j-i+1)\log (4\tau_{n,r})},\quad \text{for all } 1\le i< j \le T_{r}.
    \end{equation}
    Below, we look at the two drift settings separately.

    \paragraph{Change-point setting.}
    In this setting, the score distribution remains fixed within each time segment \(\gI_{n,r,k}\).
Consequently, for every \(l\in \gI_{n,r,k}\), the conditional exceedance probability
\(\PB\!\left(s_{n,r,l}> q_{n,r}\mid q_{n,r}\right)\) is identical (i.e., does not depend on \(l\)), which in turn implies that
\[
\sum_{l\in \gI_{n,r,k}} \big|\PB\!\left(s_{n,r,l}> q_{n,r}\mid q_{n,r}\right) - \alpha\big|
= \bigg| \sum_{l\in \gI_{n,r,k}}\big(\PB\!\left(s_{n,r,l}> q_{n,r}\mid q_{n,r}\right) - \alpha\big) \bigg|.
\]
Combining this with Eqn.~\eqref{eq:one_round_regret_2_5} yields
\[
\bigg(
\sum_{l\in \gI_{n,r,k}}\big|\PB(s_{n,r,l} > q_{n,r}\mid q_{n,r}) - \alpha\big| \bigg)\mathbbm{1}\{\gA_{n,r}\} \le 30\sqrt{\abs{\gI_{n,r,k}}\log (4\tau_{n,r})}.
\]
Summing this inequality over \(k=1,\ldots,K_{n,r}\) (i.e., summing over all segments in this round) yields the advertised bound for the change-point setting.

    \paragraph{Smooth drift setting.}
    Consider a given $q_{n,r}$. Partition the time indices $1,2,\ldots,T_{r}$ into $K$ consecutive time segments $\gI_k \coloneqq \{i_{k-1}+1,\ldots,i_{k}\}$, $k=1,\ldots,K$, with $i_0=0$ and $i_{K}=T_{r}$. This partition is chosen so that, for $l\in \gI_1,\gI_3,\ldots,\gI_{2[\frac{K-1}{2}]+1}$, the quantity $\PB(s_{n,r,l}>q_{n,r})-\alpha$ has the same sign; without loss of generality, assume that the signs are positive. Then for any $l\in \gI_2\cup\cdots\cup \gI_{2[\frac{K}{2}]}$, we have $\PB(s_{n,r,l}>q_{n,r})-\alpha<0$. Consequently, the cumulative regret within this round can be expressed by grouping terms with positive and negative signs as follows:
    \begin{align}
        \ssum{l}{1}{T_{r}}& \big|\PB(s_{n,r,l} > q_{n,r}\mid q_{n,r}) - \alpha\big| \notag\\
        &= \sum\limits_{k: k\text{ is odd}} \sum\limits_{l\in \gI_k} \big(\PB(s_{n,r,l} > q_{n,r}\mid q_{n,r}) - \alpha\big) + \sum\limits_{k: k \text{ is even}}\sum\limits_{l\in \gI_k} \big(\alpha - \PB(s_{n,r,l} > q_{n,r}\mid q_{n,r})\big) \notag\\
        &= \ssum{k}{1}{K}\sum\limits_{l\in \gI_k} \big(\PB(s_{n,r,l} > q_{n,r}\mid q_{n,r}) - \alpha\big) + 2\sum\limits_{k: k \text{ is even}}\sum\limits_{l\in \gI_k} \big(\alpha - \PB(s_{n,r,l} > q_{n,r}\mid q_{n,r})\big).
        \label{eq:regret-partition-Ik-even-odd}
    \end{align}

    The first term on the right-hand side of \eqref{eq:regret-partition-Ik-even-odd} can be readily controlled 
    on the event $\gA_{n,r}$; more specifically, it is seen from \eqref{eq:one_round_regret_2_5} that, on the event $\gA_{n,r}$, 
    \begin{align}
    \ssum{k}{1}{K}\sum\limits_{l\in \gI_k} \big(\PB(s_{n,r,l} > q_{n,r}\mid q_{n,r}) - \alpha\big) = \ssum{l}{1}{T_{r}}\big(\PB(s_{n,r,l} > q_{n,r}\mid q_{n,r}) - \alpha\big) \le 30\sqrt{T_{r}\log (4\tau_{n,r})}. 
    \label{eq:sum-Pscore-alpha-13579}
    \end{align}
    %
    We then turn to the second term on the right-hand side of \eqref{eq:regret-partition-Ik-even-odd}. For every $k\in[K]$, set $$A_k \coloneqq \frac{1}{\abs{\gI_k}}\sum_{l\in \gI_k}\big(\alpha - \PB(s_{n,r,l}>q_{n,r}\mid q_{n,r})\big).$$ On the event $\gA_{n,r}$, it agains follows from \eqref{eq:one_round_regret_2_5} that
    \begin{equation}\label{eq:one_round_regret_3}
    \frac{\sqrt{\abs{\gI_k}}}{30\sqrt{\log (4\tau_{n,r})}} A_k = \frac{\sqrt{\abs{\gI_k}}}{30\sqrt{\log (4\tau_{n,r})}}\left(\frac{1}{\abs{\gI_k}}\sum_{l\in \gI_k}\big(\alpha - \PB(s_{n,r,l}>q_{n,r})\mid q_{n,r}\big)\right) \le 1.
    \end{equation}
    This allows one to deduce that, on the event $\gA_{n,r}$,
    %
    \begin{align}
        \sum\limits_{k: k \text{ is even}}\sum\limits_{l\in \gI_k} &\big(\alpha - \PB(s_{n,r,l} > q_{n,r}\mid q_{n,r})\big) 
        = \sum\limits_{k: k \text{ is even}} \abs{\gI_k}\cdot A_k = \sum\limits_{k: k \text{ is even}}\abs{\gI_k}^{\frac{2}{3}}\left(\abs{\gI_k}^{\frac{1}{2}}A_k\right)^{\frac{2}{3}}A_k^{\frac{1}{3}} \notag\\
        &{\overset{\eqref{eq:one_round_regret_3}}{\le}} \sum\limits_{k: k \text{ is even}}30\abs{\gI_k}^{\frac{2}{3}} A_k^{\frac{1}{3}}\sqrt{\log T} \le 30\sqrt{\log T}\bigg(\sum\limits_{k: k \text{ is even}}\abs{\gI_k}\bigg)^{\frac{2}{3}}\bigg(\sum\limits_{k: k \text{ is even}}A_k\bigg)^{\frac{1}{3}},
    \label{eq:one_round_regret_4}
    \end{align}
    %
    where the last inequality follows from H\"older's inequality. 
    This leaves us with two sums to control. 
    \begin{itemize}
    \item 
    Regarding the summation of $|\mathcal{I}_k|$, it is easily seen that
    \begin{equation}\label{eq:one_round_regret_5}
    \sum\limits_{k: k \text{ is even}}\abs{\gI_k} \le \ssum{k}{1}{K}\abs{\gI_k} 
    = T_{r}.
    \end{equation}
    \item 
    Let us now turn to the summation of $A_k$. Given the way we partition the sets $\gI_k$, we see that for any even number $k~(\ge2)$,  $\PB\big(s_{n,r,i_{k-1}} > q_{n,r}\big) > \alpha$. 
 Consequently, for any even  $k$, one can bound $A_k$ as
    \begin{equation}\label{eq:one_round_regret_6}
    \begin{aligned}
        A_k &= \frac{1}{\abs{\gI_k}}\sum_{l\in \gI_k}\big(\alpha - \PB(s_{n,r,l}>q_{n,r}\mid q_{n,r})\big)\\
        &\le \frac{1}{\abs{\gI_k}}\sum_{l\in \gI_k}\big(\PB\big(s_{n,r,i_{k-1}} > q_{n,r}\mid q_{n,r}\big) - \PB(s_{n,r,l}>q_{n,r}\mid q_{n,r})\big) = \frac{1}{\abs{\gI_k}}\sum\limits_{l\in \gI_k}\ssum{i}{i_{k-1}}{l-1}\Delta_{n,r,i},
    \end{aligned}
    \end{equation}
    where we define 
    $$\Delta_{n,r,i}\coloneqq \PB\big(s_{n,r,i} > q_{n,r}\mid q_{n,r}\big) - \PB\big(s_{n,r,i+1} > q_{n,r}\mid q_{n,r}\big).$$ 
    The definition \eqref{eq:defn-KS-dist} of the KS distance tells us that
    \[
    \Delta_{n,r,i} \le \mathsf{KS}\big(s_{n,r,i} , s_{n,r,i+1}\big)\quad  \text{for all }i \in [T_{r}],
    \]
    which combined with Eqn.~\eqref{eq:one_round_regret_6} leads to
    \begin{align}
        A_k 
        &\le \frac{1}{\abs{\gI_k}}\sum\limits_{l\in \gI_k}\ssum{i}{i_{k-1}}{l-1}\mathsf{KS}\big(s_{n,r,i} , s_{n,r,i+1}\big)\notag\\
        &\le \frac{1}{\abs{\gI_k}}\sum\limits_{l\in \gI_k}\ssum{i}{i_{k-1}}{i_k - 1}\mathsf{KS}\big(s_{n,r,i} , s_{n,r,i+1}\big) 
        = \ssum{i}{i_{k-1}}{i_k - 1}\mathsf{KS}\big(s_{n,r,i} , s_{n,r,i+1}\big).
        \label{eq:one_round_regret_777}
    \end{align}
    Let $s_{n,r,0}=s_{n,r,1}$. Summing over all even $k$ yields 
    \begin{equation}\label{eq:one_round_regret_7}
    \begin{aligned}
        \sum\limits_{k: k \text{ is even}}A_k \le \ssum{k}{1}{K}\ssum{i}{i_{k-1}}{i_k-1}\mathsf{KS}\big(s_{n,r,i}, s_{n,r,i+1}\big) = \ssum{i}{0}{T_{r}}\mathsf{KS}\big(s_{n,r,i} ,  s_{n,r,i+1}\big).
    \end{aligned}
    \end{equation}
    
    \end{itemize} 
    Putting Eqns.~\eqref{eq:one_round_regret_5} and \eqref{eq:one_round_regret_7} together yields
    \[
    \left(\sum\limits_{k: k \text{ is even}}\abs{\gI_k}\right)^{\frac{2}{3}}\left(\sum\limits_{k: k \text{ is even}}A_k\right)^{\frac{1}{3}} \le T_{r}^{\frac{2}{3}}\left(\ssum{i}{1}{T_{r}}\mathsf{KS}\big(s_{n,r,i}, s_{n,r,i+1}\big)\right)^{\frac{1}{3}} = \big(\mathsf{KS}_{n,r}^{\mathsf{round}}\big)^{\frac{1}{3}}T_{r}^{\frac{2}{3}},
    \]
    which taken together with Eqns.~\eqref{eq:regret-partition-Ik-even-odd}, \eqref{eq:sum-Pscore-alpha-13579} and \eqref{eq:one_round_regret_4} establishes that
    \begin{align*}
         \ssum{l}{1}{T_{r}}\Big|\PB(s_{n,r,l} > q_{n,r}\mid q_{n,r}) - \alpha\Big|\cdot \mathbbm{1}\{\gA_{n,r}\}
        &\le 30\sqrt{\log T}\left(\sqrt{T_r} +  \big(\mathsf{KS}_{n,r}^{\mathsf{round}}\big)^{\frac{1}{3}}T_{r}^{\frac{2}{3}}\right) = \widetilde{O}\left(\sqrt{T_{r}} + \big(\mathsf{KS}_{n,r}^{\mathsf{round}}\big)^{\frac{1}{3}}T_{r}^{\frac{2}{3}}\right).
    \end{align*}

\subsubsection{Proof of \Cref{lem:one_stage_regret}}\label{sec:prf:lem:one_stage_regret}

Consider stage $n$, and
    denote by $t_n$  the number of iterations in the $r_n$-th round (recall that $r_n$ denotes the index of the last round of stage $n$). Observe that for any $r\in [r_n-1]$ (resp.~for $r=r_n$), no initiation of a new stage---i.e., no detection of distribution drift---is triggered within iterations $\{1,\ldots, T_{r}\}$ (resp.~$\{1,\ldots, t_n-1\}$). In what follows, we look at the two drift settings separately.

\paragraph{Change-point setting.}
Note that the collection of time segments \(\{\gI_{n,r,k}\}\) described in  \Cref{sec:additional-notation-split-thm-UB} can be viewed as a refinement of
\(\{\gI_{n,j}\}_{j=1}^{J_n}\); for instance, a given $\mathcal{I}_{n,j}$ might appear in more than one round, possibly due to imperfect drift detection. 
For each \(j\in[J_n]\), denote by 
$r^{(j)}$  the index of the first round that overlaps with the segment \(\gI_{n,j}\); it is straightforward to see that the last round that overlaps with \(\gI_{n,j}\) cannot exceed $r^{(j+1)}$ (here, if this is already the last time segment in stage $n$, we can simply let $r^{(j+1)}$ be the last round of this stage). Note that 
\(\gI_{n,j}\) may share its first and/or last
round with \(\gI_{n,j-1}\) or \(\gI_{n,j+1}\). Therefore, we have
\begin{align}
\gI_{n,j} \subseteq \Bigg\{\bigcup\limits_{r=r^{(j)}+1}^{r^{(j+1)}-1} \bigcup\limits_{k=1}^{K_{n,r}}\gI_{n,r,k}\Bigg\}\cup \left\{\gI_{n,r^{(j)}, K_{n,r^{(j)}}}\right\}\cup \left\{\gI_{n,r^{(j+1)},1}\right\},
\end{align}
which in turn gives
\begin{equation}\label{eq:one_stage_cp_0}
\ssum{r}{1}{r_n}\ssum{k}{1}{K_{n,r}}\sqrt{\abs{\gI_{n,r,k}}} \le 
\ssum{j}{1}{J_n}\left\{\ssum{r}{r^{(j)}+1}{r^{(j+1)}-1}\ssum{k}{1}{K_{n,r}}\sqrt{\abs{\gI_{n,r,k}}}+ \sqrt{\big|
\gI_{n,r^{(j)}, K_{n,r^{(j)}}}
\big|} + \sqrt{\abs{\gI_{n,r^{(j+1)},1}}}\right\}.
\end{equation}

Recognizing that each of the intermediate rounds \(r^{(j)}+1,\ldots,r^{(j+1)}-1\) is fully contained within $\mathcal{I}_{n,j}$, we see that, by construction, each of these rounds also contains a single time segment from the collection $\{\mathcal{I}_{n,r,k}\}$. 
This means that for each $r\in \{r^{(j)}+1,\ldots,r^{(j+1)}-1\}$, we have
\[
\ssum{k}{1}{K_{n,r}} \sqrt{\abs{\gI_{n,r,k}}}
= \sqrt{\abs{\gI_{n,r,1}}}
\leq \sqrt{T_r}.
\]
Consequently, for
each \(j\in[J_n]\), we can bound
\begin{equation}\label{eq:one_stage_cp_1}
    \ssum{r}{r^{(j)}+1}{r^{(j+1)-1}} \ssum{k}{1}{K_{n,r}} \sqrt{\abs{\gI_{n,r,k}}}+ 
    \sqrt{\big|\gI_{n,r^{(j)}, K_{n,r^{(j)}}}\big|}
    + \sqrt{\abs{\gI_{n,r^{(j+1)},1}}}
    {\le} \ssum{r}{r^{(j)}+1}{r^{(j+1)-1}}\sqrt{T_{r}} + 2\sqrt{\abs{\gI_{n,j}}},
\end{equation}
where the last inequality holds since \(\gI_{n,r^{(j)}, K_{n,r^{(j)}}} \cup \gI_{n,r^{(j+1)},1} \subseteq \gI_{n,j}\).

Next, we bound the summation of $\sqrt{T_r}$ on the right-hand side of  \eqref{eq:one_stage_cp_1}. 
If $r^{(j)}+1 > r^{(j+1)}-1$, then this summation term is equal to \(0\); otherwise, it can be seen that (by construction)
$$T_{r^{(j+1)}-1}=\abs{\gI_{n,r^{(j+1)}-1, 1}} \le \abs{\gI_{n,j}},$$  
which implies that
\begin{equation}\label{eq:one_stage_cp_2}
    \ssum{r}{r^{(j)}+1}{r^{(j+1)-1}}\sqrt{T_{r}} \le  \ssum{r}{1}{r^{(j+1)-1}}\sqrt{T_{r}}
    \overset{\mathrm{(a)}}{=} \ssum{r}{1}{r^{(j+1)-1}} 3^{\frac{r}{2}}
    \leq 3\cdot 3^{\frac{r^{(j+1)}-1}{2}}
    \overset{\mathrm{(b)}}=  3\sqrt{T_{r^{(j+1)}-1}} \le 3\sqrt{\abs{\gI_{n,j}}}.
\end{equation}
Here, (a) and (b) are valid due to our choice $T_r=3^r$ and the fact that, except for the last round of this stage, the $r$-th round has time length exactly equal to $T_r$.
%

Invoking \Cref{lem:one_round_regret} as well as \eqref{eq:auxiliary-sl-exceed}, we can demonstrate that
\begin{align*}
    &\ssum{r}{1}{r_n} \bigg\{\ssum{l}{1}{T_r}\big|\PB(s_{n,r,l} > q_{n,r}\mid q_{n,r}) - \alpha\big| \bigg\} \mathbbm{1}\{\gA_{n,r}\} \notag\\
    & \leq
        \ssum{r}{1}{r_n-1} \bigg\{\ssum{l}{1}{T_r}\big|\PB(s_{n,r,l} > q_{n,r}\mid q_{n,r}) - \alpha\big| \bigg\} \mathbbm{1}\{\gA_{n,r}\}
        + 
        \bigg\{
        \ssum{l}{1}{t_n-1}\big|\PB(s_{n,r_n,l} > q_{n,r_n}\mid q_{n,r_n}) - \alpha\big| \bigg\} \mathbbm{1}\{\gA_{n,r_n}\} + 1\notag\\
    & \leq \widetilde{O}\bigg( \ssum{r}{1}{r_n}\ssum{k}{1}{K_{n,r}}\sqrt{\abs{\gI_{n,r,k}}} \bigg) + 1
    \leq \widetilde{O}\bigg( \ssum{j}{1}{J_n} \sqrt{\abs{\gI_{n,j}}} \bigg),
\end{align*}
where the last line follows from \Cref{lem:one_round_regret} and the fact that no distribution shift has been detected before the last iteration of stage $n$. This taken collectively with \eqref{eq:one_stage_cp_0}-\eqref{eq:one_stage_cp_2}  then yields
\begin{align*}
    \ssum{r}{1}{r_n} \bigg\{\ssum{l}{1}{T_r}\big|\PB(s_{n,r,l} > q_{n,r}\mid q_{n,r}) - \alpha\big| \bigg\} \mathbbm{1}\{\gA_{n,r}\}
    \leq \widetilde{O}\bigg( \ssum{r}{1}{r_n}\ssum{k}{1}{K_{n,r}}\sqrt{\abs{\gI_{n,r,k}}} \bigg) 
    \leq \widetilde{O}\bigg( \ssum{j}{1}{J_n} \sqrt{\abs{\gI_{n,j}}} \bigg)
\end{align*}
as claimed. 
\paragraph{Smooth drift setting.}
To begin with, Lemma~\ref{lem:one_round_regret} tells us that
    \begin{subequations}\label{eq:one_stage_regret_1}
    \begin{align}
        \bigg\{ \ssum{l}{1}{T_{r}} \big| \PB(s_{n,r,l} > q_{n,r}\mid q_{n,r}) - \alpha\big| \mathbbm{1} \bigg\} \{\gA_{n,r}\} &= \widetilde{O}\left(\sqrt{T_{r}} + \big(\mathsf{KS}_{n,r}^{\mathsf{round}}\big)^{\frac{1}{3}}T_{r}^{\frac{2}{3}}\right),\quad r = 1,2, \ldots, r_n - 1
        \end{align}
        and
        \begin{align}
        \ssum{l}{1}{t_{n}} \big|\PB(s_{n,r_n,l} > q_{n,r_n}\mid q_{n,r_n}) - \alpha\big|\mathbbm{1}\{\gA_{n,r_n}\}
        &\leq 
        \ssum{l}{1}{t_{n}-1} \big|\PB(s_{n,r_n,l} > q_{n,r_n}\mid q_{n,r_n}) - \alpha\big|\mathbbm{1}\{\gA_{n,r_n}\}
        + 1 \notag\\
        &= \widetilde{O}\left(\sqrt{t_{n}} + \big(\mathsf{KS}_{n,r_n}^{\mathsf{round}}\big)^{\frac{1}{3}}t_{n}^{\frac{2}{3}}\right).
    \end{align}
    \end{subequations}
    %
   Sum Eqn.~\eqref{eq:one_stage_regret_1} over all rounds in this stage to arrive at:
    \begin{align*}
        \ssum{r}{1}{r_n-1}\ssum{l}{1}{T_{r}} &\big|\PB(s_{n,r,l} > q_{n,r}\mid q_{n,r}) - \alpha\big|\mathbbm{1}\{\gA_{n,r}\} + \ssum{l}{1}{t_n}\big|\PB(s_{n,r_n,l}> q_{n,r_n}\mid q_{n,r_n}) -\alpha\big|\mathbbm{1}\{\gA_{n,r_n}\}\\
        &= \ssum{r}{1}{r_n-1}\widetilde{O}\left(\sqrt{T_{r}} + \big(\mathsf{KS}_{n,r}^{\mathsf{round}}\big)^{\frac{1}{3}}T_{r}^{\frac{2}{3}}\right) + \widetilde{O}\left(\sqrt{t_n} + \big(\mathsf{KS}_{n,r_n}^{\mathsf{round}}\big)^{\frac{1}{3}}t_n^{\frac{2}{3}}\right)\\
        &\overset{\mathrm{(a)}}{\leq } \widetilde{O}\left(\ssum{r}{1}{r_n}3^{\frac{r}{2}} + \ssum{r}{1}{r_n - 1}\big(\mathsf{KS}_{n,r}^{\mathsf{round}}\big)^{\frac{1}{3}}T_{r}^{\frac{2}{3}} + \big(\mathsf{KS}_{n,r_n}^{\mathsf{round}}\big)^{\frac{1}{3}}t_n^{\frac{2}{3}}\right)\\
        &\overset{\mathrm{(b)}}{\leq }\widetilde{O}\left(3^{\frac{r_n}{2}} + \bigg(\ssum{r}{1}{r_n}\big(\mathsf{KS}_{n,r}^{\mathsf{round}}\big)\bigg)^{\frac{1}{3}}S_n^{\frac{2}{3}}\right)
        \overset{\mathrm{(c)}}{\leq} \widetilde{O}\left(\sqrt{S_n} + \mathsf{KS}_{n,S_n}^{\frac{1}{3}}S_n^{\frac{2}{3}}\right).
    \end{align*}
    Here, (a) holds since $T_{r} = 3^r$ and $t_n \le 3^{r_n}$, (b) follows from H\"older's inequality, and (c) holds since the total number $S_n$ of time points within stage $n$ satisfies
    \[
    \sqrt{S_n} \ge \bigg(\ssum{r}{1}{r_n-1}3^r\bigg)^{1/2} \asymp 3^{\frac{r_n}{2}}.
    \]
This taken together with \eqref{eq:auxiliary-sl-exceed} concludes the proof.


\subsubsection{Proof of \Cref{lem:seg-separation}}\label{sec:prf:lem:seg-separation}
We establish this lemma by contradiction. Suppose that round \(r_n-1\) and round \(r_n\) are completely contained within the same time segment from the collection \(\{\gI_k\}_{k=1}^{N^{\mathsf{cp}}+1}\). According to the procedure of \driftocp, on the event \(\gA_{n,r_n-1}\cap\gA_{n,r_n}\) there exists an index
\(j_n\) such that
\begin{subequations}
\begin{align}
\frac{1}{T_{r_n-1}}
\Bigg|\sum_{l=1}^{T_{r_n-1}}\Bigl(\mathbbm{1}\{s_{n,r_n-1,l}>q_{n,r_n}\}-\alpha\Bigr)\Bigg|
&\le \frac{1}{T_{r_n-1}}, \label{eq:pre-bound}\\
\frac{1}{t_n-j_n+1}
\Bigg|\sum_{l=j_n}^{t_n}\Bigl(\mathbbm{1}\{s_{n,r_n,l}>q_{n,r_n}\}-\alpha\Bigr)\Bigg|
&> \frac{24\sqrt{\log (4\tau_{n,r_n})}}{\sqrt{t_n-j_n+1}}, \label{eq:det-stat}
\end{align}
\end{subequations}
where \eqref{eq:pre-bound} is valid since $q_{n,r_n}$ is taken to be the $\alpha$-empirical-quantile of the set $\{s_{n,r_n-1,l}\}_{l=1}^{T_{r_n-1}}$, and in \eqref{eq:det-stat} we use the detection threshold $\sigma_{n,r}=24\sqrt{\log (4\tau_{n,r})}$.

Since the two rounds lie within the same time
segment from the collection \(\{\gI_k\}_{k=1}^{N^{\mathsf{cp}}+1}\), the scores in these two rounds are identically distributed.
Let \(s\) denote an independent copy of the score from this segment. Then, on
\(\gA_{n,r_n-1}\cap\gA_{n,r_n}\) (cf.~\eqref{eq:typical_set_gA_ij_nr}), we can invoke \eqref{eq:pre-bound} and the triangle inequality to obtain
\begin{align}\label{eq:pop-upper}
&\Bigl|\PB(s>q_{n,r_n}\mid q_{n,r_n})-\alpha\Bigr|
= 
\frac{1}{T_{r_n-1}}
\Bigg|\sum_{l=1}^{T_{r_n-1}}\Bigl(\mathbb{P}(s_{n,r_n-1,l}>q_{n,r_n}\mid q_{n,r_n})-\alpha\Bigr)\Bigg| 
\notag\\
&
\le 
\frac{1}{T_{r_n-1}}
\Bigg|\sum_{l=1}^{T_{r_n-1}}\Bigl(\mathbbm{1}\{s>q_{n,r_n}\}-\alpha\Bigr)\Bigg| 
+ 
\frac{1}{T_{r_n-1}}
\Bigg|\sum_{l=1}^{T_{r_n-1}}\Bigl(\mathbbm{1}\{s_{n,r_n-1,l}>q_{n,r_n}\}-\mathbb{P}(s>q_{n,r_n}\mid q_{n,r_n})\Bigr)\Bigg| 
\notag\\ 
&\le \frac{1}{T_{r_n-1}}+\frac{6\sqrt{\log \tau_{n+1}}}{\sqrt{T_{r_n-1}}}
\leq \frac{7\sqrt{\log \tau_{n+1}}}{
\sqrt{T_{r_{n-1}}}}
\leq \frac{14\sqrt{\log \tau_{n+1}}}{
\sqrt{T_{r_{n}}}}. 
\end{align}
In the meantime, applying \eqref{eq:det-stat} again on the same event and invoking the triangle inequality gives
%
\begin{align}
\Bigl|\PB(s>q_{n,r_n}\mid q_{n,r_n})-\alpha\Bigr|
&= \frac{1}{t_n - j_n +1}\abs{\ssum{l}{j_n}{t_n}\Bigl(\PB(s_{n,r_n,l}>q_{n,r_n}\mid q_{n,r_n})-\alpha\Bigr)}\notag\\
&\ge \frac{1}{t_n-j_n+1}
\Bigg|\sum_{l=j_n}^{t_n}\Bigl(\mathbbm{1}\{s_{n,r_n,l}>q_{n,r_n}\}-\alpha\Bigr)\Bigg|\notag\\
&\quad - \frac{1}{t_n-j_n+1}
\Bigg|\sum_{l=j_n}^{t_n}\Bigl(\mathbbm{1}\{s_{n,r_n,l}>q_{n,r_n}\}-\PB(s_{n,r_n,l}>q_{n,r_n}\mid q_{n,r_n})\Bigr)\Bigg|
\notag\\
&\overset{\mathrm{(a)}}{\ge} \frac{24\sqrt{\log (4\tau_{n,r_n})}}{\sqrt{t_n-j_n+1}}
      -\frac{6\sqrt{\log \tau_{n+1}}}{\sqrt{t_n-j_n+1}}
\ge \frac{18\sqrt{\log \tau_{n+1}}}{\sqrt{T_{r_n}}}, 
\label{eq:pop-lower}
\end{align}
where (a) follows by combining \eqref{eq:det-stat} with the event $\gA_{n,r_n}$. 
However, \eqref{eq:pop-lower} contradicts \eqref{eq:pop-upper}, which in turn completes the proof.

\section{Detailed proofs in \Cref{sec:full_conf}}\label{sec:prf_of_sec_full_conf}

This section is devoted to establishing the main results in \Cref{sec:full_conf}.  
Throughout this section, we define, 
 for any cumulative distribution function (CDF) $F$,  the quantile function
 \begin{align}
 \Q_{1-\alpha}(F) \coloneqq \inf\{x\in \RB: F(x) \ge 1-\alpha\}.
 \end{align}
 Also, we denote by $\gC(\cdot\mid \gS^{\mathsf{cal}}, \gS^{\mathsf{train}})$ the prediction-set mapping constructed via \eqref{eq:full_conformal_set}, where $\gS^{\mathsf{train}}$ is used to fit the model and $\gS^{\mathsf{cal}}$ is used to form the quantile.

\subsection{Proof of Proposition~\ref{prop:training_conditioned_cov}}\label{sec:prf:prop:training_conditioned_cov}

We first present the proof of Proposition~\ref{prop:training_conditioned_cov}, which concerns the training-conditional coverage guarantees for standard full conformal methods.  
Before embarking on the proof, we introduce the following convenient notation (when there is no ambiguity), which shall be used repeatedly throughout Section~\ref{sec:prf:prop:training_conditioned_cov}.

\begin{definition}[Basic notation]\label{def:basic_notation}
We introduce the following notation, all conditioned on the realization $Z_{m+1:n}^{\mathsf{train}}=z_{m+1:n}^{\mathsf{train}}$ (i.e., the portion of the training set that is disjoint from the calibration set). 
\begin{itemize}
    \item For any dataset $\gS\subseteq \gX\times\RB$, let
    $\widehat{\mu}_{\gS}(\cdot)\coloneqq \gA\big(\gS\cup z_{m+1:n}^{\mathsf{train}}\big)$ be the fitted model trained obtained by algorithm $\mathcal{A}$ on $\mathcal{S}\cup z_{m+1:n}^{\mathsf{train}}$.
    
    \item For any dataset $\gS\subseteq \gX\times\RB$ and any $Z=(X,Y)$, let
    $\widehat{\mu}^{Z}_{\gS}(\cdot)\coloneqq \widehat{\mu}_{\gS\cup \{Z\}}(\cdot) $  be the fitted model trained obtained by algorithm $\mathcal{A}$ on $\mathcal{S}\cup \{(X,Y)\}\cup z_{m+1:n}^{\mathsf{train}}$.
    \item For any dataset $\gS\subseteq \gX\times\RB$ and any $Z=(X,Y)$, $Z'=(X',Y')$, define the scores
    \begin{equation}
    s_{\gS}(Z')\coloneqq \bigl|Y'-\widehat{\mu}_{\gS}(X')\bigr| \qquad \text{and}\qquad  
    s^{Z}_{\gS}(Z')\coloneqq \bigl|Y'-\widehat{\mu}^{Z}_{\gS}(X')\bigr|.
    \label{eq:defn-sS-Z-1357}
    \end{equation}
    %
\end{itemize}
\end{definition}
\begin{remark}
Note that we introduce $\widehat{\mu}_{\gS}^{Z}(\cdot)$ in addition to $\widehat{\mu}_{\gS}(\cdot)$. This is because, in the full conformal algorithm, the target sample $Z$ is used for both model fitting and construction of the calibration quantile. To emphasize this role and distinguish it from the pretrained-score setting, we adopt a separate notation.
\end{remark}

\subsubsection{Key lemmas}

We first single out two key lemmas that play a pivotal role in the proof of Proposition~\ref{prop:training_conditioned_cov}. Here and throughout, we take $L=L_1L_2$. 

The first lemma characterizes the discrepancy between the tail distribution of the scores conditional on a random calibration set and the corresponding tail distribution obtained after averaging over the randomness of the calibration set, provided that a stable learning algorithm is used for model fitting. The proof is deferred to  \Cref{sec:prf:lem:cond_prob_concentration}. 
\begin{lemma}\label{lem:cond_prob_concentration}
    Consider the same setting as in \Cref{prop:training_conditioned_cov}. Let $\widehat{\mu}^{(X,Y)}(\cdot)$ denote a fitted model trained on the data $Z_{1:n}^{\mathsf{train}}$ together with the target sample $Z=(X,Y)$, and assume that $\widehat{\mu}^{(X,Y)}(\cdot)$ satisfies Assumption~\ref{ass:fair_alg} with coefficient $L_2$. Further, suppose that for each $i\in [m]$, $Z_i^{\mathsf{cal}} = (X_i^{\mathsf{cal}}, Y_i^{\mathsf{cal}})$ is independently drawn from $\gD_i$, and let the target pair $Z=(X,Y)\sim \gD$.   
Then, for any $\delta \in (0,1)$ and conditional on any given realization $Z_{m+1:n}^{\mathsf{train}}=z_{m+1:n}^{\mathsf{train}}$, we have
\begin{align}
&\sup\limits_{x\in \RB}\left\{
\bigg|
\PB_{\gD}\left(\big|Y-\widehat{\mu}^{(X,Y)}(X)\big|> x \,\Big|\, Z_{1:m}^{\mathsf{cal}}, 
Z_{m+1:n}^{\mathsf{train}}=z_{m+1:n}^{\mathsf{train}} \right)  -  \PB_{\gD_{1:m}\times \gD}\left(\big|Y-\widehat{\mu}^{(X,Y)}(X)\big|> x \right)
\bigg|
\right\} \notag\\
& \qquad\qquad \le \frac{16{L}}{n}\sqrt{m\log \frac{1}{\delta}}
\end{align}
with probability at least $1 -\delta$. Here, we adopt the notation 
\[
\PB_{\gD_{1:m}\times \gD}\left(\big|Y-\widehat{\mu}^{(X,Y)}(X)\big|> x \right) \coloneqq  \EB_{Z_{1:m}^{\mathsf{cal}}}\left[
\PB_{\gD}\left(\big|Y-\widehat{\mu}^{(X,Y)}(X)\big|> x \,\Big|\, Z_{1:m}^{\mathsf{cal}},Z_{m+1:n}^{\mathsf{train}}=z_{m+1:n}^{\mathsf{train}}\right)
\right].
\]
\end{lemma}

Another key lemma establishes a high-probability upper bound on the deviation of the average empirical scores from their mean over a given time window.  In contrast to the pretrained-score setting, full conformal methods induce complicated statistical dependency among the scores $\{s_i\}$, leading to additional technical difficulties. The proof of this lemma is provided in \Cref{sec:prf:lem:empr_prob_concentration}.   
\begin{lemma}\label{lem:empr_prob_concentration}
Consider the same setting as in \Cref{prop:training_conditioned_cov}. Let 
$\widehat{\mu}(\cdot)$ represent a fitted model trained on $Z_{1:n}^{\mathsf{train}} $ satisfying Assumption~\ref{ass:fair_alg} with coefficient $L_2$.
Further, suppose that for each $i\in [m]$, $Z_i^{\mathsf{cal}} = (X_i^{\mathsf{cal}}, Y_i^{\mathsf{cal}}) $ is independently drawn from $ \gD_i$.
For every $i = 1,\ldots, m$, we let
\begin{align*}
s_{i} & \coloneqq\big|Y_{i}^{\mathsf{cal}}-\widehat{\mu}(X_{i}^{\mathsf{cal}})\big|,\qquad i=1,\ldots,m.
\end{align*}
Then, for any $\delta\in(0,1)$ and conditional on any realization $Z_{m+1:n}^{\mathsf{train}}=z_{m+1:n}^{\mathsf{train}}$, the following event 
\begin{align*}
\sup_{x\in \RB}\Biggl\{\frac{1}{m}
\bigg|\,
\sum_{i=1}^{m}\Big(
\mathbbm{1}\left\{s_i\le x\right\}
&- \PB_{\gD_{1:m}}\big(s_i \le x\big)
\Big)
\,\biggr|
\Biggr\}\le\;
24\sqrt{\frac{\log(10/\delta)}{m}} + \frac{24L}{n}\sqrt{m\log \left(\frac{10m}{\delta} + n \right)} 
\end{align*}
happens with probability at least $1 - \delta$.
\end{lemma}
\subsubsection{Proof of Proposition~\ref{prop:training_conditioned_cov}}
For notational convenience,  we write $Z_{1:m}$ in place of $Z_{1:m}^{\mathsf{cal}}$ throughout this proof when it is clear from the context.
Fix an auxiliary sample $z_0=(x_0,y_0)$, which shall be treated as deterministic in the following. We introduce several addition notation: 
\begin{itemize}
    \item $s_i \coloneqq s_{Z_{1:m}\cup \{z_0\}}(Z_i) = \bigl|\, Y_i - \widehat{\mu}_{Z_{1:m}\cup \{z_0\}}(X_i)\,\bigr|$ for $i = 1,\ldots, m$, and $s_{\mathsf{test}} \coloneqq \bigl|\, Y - \widehat{\mu}_{Z_{1:m}\cup\{z_0\}}(X)\,\bigr|$, where $(X,Y)$ is not used for model fitting; 

 \item  $s_i^{(X,Y)} \coloneqq s_{Z_{1:m}}^{Z}(Z_i) = \bigl|\, Y_i - \widehat{\mu}^{(X,Y)}_{Z_{1:m}}(X_i)\,\bigr|$ for $i = 1,\ldots, m$, and $ s_{\mathsf{test}}^{(X,Y)} \coloneqq \bigl|\, Y - \widehat{\mu}_{Z_{1:m}}^{(X,Y)}(X)\,\bigr|$, where $(X,Y)$ is used for for model fitting;
    

    \item $\widehat{\Q}_{1-\alpha} \coloneqq\Q_{1-\alpha}\left(\frac{1}{m+1}\Bigl\{\delta\{s_\mathsf{test}^{(X,Y)}\} +\ssum{i}{1}{m}\delta\{s_{i}^{(X,Y)}\} \Bigr\}\right)$, which indicates the quantile when $(X,Y)$ is also used for model fitting; 
    
    \item $ \widetilde{\Q}_{1-\alpha}\coloneqq\Q_{1-\alpha}\left(\frac{1}{m+1}\Bigl\{\delta\{s_\mathsf{test}^{(X,Y)}\} + \ssum{i}{1}{m}\delta\{s_{i}\} \Bigr\}\right)$; note that except for $s_{\mathsf{test}}^{(X,Y)}$, the remaining scores  $\{s_i\}_{i=1}^m$ are computed when $(X,Y)$ is not used for training;

    \item $F_{\mathsf{test}}(u; z_{1:m}) \coloneqq \PB_{(X,Y)\sim \gD}\Bigl(\bigl|\,
    Y - \widehat{\mu}_{z_{1:m}}^{(X,Y)}(X)
    \,\bigr|\le u\Bigr),~ F_{\mathsf{test}}(u)\coloneqq \EB_{Z_{1:m}\sim \gD_{1:m}}[F_{\mathsf{test}}(u;Z_{1:m})]$;
    \item $F_i(u)\coloneqq \PB_{\gD\times \gD_{1:m}}\bigl(s_i^{(X,Y)} \le u\bigr)$ and $ F_i^0(u)\coloneqq \PB_{\gD_{1:m}}\bigl(s_i \le u\bigr)$ for  $i=1,\ldots, m$.
\end{itemize}

\paragraph{Step 1: eliminating the dependence of the fitted model on $(X,Y)$.}
The first step of the proof is to examine the effect of removing the dependence of the fitted model $\widehat{\mu}^{(X,Y)}(\cdot)$ on the target sample $(X,Y)$.
To be precise, consider the discrepancy between 
\[
\PB_{\gD}(Y \in \gC(X)\mid Z_{1:m}) = \PB_{\gD}(s_{\mathsf{test}}^{(X,Y)}\le \widehat{\Q}_{1-\alpha}\mid Z_{1:m}) 
\qquad \text{and} \qquad \PB_{\gD}(s_{\mathsf{test}}^{(X,Y)}\le \widetilde{\Q}_{1-\alpha}).
\]
For the two score sets $\{s_{\mathsf{test}}^{(X,Y)}\}\cup \{s_i^{(X,Y)}\}_{i=1}^m$ and $\{s_{\mathsf{test}}^{(X,Y)}\}\cup \{s_i\}_{i=1}^m$, Assumption~\ref{ass:fair_alg} tells us that
\[
\max\limits_{i\in [m]}\bigl\{\bigl|\,
s_{i}^{(X,Y)} - s_i
\,\bigr|\bigr\}\le \max\limits_{i\in[m]}
\bigl\{\bigl|\,
\widehat{\mu}_{Z_{1:m}}^{(X,Y)}(X_i) - \widehat{\mu}_{Z_{1:m}}(X_i)
\,\bigr|\bigr\} \le \frac{L_2}{n}.
\]
Then by virtue of \citet[Lemma B.1]{han2024distribution}, we obtain 
\begin{align}
\bigl|\,\widehat{\Q}_{1-\alpha}-\widetilde{\Q}_{1-\alpha}\,\bigr| \le \frac{L_2}{n}
\label{eq:Q-hat-Q-tilde-diff}
\end{align}
for any $(X,Y)$ and $ Z_{1:m}$. Combining this with Assumption~\ref{ass:lip_cond_distr} yields
\begin{equation}\label{eq:decorr_XY_1}
\begin{aligned}
\Bigl|\, \PB_{\gD}(Y \in \gC(X)\mid Z_{1:m}) &- \PB_{\gD}\big(s_{\mathsf{test}}^{(X,Y)}\le \widetilde{\Q}_{1-\alpha}\mid Z_{1:m}\big)\,\Bigr|\\
&= \Bigl|\, \PB_{\gD}\big(s_{\mathsf{test}}^{(X,Y)}\le \widehat{\Q}_{1-\alpha}\mid Z_{1:m}\big) - \PB_{\gD}\big(s_{\mathsf{test}}^{(X,Y)}\le \widetilde{\Q}_{1-\alpha}\mid Z_{1:m}\big)\,\Bigr|\\
&\le \PB_{\gD}\Big(s_{\mathsf{test}}^{(X,Y)}\in \big[ \widetilde{\Q}_{1-\alpha} - \bigl|\,\widehat{\Q}_{1-\alpha}-\widetilde{\Q}_{1-\alpha}\,\bigr|, \widetilde{\Q}_{1-\alpha} + \bigl|\,\widehat{\Q}_{1-\alpha}-\widetilde{\Q}_{1-\alpha}\,\bigr|\big]\mid Z_{1:m}\Big)\\
&\le \PB_{\gD}\Big(s_{\mathsf{test}}^{(X,Y)}\in \Big[ \widetilde{\Q}_{1-\alpha} - \frac{L_2}{n}, \widetilde{\Q}_{1-\alpha} + \frac{L_2}{n}\Big]\,\Big|\, Z_{1:m}\Big)
\\
&\leq \frac{4L_1L_2}{n} = \frac{4L}{n},
\end{aligned}
\end{equation}
where the penultimate line results from \eqref{eq:Q-hat-Q-tilde-diff}, and the last line is due to Assumption~\ref{ass:lip_cond_distr}. This inequality allows us to switch attention to $\PB_{\gD}(s_{\mathsf{test}}^{(X,Y)}\le \widetilde{\Q}_{1-\alpha}\mid Z_{1:m})$.

\paragraph{Step 2: replacing $\widetilde{\Q}_{1-\alpha}$ with an adjusted quantile independent of $(X,Y)$.}
Note that by definition,  $\widetilde{\Q}_{1-\alpha}$ still depends on the test score $s_{\mathsf{test}}^{(X,Y)}$, which motivates us to consider replacing $\widetilde{\Q}_{1-\alpha}$ with an alternative quantile independent of $s_{\mathsf{test}}^{(X,Y)}$. More precisely, define the following adjusted quantile
\begin{align}
\check{\Q}_{1-\alpha}
\;\coloneqq\; \inf\Big\{x\in\RB: \ssum{n}{1}{m} \mathbbm{1}\{s_i\le x\}\ge \lceil (1-\alpha)m - \alpha \rceil\Big\}.
,
\label{eq:defn-Q-check-prop}
\end{align}
which satisfies the following property. 
\begin{claim}\label{clm:equiv-quantile}
The adjusted quantile defined in \eqref{eq:defn-Q-check-prop} satisfies
\[
\bigl\{s_{\mathsf{test}}^{(x,y)}\le \widetilde{\Q}_{1-\alpha}\bigr\}
\quad\Longleftrightarrow\quad
\bigl\{s_{\mathsf{test}}^{(x,y)}\le \check{\Q}_{1-\alpha}\bigr\}.
\]
\end{claim}

It then follows immediately from Claim~\ref{clm:equiv-quantile} that 
\begin{equation}\PB_{\gD}\big(s_{\mathsf{test}}^{(X,Y)}\le \widetilde{\Q}_{1-\alpha}\mid Z_{1:m}\big) = \PB_{\gD}\big(s_{\mathsf{test}}^{(X,Y)}\le \check{\Q}_{1-\alpha}\mid Z_{1:m}\big) = F_{\mathsf{test}}(\check{\Q}_{1-\alpha}; Z_{1:m}),
\label{eq:equiv-stest-Ftest-bound}
\end{equation}
where $F_{\mathsf{test}}(\cdot;\cdot)$ is defined at the beginning of this subsection. It then boils down to controlling $F_{\mathsf{test}}(\check{\Q}_{1-\alpha}; Z_{1:m})$. 

\begin{proof}[Proof of Claim~\ref{clm:equiv-quantile}]
From the definition of the quantile functional, one has
\begin{align*}
\bigl\{s_{\mathsf{test}}^{(x,y)}\le \widetilde{\Q}_{1-\alpha}\bigr\}\quad 
&\Longleftrightarrow\quad
s_{\mathsf{test}}^{(x,y)}
\le
\inf\Bigl\{q\in\R:\;
\frac{1}{m+1}\mathbbm{1}\{s_{\mathsf{test}}^{(x,y)}\le q\}
+\frac{1}{m+1}\sum_{i=1}^m \mathbbm{1}\{s_i\le q\}
\ge 1-\alpha
\Bigr\} \\
&\Longleftrightarrow\quad
\mathbbm{1}\{s_{\mathsf{test}}^{(x,y)}\le s_{\mathsf{test}}^{(x,y)}\}
+\sum_{i=1}^m \mathbbm{1}\{s_i\le s_{\mathsf{test}}^{(x,y)}\}
\le \lceil (1-\alpha)(m+1) \rceil.
\end{align*}
Given the trivial fact $\mathbbm{1}\{s_{\mathsf{test}}^{(x,y)}\le s_{\mathsf{test}}^{(x,y)}\}=1$, the last display is equivalent to
\[
\sum_{i=1}^m \mathbbm{1}\{s_i\le s_{\mathsf{test}}^{(x,y)}\}
\le \lceil (1-\alpha)(m+1)-1 \rceil = \lceil (1-\alpha)m - \alpha \rceil.
\]
By the definition \eqref{eq:defn-Q-check-prop} of $\check{\mathsf{Q}}_{1-\alpha}$, this inequality holds if and only if
$s_{\mathsf{test}}^{(x,y)} \le \check{\Q}_{1-\alpha}$.
This proves the claim.
\end{proof}

\paragraph{Step 3: controlling $F_{\mathsf{test}}(\check{\Q}_{1-\alpha}; Z_{1:m})$.}
In order to control $F_{\mathsf{test}}(\check{\Q}_{1-\alpha}; Z_{1:m})$, we begin with the following decomposition:
\begin{align}
    \Big|F_{\mathsf{test}}(\check{\Q}_{1-\alpha};Z_{1:m}) {-}(1- \alpha)\Big|
 &\le
    \underbrace{\abs{
     F_{\mathsf{test}}(\check{\Q}_{1-\alpha};Z_{1:m}) - F_{\mathsf{test}}(\check{\Q}_{1-\alpha})
    }}_{\eqqcolon \gT_{1}} + \underbrace{\bigg|
    \frac{1}{m}\ssum{i}{1}{m}\left[F_i^0(\check{\Q}_{1-\alpha})
    - \mathbbm{1}\left\{s_{i} \le \check{\Q}_{1-\alpha}\right\}
    \right]
    \bigg|}_{\eqqcolon\gT_{2}} \notag\\
    &\quad + \underbrace{\bigg|
    F_{\mathsf{test}}(\check{\Q}_{1-\alpha}) - \frac{1}{m}\ssum{i}{1}{m}F_i^0(\check{\Q}_{1-\alpha})
    \bigg|}_{\eqqcolon\gT_{3}} + \underbrace{\bigg|
    \frac{1}{m}\ssum{i}{1}{m}\mathbbm{1}\left\{s_{i} \le \check{\Q}_{1-\alpha}\right\} - (1-\alpha)
    \bigg|}_{\eqqcolon\gT_{4}},
    \label{eq:Ftest-decompose-4-terms}
\end{align}
where both $F_{\mathsf{test}}(\cdot)$ and $F_i^0(\cdot)$ are defined at the beginning of this subsection.  This decomposition leaves us with four terms to cope with. 
\begin{itemize}
\item {\em Bounding $\mathcal{T}_1$ and $\mathcal{T}_2$.}
Define the typical events $\gE_1$ and $\gE_2$ as:
\begin{align*}
\gE_{1} &\coloneqq \Bigg\{\sup\limits_{u\in \RB}\bigg\{\Big|F_{\mathsf{test}}(u;Z_{1:m})- F_{\mathsf{test}}(u)\Big|\bigg\}
\le \frac{16{L}}{n}\sqrt{m\log \frac{2}{\delta}}
\Bigg\};
\\
\gE_2 &\coloneqq \Bigg\{
\sup\limits_{u\in \RB}\bigg\{
\Big|\frac{1}{m}\ssum{i}{1}{m}\Big(\mathbbm{1}\left\{s_{i} \le u\right\} - F_i^0(u)\Big)\Big|
\bigg\} \le 24\sqrt{\frac{\log (40/\delta)}{m}} + \frac{24{L}}{n}\sqrt{m\log \left(\frac{40m}{\delta} + n\right)}
\Bigg\}.
\end{align*}
Lemmas~\ref{lem:cond_prob_concentration} and~\ref{lem:empr_prob_concentration} imply that, with probability at least \(1-\delta\), these two events occur simultaneously.
On the event $\gE_1 \cap \gE_2$, the terms $\gT_{1}$ and $\gT_{2}$ satisfy
\begin{align}
\gT_{1} & = \abs{
     F_{\mathsf{test}}(\check{\Q}_{1-\alpha};Z_{1:m}) - F_{\mathsf{test}}(\check{\Q}_{1-\alpha})
    }
\le \sup\limits_{u\in\RB}\Big|\,F_{\mathsf{test}}(u; Z_{1:m}) - F_{\mathsf{test}}(u)\,\Big|
\le \frac{16{L}\sqrt{m\log (2/\delta)}}{n};\label{eq:UB-T1-1479}\\
\gT_{2} &= \bigg|
    \frac{1}{m}\ssum{i}{1}{m}\left[F_i^0(\check{\Q}_{1-\alpha})
    - \mathbbm{1}\left\{s_{i} \le \check{\Q}_{1-\alpha}\right\}
    \right]
    \bigg|
\le \sup\limits_{u\in\RB}\bigg|
\frac{1}{m}\ssum{i}{1}{m}\left[F_i^0(u)  - \mathbbm{1}\left\{s_i \le u\right\}\right]
\bigg|\notag\\
&\le 24\sqrt{\frac{\log (40/\delta)}{m}} + \frac{24{L}}{n}\sqrt{m\log \left(\frac{40m}{\delta} + n\right)}.
\label{eq:UB-T2-1479}
\end{align}

\item {\em Bounding $\mathcal{T}_3$.}
Regarding $\gT_{3}$, it follows from the triangle inequality and the definition of the total-variation distance that
\begin{align}
\gT_{3} &= \bigg|
    F_{\mathsf{test}}(\check{\Q}_{1-\alpha}) - \frac{1}{m}\ssum{i}{1}{m}F_i^0(\check{\Q}_{1-\alpha})
    \bigg| \le \frac{1}{m}\ssum{i}{1}{m}\abs{F_{\mathsf{test}}(\check{\Q}_{1-\alpha}) - F_i^0(\check{\Q}_{1-\alpha})} \notag\\
&\le \frac{1}{m}\ssum{i}{1}{m}\Big(\sup\limits_{u\in \RB}\big\{\abs{
F_{\mathsf{test}}(u) - F_i(u)
}
\big\} + \sup\limits_{u\in \RB}\big\{\abs{
F_{i}^0(u) - F_i(u)
}
\big\}
\Big) \notag\\
&\overset{\mathrm{(a)}}\le \frac{1}{m}\ssum{i}{1}{m}\Big(\mathsf{KS}\big(s_{\mathsf{test}}^{(X,Y)}, s_i^{(X,Y)}\big) + \frac{4L}{n} \Big) \overset{\mathrm{(b)}}{\le} \frac{2}{m}\ssum{i}{0}{m}\mathsf{TV}(Z, Z_{i}) + \frac{8L}{n},
\label{eq:UB-T3-1479}
\end{align}
where $F_i(\cdot)$ is defined at the beginning of this subsection. Inequalities (a) and (b) are justified below.
\begin{itemize}
\item To validate inequality (a) in \eqref{eq:UB-T3-1479}, observe that
\begin{align*}
    F_i^0(u) - F_i(u) &= \PB\big(s_i^{(X,Y)}\le u\big) - \PB(s_i \le u)\\
    &\overset{\mathrm{(i)}}{\le} \PB\Big(
    u-\big|\,\widehat{\mu}_{Z_{1:m}}^{(X,Y)}(X_i) - \widehat{\mu}_{Z_{1:m}\cup \{z_0\}}(X_i) \,\big| < s_i
    \le u+ \big|\,\widehat{\mu}_{Z_{1:m}}^{(X,Y)}(X_i) - \widehat{\mu}_{Z_{1:m}\cup \{z_0\}}  (X_i) \,\big|
    \Big)\\
    &\overset{\mathrm{(ii)}}{\le} \PB\Big(u-\frac{L_2}{n}< s_i \le u+\frac{L_2}{n}\Big)\\
    &\le \PB \Big(\big(\widehat{\mu}_{Z_{1:m}\cup\{z_0\}}(X_i) - u\big) - \frac{L_2}{n} < Y_i \le \big(\widehat{\mu}_{Z_{1:m}\cup\{z_0\}}(X_i) - u\big) + \frac{L_2}{n}\Big)\\
    &\quad + \PB \Big(\big(\widehat{\mu}_{Z_{1:m}\cup\{z_0\}}(X_i) + u\big) - \frac{L_2}{n} < Y_i \le \big(\widehat{\mu}_{Z_{1:m}\cup\{z_0\}}(X_i) + u\big) + \frac{L_2}{n}\Big) \\
    & \overset{\mathrm{(iii)}}\le \frac{4L_1L_2}{n}=\frac{4L}{n},
\end{align*}
where (ii) arises from Assumption~\ref{ass:fair_alg},   (iii) follows from Assumption~\ref{ass:lip_cond_distr}, respectively, and (i) is a direct consequence of the definition of $s_i^{(X,Y)}$ and $s_i$ (see the beginning of the subsection) and 
the following elementary fact.
\begin{fact}\label{fact:ind_diff}
$\big|\PB(\abs{a}>u) - \PB(\abs{a + \delta}>u) \big| \le 
\PB(u-\abs{\delta} \le \abs{a} \le u + \abs{\delta}).
$
\end{fact}
\begin{proof}[Proof of Fact~\ref{fact:ind_diff}] We observe that
\begin{align*}
\big|\PB(\abs{a}>u) - \PB(\abs{a + \delta}>u) \big|
&= \max\Big\{\PB(\abs{a} > u) - \PB(\abs{a+\delta} > u),\, \PB(\abs{a+\delta} > u) - \PB(\abs{a} > u)\Big\}\\
 &\le \PB(\abs{a}+\abs{\delta} > u) - \PB(\abs{a}-\abs{\delta} > u)
 \le \PB(u-\abs{\delta} \le \abs{a} \le u + \abs{\delta})
\end{align*}
as claimed. 
\end{proof}
\item 
We now justify inequality (b) in \eqref{eq:UB-T3-1479}. Consider any index $i$, and define ${Z}_{1:m}^i$ as the dataset obtained from $Z_{1:m}$ by replacing the $i$-th sample $Z_i$ with the target sample $Z$. Based on ${Z}_{1:m}^i$, introduce the auxiliary score
\[
{s}_i'\coloneqq \big|\, Y_i - \widehat{\mu}_{{Z}_{1:m}^i}^{(X_i,Y_i)}(X_i)\,\big|.
\]
 In view of Assumption~\ref{ass:fair_alg}, $s_i'$ differs from $s_i^{(X,Y)}$ by at most $2L_2/n$. Combining this with Assumption~\ref{ass:lip_cond_distr} immediately yields
\[
\mathsf{KS}\big(s_i', s_i^{(X,Y)}\big) \le \frac{4L_1L_2}{n}=\frac{4L}{n}.
\]
As a consequence,  for each $i = 1,\ldots, m$ we have
\begin{align*}
    \mathsf{KS}\big(s_{\mathsf{test}}^{(X,Y)}, s_i^{(X,Y)}\big) &\le \mathsf{KS}\big(s_{\mathsf{test}}^{(X,Y)}, s_i'\big) + \mathsf{KS}\big(s_i', s_i^{(X,Y)}\big) \\
    &\overset{\mathrm{(iv)}}{\le} \mathsf{TV}\Big(\big(Z_{1:m},Z\big), \big(Z_{1:m}^i,Z_i\big)\Big) + \frac{4L}{n} \\
    &\overset{(\mathrm{v})}{\le} 
    2\mathsf{TV}(Z,Z_i) + \frac{4L}{n}.
\end{align*}
Here, (iv) is valid since $s_{\mathsf{test}}^{(X,Y)}$ and $s_i'$ are outputs of the same measurable mapping, evaluated at $(Z_{1:m},Z)$ and $(Z_{1:m}^i, Z_i)$, respectively, whereas (v) invokes \citet[Lemma 1]{barber2023conformal}.

\end{itemize}

\item {\em Bounding $\mathcal{T}_4$.}
According to the definition of $\check{\Q}_{1-\alpha}$, the term $\gT_{4}$ can be bounded by
\begin{align}
\gT_4 &= \bigg|
    \frac{1}{m}\ssum{i}{1}{m}\mathbbm{1}\left\{s_{i} \le \check{\Q}_{1-\alpha}\right\} - (1-\alpha)
    \bigg| = 
    \bigg|
    \frac{\lceil (1-\alpha)m - \alpha \rceil}{m} - \left(1-\alpha\right)
    \bigg| < \frac{1}{m},
    \label{eq:UB-T4-1479}
\end{align}
where the last inequality follows since
\begin{align*}
    -1 < {(1-\alpha)m - \alpha} - (1-\alpha)m
    &\le {\lceil (1-\alpha)m - \alpha \rceil} - (1-\alpha)m\\
    &\le (1-\alpha)m - \alpha +1 - (1-\alpha)m <1.
\end{align*}
\end{itemize}
Substituting the preceding bounds on $\mathcal{T}_1,\dots,\mathcal{T}_4$ into \eqref{eq:Ftest-decompose-4-terms}, we arrive at
\begin{align}
\Big|F_{\mathsf{test}}(\check{\Q}_{1-\alpha};Z_{1:m}) {-}(1- \alpha)\Big|
 &\le  
 \frac{16{L}\sqrt{m\log (2/\delta)}}{n} + 
 24\sqrt{\frac{\log (40/\delta)}{m}} + \frac{24{L}}{n}\sqrt{m\log \left(\frac{40m}{\delta} + n\right)}\notag\\
 &\quad + 
 \frac{2}{m}\ssum{i}{0}{m}\mathsf{TV}(Z, Z_{i}) + \frac{8L}{n} + \frac{1}{m}\notag\\
 &\le \frac{40L\sqrt{m\log (45n/\delta)}}{n} + 24\sqrt{\frac{\log (40/\delta)}{m}} +  
 \frac{2}{m}\ssum{i}{0}{m}\mathsf{TV}(Z, Z_{i}) + \frac{8L}{n} + \frac{1}{m}\notag\\
 &\le \frac{48L\sqrt{m\log (45n/\delta)}}{n} + 25\sqrt{\frac{\log (40/\delta)}{m}} + \frac{2}{m}\ssum{i}{0}{m}\mathsf{TV}(Z, Z_{i}).
 \label{eq:F-test-gap-alpha}
\end{align}

\paragraph{Step 4: putting all this together.} 
Finally, taking 
  the above results \eqref{eq:decorr_XY_1}, \eqref{eq:equiv-stest-Ftest-bound} and \eqref{eq:F-test-gap-alpha} collectively, we can readily finish the proof of Proposition~\ref{prop:training_conditioned_cov}.

\subsection{Proof of Theorem~\ref{thm:full_regret}}

This section is dedicated to establishing Theorem~\ref{thm:full_regret}. 

\paragraph{Notation.} 
For ease of presentation, we adopt the notation introduced in \Cref{def:notation_of_cp} as well as in \Cref{sec:full_conf_alg}. In addition, we shall adopt the following notation:
\begin{itemize}
    \item For any $k\le m<t\in[T]$, let
    \begin{align}
    \Q_{1-\alpha}^{k,m,t}\coloneqq \Q_{1-\alpha}\!\left(
    \frac{1}{m-k+2}\Bigl(\delta\{s^{Z_t}_{Z_{1:m}}(Z_t)\}+\sum_{l=k}^{m}\delta\{s^{Z_t}_{Z_{1:m}}(Z_l)\}\Bigr)
    \right),
    \end{align}
    representing the quantile when (i) the data $Z_{1:m}\cup \{Z_t\}$ are used for training; and (ii) the data $Z_{k:m}\cup \{Z_t\}$ are used for calibration. 
    \item For any $k\le m<i\le j$, define the event
        \begin{subequations}
    \label{eq:event-Anr-full}
    \begin{align}
    \gA(k,m;i,j)\coloneqq
    \Bigg\{
    \bigg|
    \sum_{t=i}^{j}\Bigl(\mathbbm{1}\big\{s^{Z_t}_{Z_{1:m}}(Z_t)\le \Q_{1-\alpha}^{k,m,t}\big\}
    -\PB_{\gD_t}\big(s^{Z_t}_{Z_{1:m}}(Z_t)\le \Q_{1-\alpha}^{k,m,t}\mid Z_{1:m}\big)\Bigr)
    \bigg| \notag\\
    \le 2\sqrt{(j-i+1)\log(2j)}
    \Bigg\},
    \end{align}
    which is concerned with the deviation of the empirical coverage from the training-conditional mean coverage over the time window $[i,j]$. As we shall see momentarily, this is a high-probability event. 
    
    \item For any stage-round pair $(n,r)$, take
    \begin{equation}
    \gA_{n,r}\coloneqq
    \bigcap_{i=\tau_{n,r}}^{\tau_{n,r+1}-1}\ \bigcap_{j=i+1}^{\tau_{n,r+1}-1}
    \gA(\tau_{n,r-1},\tau_{n,r}-1;i,j).
    \end{equation}
    \end{subequations}
\end{itemize}

\paragraph{Step 1: regret decomposition.} 
Following the proof structure of \Cref{thm:split_regret}, we decompose the cumulative regret of interest as
\begin{equation}\label{eq:typical_and_rare_decomp_full}
\begin{aligned}
\ssum{t}{1}{T}\big|\PB(Y_t \in \gC_t(X_t)\mid \gC_t) -(1- \alpha)\big| &= \ssum{n}{1}{N}\ssum{r}{1}{r_n}\ssum{l}{1}{T_r}\Big|\PB(Y_{n,r,l} \in \gC_{n,r}(X_{n,r,l})\mid \gC_{n,r}) - (1- \alpha)\Big|\,\mathbbm{1}\{\gA_{n,r}\}\\
&~ + \ssum{n}{1}{N}\ssum{r}{1}{r_n}\ssum{l}{1}{T_r}\Big|\PB(Y_{n,r,l} \in \gC_{n,r}(X_{n,r,l})\mid \gC_{n,r}) - (1-\alpha)\Big|\,\mathbbm{1}\{\gA_{n,r}^{\mathrm{c}}\}. 
\end{aligned}
\end{equation}
Here we write $\PB(\,\cdot \mid \gC_{n,r})$ (or $\PB(\,\cdot \mid \gC_t)$) for the conditional probability given the set-valued mapping $\gC_{n,r}(\cdot)$ (or $\gC_t(\cdot)$).
Further, similar to \eqref{eq:auxiliary-sl-exceed}, we augment the data to simplify notation: although the last round of stage $n$ contains only $t_n\leq T_r$ time instances, we still generate $Z_{n,r,l}=(X_{n,r,l},Y_{n,r,l})$ for every $l>t_n$ in an i.i.d.~manner obeying
\begin{align}
    \PB(Y_{n,r,l} \in \gC_{n,r}(X_{n,r,l})\mid \gC_{n,r}) = 1- \alpha
    \qquad \text{for all }l>t_n.
\end{align}

\paragraph{Step 2: bounding the second term on the right-hand side of \eqref{eq:typical_and_rare_decomp_full}. }
Akin to the pretrained-score setting, the first term on the right-hand side of \eqref{eq:typical_and_rare_decomp_full} is the dominant term in the above regret decomposition. To justify this, let us look at the second term on the right-hand side of \eqref{eq:typical_and_rare_decomp_full}. 
Fix a realization $Z_{1:m}=z_{1:m}$. Then for each $t\in[i,j]$ with $m<i$, the indicator
$\mathbbm{1}\!\big\{s_{z_{1:m}}^{Z_t}(Z_t)\le Q_{1-\alpha}^{k,m,t}\big\}$ is a function of $Z_t$ only, and hence this collection of indicator variables over the time window $[i,j]$ are statistically independent.
Moreover, the conditional mean of this indicator variable at time $t$ is
$\PB_{\gD_t}\!\big(s_{z_{1:m}}^{Z_t}(Z_t)\le Q_{1-\alpha}^{k,m,t}\mid Z_{1:m}=z_{1:m}\big)$.
Therefore, Hoeffding's inequality readily yields
\begin{subequations}\label{eq:up_bd_A_kmij_c}
\begin{align}
    & \PB\big(\gA(k,m;i,j)^{\mathrm{c}} \mid Z_{1:m}=z_{1:m}\big) \le j^{-8}\quad \text{for all }z_{1:m} \\
    \Longrightarrow \qquad &\PB\big(\gA(k,m;i,j)^{\mathrm{c}}\big) \le \EB_{Z_{1:m}}\big[\PB\big(\gA(k,m;i,j)^{\mathrm{c}} \mid Z_{1:m}\big)\big] \le j^{-8}.
\end{align}
\end{subequations}

Now consider any time point $t\ge 4$ that resides within round $r$ of stage $n$. By the construction of \driftocpfull, we have
\[
\frac{t}{16}\le \frac{\tau_{n,r}}{4}\le \tau_{n,r-1}<\tau_{n,r}\le t\le\tau_{n,r+1}\le 4t.
\]
Consequently, defining
\[
\gE_t \coloneqq 
\bigcap\limits_{m=\frac{t}{4}}^{t}\bigcap\limits_{k=\frac{t}{16}}^m
\bigcap\limits_{j=\frac{t}{4}}^{4t}\bigcap\limits_{i=\frac{t}{4}}^t \gA(k,m;i,j)
\]
in which the index pair $(k,m)$ ranges over all values that
$(\tau_{n,r-1},\tau_{n,r})$ may take, 
we have 
$$\gE_t \subseteq \gA_{n,r}.$$
%
Regarding this event $\mathcal{E}_t$, it follows from \eqref{eq:up_bd_A_kmij_c} that
\begin{align*}
\PB(\gE_t^\mathrm{c})&\le \ssum{m}{\frac{t}{4}}{t}\ssum{k}{\frac{t}{16}}{m}\ssum{j}{\frac{t}{4}}{4t}\ssum{i}{\frac{t}{4}}{j} \PB\big(
\gA(k,m;i,j)^{\mathrm{c}}
\big)\le t^2 \ssum{j}{\frac{t}{4}}{4t}j\PB\big(
\gA(k,m;i,j)^{\mathrm{c}}
\big)
\le t^2\ssum{j}{\frac{t}{4}}{4t}\frac{1}{j^7} = O\big(t^{-4}\big).
\end{align*}
Therefore, the second term on the right-hand side of  \eqref{eq:typical_and_rare_decomp_full} can be bounded above by
\begin{equation}\label{eq:rare_eve_bd_full}
\begin{aligned}
    \EB\bigg[\ssum{n}{1}{N}\ssum{r}{1}{r_n}\ssum{l}{1}{T_r}&\Big|\PB(Y_{n,r,l} \in \gC_{n,r}(X_{n,r,l})\mid \gC_{n,r}) - (1- \alpha)\Big|\,\mathbbm{1}\{\gA_{n,r}^{\mathrm{c}}\}\bigg]\\&\le \EB\left[\ssum{t}{1}{T} \big|\PB(Y_t\in \gC_t(X_t)\mid \gC_{n,r}) - \alpha\big|\,\mathbbm{1}\{\gE_t^{\mathrm{c}}\}\right]\le \ssum{t}{1}{T}\PB\big(\gE_t^{\mathrm{c}}\big) 
    = O\left( \ssum{t}{1}{T}\frac{1}{t^4} \right)= O(1).
\end{aligned}
\end{equation}


\paragraph{Step 3: bounding the first term on the right-hand side of \eqref{eq:typical_and_rare_decomp_full}. }
It remains to bound the first term on the right-hand side of  \eqref{eq:typical_and_rare_decomp_full}.
The overall proof follows a similar strategy to that used in the pretrained-score setting.
To avoid unnecessary repetition, we shall focus primarily on the steps that differ nontrivially from the pretrained-score case. 

To begin with, 
by adapting the arguments in the proof of  Lemma~\ref{lem:one_stage_regret}, we obtain the following result, whose proof of \Cref{lem:one_stage_regret_full} is deferred to \Cref{sec:prf:lem:one_stage_regret_full}. 
\begin{lemma}\label{lem:one_stage_regret_full}
    Consider any stage $n$ in Algorithm~\ref{alg:FOCID}, which comprises $r_n$ rounds and $S_n$ time points.  
    Reusing the notation introduced in \Cref{def:notation_of_cp}, we have 
    \begin{align}
    &\ssum{r}{1}{r_n}\bigg(\ssum{l}{1}{T_{r}}\Big|\PB(Y_{n,r,l} \in \gC_{n,r}\mid \gC_{n,r}) - (1-\alpha)\Big| \bigg)\mathbbm{1}\{\gA_{n,r}\} \notag\\
    &\qquad\qquad \leq 
    \begin{cases}
        \widetilde{O}\bigg(\ssum{j}{1}{J_n}\sqrt{\abs{\gI_{n,j}}}\bigg),~ &\text{for the change-point setting},\\
        \widetilde{O}\left(\sqrt{S_n}  + (\mathsf{TV}_{n}^{\mathsf{stage}})^{\frac{1}{3}}S_n^{\frac{2}{3}}\right),~ &\text{for the smooth drift setting},
    \end{cases}
    \end{align}
    where 
    we set
    \begin{equation}
    \mathsf{TV}_n^{\mathsf{stage}}\coloneqq \sum_{r=1}^{r_n}\sum_{l=1}^{S_{n,r}-1} \mathsf{TV}(Z_{n,r,l}, Z_{n,r,l+1})
    .
    \label{eq:defn-TV-per-stage-full-conformal}
    \end{equation}
    %
\end{lemma}
Given that \Cref{lem:one_stage_regret_full} controls the cumulative regret within a single stage, it remains to extend this bound to the entire time horizon, which we handle separately for the two drift settings in Steps 4 and 5. 

Before proceeding, let us introduce a collection of typical events that will be used in both settings.
For any two time points $1 \le k< m\le T$, define the event
\begin{align}
\gG(k,m) \coloneqq \bigg\{
    \Big|\,\PB\Big(Y_{m}\notin \gC(X_m &\mid Z_{k:m-1}, Z_{1:m-1})\mid Z_{1:m-1}\Big) - \alpha\,\Big| \le 2^6\sqrt{\frac{\log(40m)}{m-k}}\notag\\ 
    &+ \frac{2^7L\sqrt{(m-k)\log (40m)}}{m} + 
    \frac{1}{m-k}\ssum{l}{k}{m-1}\mathsf{TV}\big(Z_{m}, Z_{l}\big)
    \bigg\},
    \label{eq:defn-Gkm-full}
\end{align}
as motivated by \Cref{prop:training_conditioned_cov}. Here the notation $\gC(\cdot\mid \cdot, \cdot)$ is defined at the beginning of \Cref{sec:prf_of_sec_full_conf}.
For any stage $n$, define the typical event $\gB_{n}$ as follows
\begin{equation}\label{eq:def_gB_n-full}
    \gB_{n}\coloneqq \gA_{n,r_n}\cap \gG(\tau_{n,r_n-1},\tau_{n,r_n}).
\end{equation}
In addition, the following two lemmas will be used in the analysis for both drift settings, and we therefore state them here for subsequent use. The first lemma shows that $\mathcal{B}_n$ is an event with sufficiently high probability (even when suitably weighted by $\tau_{n+1}$); the proof can be found in \Cref{sec:prf:lem:control_rare_event_full}.
\begin{lemma}\label{lem:control_rare_event_full}
For any $n\ge 1$,
    recall that $\tau_n$ denotes the starting time of stage $n$. Then we have
    \[
    \EB \big[\tau_{n+1}\mathbbm{1}\{\gB_n^{\mathrm{c}}\}\big] \le O\big(n^{-2}\big).
    \]
\end{lemma}

Another useful lemma shows that the aggregate total variation over the last two rounds of each stage is sufficiently large (at least on some high-probability event). The proof can be found in \Cref{sec:prf:lem:tv_lower_bound_on_typical}.
\begin{lemma}\label{lem:tv_lower_bound_on_typical}
For any stage $n \le N-1$, define
\begin{equation}\label{eq:def-TV_n-tail}
\mathsf{TV}_{n}^{\mathsf{tail}}\coloneqq \ssum{j}{1}{T_{r_n}-1}\mathsf{TV}(Z_{n,r_n,j}, Z_{n,r_n,j{+}1}) + \ssum{i}{1}{T_{r_n{-}1}}\mathsf{TV}(Z_{n,r_n-1,i}, Z_{n,r_n{-}1,i+1}),
\end{equation}
and introduce the event
\begin{equation}\gH_n\coloneqq \Big\{T_{r_n-1}\sqrt{\log(40\tau_{n,r_n})}\le \frac{\tau_{n,r_n}}{256}\Big\}.
\label{eq:defn-gHn-full-conformal}
\end{equation}
Recalling that $t_n\leq T_{r_n}$ is the number of iterations in round $r_n$ of stage $n$, 
one has
\[
\sqrt{t_n}\,\mathsf{TV}_n^{\mathsf{tail}}\mathbbm{1}\big\{\gB_n\cap \gH_n\big\} \ge 3\cdot \mathbbm{1}\big\{\gB_n\cap \gH_n\big\}.
\]
\end{lemma}

\paragraph{Step 4: analysis for the change-point setting.} 
In this setting, 
\Cref{lem:one_stage_regret_full} allows one to decompose
\begin{align}
\sum_{n=1}^N \sum_{r=1}^{r_n}\sum_{l=1}^{T_r}
&\big|\mathbb{P}(Y_{n,r,l} \notin \gC_{n,r}(X_{n,r,l})\,|\, \gC_{n,r}) - \alpha\big|\,
\mathbbm{1}\{\mathcal{A}_{n,r}\}
\le  \widetilde{O}\Bigg(
\sum_{n=1}^N\bigg(\sum_{j=1}^{J_n}\sqrt{|\mathcal{I}_{n,j}|}\bigg)
\Bigg)\notag\\
&\le \widetilde{O}\Bigg(
\sum_{n=1}^N\bigg(\sum_{j=1}^{J_n}\sqrt{|\mathcal{I}_{n,j}|}\bigg)\,
\mathbbm{1}\{\gB_{n}\}
+\sum_{n=1}^N (\tau_{n+1}-\tau_{n})\,
\mathbbm{1}\{\gB_n^{\mathrm{c}}\}\Bigg),
\label{eq:cp-typical_decomp_full}
\end{align}
where the last line follows since, by Cauchy-Schwarz, 
\[
\ssum{j}{1}{J_n}\sqrt{\big|\,\gI_{n,j}\,\big|}
\le \sqrt{J_n \ssum{j}{1}{J_n} \big|\,\gI_{n,j}\,\big|}
= \sqrt{J_n(\tau_{n+1} - \tau_n)}
\le \tau_{n+1} - \tau_n.
\]

\begin{itemize}
\item 
We start with the last term on the right-hand side of \eqref{eq:cp-typical_decomp_full}, for which \Cref{lem:control_rare_event_full} indicates that
\begin{equation}\label{eq:change-pt_rare_event_bd}
    \EB\bigg[\ssum{n}{1}{N}(\tau_{n+1}-\tau_n)\mathbbm{1}\{\gB_n^{\mathrm{c}}\}\bigg] \le \ssum{n}{1}{\infty} \EB\big[\tau_{n+1}\mathbbm{1}\{\gB_n^{\mathrm{c}}\}\big] \leq  O\left( \ssum{n}{1}{\infty}\frac{1}{n^2}\right) = O(1).
\end{equation}

\item 
When it comes to the first term on the right-hand side of \eqref{eq:cp-typical_decomp_full}, we first divide it into
\begin{equation}\label{eq:cp-typical_decomp_full_1}
\begin{aligned}
\sum_{n=1}^N\bigg(\sum_{j=1}^{J_n}\sqrt{|\mathcal{I}_{n,j}|}\bigg)\,
\mathbbm{1}\{\gB_{n}\} &\le \ssum{n}{1}{N}\bigg(\ssum{j}{1}{J_n-1}\sqrt{|\gI_{n,j}|}\bigg)\\
&\quad + \ssum{n}{1}{N}\sqrt{|\gI_{n,J_n}|}\mathbbm{1}\{\gB_{n}\cap \gH_n\}
+ \ssum{n}{1}{N}\sqrt{|\gI_{n,J_n}|}\mathbbm{1}\{ \gH_n^\mathrm{c}\},
\end{aligned}
\end{equation}
where $\gH_n
$ is defined in \eqref{eq:defn-gHn-full-conformal}. 
We shall bound the three terms on the right-hand side of  \eqref{eq:cp-typical_decomp_full_1} separately. 

\begin{itemize}
\item 
As for the first term on the right-hand side of \eqref{eq:cp-typical_decomp_full_1}, we first make the observation that: the time segments $\gI_{n,j}$  ($n=1,\ldots, N$ and $j=1,\ldots, J_n-1$) belong to distinct time segments in $\{\gI_k\}_{k=1}^{N^{\mathsf{cp}}+1}$. As a result, we can derive
\begin{equation}\label{eq:cp-typical_decomp_full_2}
    \ssum{n}{1}{N}\ssum{j}{1}{J_n-1}\sqrt{|\gI_{n,j}|} \le \ssum{k}{1}{N^{\mathsf{cp}}+1}\sqrt{|\gI_k|} \le \sqrt{(N^{\mathsf{cp}}+1)T},
\end{equation}
where the last relation arises from the Cauchy-Schwarz inequality. 

\item With regards to the second term on the right-hand side of \eqref{eq:cp-typical_decomp_full_1}, by the construction of $\gI_{n,j}$ we know that, for each terminal interval $\gI_{n,J_n}$, there exists a unique time segment $\gI_{k_n}$ defined in \Cref{def:notation_of_cp} such that $\gI_{n,J_n}\subseteq \gI_{k_n}$. Moreover, the indices are nondecreasing, namely $k_n \le k_{n+1}$ for $n=1,\ldots,N-1$. In particular, when $J_n\ge 2$, stage $n$ must contain a distribution change, which implies $k_n>k_{n-1}$. Using these properties, we can obtain
\begin{equation}\label{eq:cp-typical_decomp_full_3}
\begin{aligned}
    \ssum{n}{1}{N}\sqrt{|\gI_{n,J_n}|}\mathbbm{1}\{\gB_n\cap \gH_n\} &\overset{\mathrm{(a)}}{\le} \ssum{n}{1}{N}\sqrt{|\gI_{n,J_n}|}\mathbbm{1}\{J_n \ge 2\}\\
    &\le \ssum{n}{1}{N}\sqrt{\big|\, \gI_{k_n}\,\big|}\mathbbm{1}\{k_n > k_{n-1}\}
    \le \ssum{k}{1}{N^{\mathsf{cp}}+1}\sqrt{|\gI_k|} \le \sqrt{(N^{\mathsf{cp}}+1)T}.
\end{aligned}
\end{equation}
Here, (a) follows since, on the event $\gB_n\cap \gH_n$, we have $\sqrt{t_n}\mathsf{TV}_n^{\mathsf{tail}} \ge 3 > 0$ (according to \Cref{lem:tv_lower_bound_on_typical}), which implies that stage $n$ must contain a distribution shift and hence necessarily requires $J_n \ge 2$.

\item 
It remains to bound the third term on the right-hand side of  \eqref{eq:cp-typical_decomp_full_1}. To this end, we make note of the following relations between the two sequences $\{\tau_{n,r}\}_{n,r}$ and $\{T_{r}\}_{r}$:
\begin{align*}
    \tau_{n-1, r_{n-1}} \le \tau_{n,r_n-1};\quad 
    \tau_{n,r_n} - \tau_{n,r_n-1} = T_{r_n-1}.
\end{align*}
In particular, the intervals $\{[\tau_{n,r_n-1}+1, \tau_{n,r_n}]\}_{n=1}^N$ are pairwise disjoint, and as a result, 
\begin{align}
    \ssum{n}{1}{N}\mathbbm{1}\{\gH_n^{\mathrm{c}}\} &= \ssum{n}{1}{N}\mathbbm{1}\left\{{T_{r_n-1}\sqrt{\log (40\tau_{n,r_n})}} > \frac{\tau_{n,r_n}}{256{L}}\right\}\le \ssum{n}{1}{N}\frac{256{L}\sqrt{\log (40T)}(\tau_{n,r_n} - \tau_{n,r_n-1})}{\tau_{n,r_n}} \notag\\
    &\le 256{L}\sqrt{\log (40T)}\ssum{n}{1}{N}\ssum{i}{\tau_{n,r_n-1}+1}{\tau_{n,r_n}} \frac{1}{\tau_{n,r_n}}
    \le 256{L}\sqrt{\log (40T)}\ssum{n}{1}{N}\ssum{i}{\tau_{n,r_n-1}+1}{\tau_{n,r_n}} \frac{1}{i} \notag\\
    &\le 256{L}\sqrt{\log (40T)}\ssum{i}{1}{T}\frac{1}{i} \le 256{L}\big(\log (40T)\big)^{\frac{3}{2}}.
\label{eq:full_conf_B_n_bound_2}
\end{align}
Consequently, taking this together with the Cauchy-Schwarz inequality yields
\begin{equation}\label{eq:cp-typical_decomp_full_4}
\begin{aligned}
    \ssum{n}{1}{N}\sqrt{|\gI_{n,J_n}|}\mathbbm{1}\{ \gH_n^\mathrm{c}\} &\le \ssum{n}{1}{N}\sqrt{S_n}\mathbbm{1}\left\{\gH_n^{\mathrm{c}}\right\}
    {\le}  \sqrt{\left(\ssum{n}{1}{N}S_n\right)\left(\ssum{n}{1}{N}\mathbbm{1}\{\gH_n^{\mathrm{c}}\}\right)}
    \\
    &\le \sqrt{T\left(\ssum{n}{1}{N}\mathbbm{1}\{\gH_n^{\mathrm{c}}\}\right)}
    \overset{\eqref{eq:full_conf_B_n_bound_2}}{\leq } \widetilde{O}\big(\sqrt{LT}\big).
\end{aligned}
\end{equation}
\end{itemize}
\end{itemize}
Combining \eqref{eq:cp-typical_decomp_full}--\eqref{eq:cp-typical_decomp_full_3} and \eqref{eq:cp-typical_decomp_full_4} reveals that
\begin{align}
\mathbb{E}\Bigg[ \sum_{n=1}^N \sum_{r=1}^{r_n}\sum_{l=1}^{T_r}
&\big|\mathbb{P}(Y_{n,r,l} \notin \gC_{n,r}(X_{n,r,l})\,|\, \gC_{n,r}) - \alpha\big|\,
\mathbbm{1}\{\mathcal{A}_{n,r}\}
\Bigg] \le \widetilde{O}(\sqrt{(N^{\mathsf{cp}}+L+1)T}),
\end{align}
which together with 
\eqref{eq:typical_and_rare_decomp_full} and \eqref{eq:rare_eve_bd_full} 
establishes the advertised regret bound for the change-point setting. 

\paragraph{Step 5: analysis for the smooth drift setting.}
With \Cref{lem:one_stage_regret_full} in mind, we first analyze 
\(\sum_{n=1}^N \sqrt{S_n}\).
Recalling the definition of $\mathsf{TV}_n^{\mathsf{tail}}$ in \eqref{eq:def-TV_n-tail}, we make the observation that
\begin{equation}\label{eq:ssum_sqrt_T_n_bound_1}
\begin{aligned}
    \ssum{n}{1}{N}\sqrt{S_n}\mathbbm{1}\left\{\gB_n\right\}
    &\le \sqrt{S_N} + \ssum{n}{1}{N-1}\sqrt{S_n}\mathbbm{1}\left\{\gH_n\right\}\mathbbm{1}\left\{\gB_n\right\}
    + \ssum{n}{1}{N}\sqrt{S_n}\mathbbm{1}\left\{\gH_n^{\mathrm{c}}\right\}\mathbbm{1}\left\{\gB_n\right\}\\
    &\le \sqrt{T} + \ssum{n}{1}{N-1} \sqrt{S_n}\left(\sqrt{t_n}\mathsf{TV}_n^{\mathsf{tail}}\right)^{\frac{1}{3}} + \sqrt{\left(\ssum{n}{1}{N}S_n\right)\left(\ssum{n}{1}{N}\mathbbm{1}\{\gH_n^{\mathrm{c}}\}\right)}
    \\
    &\le \sqrt{T} +\ssum{n}{1}{N-1}S_n^{\frac{2}{3}}\big(\mathsf{TV}_n^{\mathsf{tail}}\big)^{\frac{1}{3}} + \sqrt{T\left(\ssum{n}{1}{N}\mathbbm{1}\{\gH_n^{\mathrm{c}}\}\right)}\\
    &\le \sqrt{T} +\left(\ssum{n}{1}{N-1}S_n\right)^{\frac{2}{3}}\left(\ssum{n}{1}{N-1}\mathsf{TV}_n^{\mathsf{tail}}\right)^\frac{1}{3} + \sqrt{T\left(\ssum{n}{1}{N}\mathbbm{1}\{\gH_n^{\mathrm{c}}\}\right)}\\
    &\le \sqrt{T} +2T^{\frac{2}{3}}\mathsf{TV}_T^{\frac{1}{3}} + \sqrt{T\left(\ssum{n}{1}{N}\mathbbm{1}\{\gH_n^{\mathrm{c}}\}\right)} \overset{\eqref{eq:full_conf_B_n_bound_2}}{=} \widetilde{O}\Big(T^{\frac{2}{3}} \mathsf{TV}_T^{\frac{1}{3}}+ \sqrt{(L+1)T}\Big),
\end{aligned}
\end{equation}
where the second line arises from \Cref{lem:tv_lower_bound_on_typical} and the Cauchy-Schwarz inequality, the penultimate line results from H\"older’s inequality, and the last inequality holds because for any $n$, $\mathsf{TV}_{n}^{\mathsf{tail}}$ is counted at most twice in the summation from $1$ to $N$. Taking this collectively with \Cref{lem:control_rare_event_full}, we can demonstrate that
\begin{align*}
\EB\left[\ssum{n}{1}{N}\sqrt{S_n}\right] 
&\le \EB\left[\ssum{n}{1}{N}\sqrt{S_n}\mathbbm{1}\{\gB_n\}\right]
+ \EB\left[\ssum{n}{1}{N}\sqrt{S_n}\mathbbm{1}\{\gB_n^{\mathrm{c}}\}\right]\\
&\le \widetilde{O}\left( \sqrt{(L+1)T} + T^{\frac{3}{2}}\mathsf{TV}_{T}^{\frac{1}{3}}\right)
+ \EB\left[\ssum{n}{1}{N}\tau_{n+1}\mathbbm{1}\{\gB_n^c\}\right]\\
&\le \widetilde{O}\left( \sqrt{(L+1)T} + T^{\frac{3}{2}}\mathsf{TV}_{T}^{\frac{1}{3}}\right)
+ \ssum{n}{1}{\infty}\EB\big[\tau_{n+1}\mathbbm{1}\{\gB_n^c\}\big]\\
&= \widetilde{O}\left( \sqrt{(L+1)T} + T^{\frac{3}{2}}\mathsf{TV}_{T}^{\frac{1}{3}} + \ssum{n}{1}{\infty}\frac{1}{n^2}\right)
= \widetilde{O}\left(\sqrt{(L+1)T} + \mathsf{TV}_T^{\frac{1}{3}}T^{\frac{2}{3}}\right).
\end{align*}

Armed with the above bound, we can readily invoke \Cref{lem:one_stage_regret_full} and apply H\"older's inequality to yield
\begin{align*}
    \EB\Bigg[\ssum{n}{1}{N}\ssum{r}{1}{r_n}\ssum{l}{1}{T_r}\Big|\, \PB(Y_{n,r,l}\in &\gC_{n,r}(X_{n,r,l})\,|\, \gC_{n,r}) - (1-\alpha) \,\Big|\mathbbm{1}\{\gA_{n,r}\}\Bigg]
    = \widetilde{O}\bigg(
    \EB\bigg[\ssum{n}{1}{N}\sqrt{S_n}\bigg] + \ssum{n}{1}{N}\big(\mathsf{TV}_n^{\mathsf{stage}}\big)^{\frac{1}{3}}S_n^{\frac{2}{3}}
    \bigg)\\
    &\le \widetilde{O}\Bigg(
    \sqrt{(L+1)T}+ \mathsf{TV}_T^{\frac{1}{3}}T^{\frac{2}{3}} + \bigg(\ssum{n}{1}{N}\mathsf{TV}_n^{\mathsf{stage}}\bigg)^{\frac{1}{3}}\bigg(\ssum{n}{1}{N}S_n\bigg)^{\frac{2}{3}}
    \Bigg)
    \notag\\
    &\le \widetilde{O}\Big(\sqrt{(L+1)T} + \mathsf{TV}_T^{\frac{1}{3}}T^{\frac{2}{3}}\Big).
\end{align*}
Taking this together with \eqref{eq:typical_and_rare_decomp_full} and \eqref{eq:rare_eve_bd_full} establishes the claimed regret bound for the smooth drift setting.

\subsection{Proof of auxiliary lemmas} 
\subsubsection{Proof of Lemma~\ref{lem:cond_prob_concentration}}\label{sec:prf:lem:cond_prob_concentration}
For notational convenience, we write $Z_{1:m}$ here in place of $Z_{1:m}^{\mathsf{cal}}$ as long as it is clear from the context.

To apply McDiarmid’s inequality,  consider two given calibration datasets, 
$z_{1:m}$ and $z_{1:m}^\prime$, which differ in exactly one sample. 
For any given point $(x,y)\in\gZ$, define the corresponding scores as
\[
s^{(X,Y)}_{z_{1:m}}(x,y)\coloneqq \bigl|y-\widehat{\mu}^{(X,Y)}_{z_{1:m}}(x)\bigr|,
\qquad
s^{(X,Y)}_{z'_{1:m}}(x,y)\coloneqq \bigl|y-\widehat{\mu}^{(X,Y)}_{z'_{1:m}}(x)\bigr|.
\]
Note that both $\widehat{\mu}^{(X,Y)}_{z_{1:m}}(\cdot)$ and $\widehat{\mu}^{(X,Y)}_{z_{1:m}'}(\cdot)$ are trained on $n+1$ data points. Then, by Assumption~\ref{ass:fair_alg}, for any $x\in \gX$ we have
\[
\abs{\widehat{\mu}_{z_{1:m}}^{(X,Y)}(x) - \widehat{\mu}_{z_{1:m}^\prime}^{(X,Y)} (x)} \le \frac{L_2}{n}.
\]
Combining this with Assumption~\ref{ass:lip_cond_distr}, which assumes the Lipschitz continuity of the distribution function, we see that: for $(X,Y)\sim \mathcal{D}$, 
\begin{align}
    \abs{
    \PB_{\gD}\left(s_{z_{1:m}}^{(X,Y)}(X,Y)> u \right) - \PB_{\gD}\left(s_{z_{1:m}'}^{(X,Y)}(X,Y)> u \right)
    }
    &\overset{\mathrm{\mathrm{(a)}}}{\le} 
    \PB_{\gD}\left(u - \Delta \le s_{z_{1:m}}^{(X,Y)}(X,Y) \le u + \Delta \right) \notag\\ 
    &\overset{\mathrm{(b)}}{\le} 4L_1 \sup\limits_{X,Y}\left\{\abs{\widehat{\mu}_{z_{1:m}}^{(X,Y)}(X) - \widehat{\mu}_{z_{1:m}'}^{(X,Y)}(X)}\right\} \notag\\ 
    &\le \frac{4L_1L_2}{n} = \frac{4L}{n},
    \label{eq:PD-s-sprime-UB-4L}
\end{align}
where $\Delta \coloneqq \big|\widehat{\mu}_{z_{1:m}}^{(X,Y)}(X) - \widehat{\mu}_{z_{1:m}^\prime}^{(X,Y)}(X)\big|$. Here, (a) 
is a result of Fact~\ref{fact:ind_diff} 
%
%
whereas (b) follows since
\begin{align}
    \PB_{\gD}\left(u - \Delta \le s_{z_{1:m}}^{(X,Y)}(X,Y) \le u + \Delta \right) &\le \PB_{\gD}\left(-u-\Delta \le Y- \widehat{\mu}_{z_{1:m}}^{(X,Y)}\le  - u+\Delta\right) \notag\\
    &\qquad + \PB_{\gD}\left(u-\Delta \le Y- \widehat{\mu}_{z_{1:m}}^{(X,Y)}\le u+ \Delta \right) \notag\\
    &\le \sup\limits_{\mu \in \RB} \PB_{\gD}\Bigl( (\mu-u)-\Delta' \le Y \le \Delta' + (\mu - u)\Bigr) \notag\\
    &\qquad + \sup\limits_{\mu \in \RB}\PB_{\gD}\Bigl( (\mu + u)-\Delta' \le Y \le \Delta' + (\mu + u)\Bigr) \leq 4L_1 \Delta'
    \label{eq:PD-Delta-Deltaprime-UB-XY}
\end{align}
with $\Delta' \coloneqq \sup\limits_{X,Y}\left\{\big|\widehat{\mu}_{z_{1:m}}^{(X,Y)}(X) - \widehat{\mu}_{z_{1:m}'}^{(X,Y)}(X)\big|\right\}$.

From \eqref{eq:PD-s-sprime-UB-4L}, we observe that $\PB_{\mathcal{D}}\big(s_{z_{1:m}}^{(X,Y)}(X,Y) > u\big)$---when viewed as a function of $z_{1:m}$---satisfies the bounded difference property with coefficient $4L/n$.  
Now, we make the following definition:
\begin{align*}
    Q(z_{1:m},u) &\coloneqq \abs{
    \PB_{\gD}\left(s_{z_{1:m}}^{(X,Y)}(X,Y)> u\right) - \PB_{\gD_{1:m}\times \gD}\left(s_{Z_{1:m}}^{(X,Y)}(X,Y)>u\right)
    }, \notag\\
    Q(z_{1:m}) &\coloneqq \sup\limits_{u\in \RB}\big\{Q(z_{1:m}, u) \big\},
\end{align*}
where
\begin{equation}
\PB_{\gD_{1:m}\times \gD}\left(s_{Z_{1:m}}^{(X,Y)}(X,Y)>u\right) \coloneqq  \EB_{Z_{1:m}\sim \mathcal{D}_{1:m}}\left[
\PB_{(X,Y)\sim\gD}\left(s_{Z_{1:m}}^{(X,Y)}(X,Y)>u \,\Big|\, Z_{1:m}\right)
\right].
\end{equation}
Then,  basic calculation yields
\begin{align*}
    Q(z_{1:m}) - Q(z_{1:m}') &= \sup\limits_{u\in \RB}\big\{Q(z_{1:m},u)\big\} - \sup\limits_{u\in \RB}\big\{Q(z_{1:m}',u)\big\}\\
    &\le \sup\limits_{u\in \RB}\big\{Q(z_{1:m},u) - Q(z_{1:m}',u)\big\}
    \le \frac{4L}{n}.
\end{align*}
Hence, by applying McDiarmid’s inequality (Lemma~\ref{lem:mcdiarmid}), we can demonstrate that
\begin{equation}\label{eq:cond_conc_pt}
\begin{aligned}
\sup\limits_{u\in \RB}&\left\{\abs{\PB_{\gD}\left(s_{Z_{1:m}}^{(X,Y)}(X,Y)> u ~\Big|~ Z_{1:m}\right)
- \PB_{\gD_{1:m}\times\gD}\left(s_{Z_{1:m}}^{(X,Y)}(X,Y)> u \right)
}\right\}
=Q(Z_{1:m})\\
&\le \EB_{Z_{1:m}}\left[Q(Z_{1:m})\right] + 4{L}\frac{\sqrt{m\log \frac{1}{\delta}}}{n}
\end{aligned}
\end{equation}
holds with probability exceeding $1 - {\delta}$. 

Now consider any given $z_0 = (x_0,y_0)$, and denote $\widehat{\mu}_{z_{1:m}}^0(\cdot) \coloneqq \widehat{\mu}_{z_{1:m}}^{(x_0,y_0)}(\cdot)$ and $s_{z_{1:m}}^0(x,y) = \big|y-\widehat{\mu}_{z_{1:m}}^0(x)\big|$.
Let $Z_{1:m}^*$ be an independent copy of $Z_{1:m}$.
Then one can show that
\begin{equation}\label{eq:EB_Q_1}
\begin{aligned}
    \EB_{Z_{1:m}}[Q(Z_{1:m})] &= \EB_{Z_{1:m}}\left[\sup\limits_{u\in \RB}\abs{\PB_{\gD}\big(s_{Z_{1:m}}^{(X,Y)}(X,Y)>u\mid Z_{1:m}\big) - \PB_{\gD_{1:m}\times\gD}\big(s_{Z_{1:m}}^{(X,Y)}(X,Y)>u\big)}\right]\\
    &\le \mathop{\EB}\limits_{Z_{1:m},Z_{1:m}^*}\left[\sup\limits_{u\in \RB}\abs{
    \PB_{\gD}\big(s_{Z_{1:m}}^{(X,Y)}(X,Y)>u\mid Z_{1:m}\big) - \PB_{\gD}\big(s_{Z_{1:m}^*}^{(X,Y)}(X,Y)>u\mid Z_{1:m}^*\big)
    }\right]\\
    &\le 4L_1\mathop{\EB}\limits_{Z_{1:m},Z_{1:m}^*}\left[
    \sup\limits_{u\in \RB}\left\{\EB_{X}\left[\abs{\widehat{\mu}_{Z_{1:m}}^0(X) - \widehat{\mu}_{Z_{1:m}^*}^0(X)} + \frac{L_2}{n}\right]\right\}
    \right]\\
    &= 4L_1\mathop{\EB}\limits_{Z_{1:m}, Z_{1:m}^*, X}\left[\abs{\widehat{\mu}_{Z_{1:m}}^0(X) - \widehat{\mu}_{Z_{1:m}^*}^0(X)}\right] + \frac{8L}{n}.
\end{aligned}
\end{equation}
Here, the second line follows from Jensen's inequality; the penultimate line follows since, for any given two arrays $z_{1:m}$ and $z_{1:m}^*$, one has
\begin{align*}
    &\bigg|\,\PB_{\gD}\Big(s_{z_{1:m}}^{(X,Y)}(X,Y)>u\Big) - \PB_{\gD}\Big(s_{z_{1:m}^*}^{(X,Y)}(X,Y)>u\Big)\,\bigg|\\
    &\le \bigg|\,\PB_{\gD}\Big(s_{z_{1:m}}^{(X,Y)}(X,Y)>u\Big) - \PB_{\gD}\Big(s_{z_{1:m}}^{0}(X,Y)>u\Big)\,\bigg| {+} 
    \bigg|\,\PB_{\gD}\Big(s_{z_{1:m}^*}^{(X,Y)}(X,Y)>u\Big) - \PB_{\gD}\Big(s_{z_{1:m}^*}^{0}(X,Y)>u\Big)\,\bigg|\\
    &\quad + \bigg|\,\PB_{\gD}\Big(s_{z_{1:m}}^{0}(X,Y)>u\Big) - \PB_{\gD}\Big(s_{z_{1:m}^*}^{0}(X,Y)>u\Big)\,\bigg|\\
    &\overset{\mathrm{(c)}}{\le} \frac{8L}{n} + \bigg|\,\PB_{\gD}\Big(s_{z_{1:m}}^{0}(X,Y)>u\Big) - \PB_{\gD}\Big(s_{z_{1:m}^*}^{0}(X,Y)>u\Big)\,\bigg|\\
    &\overset{\text{Fact}~\ref{fact:ind_diff}}{\le}
    \frac{8L}{n}+ \EB_{X}\left[\PB\left(
    u - \big|\widehat{\mu}_{z_{1:m}}^0(X) - \widehat{\mu}_{z_{1:m}*}^0(X)\big|
    \le s_{z_{1:m}}^0(X,Y) \le u + \big|\widehat{\mu}_{z_{1:m}}^0(X) - \widehat{\mu}_{z_{1:m}*}^0(X)\big|\,\Big|\, X
    \right)\right]\\
    &\overset{\mathrm{(d)}}{\le} \frac{8L}{n} + 4L_1 \EB_X\left[\abs{\widehat{\mu}_{z_{1:m}}^0(X) - \widehat{\mu}_{z_{1:m}^*}^0(X)}\right],
\end{align*}
where (c) holds by Fact~\ref{fact:ind_diff}, 
Assumptions~\ref{ass:lip_cond_distr} and \ref{ass:fair_alg} (similar to the arguments for \eqref{eq:PD-s-sprime-UB-4L}), and (d) makes use of Assumptions~\ref{ass:lip_cond_distr} (similar to the arguments for \eqref{eq:PD-Delta-Deltaprime-UB-XY}).

To control the first term on the right-hand side of \eqref{eq:EB_Q_1}, we introduce the quantity below for any given $X$:
\[
\nu_i(X)\coloneqq \EB\left[\widehat{\mu}_{Z_{1:m}}^0(X)\,\big|\, Z_{1:i+1}\right] - \EB\left[\widehat{\mu}_{Z_{1:m}}^0(X)\,\big|\, Z_{1:i}\right],\quad i = 0,\ldots, m-1.
\]
It is readily seen that $\{\nu_i(X)\}_{i=0}^{m-1}$ forms a martingale difference sequence.
Moreover, Assumption~\ref{ass:lip_cond_distr} tells us that, for any $i= 0,\ldots, m-1$ and $X\in \gX$,
\[
\abs{\nu_i(X)} \le \sup\limits_{z_{1:i+1}}\EB_{Z_{i+2:m}}\left[\abs{\widehat{\mu}_{z_{1:i}\cup z_{i+1}\cup Z_{i+2:m}}^0(X) - \EB_{Z_{i+1}}\left[\widehat{\mu}_{z_{1:i}\cup Z_{i+1}\cup Z_{i+2:m}}^0(X)\right]}\right] \le \frac{L_2}{n}.
\]
Therefore, it holds that
\begin{equation}\label{eq:EB_Q_2}
\begin{aligned}
  \mathop{\EB}\limits_{Z_{1:m}, Z_{1:m}^*, X}\left[\abs{\widehat{\mu}_{Z_{1:m}}^0(X) - \widehat{\mu}_{Z_{1:m}^*}^0(X)}\right] &\le 2\EB_{\gD_{1:m}}\left[\abs{\widehat{\mu}_{Z_{1:m}}^0(X) - \EB_{\widetilde{Z}_{1:m}\sim \gD_{1:m}}\big[\widehat{\mu}_{\widetilde{Z}_{1:m}}^0(X)\big]}\right]\\
  &\le 2\EB_{\gD_{1:m}}\left[\abs{\ssum{i}{0}{m-1}\nu_i(X)}\right]
  \le 2\left(\EB_{\gD_{1:m}}\left[
  \left(\ssum{i}{0}{m-1}\nu_i(X)\right)^2
  \right]\right)^{\frac{1}{2}}\\
  &= 2\left(
  \ssum{i}{0}{m-1}\EB_{\gD_{1:m}}\left[\nu_i(X)^2\right]
  \right)^{\frac{1}{2}} \le \frac{2L_2\sqrt{m}}{n},
\end{aligned}
\end{equation}
where the last equality holds since $\{\nu_i(X)\}$ is a martingale difference sequence. 
Taking \eqref{eq:cond_conc_pt}, \eqref{eq:EB_Q_1} and \eqref{eq:EB_Q_2} together yields that, with probability at least $1-\delta$, 
\begin{align*}
    \sup\limits_{u\in \RB}&\left\{\abs{\PB_{\gD}\left(s_{Z_{1:m}}^{(X,Y)}(X,Y)> u ~\Big|~ Z_{1:m}\right)
- \PB_{\gD_{1:m}\times\gD}\left(s_{Z_{1:m}}^{(X,Y)}(X,Y)> u \right)
}\right\}\\
&\le 4L\frac{\sqrt{m\log (1/\delta)}}{n} + \frac{8L}{n} + 8L\frac{\sqrt{m\log(1/\delta)}}{n} \le 16L\frac{\sqrt{m\log (1/\delta)}}{n},
\end{align*}
thereby concluding the proof of Lemma~\ref{lem:cond_prob_concentration}.

\subsubsection{Proof of Lemma~\ref{lem:empr_prob_concentration}}\label{sec:prf:lem:empr_prob_concentration}

Before proceeding, we introduce several additional convenient notations below. 
\begin{itemize}

    \item $Z_{1:m}$: we often use it in place of $Z_{1:m}^{\mathsf{cal}}$ when there is no ambiguity. 

    \item $\widehat{\mu}_{Z_{1:m}}(\cdot)$: we remind the reader that this indicates the fitted model trained on $Z_{1:m}\cup z_{m+1:n}^{\mathsf{train}}$ (see Definition~\ref{def:basic_notation}). 

    \item $Z_{1:m}^x$: this refers to $\{(x_i,Y_i)\}_{i=1}^m$, where the features are frozen to be $\{x_1,\dots, x_m\}$; with this notation one clearly has $Z_{1:m}^X = Z_{1:m}$. 

    \item $\widehat{\mu}_{Z_{1:m}^x}(\cdot)$: the fitted model trained on $\{(x_i,Y_i)\}_{i=1}^m$, with fixed features $\{x_i\}_{i=1}^m$ and random responses $\{Y_i\}_{i=1}^m$.

    \item $\widetilde{\mu}_{x_{1:m}}(x_i)$: the expected prediction of $\widehat{\mu}_{Z_{1:m}^x}(\cdot)$ w.r.t.~$x_i$, with the expectation taken over the randomness of $Y_{1:m}$, i.e.,  
        \begin{align}
    \widetilde{\mu}_{x_{1:m}}(x_i) \coloneqq \EB_{Y_{1:m}\mid X_{1:m}}
    \left[\widehat{\mu}_{Z_{1:m}^x}(x_i) \mid X_{1:m} = x_{1:m}\right], \qquad \text{  } i = 1, \ldots, m,
    \label{eq:defn-mu-tilde-exp}
    \end{align}

\end{itemize}
%


\noindent 
The proof of this lemma is organized into several steps below.  
\paragraph{Step 1: proximity of $\widehat{\mu}_{Z_{1:m}^x}$ and its conditional expectation.}

Equipped with the above set of notation, we immediately note that:
\begin{itemize}
\item $\widehat{\mu}_{Z_{1:m}}(\cdot)$ is a function jointly dependent on the random objects $X_{1:m}$ and $Y_{1:m}$; 

\item For any fixed realization $x_{1:m}$,  the collection 
$\{\widehat{\mu}_{Z_{1:m}^x}(x_i)\}_{i=1}^m$ 
can be viewed as a family of functions dependent on the random variables $Y_{1:m}$;   

 \item Conditional on $X_{1:m} = x_{1:m}$, 
the random variables $Y_1, \ldots, Y_m$ are mutually independent.  
\end{itemize}
We now make note of the following basic fact. 
\begin{claim}\label{clm:mu_conc}
    Recall the definition of $\widehat{\mu}_{Z_{1:m}^x}(x_i)$ and $ \widetilde{\mu}_{x_{1:m}}(x_i)
    $ defined at the beginning of this subsection. Then 
    %
    %
    for any fixed $x_{1:m}$ and any $0<\delta<1$, the event
    \begin{align}
    \gE_1(x_{1:m}) \coloneqq \left\{
    \sup\limits_{i\in \{1,\dots,m\}} \left\{
    \abs{
    \widehat{\mu}_{Z_{1:m}^x}(x_i) - \widetilde{\mu}_{x_{1:m}}(x_i)
    }\right\}\le \frac{L_2}{n}\sqrt{m\log \frac{10m}{\delta}}
    \right\}
    \label{eq:E-event-1-to-m-C2}
    \end{align}
    occurs with probability at least $1 - \delta/5$.
\end{claim}
\begin{proof} This claim follows directly by invoking McDiarmid’s inequality (Lemma~\ref{lem:mcdiarmid}) along with Assumption~\ref{ass:fair_alg} and a union bound over $i\in\{1,\dots,m\}$. We omit the details for brevity. 
\end{proof}
%

\paragraph{Step 2: a surrogate empirical distribution.}
With Claim~\ref{clm:mu_conc} in place, our next step is to approximate the target empirical distribution
\begin{equation}
\widehat{F}_{Z_{1:m}}(u) \coloneqq \frac{1}{m}\ssum{i}{1}{m}\mathbbm{1}\left\{
s_i \le u
\right\}
\end{equation}
using a surrogate empirical distribution
\begin{equation}
\widetilde{F}_{Z_{1:m}}(u) \coloneqq \frac{1}{m}\ssum{i}{1}{m}\mathbbm{1}\left\{
\widetilde{s}_{i} \le u
\right\},\quad  \text{ where } \widetilde{s}_{i} \coloneqq \big|Y_i - \widetilde{\mu}_{X_{1:m}}(X_i)\big|,~~ i=1,\ldots, m.
\label{eq:defn-Ftilde-si-tilde}
\end{equation}
In words, the fitted outcome $\widehat{\mu}_{Z_{1:m}}(X_i)$ is now replaced with $\widetilde{\mu}_{X_{1:m}}(X_i)$, the latter of which averages out the randomness over $Y_{1:m}$ (cf.~\eqref{eq:defn-mu-tilde-exp}).  
On the event $\gE_1(X_{1:m})$ defined in \eqref{eq:E-event-1-to-m-C2}, we can bound the difference between these two quantities as follows
\begin{equation}\label{eq:tilde_F-hat_F_1}
\begin{aligned}
    \Big|&\widetilde{F}_{Z_{1:m}}(u) - \widehat{F}_{Z_{1:m}}(u)\Big|
    \le \frac{1}{m}\ssum{i}{1}{m}\big|\mathbbm{1}\{s_i > u\} - \mathbbm{1}\{\widetilde{s}_{i} > u\}\big|\\
    &\overset{\mathrm{(a)}}{\le}
    \frac{1}{m}\ssum{i}{1}{m}\mathbbm{1}\big\{
    u - \abs{s_i - \widetilde{s}_{i}}
    < \widetilde{s}_{i} \le u +\abs{s_i - \widetilde{s}_{i}}
    \big\}\\
    &\le \frac{1}{m}\ssum{i}{1}{m}\mathbbm{1}\big\{
    u - \abs{\widehat{\mu}_{Z_{1:m}}(X_i) {-} \widetilde{\mu}_{X_{1:m}}(X_i)} < \widetilde{s}_{i} \le u + \abs{\widehat{\mu}_{Z_{1:m}}(X_i) {-} \widetilde{\mu}_{X_{1:m}}(X_i)}
    \big\}\\
    &\overset{\mathrm{(b)}}{\le} 
    \frac{1}{m}\ssum{i}{1}{m}\mathbbm{1}\left\{
    u - \frac{L_2}{n}\sqrt{{m}\log \frac{10{m}}{\delta}}
    < \widetilde{s}_{i} \le u + \frac{L_2}{n}\sqrt{{m}\log \frac{10{m}}{\delta}}
    \right\},
\end{aligned}
\end{equation}
where (a) holds because of Fact~\ref{fact:ind_diff} 
%
and (b) arises from the definition of the event $\gE_1(X_{1:m})$.   

For simplicity of notation, denote 
\begin{align}
\Delta_{n,m} \coloneqq 
\frac{L_2}{n}\sqrt{{m}\log\frac{10{m}}{\delta}}
\qquad \text{and} \qquad 
\gB(u, \varepsilon) \coloneqq (u - \varepsilon, u + \varepsilon], 
\end{align}
%
and consider the event
\[
\gE_2(x_{1:m}) \coloneqq \left\{
\sup\limits_{u\in \RB}\left\{\frac{1}{m}\abs{\ssum{i}{1}{m}
\Big(\mathbbm{1}\{\widetilde{s}_i \in \gB(u,\Delta_{n,m})\} {-} \PB\big(\widetilde{s}_i \in \gB(u,\Delta_{n,m})\,|\, X_{1:m}=x_{1:m}\big)\Big)}\right\} \le 10\sqrt{\frac{\log(10/\delta)}{m}}
\right\}.
\]
We would like to prove that this event occurs with high probability conditional on $X_{1:m} = x_{1:m}$. 
Towards this end, we first observe that, for any $i=1,\ldots, m$ and any $x\in \RB$, 
\begin{align*}
    \mathbbm{1}\left\{\widetilde{s}_i \in \gB(u,\Delta_{n,m})\right\}
    &- \PB\bigl(\widetilde{s}_i \in \gB(u, \Delta_{n,m})\,|\, X_{1:m}=x_{1:m}\bigr)\\
    &= \Bigl(
    \mathbbm{1}\{\widetilde{s}_i \le u + \Delta_{n,m}\} - \PB\big(\widetilde{s}_i \le u + \Delta_{n,m}\,|\,X_{1:m} = x_{1:m}\big)
    \Bigr)\\ &\quad - \Bigl(
    \mathbbm{1}\{\widetilde{s}_i \le u - \Delta_{n,m}\} - \PB\big(\widetilde{s}_i \le u - \Delta_{n,m}\,|\,X_{1:m} = x_{1:m}\big)
    \Bigr),
\end{align*}
which in turn implies that
\begin{align}
    \frac{1}{m}\biggl|
    \ssum{i}{1}{m}\Bigl(
    \mathbbm{1}\{\widetilde{s}_i \in \gB(u,\Delta_{n,m})\} &- \PB\bigl(\widetilde{s}_i \in \gB(u,\Delta_{n,m})\,|\, X_{1:m} = x_{1:m}\bigr)
    \Bigr)\biggr| \notag\\
    &\le \frac{1}{m}\biggl|
    \ssum{i}{1}{m}\Bigl(
    \mathbbm{1}\{\widetilde{s}_i \le u+\Delta_{n,m})\} - \PB\big(\widetilde{s}_i \le u + \Delta_{n,m}\,|\,X_{1:m} = x_{1:m}\big)
    \Bigr)\biggr| \notag\\
    &\quad+ \frac{1}{m}\biggl|
    \ssum{i}{1}{m}\Bigl(
    \mathbbm{1}\{\widetilde{s}_i \le u-\Delta_{n,m})\} - \PB\big(\widetilde{s}_i \le u - \Delta_{n,m}\,|\,X_{1:m} = x_{1:m}\big)
    \Bigr)\biggr|.
    \label{eq:average-ind-sitilde-deviation}
\end{align}
In addition, it is straightforward to verify that: conditional on $X_{1:m} = x_{1:m}$, 
the quantity $\widetilde{s}_{i}$ (see \eqref{eq:defn-Ftilde-si-tilde} and \eqref{eq:defn-mu-tilde-exp}) is independent of $Y_{1:m} \setminus \{Y_i\}$, so that 
the collection $\{\widetilde{s}_{i}\}_{i=1}^m$ 
forms a set of mutually independent random variables.  
Hence, \Cref{lem:DKW_ineq} readily tells us that, conditional on $X_{1:m} = x_{1:m}$, 
\begin{align*}
    &\sup\limits_{u\in\RB}\left\{\frac{1}{m}\abs{\ssum{i}{1}{m}\Bigl(
    \mathbbm{1}\{\widetilde{s}_i \le u+ \Delta_{m,n}\} - 
    \PB\big(\widetilde{s}_i \le u + \Delta_{n,m}\,|\,X_{1:m}= x_{1:m}\big)
    \Bigr)}\right\} \le 5\sqrt{\frac{\log (10/\delta)}{m}}\\
    &\sup\limits_{u\in\RB}\left\{\frac{1}{m}\abs{\ssum{i}{1}{m}\Bigl(
    \mathbbm{1}\{\widetilde{s}_i \le u- \Delta_{m,n}\} - 
    \PB\big(\widetilde{s}_i \le u - \Delta_{n,m}\,|\,X_{1:m}= x_{1:m}\big)
    \Bigr)}\right\} \le 5\sqrt{\frac{\log (10/\delta)}{m}}
\end{align*}
hold with probability exceeding $1-\delta/5$, which taken together with \eqref{eq:average-ind-sitilde-deviation} shows that for any realization $x_{1:m}$, 
\begin{equation}\label{eq:gE_2_ge_1-delta}
    \PB\left(\gE_2(x_{1:m})\mid X_{1:m}=x_{1:m}\right) \ge 1- \frac{\delta}{5}.
\end{equation}

Continuing from the derivation in Eqn.~\eqref{eq:tilde_F-hat_F_1}, we can now see that: on the event $\gE_1(x_{1:m}) \cap \gE_2(x_{1:m})$,  
for any $u \in \RB$ we have
\begin{equation}\label{eq:tilde_F-hat_F_2}
\begin{aligned}
    \abs{\widetilde{F}_{Z_{1:m}}(u) - \widehat{F}_{Z_{1:m}}(u)}
    &\le \frac{1}{m}\ssum{i}{1}{m}\mathbbm{1}\left\{\widetilde{s}_i \in \gB(u,\Delta_{n,m})\right\}\\
    &\le \frac{1}{m}\ssum{i}{1}{m}\PB\big(\widetilde{s}_i \in \gB(u,\Delta_{n,m})\,|\, X_{1:m}=x_{1:m}\big) + 10\sqrt{\frac{\log ({10}/{\delta})}{m}}\\
    &\le \frac{1}{m}\ssum{i}{1}{m}2L_1\Delta_{n,m} + 10\sqrt{\frac{\log (10/\delta)}{m}}\\
    &= \frac{2{L}}{n}\sqrt{{m}\log \frac{10{m}}{\delta}} + 10\sqrt{\frac{\log (10/\delta)}{m}},
\end{aligned}
\end{equation}
where the second line is valid on $\gE_2(x_{1:m})$,  and the third line makes use of the assumption stated in Assumption~\ref{ass:lip_cond_distr}, 
along with the fact that the interval $\gB(u, \varepsilon)$ has length $2\varepsilon$. 
As a remark, in view of Claim~\ref{clm:mu_conc} and \eqref{eq:gE_2_ge_1-delta}, we know that \eqref{eq:tilde_F-hat_F_2} holds with high probability when conditioned on $X_{1:m}=x_{1:m}$.

\paragraph{Step 3: the surrogate empirical distribution vs.~the marginal coverage rate.}
Next, we would like to bound the discrepancy between the surrogate empirical distribution $\widetilde{F}_{Z_{1:m}}(u)$ (cf.~\eqref{eq:defn-Ftilde-si-tilde}) and the marginal coverage rate 
\begin{align}
\overline{F}_m(u) \coloneqq \frac{1}{m}\sum_{i=1}^m 
\PB_{\gD_{1:m}}\big(\widetilde{s}_{i} \le u\big).
\label{eq:defn-Fm-bar-u}
\end{align}
Towards this end, we find it convenient to introduce two auxiliary coverage rates conditional on \(X_{1:m}\) as intermediary quantities, namely, 
\begin{align}
\widetilde{F}_{X_{1:m}}(u) \coloneqq \frac{1}{m}\ssum{i}{1}{m}\PB_{Y_i\mid X_{i}}\left(\widetilde{s}_{i}\le u~\big|~ X_{1:m}\right),\quad 
\overline{F}_{X_{1:m}}(u) \coloneqq \frac{1}{m}\ssum{i}{1}{m} \PB_{Y_i\mid X_i}(\widetilde{s}_i \le u\mid X_i),
\label{eq:defn-Ftilde-Fbar-X-m}
\end{align}
where for any $i= 1,\ldots, m$,  we define
\begin{align}
\PB_{Y_i\mid X_i}(\widetilde{s}_i \le u\mid X_i)\coloneqq \EB_{X_{1:m}\backslash \{X_i\}}\big[\PB(\widetilde{s}_i \le u\mid X_{1:m})\big].
\label{eq:defn-P-Yi-Xi-stilde}
\end{align}
It then follows from the triangle inequality that
\begin{equation}\label{eq:F_X-F_m_decomp}
\begin{aligned}
    \abs{\widetilde{F}_{Z_{1:m}}(u) - \overline{F}_m(u)} &\le
    \abs{\widetilde{F}_{Z_{1:m}}(u) - \widetilde{F}_{X_{1:m}}(u)} + 
    \abs{\widetilde{F}_{X_{1:m}}(u) - \overline{F}_{X_{1:m}}(u)} +
    \abs{\overline{F}_{X_{1:m}}(u) - \overline{F}_m(u)}
    \\
    &=\underbrace{\frac{1}{m}\biggl|\,\ssum{i}{1}{m}\Bigl(\mathbbm{1}\big\{\widetilde{s}_{i}\le u\big\}
    - \PB_{Y_i\mid X_{1:m}}\left(\widetilde{s}_{i} \le u~\big|~ X_{1:m}\right)\Bigr)\,\biggr|
    }_{\eqqcolon \gT_1(u, Z_{1:m})}\\
    &\quad + \underbrace{
    \frac{1}{m}\biggl|\,\ssum{i}{1}{m}\Bigl(\PB_{Y_i\mid X_{1:m}}\left(\widetilde{s}_{i} \le u\,\big|\, X_{1:m}\right)
    - \PB_{Y_i\mid X_i}(\widetilde{s}_i \le u\mid X_i)
    \Bigr)
    \,\biggr|
    }_{\eqqcolon \gT_2(u,X_{1:m})}\\
    &\quad + \underbrace{
    \frac{1}{m}\biggl|\,\ssum{i}{1}{m}\left(
     \PB_{Y_i\mid X_{i}}\left(\widetilde{s}_{i} \le u\,\big|\, X_{i}\right) - \PB_{\gD_{1:m}}\left(\widetilde{s}_{i} \le u\right)
    \right)\,\biggr|
    }_{\eqqcolon \gT_3(u,X_{1:m})},
\end{aligned}
\end{equation}
thus leaving us with three terms to cope with.

\paragraph{Step 4: a bound on the first term in \eqref{eq:F_X-F_m_decomp}.} 
Regarding the first term $\gT_1(u, Z_{1:m})$ on the right-hand side of \eqref{eq:F_X-F_m_decomp}, consider the following event w.r.t.~a given realization $x_{1:m}$:
\begin{equation}
\gE_3(x_{1:m}) \coloneqq \left\{\sup\limits_{u\in\RB}\gT_1(u, Z_{1:m}) \le 5\sqrt{\frac{\log (5/\delta)}{m}}\right\}.
\label{eq:defn-eG3-x1m}
\end{equation}
As mentioned earlier, when conditioned on $X_{1:m} = x_{1:m}$,  
the random variables $\{\widetilde{s}_{i}\}_{i=1}^m$ are mutually independent, allowing us to invoke the generalized DKW inequality (\Cref{lem:DKW_ineq}) to show that 
\begin{equation}
\PB\left(\gE_3(x_{1:m}) \mid X_{1:m} = x_{1:m}\right) \ge 1 - \delta/5.
\label{eq:P-E3-event-prob}
\end{equation}


\paragraph{Step 5: a bound on the second term in \eqref{eq:F_X-F_m_decomp}.} 

Regarding the second term $\gT_2(u, X_{1:m})$ on the right-hand side of \eqref{eq:F_X-F_m_decomp}, we first single out a few properties about $\overline{F}_{X_{1:m}}(u)$ (cf.~\eqref{eq:defn-Ftilde-Fbar-X-m}).  
Without loss of generality, consider two realizations $x_{1:m}$ and $x_{1:m}^\prime$ that differ only in the first sample (i.e., $x_1 \neq x_1'$). We would like to show that, for any $i \ge 1$, the function 
$\PB_{Y_i \mid x_{1:m}}\big(\widetilde{s}_{i}\le u \mid X_{1:m}=x_{1:m}\big)$, 
viewed as a function of $x_{1:m}$, satisfies the bounded difference property. In fact, in view of the definition of $\widetilde{\mu}_{x_{1:m}}(x_i)$ (cf.~\eqref{eq:defn-mu-tilde-exp}), we have
\begin{align}
    &\abs{\widetilde{\mu}_{ x_{1:m}}(x_i) - \widetilde{\mu}_{x_{1:m}^\prime}(x_i)} = \abs{
    \EB_{Y_{1:m}\mid x_{1:m}}\left[\widehat{\mu}_{Z_{1:m}^x}(x_i)\right] - \EB_{Y_{1:m}^\prime\mid x_{1:m}^\prime}\left[\widehat{\mu}_{Z_{1,m}^{x'}}(x_i)\right]
    } \notag\\
    &\quad \overset{\mathrm{(a)}}{\le}
    \abs{
    \EB_{Y_{2:m}\mid x_{2:m}}\Big[
    \EB_{Y_1 \mid x_1}[\widehat{\mu}_{Z_{1:m}^x}(x_i) \mid Y_{2:m}]
    - \EB_{Y_1^\prime\mid x_1^\prime}[\widehat{\mu}_{Z_{1:m}^{x'}}(x_i)\mid Y_{2:m}]
    \Big]
    } \notag\\
    &\quad \le \EB_{Y_{2:m}\mid x_{2:m}}\Big[
    \EB_{Y_1\times Y_1^\prime|(x_1, x_1^\prime)}\left[
    \abs{\widehat{\mu}_{Z_1^x\cup Z_{2:m}^x}(x_i) - \widehat{\mu}_{Z_1^{x'}\cup Z_{2:m}^x}(x_i)} ~\big|~ Y_{2:m}
    \right]
    \Big] \notag\\
    &\quad \overset{\mathrm{(b)}}{\le} \EB_{Y_{2:m}\mid x_{2:m}}\left[\EB_{Y_1\times Y_1^\prime\mid (x_1,x_1^\prime)}\left[\frac{L_2}{n}~\Big|~ Y_{2:m}\right]\right] = \frac{L_2}{n},
    \label{eq:tilde-mu-tilde-mu-gap}
\end{align}
where we denote $Z_{1:m}^{x'}\coloneqq \{(x_i',Y_i')\}_{i=1}^m$, $Z_1^{x}\coloneqq (x_1,Y_1)$, and $Z_1^{x'}\coloneqq (x_1',Y_1')$.
Here, (b) results from Assumption~\ref{ass:fair_alg}. Regarding (a), it follows from the fact that $x_{2:m}=x'_{2:m}$; 
in particular, $Y_{2:m}$ and $Y'_{2:m}$ have the same joint distribution, and are both independent of $(Y_1,Y'_1)$, which allow us to couple $Y_{1:m}$ and $Y'_{1:m}$, so that the two samples differ only at the first data point, i.e., $(x_1,Y_1)\neq(x'_1,Y'_1)$. 
 Consequently,  combining \eqref{eq:tilde-mu-tilde-mu-gap} with Assumption~\ref{ass:lip_cond_distr} reveals that, for any $i=2,\ldots, m$, 
\begin{align}
    &\Big|~\PB_{Y_i\mid x_i}\left(\abs{Y_i - \widetilde{\mu}_{x_{1:m}}(x_i)} > u\right) - \PB_{Y_i\mid x_i}\left(\abs{Y_i - \widetilde{\mu}_{x_{1:m}^\prime}(x_i )} > u\right) \,\Big| \notag\\
    &\le \PB\Big(
    u - \abs{\widetilde{\mu}_{x_{1:m}}(x_i ) - \widetilde{\mu}_{x_{1:m}^\prime}(x_i)} \le \abs{Y_i - \widetilde{\mu}_{x_{1:m}}(x_i)}
    \le u + \abs{\widetilde{\mu}_{x_{1:m}}(x_i) - \widetilde{\mu}_{x_{1:m}^\prime}(x_i)}
    \Big) \notag\\
    &\le 2L_1 
    \abs{\widetilde{\mu}_{x_{1:m}}(x_i ) - \widetilde{\mu}_{x_{1:m}^\prime}(x_i)}
    \le \frac{2L_1L_2}{n}=\frac{2{L}}{n}.
    \label{eq:P-Ycondx-diff-tildemu}
\end{align}

Now, let us return to $\overline{F}_{x_{1:m}}(u)$. Applying the above bound \eqref{eq:P-Ycondx-diff-tildemu} along with straightforward calculations reveals that, for any $x$,
\begin{align*}
    \abs{\widetilde{F}_{x_{1:m}}(u) - \widetilde{F}_{x_{1:m}^\prime}(u)} &\le \frac{1}{m}\abs{\PB_{Y_1\mid x_1}\big(\widetilde{s}_{1} \le u\mid x_{1:m}\big)
     - \PB_{Y_1^\prime\mid x_1^\prime}\big(\widetilde{s}_{1}' \le u \mid x_{1:m}'\big)
    }\\
    &+ \frac{1}{m}\abs{\ssum{i}{2}{m}\Big(
    \PB_{Y_i\mid x_i}\big(\widetilde{s}_{i}\le u\mid x_{1:m}'\big) - \PB_{Y_i\mid x_i}\big(\widetilde{s}_{i}'\le u\mid x_{1:m}'\big)
    \Big)}
    \le \frac{1}{m} + \frac{2{L}}{n},
\end{align*}
where $\widetilde{s}_i' \coloneqq \bigl|Y_i' - \widetilde{\mu}_{x_{1:m}'}(x_i')\bigr|$.
Also, note that $\PB_{Y_i\mid x_i}(\widetilde{s}_i\le u \mid x_i)$ is a function of $x_i$ only, and hence
\[
\bigl|\,\overline{F}_{x_{1:m}}(u)  - \overline{F}_{x_{1:m}'}(u)\,\bigr|
= \frac{1}{m}\bigl|\, \PB_{Y_1\mid x_1}(\widetilde{s}_1 \le u \mid x_1) - \PB_{Y_1\mid x_1'}(\widetilde{s}_1' \le u \mid x_1')\,\bigr| \le \frac{1}{m}.
\]
As a consequence, for any $x$, the function $\bigl|\,\widetilde{F}_{x_{1:m}}(u) - \overline{F}_{x_{1:m}}(u)\,\bigr|$ satisfies the bounded difference property 
with coefficient $\tfrac{2}{m} + \tfrac{2{L}}{n}$.
If we define $$\gT_2(x_{1:m}) \coloneqq \sup\limits_{u\in \RB}\left\{\gT_2(u, x_{1:m})\right\},$$
then simple computation yields
\begin{align*}
    \gT_2(x_{1:m}) - \gT_2(x_{1:m}') &= \sup\limits_{u\in \RB}\left\{\gT_2(u,x_{1:m})\right\} - \sup\limits_{u'\in \RB}\left\{\gT_2(u',x_{1:m}')\right\}\\
    &\le \sup\limits_{u\in \RB}\left\{\gT_2(u,x_{1:m}) - \gT_2(u,x_{1:m}')\right\}
    \le \frac{2}{m}+\frac{2L}{n}.
\end{align*}
Thus, we can apply McDiarmid’s inequality (Lemma~\ref{lem:mcdiarmid}) to derive
\begin{equation}\label{eq:cond_conc_pt_2}
\begin{aligned}
\sup\limits_{u\in \RB}\left\{\abs{\widetilde{F}_{X_{1:m}}(u) - \overline{F}_{X_{1:m}}(u)
}\right\}
&=\gT_2(X_{1:m})\\
&\le \EB_{X_{1:m}}\left[\gT_2(X_{1:m})\right] + \left(\frac{2}{m}+\frac{2L}{n}\right)\sqrt{{m}\log \frac{5}{\delta}}
\end{aligned}
\end{equation}
holds with probability at least $1 - \delta/5$.

It then comes down to bounding $\EB[\gT_2(X_{1:m})]$. From the definition of $\gT_2(X_{1:m})$, we have
\begin{align}
    \EB[\gT_2(X_{1:m})] &= 
    \EB_{X_{1:m}}\left[\sup\limits_{u\in \RB}\abs{\frac{1}{m} \ssum{i}{1}{m} \Bigl(\,\PB_{Y_i\mid X_i}(\widetilde{s}_i \le u \mid X_{1:m}) - \PB_{Y_i\mid X_i}(\widetilde{s}_i \le u\mid X_i)\,\Bigr) }\right] \notag\\
    &\le \EB_{X_{1:m}}\left[\sup\limits_{u\in \RB}\frac{1}{m} \ssum{i}{1}{m} \Bigl|\,\PB_{Y_i\mid X_i}(\widetilde{s}_i \le u \mid X_{1:m}) - \PB_{Y_i\mid X_i}(\widetilde{s}_i \le u\mid X_i)\,\Bigr| \right]\notag\\
    &\le \frac{1}{m}\ssum{i}{1}{m}\EB_{X_{1:m}}\left[
    \sup\limits_{u\in \RB}\Bigl|\, \PB_{Y_i\mid X_i}(\widetilde{s}_i \le u \mid X_{1:m}) - \PB_{Y_i\mid X_i}(\widetilde{s}_i \le u\mid X_i)
    \,\Bigr|\right]\notag\\
    &\le \frac{1}{m}\ssum{i}{1}{m}\EB_{X_i}\Biggl[\,
    \mathop{\EB}\limits_{X_{1:m}\backslash \{X_i\}}\Bigl[\,\underbrace{\sup\limits_{u\in \RB}\biggl|\, \PB_{Y_i\mid X_i}(\widetilde{s}_i \le u \mid X_{1:m}) - \PB_{Y_i\mid X_i}(\widetilde{s}_i \le u\mid X_i)
    \,\Bigr|}_{\eqqcolon \,\gK_i(X_{1:m})}\,\big|\, X_i\,\biggr]
    \,\Biggr],
    \label{eq:E-T2-X-1m-UB135}
\end{align}
which motivates us to control $\gK_i(X_{1:m}),~ i = 1,\ldots, m$. Without loss of generality, it suffices to analyze $\gK_1(X_{1:m})$, since the same argument applies to the remaining $i$. Fix $X_1 = x_1$, it is seen that $\widetilde{\mu}_{X_{1:m}}(\cdot)$ satisfies Assumption~\ref{ass:fair_alg} with parameter $L_2/n$ with respect to the remaining samples $X_{2:m}$. Applying \Cref{lem:cond_prob_concentration} reveals that, for any $x_1$ (it can be regarded as the target sample is $(X_1,Y_1)$ and $X_1 \sim \delta_{\{x_1\}}$ at this time), 
\[
\gK_1(\{x_1\}\cup X_{2:m}) \le \frac{16L}{n}\sqrt{m\log n}
\]
holds with probability at least $1-1/n$, which combined with \eqref{eq:E-T2-X-1m-UB135} gives
\[
\EB[\gT_2(X_{1:m})] \le \frac{1}{m}\ssum{i}{1}{m}\EB_{X_i}\left[\mathop{\EB}_{X_{1:m}\backslash \{X_i\}}\left[\gK_i(X_{1:m})\mid X_i \right]\right] \le \frac{16L}{n}\sqrt{m\log n}+ \frac{1}{n}.
\]
Plugging this into \eqref{eq:cond_conc_pt_2} yields that the following event
\begin{equation}
\gE_4\coloneqq \left\{\sup\limits_{u\in \RB}\{\gT_2(u,X_{1:m})\}\le 3\sqrt{\frac{\log ({5}/{\delta})}{m}} + \frac{18L}{n}\sqrt{{m}\log \left(\frac{5}{\delta}+n\right)}\right\}
\label{eq:defn-E4-claim-135}
\end{equation}
happens with probability at least $1 - \delta / 5$.

\paragraph{Step 6: a bound on the last term in \eqref{eq:F_X-F_m_decomp}.}  We now turn attention to the last term $\gT_3(u,X_{1:m})$ on the right-hand side of \eqref{eq:F_X-F_m_decomp}. 
Define 
$$H(u,X_i)\coloneqq \PB_{Y_i\mid X_i}(\widetilde{s}_i \le u\mid X_i),\qquad i=1,\ldots, m.$$ 
It is easily seen that the random variables $\{H(u,X_i)\}_{1\leq i\leq m}$ 
are mutually independent. Moreover, for any $i$ and any fixed $X_i$, $H(u,X_i)$ is a non-decreasing function in $u$.
Applying \Cref{lem:DKW_ineq} reveals that the event
\begin{equation}
\gE_5 \coloneqq \left\{
\sup\limits_{u\in \RB}\left\{\big|\overline{F}_{X_{1:m}}(u) - \overline{F}_m(u)\big|\right\} \le 5\sqrt{\frac{\log(5/\delta)}{m}}
\right\}
\label{eq:defn-E5-x1m}
\end{equation}
happens with probability at least $1- \delta / 5$, where we remind the reader of the definitions of $\overline{F}_{X_{1:m}}$ and $\overline{F}_m$ in \eqref{eq:defn-Ftilde-Fbar-X-m} and \eqref{eq:defn-Fm-bar-u}, respectively.


\paragraph{Step 7: putting all pieces together.}
To finish up,  let us put together the preceding results.
First, define
\begin{equation}
F_m(u) \coloneqq \frac{1}{m}\ssum{i}{1}{m}\PB_{\gD_{1:m}}\bigl( s_i \le u\bigr).
\end{equation}
On the event $\left(\bigcap_{i=1}^3 \gE_i(x_{1:m})\right)\cap \gE_4\cap \gE_5$, we see that for any $x \in \RB$, it always holds that
\begin{align*}
    \sup\limits_{u\in \RB}&\abs{\widehat{F}_{Z_{1:m}}(u) - {F}_m(u)} \le 
    \abs{\widehat{F}_{Z_{1:m}}(u) - \widetilde{F}_{Z_{1:m}}(u)} + \abs{\widetilde{F}_{Z_{1:m}}(u) - F_m(u)}\\
    &\overset{\eqref{eq:F_X-F_m_decomp}}{\le}
    \sup\limits_{u\in \RB}\Bigl\{\abs{\widehat{F}_{Z_{1:m}}(u) - \widetilde{F}_{Z_{1:m}}(u)} \Bigr\} + \sup\limits_{u\in\RB}\left\{\gT_1(u, Z_{1:m})\right\}\\
    &\quad + \sup\limits_{u\in\RB}\left\{\gT_2(u, Z_{1:m})\right\} + \sup\limits_{u\in\RB}\left\{\gT_3(u, Z_{1:m})\right\}
    + \sup\limits_{u\in \RB}\left\{\abs{\overline{F}_m(u) - F_m(u)}\right\}\\
    &\le \frac{2{L}}{n}\sqrt{{m}\log \frac{10{m}}{\delta}} + 10\sqrt{\frac{\log ({10}/{\delta})}{m}} + 5\sqrt{\frac{\log (5/\delta)}{m}}\\
    &\quad + 3\sqrt{\frac{\log(5/\delta)}{m}} + \frac{18L}{n}\sqrt{{m}\log\left(\frac{5}{\delta}+ n\right)} + 5\sqrt{\frac{\log (5/\delta)}{m}} + \sup\limits_{u\in \RB}\left\{\abs{\overline{F}_m(u) - F_m(u)}\right\}\\
    &\le 24\sqrt{\frac{\log(10/\delta)}{m}} + \frac{24L}{n}\sqrt{{m}\log \left(\frac{10{m}}{\delta} + n \right)}.
\end{align*}
To justify the last inequality, we observe that, for any $u$,
\begin{align*}
    &\abs{\overline{F}_m(u) - F_m(u)} \le \frac{1}{m}\ssum{i}{1}{m}
    \abs{\PB(\widetilde{s}_{i}\le u) - \PB({s}_{i} \le u)}\\
    &\le \frac{1}{m}\ssum{i}{1}{m}\Bigg\{
    \PB\Bigl(\widetilde{s}_{i} - \abs{\widetilde{\mu}_{X_{1:m}}(X_i) - \widehat{\mu}_{Z_{1:m}}(X_i)} \le u\Bigr)
    - \PB\Bigl(\widetilde{s}_{i} + \abs{\widetilde{\mu}_{X_{1:m}}(X_i) - \widehat{\mu}_{Z_{1:m}}(X_i)} \le u\Bigr)
    \Bigg\}\\
    &\overset{\mathrm{(a)}}{\le} \frac{1}{m}\ssum{i}{1}{m}\PB\left(
    \widetilde{s}_i \in \gB\left(u, \Delta\right)
    \right) 
    + \frac{1}{m}\ssum{i}{1}{m}\PB\bigl(\abs{\widetilde{\mu}_{X_{1:m}}(X_i) - \widehat{\mu}_{Z_{1:m}}(X_i)} > \Delta\bigr)
    \le \frac{4{L}}{n}\sqrt{{m}\log n}+ \frac{1}{n},
\end{align*}
where $\Delta \coloneqq \frac{2L_2}{n}\sqrt{{m}\log n}$, and (a) follows by invoking the same argument as in the analysis of $\gE_1(X_{1:m})$.
By combining our uniform high-probability bounds on $\gE_i(x_{1:m})$ for $i = 1, 2, 3$ given any $X_{1:m}=x_{1:m}$, and applying the high-probability bound of $\gE_4$ and $\gE_5$ as well as the union bound, we arrive at
\begin{align*}
    \PB_{\gD_{1:m}}&\left(\left\{\left(\bigcap\limits_{i=1}^3 \gE_i(X_{1:m})\right)\cap \gE_4\cap \gE_5\right\}^{\mathrm{c}}\right) =
    \PB_{\gD_{1:m}}\left(\left(\bigcup \gE_i(X_{1:m})^{\mathrm{c}}\right)\cup \gE_4^{\mathrm{c}} \cup \gE_5^{\mathrm{c}}\right)\\
    &\le \ssum{i}{1}{3} \PB_{\gD_{1:m}}\big(\gE_i(X_{1:m})^{\mathrm{c}}\big) + \frac{\delta}{5} + \frac{\delta}{5} = \ssum{i}{1}{3}\EB_{X_{1:m}}\left[\PB_{Y_{1:m}\mid X_{1:m}}\left(\gE_i(X_{1:m})^{\mathrm{c}}\mid X_{1:m}\right)\right] + \frac{2\delta}{5}\\
    &\le \frac{3\delta}{5} + \frac{2\delta}{5} = \delta.
\end{align*}
This completes the proof of \Cref{lem:empr_prob_concentration}.

\subsubsection{Proof of \Cref{lem:one_stage_regret_full}}\label{sec:prf:lem:one_stage_regret_full}
The proof closely follows the arguments in the proof of  \Cref{lem:one_round_regret,lem:one_stage_regret}. The only difference lies in that, for the smooth drift setting, inequality~\eqref{eq:one_round_regret_6} in the pretrained-score setting needs to be modified as follows:
\begin{align}
    A_k &= \frac{1}{\abs{\gI_k}}\sum_{l\in \gI_k}\big(\alpha - \PB(Y_{n,r,l}\notin \gC_{n,r}(X_{n,r,l})\mid \gC_{n,r})\big) \notag\\
    &\le \frac{1}{\abs{\gI_k}}\sum_{l\in \gI_k}\big(\PB\big(Y_{n,r,i_{k-1}} \notin \gC_{n,r}(X_{n,r,i_{k-1}})\mid \gC_{n,r}\big) - \PB(Y_{n,r,l}\notin \gC_{n,r}(X_{n,r,l}\mid \gC_{n,r}))\big) \notag\\
    &= \frac{1}{\abs{\gI_k}}\sum\limits_{l\in \gI_k}\ssum{i}{i_{k-1}}{l-1}\Big\{\PB\big(Y_{n,r,i} \notin \gC_{n,r}(X_{n,r,i})\mid \gC_{n,r}\big) - \PB\big(Y_{n,r,i+1} \notin \gC_{n,r}(X_{n,r,i+1})\mid \gC_{n,r}\big)\Big\}
\end{align}
for any even $k$, 
where we invoke the same arguments as in \eqref{eq:one_round_regret_6}, albeit with notation adjusted to the full conformal setting. 
We observe that, for any $i\in i_{k-1},\ldots, l-1$,
\begin{align*}
\PB\big(Y_{n,r,i} \notin \gC_{n,r}(X_{n,r,i})\mid \gC_{n,r}\big) &- \PB\big(Y_{n,r,i+1} \notin \gC_{n,r}(X_{n,r,i+1})\mid \gC_{n,r}\big)\\
&= \EB\big[
\mathbbm{1}\{Y_{n,r,i}\notin \gC_{n,r}(X_{n,r,i})\} 
\mid \gC_{n,r}\big] - 
\EB\big[
\mathbbm{1}\{Y_{n,r,i+1}\notin \gC_{n,r}(X_{n,r,i+1})\} 
\mid \gC_{n,r}\big]
\\
&\le \sup\limits_{h\in \gM([0,1])}\Big\{\EB\big[h(Z_{n,r,i})\big]-
\EB\big[h(Z_{n,r,i+1})\big]\Big\} = \mathsf{TV}\big(Z_{n,r,i}, Z_{n,r,i+1}\big),
\end{align*}
where $\gM([0,1])$ denotes all measurable functions of $Z\in \gX\times \RB$ that are bounded in $[0,1]$. Accordingly, the bound \eqref{eq:one_round_regret_777} on $A_k$ in the pretrained-score case  can now be replaced by
\begin{align*}
    A_k \le \frac{1}{\abs{\gI_k}}\sum\limits_{l\in \gI_k}\ssum{i}{i_{k-1}}{l-1} \mathsf{TV}(Z_{n,r,i}, Z_{n,r,i+1}) \le \ssum{i}{i_{k-1}}{i_k - 1} \mathsf{TV}(Z_{n,r,i}, Z_{n,r,i+1})
\end{align*}
for each even $k$. 
As a result, under smooth drift, the complexity measure used to control the cumulative regret in stage $n$ should now be  $\mathsf{TV}_n^{\mathsf{stage}}$ (cf.~\eqref{eq:defn-TV-per-stage-full-conformal}) rather than $\mathsf{KS}_n^{\mathsf{stage}}$ (cf.~\eqref{eq:defn-KS-tau-n-r}).

The remaining arguments are the same as for \Cref{lem:one_round_regret,lem:one_stage_regret}, and are hence omitted for brevity.

\subsubsection{Proof of \Cref{lem:control_rare_event_full}}\label{sec:prf:lem:control_rare_event_full}
For any $n$, given that $\gB_n=\gA_{n,r_n}\cap \gG(\tau_{n,r_n-1},\tau_{n,r_n})$, one has
\begin{equation}\label{eq:control_rare_1}
    \EB\big[\tau_{n+1}\mathbbm{1}\{\gB_n^{\mathrm{c}}\}\big]
    \le \EB\big[\tau_{n+1}\mathbbm{1}\{\gA_{n,r_n}^{\mathrm{c}}\}\big]
    + \EB\big[\tau_{n+1}\mathbbm{1}\{\gG(\tau_{n,r_n-1},\tau_{n,r_n})^{\mathrm{c}}\}\big],
\end{equation}
leaving us with two terms to control. 

We first bound $\EB\big[\tau_{n+1}\mathbbm{1}\{\gA_{n,r_n}^{\mathrm{c}}\}\big]$. Recognizing that
$n\le \tau_{n,r_n-1} < \tau_{n,r_n}$ and $\tau_{n+1}\le 4\tau_{n,r_n}$ (since the round lengths grow geometrically), we can deduce that
\begin{equation}\label{eq:control_rare_2}
\begin{aligned}
    \EB\big[\tau_{n+1}\mathbbm{1}\{\gA_{n,r_n}^{\mathrm{c}}\}\big]
    &\le \sum\limits_{1\le k<m\le T}\EB\Big[4m\,\mathbbm{1}\big\{\gA_{n,r_n}^\mathrm{c}\big\}\mathbbm{1}\{\tau_{n,r_n-1}=k;~ \tau_{n,r_n}=m\}\Big]\\
    &= \sum\limits_{1\le k<m\le T}4m\EB \Big[\mathbbm{1}\{\tau_{n,r_n-1}=k;~ \tau_{n,r_n}=m\}\EB\big[\mathbbm{1}\big\{\gA_{n,r_n}^{\mathrm{c}}\big\}\mid Z_{1:m-1}\big]\Big]\\
    &\le \sum\limits_{1\le k<m\le T}4m\EB \bigg[\mathbbm{1}\{\tau_{n,r_n-1}=k; \tau_{n,r_n}=m\}
    \Big(\sum\limits_{m\le i<j<\infty}\PB\big(\gA(k,m;i,j)^{\mathrm{c}}\,|\, Z_{1:m-1}\big)\Big)\bigg]\\
    &\overset{\eqref{eq:up_bd_A_kmij_c}}{\le } 
    \sum\limits_{1\le k<m\le T}4m\EB \bigg[\mathbbm{1}\{\tau_{n,r_n-1}=k;~ \tau_{n,r_n}=m\}
    \Big(\sum\limits_{m\le i<j<\infty}j^{-8}\Big)\bigg]\\
    &\le \sum\limits_{1\le k<m\le T}4m\EB \bigg[\mathbbm{1}\{\tau_{n,r_n-1}=k;~ \tau_{n,r_n}=m\}
    \big(m^{-6}/42\big)\bigg]
    = \frac{2}{21}\EB\big[\tau_{n,r_n}^{-5}\big]\le \frac{2}{21n^5}.
\end{aligned}
\end{equation}

Next, we turn attention to the term $\EB\big[\tau_{n+1}\mathbbm{1}\{\gG(\tau_{n,r_n-1},\tau_{n,r_n})^{\mathrm{c}}\}\big]$.
Using $\tau_{n+1}\le 16\tau_{n,r_n-1}$ (which is again due to the geometric growth of the round lengths), we obtain
\begin{equation}\label{eq:control_rare_3}
\begin{aligned}
    \EB\big[\tau_{n+1}\mathbbm{1}\{\gG(\tau_{n,r_n-1},\tau_{n,r_n})^{\mathrm{c}}\}\big]
    &\le \ssum{k}{1}{T}\EB\Big[16k\,\mathbbm{1}\{\tau_{n,r_n-1}=k\}\mathbbm{1}\Big\{\gG(k,\tau_{n,r_n})^{\mathrm{c}}\Big\}\Big]\\
    &\le \ssum{k}{1}{T}\EB\bigg[16k\,\mathbbm{1}\{\tau_{n,r_n-1}=k\}\Big(\ssum{m}{k+1}{\infty}\PB\big(\gG(k,m)^{\mathrm{c}}\,|\, Z_{1:k-1}\big)\Big)\Big]\\
    &\overset{\mathrm{(a)}}{\le} \ssum{k}{1}{T}\EB\bigg[16k\,\mathbbm{1}\{\tau_{n,r_n-1}=k\}\Big(\sum_{m=k+1}^{\infty}m^{-4}\Big)\Big]\\
    &\le \ssum{k}{1}{T}6\EB\Big[k^{-2}\mathbbm{1}\{\tau_{n,r_n-1}=k\}\Big]
    \le 6\EB[\tau_{n,r_n-1}^{-2}] \le \frac{6}{n^2},
\end{aligned}
\end{equation}
where (a) follows from \Cref{prop:training_conditioned_cov} with $\delta = m^{-4}$.

Taking together \eqref{eq:control_rare_1}--\eqref{eq:control_rare_3} thus completes the proof.

\subsubsection{Proof of \Cref{lem:tv_lower_bound_on_typical}}\label{sec:prf:lem:tv_lower_bound_on_typical}

Recall the algorithm procedure of \driftocpfull: in stage $n~(\le N-1)$, the distribution shift is detected in round $r_n$, and this round contains $t_n$ iterations. Then in light of our drift detection subroutine (see Algorithm~\ref{alg:full_drift-detection}), there exists some $j_n \in [t_n]$ such that:
\[
\Bigg|\ssum{l}{j_n}{t_n}\big(\mathbbm{1}\{Y_{n,r_n,l}\notin \gC_{n,r_n}(X_{n,r_n,l})\} - \alpha\big)\Bigg| > 10\sqrt{t_n - j_n + 1}\log^3(40\tau_{n,r_n}).
\]
Then on the event $\gA_{n,r_n}$ (cf.~\eqref{eq:event-Anr-full}), it holds that
%
\begin{align}
    \bigg|\,\ssum{l}{j_n}{t_n}& \big(\PB(Y_{n,r_n,l} \notin \gC_{n,r_n}(X_{n,r_n,l})\,|\, \gC_{n,r_n}) - \alpha\big)\,\bigg|
    \ge \Bigg|\ssum{l}{j_n}{t_n} \big(\mathbbm{1}\{Y_{n,r_n,l} \notin \gC_{n,r_n}(X_{n,r_n,l})\}- \alpha\big)\Bigg| \notag\\
    & \qquad - \Bigg|\ssum{l}{j_n}{t_n} \Big(\mathbbm{1}\big\{Y_{n,r_n,l} \notin \gC_{n,r_n}(X_{n,r_n,l})\big\} - \PB\big(Y_{n,r_n,l} \notin \gC_{n,r_n}(X_{n,r_n,l})\,|\, \gC_{n,r_n}\big)\Big)\Bigg| \notag\\
    &> 10\sqrt{t_n - j_n+1}\log^3 (40\tau_{n,r_n}) - 2\sqrt{(t_n - j_n+1)\log(2\tau_{n+1,1})}
    \notag\\
    &\ge 8\sqrt{t_n - j_n+1}\log^3 (40\tau_{n,r_n}).
    \label{eq:full_regret_2}
\end{align}
As a consequence, if we define
\begin{align}
B_n \coloneqq \frac{1}{t_n - j_n + 1}\ssum{l}{j_n}{t_n}\Big(
\PB\big(Y_{n,r_n,l} \notin \gC_{n,r_n}(X_{n,r_n,l})\,|\, \gC_{n,r_n}\big) - \alpha
\Big),
\label{eq:defn-Bn-full-conformal}
\end{align}
then \eqref{eq:full_regret_2} implies that
\begin{equation}\label{eq:full_regret_3}
    \frac{\sqrt{t_n - j_n + 1}}{8\log^3( 40\tau_{n,r_n})}\abs{B_n} \ge 1.
\end{equation}

The next step is to analyze the quantity $B_n$ defined in \eqref{eq:defn-Bn-full-conformal}. Note that $\tau_{n,r_n} - \tau_{n,r_n-1} = T_{r_n-1}$. Then on the event $\gB_n$ defined in \eqref{eq:def_gB_n-full}---more precisely, on the event $\gG(\tau_{n,r_n-1},\tau_{n,r_n})$ defined in \eqref{eq:defn-Gkm-full}---we have
%
\begin{align}
    \abs{B_n} &\le \Big|\, \PB\big(Y_{n,r_n,1}\notin \gC_{n,r_n}(X_{n,r_n,1})\,|\, \gC_{n,r_n}\big) - \alpha\,\Big| \notag\\
    & \quad+ \frac{1}{t_n-j_n + 1}\ssum{l}{j_n}{t_n}\Big|\, 
    \PB\big(Y_{n,r_n,l} \notin \gC_{n,r_n}(X_{n,r_n,l})\,|\, \gC_{n,r_n}\big) - \PB\big(Y_{n,r_n,1}\notin \gC_{n,r_n}(X_{n,r_n,1})\,|\, \gC_{n,r_n}\big)
    \,\Big| \notag\\
    &\overset{\mathrm{(a)}}{\le}\Big|\, \PB\big(Y_{n,r_n,1}\notin \gC_{n,r_n}(X_{n,r_n,1})\,|\, \gC_{n,r_n}\big) - \alpha\,\Big|
     + \frac{1}{t_n-j_n+1}\ssum{l}{j_n}{t_n}\mathsf{TV}(Z_{n,r_n,l}, Z_{n,r_n,1}) \notag\\
    &\overset{\mathrm{(b)}}{\le} \frac{1}{t_n-j_n+1}\ssum{l}{j_n}{t_n}\left(
    2^6\sqrt{\frac{\log (40\tau_{n,r_n})}{T_{r_n-1}}} +
    \frac{2^7{L}}{\tau_{n,r_n}}\sqrt{T_{r_n-1}\log (40\tau_{n,r_n})}\right) \notag\\
    & \quad+ \frac{1}{t_n-j_n+1}\ssum{l}{j_n}{t_n}\ssum{j}{1}{l-1}\mathsf{TV}\big(Z_{n,r_n,j}, Z_{n,r_n,j+1}\big) + \frac{1}{T_{r_n-1}}\ssum{l}{1}{T_{r_n-1}}\ssum{i}{l}{T_{r_n-1}}\mathsf{TV}\big(Z_{n,r_n-1,i}, Z_{n,r_n-1,i+1}\big) \notag\\
    &\le 2^6\sqrt{\frac{\log (40\tau_{n,r_n})}{T_{r_n-1}}} +
    \frac{2^7{L}}{\tau_{n,r_n}}\sqrt{T_{r_n-1}\log (40\tau_{n,r_n})}+ \mathsf{TV}_n^{\mathsf{tail}},
    \label{eq:full_conf_B_n_bound}
\end{align}
%
where (b) is valid on $\gB_n$, and (a) results from the fact that, for any $l= j_n,\ldots, t_n$,
\begin{align*}
\Big|\,\PB\big(Y_{n,r,l} \notin \gC_{n,r}(X_{n,r,l})\mid \gC_{n,r}\big) &- \PB\big(Y_{n,r,1} \notin \gC_{n,r}(X_{n,r,1})\mid \gC_{n,r}\big)\,\Big|\\
&= \Big|\,\EB\big[
\mathbbm{1}\{Y_{n,r,l}\notin \gC_{n,r}(X_{n,r,l})\} 
\mid \gC_{n,r}\big] - 
\EB\big[
\mathbbm{1}\{Y_{n,r,1}\notin \gC_{n,r}(X_{n,r,1})\} 
\mid \gC_{n,r}\big]\,\Big|
\\
&\le \sup\limits_{h\in \gM([0,1])}\Big\{\EB\big[h(Z_{n,r,l})\big]-
\EB\big[h(Z_{n,r,1})\big]\Big\} = \mathsf{TV}\big(Z_{n,r,l}, Z_{n,r,1}\big),
\end{align*}
where $\gM([0,1])$ denotes the set of all measurable functions $h:\gX\times \RB\to[0,1]$.

 Regarding the first term on the right-hand side of \eqref{eq:full_conf_B_n_bound}, 
 one can apply \eqref{eq:full_regret_3} together with a little algebra to show that: when $\tau_{n,r_n-1}\ge 2$:
\begin{equation}\label{eq:full_conf_B_n_bound_1}
\begin{aligned}
    2^6\sqrt{\frac{\log (40\tau_{n,r_n})}{T_{r_n-1}}} \overset{\eqref{eq:full_regret_3}}{\le} 
    \frac{2^6}{8\log^{5/2}(40\tau_{n,r_n})}\sqrt{\frac{t_n - j_n +1}{T_{r_n-1}}}\abs{B_n} \le \frac{1}{2}\abs{B_n}.
\end{aligned}
\end{equation}
Combine this with~\eqref{eq:full_conf_B_n_bound} and rearrange terms to reach
\[
\abs{B_n} \le \frac{2^8{L}}{\tau_{n,r_n}}\sqrt{T_{r_n-1}\log (40\tau_{n,r_n})} + 2\mathsf{TV}_n^{\mathsf{tail}}.
\]

To finish up, recall that $\gH_n \coloneqq \left\{{T_{r_n-1}\sqrt{\log (40\tau_{n,r_n})}} \le \frac{\tau_{n,r_n}}{256{L}}\right\}$. On the event $\gB_n \cap \gH_n$, combining the above expression with~\eqref{eq:full_regret_3} allows us to establish the following inequality:
\begin{align*}
    2 + 2\sqrt{t_n-j_n+1}\,\mathsf{TV}_n^{\mathsf{tail}}
    &\overset{\gH_n}{\ge} \sqrt{t_n-j_n+1}\left(\frac{2^{8}{L}}{\tau_{n,r_n}}\sqrt{T_{r_n-1}\log (40\tau_{n,r_n})} + 2\mathsf{TV}_{n}^{\mathsf{tail}}\right)\\
    &\ge\sqrt{t_n-j_n+1}\abs{B_n} \overset{\eqref{eq:full_regret_3}}{\ge} 8,
\end{align*}
which in turn implies that
\begin{equation}
\sqrt{t_n-j_n+1}\,\mathsf{TV}_n^{\mathsf{tail}}\mathbbm{1}\{\gB_n\cap\gH_n\} \ge 3\cdot\mathbbm{1}\{\gB_n\cap\gH_n\}.
\end{equation}
This immediately concludes the proof of this lemma. 
\subsection{Proof of Theorem~\ref{thm:lower_bd_gnr_regret}}\label{sec:prf:thm:lower_bd_gnr_regret}


%

Throughout this subsection, we consider the case where no features $\{X_t\}_{t=1}^T$ are observed; instead, only the response $Y_t$ is available at time $t$. Under this simplification,  the set-valued functions $\{\gC_t(\cdot)\}_{t\geq 1}$ induced by algorithm $\pi=\{\pi_t\}_{t\geq 1}$ in \eqref{eq:plcy_construct_C} admit the simpler representation
\begin{equation}\label{eq:plcy_construct_C_simp}
    \gC_t = \begin{cases}
        \pi_1(U), & \text{if }t=1,\\
        \pi_t(Y_{t-1}, \ldots, Y_1, U), & \text{if }t \ge 2,
    \end{cases}
\end{equation}
where 
\(\pi_{t}:\RB^{t-1}\times[0,1]\to 
\gB(\RB)
\) for $t\geq 2$. 
In this setting, $\gC_t$ is uniquely determined by $Y_{1:t-1}$ and $U$; accordingly, we often write it as $\gC(Y_{1:t-1},U)$ as long as it is clear from the context. 


Our proof of Theorem~\ref{thm:lower_bd_gnr_regret} is organized into several steps, presented below.

\paragraph{Step 1: constructing a class of distributions with piecewise flat density.}
We begin by introducing the distribution class $\gI$, constructed as follows. 

\begin{itemize}
\item First, divide the interval $[0,1]$ into $k$ subintervals:
\begin{equation}\label{eq:def_I_j}
I_j \coloneqq \left[\frac{j-1}{k}, \frac{j}{k}\right), \quad j = 1, \ldots, k-1;\quad I_k \coloneqq \left[\frac{k-1}{k},1\right].
\end{equation}

\item 
For any given sequence $V_1,\dots,V_k\in \{-1,1\}$, generate a distribution with probability density function
\begin{subequations}
\label{eq:construct-fx-V}
\begin{align}
f(y \mid V_{1:k}) \coloneqq \ssum{j}{1}{k} f_j(y), 
\end{align}
where $f_j(\cdot)$ is nonzero only within the subinterval $I_j$ as follows:
\[
f_j(y) \propto (1 + \epsilon V_j)\,\mathbbm{1}_{I_j}(y),
\]
with $\epsilon$ a small positive constant to be specified shortly, and 
$\mathbbm{1}_{I_j}(y) \coloneqq \mathbbm{1}\{y \in I_j\}$ the indicator function of $I_j$.
%
The normalization constant can then be computed as
\[
\ssum{j}{1}{k}\int_{I_j} (1 + \epsilon V_j)\, \rd y 
= \ssum{j}{1}{k} \frac{1 + \epsilon V_j}{k} 
\eqqcolon 1 + \epsilon \overline{V},
\]
thereby allowing us to express
\begin{align}
f_j(y) = \frac{(1 + \epsilon V_j)\mathbbm{1}_{I_j}(y)}{1 + \epsilon \overline{V}},
\quad j = 1, 2, \ldots, k.
\end{align}
\end{subequations}
%

\item 
Accordingly, we construct a distribution class as follows
\begin{align}
\gI \coloneqq 
\left\{\, f(y \mid V_{1:k}) \,\big\vert\, V_{1:k} \in \{-1, 1\}^k \,\right\}.
\label{eq:defn-I-dist-class-each-V}
\end{align}
\end{itemize}

\paragraph{Step 2: constructing a family of distribution sequences contained in $\mathcal{L}_3(N^{\mathsf{cp}})$ and $\mathcal{L}_4(\mathsf{TV}_T)$.}
With $\mathcal{I}$ in place, we would like to construct a family $\gL^\prime$ of distribution sequences such that it is composed of all  $\{\mathcal{D}_t\}_{t=1}^T$ satisfying the following two conditions:  
\begin{enumerate}
    \item $\gD_t \in \gI\text{ for every }t = 1, \ldots, T$;
    \item For every $l = 1 ,\ldots, m+1$,  it holds that $\gD_t = \gD_{t+1}$ for any $t$ obeying 
    $(l-1)\lfloor T/m\rfloor + 1 \leq t < \min\left\{l\lfloor T/m\rfloor, T\right\}$. 
    In other words, the distributions are identical within each segment
    \begin{equation}
    \gT_l \coloneqq \big[(l-1)\lfloor T/m\rfloor + 1 ,\min\left\{l\lfloor T/m\rfloor, T\right\}\big],
    \label{eq:defn-Tl-batch-lb}
    \end{equation}
    where each batch $\mathcal{T}_l$ (except for the $(m+1)$-th batch) contains $\lfloor T/m \rfloor$ time instances.   
\end{enumerate}


We now verify that $\mathcal{L}'$ belongs to both  $\mathcal{L}_3(N^{\mathsf{cp}})$ and $\mathcal{L}_4(\mathsf{TV}_T)$ under an appropriate choice of parameters.
\begin{itemize}
\item Regarding the change-point setting, it is clearly seen that $\mathcal{L}'\subset \mathcal{L}_3(N^{\mathsf{cp}})$ when $m=N^{\mathsf{cp}}$.

\item 
Turning to the smooth drift setting, we claim that $\gL'\subset \gL_4(\mathsf{TV}_T)$ for sufficiently small $\epsilon$.
To justify this, consider any two distributions $\gD_1,\gD_2\in \gI$, and suppose that their densities can be written as
$p_{\gD_i}(y)=f(y\mid V_{1:k}^{(i)})$ for $i=1,2$. Assuming that $\epsilon\leq 1/2$, we can calculate
\begin{equation}\label{eq:lower_bd_1}
\begin{aligned}
    &\mathsf{TV}(\gD_1 , \gD_2) 
    = \frac{1}{2}\int_0^1\abs{f_1(y\mid V_{1:k}^{(1)}) - f_1(y\mid V_{1:k}^{(2)})}\rd y= \frac{1}{2}\ssum{j}{1}{k}\int_{I_j}\abs{\frac{1+\epsilon V_j^{(1)}}{1 + \epsilon \overline{V}^{(1)}} - \frac{1+\epsilon V_j^{(2)}}{1 + \epsilon \overline{V}^{(2)}}}\rd y\\
    &\le \ssum{j}{1}{k}\frac{2}{k}\abs{\epsilon \big(V_j^{(1)} + \overline{V}^{(2)} - V_j^{(2)} - \overline{V}^{(1)}\big) + \epsilon^2\big(V_j^{(1)} \overline{V}^{(2)} - V_j^{(2)}\overline{V}^{(1)}\big)}
    \le
    \ssum{j}{1}{k}\frac{2}{k}(4\epsilon + 2\epsilon^2) \le 10\epsilon,
\end{aligned}
\end{equation}
where the inequalities in the last line result from $\epsilon \le 1/2$ and $\max\limits_{j\in [k] , i = 1,2}\{|V_j^{(i)}|,|\overline{V}^{(i)}|\} \le 1$.  As a result, if we take $\epsilon \le \min\{\mathsf{TV}_T/(20m), 1/2\}$, then for any $\{\gD_t\}_{t=1}^T$ we have
\begin{align*}
\ssum{t}{1}{T-1}\mathsf{TV}(\gD_t, \gD_{t+1}) &= \ssum{j}{1}{m}\mathsf{TV}\left(\gD_{(j-1)\lfloor\frac{T}{m}\rfloor + 1}, \gD_{j\lfloor\frac{T}{m}\rfloor+1}\right)\\
& \le m \sup\limits_{\gD_1, \gD_2 \in \gI}\mathsf{TV}(\gD_1, \gD_2)
\overset{\eqref{eq:lower_bd_1}}\le 20m\epsilon \le \mathsf{TV}_T,
\end{align*}
thus ensuring that $\gL^\prime \subset \gL_4(\mathsf{TV}_T)$.

\end{itemize}

\paragraph{Step 3: establishing a general regret lower bound.}
We now look at the cumulative regret within each batch $\gT_i$ ($i = 1, \ldots, m+1$) defined in \eqref{eq:defn-Tl-batch-lb}.  
To this end, we first establish the following lower bound on the coverage gap for a single time point; the proof is deferred to \Cref{sec:prf:lem:coverage_gap_lower_bd}.  

\begin{lemma}\label{lem:coverage_gap_lower_bd}
Suppose that $0<\alpha \leq 1/2$, $k \ge \frac{256K}{\alpha}$ and $\epsilon \le  \min\left\{\frac{\alpha^{5/2}}{200},\frac{1}{64}\sqrt{\frac{\alpha k}{n\log (2nk/\alpha^2)}}\right\}$. 
Consider any $n\geq 1$. 
Then, for any admissible algorithm $\pi \in \mathcal{P}_K$ (cf.~\eqref{eq:defn-PK-interval}), the set-value mapping $\gC$ induced by $\pi$ satisfies
    \begin{equation}\label{eq:coverage_gap_lower_bd}
    \begin{aligned}
    \frac{1}{\abs{\gI}}\sum\limits_{\gD\in \gI}\Big\{ \mathop{\EB}\limits_{Y_{1:n}\sim \gD^n,\, U \sim p_U}
\big[\big|\,\PB\big(Y_{n+1} \in \gC(Y_{1:n}, U)\,|\, Y_{1:n},U\big) - (1-\alpha)\,\big|\big]\Big\}\ge \frac{\alpha^{\frac{5}{2}}\epsilon}{144\sqrt{k}} - \frac{\alpha^{6}}{4n^3k^3},
\end{aligned}
    \end{equation}
    where $U$ is independently drawn from an arbitrary continuous distribution with density function $p_U(\cdot)$, and $\mathcal{I}$ is defined in Step 1 (with the parameter $\epsilon$). 
\end{lemma}
With this intermediate result in hand, we can now analyze the cumulative coverage gap within each batch $\gT_i$ (cf.~\eqref{eq:defn-Tl-batch-lb}). In fact, letting $\tau_i \coloneqq (i-1)\left\lfloor T/m\right\rfloor + 1$, one can write
\begin{align*}
\sum_{t\in \gT_i} \EB\Big[\big|\,\PB\big(Y_t \in \gC(Y_{1:t-1},U)&\,|\, Y_{1:t-1}, U\big) -(1-\alpha)\,\big|\Big]\\
&= \sum_{t \in \gT_i}
\EB\bigg[
\Big|\,
\PB\Big(Y_t \in \gC(Y_{1:\tau_i-1}, Y_{\tau_i:t-1}, U)\,\big|\, Y_{1:t-1},U\Big)
- (1 - \alpha)
\,\Big|
\bigg].
\end{align*}
Note that by construction, the distribution selected for each batch is independent of the distributions assigned to all preceding batches. 
Therefore, we can view $(Y_{1:\tau_i-1}, U)$ jointly as a new random variable 
$\widetilde{U} \sim P_{\widetilde{u}}$ for some distribution $P_{\widetilde{u}}$,  which is independent of all randomness within $\gT_i$. 
Consequently, we shall write the prediction set 
$\gC(Y_{1:\tau_i-1},Y_{\tau_i:t-1},  U)$ 
as $\gC(Y_{\tau_i:t-1}, \widetilde{U})$ in the sequel (as long as it is clear from the context), in order to underscore the role of $Y_{\tau_i:t-1}$. 
Armed with this simplified notation, we define, 
for any given distribution $\gD$ and any index $i \in [m+1]$, the cumulative regret over batch $\mathcal{T}_i$ as
\begin{align*}
\mathsf{regret}_{\pi}(\gD, \gT_i) &\coloneqq \sum\limits_{t\in \gT_i}\EB\left[\abs{
\PB\Big(Y_t \in \gC(Y_{1:\tau_i-1}, Y_{\tau_i:t-1}, U)\,\big|\, Y_{1:t-1},U\Big)
- (1 - \alpha)
}\right] \\
& = \sum\limits_{t\in \gT_i}\EB\left[\abs{
\PB\Big(Y_t \in \gC( Y_{\tau_i:t-1}, \widetilde{U})\,\big|\, Y_{\tau_i:t-1}, \widetilde{U}\Big)
- (1 - \alpha)
}\right].
\end{align*}
Applying Lemma~\ref{lem:coverage_gap_lower_bd} for each time $t$, we reach
\begin{equation}\label{eq:one_seg_lower_bound_full}
\begin{aligned}
\frac{1}{\abs{\gI}}\sum\limits_{\gD\in \gI}\mathsf{regret}_{\pi}(\gD, \gT_i) &=
    \frac{1}{\abs{\gI}}\sum\limits_{\gD\in \gI}\sum\limits_{t\in \gT_i}\EB\Big[
\abs{
\PB\Big(Y_t \in \gC(Y_{\tau_i:t-1}, \widetilde{U})\,\big|\, Y_{\tau_i:t-1}, \widetilde{U}\Big)
- (1 - \alpha)
}
\Big]\\
&= \sum\limits_{t\in \gT_i}\left\{\frac{1}{\abs{\gI}}\sum\limits_{\gD\in \gI}
\EB\Big[
\abs{
\PB\Big(Y_t \in \gC(Y_{\tau_i:t-1}, \widetilde{U})\,\big|\, Y_{\tau_i,t-1}, \widetilde{U}\Big)
- (1 - \alpha)
}
\Big]
\right\}\\
&\ge \sum\limits_{t\in \gT_i}  \left( \frac{\alpha^{\frac{5}{2}}\epsilon}{144\sqrt{k}} - \frac{\alpha^{6}}{4(T/m)^3k^3} \right) 
= \left( \frac{\alpha^{\frac{5}{2}}\epsilon}{144\sqrt{k}} - \frac{\alpha^{6}}{4(T/m)^3k^3} \right)|\gT_i|,
\end{aligned}
\end{equation}
provided that %
\begin{align}\epsilon \le \min\left\{\frac{\alpha^{\frac{5}{2}}}{200}, \frac{1}{64}\sqrt{\frac{\alpha m  k}{T\log(Tk/\alpha m)}}\right\}.
\label{eq:eps-UB-minimax-lb-full}
\end{align}

Now, putting all batches together yields the following regret lower bound: for any algorithm $\pi \in \mathcal{P}_K$, 
\begin{equation}\label{eq:regret_lower_bound_with_eps}
\begin{aligned}
\mathsf{regret}_\pi(\gL', T, K) &= 
    \sup\limits_{\{\gD_t\}_{t=1}^T\in \gL^\prime} \mathsf{regret}_\pi \big(\gD_{1:T}, T\big) = \ssum{i}{1}{m+1}\sup\limits_{\gD \in \gI}\mathsf{regret}_{\pi}(\gD, \gT_i)\\
    &\ge \ssum{i}{1}{m+1}\frac{1}{\abs{\gI}}\sum\limits_{\gD\in\gI}\mathsf{regret}_{\pi}(\gD, \gT_i) \overset{\eqref{eq:one_seg_lower_bound_full}}{\ge} \left( \frac{\alpha^{\frac{5}{2}}\epsilon}{144\sqrt{k}} - \frac{\alpha^{6}}{4n^3k^3} \right)T.
\end{aligned}
\end{equation}

\paragraph{Step 4: instantiating the general lower bound to two drift settings.} 
It remains to connect the above lower bound to the two distribution-drift settings, which we discuss separately below.
\begin{itemize}
    \item \textit{The change-point setting.}
    In this drift scenario, setting $m=N^{\mathsf{cp}}$ ensures that $\mathcal{L}'\subset \mathcal{L}_3(N^{\mathsf{cp}})$ (as discussed in Step 2). Then, taking $k = \frac{256K}{\alpha}$ and $\epsilon = \min\Big\{\frac{\alpha^{5/2}}{200},\frac{1}{64}\sqrt{\frac{\alpha k (N^{\mathsf{cp}}+1)}{T\log (Tk/\alpha)}}\Big\}$  in \eqref{eq:regret_lower_bound_with_eps} yields
    \begin{align*}
        \mathsf{regret}_{\pi}\big(\gL_3(N^{\mathsf{cp}}), T,K\big) \ge \mathsf{regret}_{\pi}\big(\gL', T,K\big) &\ge \frac{1}{300}\min\Bigg\{\frac{\alpha^5 T}{200\sqrt{k}}, \frac{\alpha^3}{16} \sqrt{\frac{(N^{\mathsf{cp}}+1)T}{\log (Tk/\alpha)}}\Bigg\}\\
        &= \widetilde{\Omega}\Bigg(\min\left\{\frac{T}{\sqrt{K}},  \sqrt{(N^{\mathsf{cp}}+1)T}\right\}\Bigg).
    \end{align*}
    
    \item  \textit{The smooth drift setting.}
    In order to simultaneously satisfy \eqref{eq:eps-UB-minimax-lb-full} and $\gL^\prime \subset \gL_4(\mathsf{TV}_T)$ (which needs $\epsilon \le \min\{\mathsf{TV}_T/(20m), 1/2\}$ as discussed in Step 2), we take $\epsilon$ to be
    $$\epsilon = \min\left\{\frac{\alpha^{\frac{5}{2}}}{200}, \frac{\mathsf{TV}_T}{20m}, \frac{1}{64}\sqrt{\frac{\alpha mk}{T\log(Tk/\alpha m)}}\right\}.
    $$
    Substitution into inequality~\eqref{eq:regret_lower_bound_with_eps} leads to
\begin{equation}\label{eq:regret_lower_bound_smooth_1}
\mathsf{regret}_{\pi}\big(\gL_4(\mathsf{TV}_T), T,K\big) \ge
\mathsf{regret}_{\pi}(\gL', T,K) \ge \frac{\alpha^{\frac{5}{2}}T}{300\sqrt{k}}\min\left\{\frac{\alpha^{\frac{5}{2}}}{200},\frac{\mathsf{TV}_T}{20 m}, \frac{1}{64}\sqrt{\frac{\alpha mk}{T\log(Tk/\alpha m)}}\right\}.
\end{equation}
We now divide into two cases. 
\begin{itemize}
\item If $\mathsf{TV}_T\sqrt{\frac{\alpha T}{12K}} \ge 1$,
then let us take $m = \frac{\mathsf{TV}_T^{\frac{2}{3}}T^{\frac{1}{3}}\log^{\frac{2}{3}}(Tk/\alpha)}{(\alpha k)^{\frac{1}{3}}}$ and $k = \frac{256 K}{ \alpha}$, giving rise to
\begin{align*}
\frac{\alpha^{\frac{5}{2}}T}{300\sqrt{k}}\min\left\{\frac{\alpha^{\frac{5}{2}}}{200},\frac{\mathsf{TV}_T}{20 m}, \sqrt{\frac{\alpha mk}{2^8T\log(Tk/\alpha m)}}\right\}
&= \widetilde{\Omega}\Bigg(\frac{T}{\sqrt{k}}\min\bigg\{1, \Big( k T^{-1}\mathsf{TV}_T\Big)^{\frac{1}{3}}\bigg\}\Bigg)\\
&= \widetilde{\Omega}\left( \min\left\{ \frac{ T}{\sqrt{K}},
\frac{ \mathsf{TV}_T^{\frac{1}{3}} T^{\frac{2}{3}}}{K^{\frac{1}{6}}}
\right\}\right).
\end{align*}
Plugging this into \eqref{eq:regret_lower_bound_smooth_1} yields
\begin{equation}\label{eq:regret_lower_bound_smooth_2}
\mathsf{regret}_{\pi}\big(\gL_4(\mathsf{TV}_T), T,K\big) \geq  \widetilde{\Omega}\left( \min\left\{ \frac{ T}{\sqrt{K}},
\frac{ \mathsf{TV}_T^{\frac{1}{3}} T^{\frac{2}{3}}}{K^{\frac{1}{6}}}
\right\}\right)
\end{equation}
\item Next, consider the case with $\mathsf{TV}_T\sqrt{\frac{\alpha T}{12 K}} < 1$. Note that for any $\gD\in \gI$, the constant distribution sequence $\gD_{1:T}$ with $\gD_1=\cdots=\gD_T=\gD$ belongs to $\gL_4(\mathsf{TV}_T)$, since its cumulative variation equals $0$.
Then one can apply \Cref{lem:coverage_gap_lower_bd} with $\epsilon = \min\left\{\frac{\alpha^{5/2}}{200}, \sqrt{\frac{\alpha K}{2^8T\log(TK/\alpha)}}\right\}$ and $k = \frac{256}{\alpha}K$ to the entire horizon $[T]$ to arrive at
\begin{align}
\mathsf{regret}_{\pi}\big(\gL_4(\mathsf{TV}_T), T, K\big) &\ge \sup\limits_{\gD\in \gI}\mathsf{regret}_{\pi}(\gD^T, T, K) \ge \mathop{\EB}\limits_{\gD \in \gI}\bigg[\ssum{t}{1}{T}\Big|\, \PB\big(Y_t \in \gC(Y_{1:t-1})\big)-(1-\alpha)\,\Big|\bigg]
\notag \\
 &=\widetilde{\Omega}\bigg(\min\left\{\frac{ T}{\sqrt{K}}, \sqrt{T}\right\}\bigg).
 \label{eq:regret_lower_bound_smooth_3}
\end{align}
\end{itemize}
Combining the bounds \eqref{eq:regret_lower_bound_smooth_2} and \eqref{eq:regret_lower_bound_smooth_3} 
 establishes the desired lower bound for the smooth drift setting.
\end{itemize}
The proof of Theorem~\ref{thm:lower_bd_gnr_regret} is thus complete.

\subsubsection{Proof of Lemma~\ref{lem:coverage_gap_lower_bd}}\label{sec:prf:lem:coverage_gap_lower_bd}
Throughout this proof, we shall often write $\gC(Y_{1:n}, U)$ simply as $\gC$, as long as it is clear from the context.  
For each $j = 1, \ldots, k$, we define
\begin{subequations}
\label{eq:defn-aj-Y-U}
\begin{align}
a_j= a_j(Y_{1:n}, U) \coloneqq \mathsf{Leb}(\gC \cap I_j),
\end{align}
where we often abbreviate $a_j(Y_{1:n},U)$ as $a_j$. 
Here, $I_j$ is given in \eqref{eq:def_I_j} and $\mathsf{Leb}(\cdot)$ denotes the Lebesgue measure on the interval $[0,1]$. 
Further, let 
\begin{equation}\overline{a} \coloneqq \frac{1}{k}\sum_{j=1}^{k} a_j
\qquad \text{and}\qquad 
\widetilde{a}_j \coloneqq a_j - \overline{a} ~~(j = 1,\ldots, k).
\end{equation}
\end{subequations}
The following lemma characterizes the range of $\sum_{j=1}^{k}\widetilde{a}_j^2$, which will be used repeatedly in our analysis. Its proof is deferred to \Cref{sec:prf:lem:range_of_sum_tilde_a_j2}.
\begin{lemma}\label{lem:range_of_sum_tilde_a_j2}
For any given realization of $Y_{1:n}$ and $U$, it always holds that
    \[
     \frac{1}{k}\bigg(\mathsf{Leb}(\gC)\big(1-\mathsf{Leb}(\gC)\big) - \frac{2K}{k}\bigg)_+\le \sum_{j=1}^k \widetilde{a}_j^2 \le \frac{\mathsf{Leb}(\gC)\big(1 - \mathsf{Leb}(\gC)\big)}{k}
     \leq \frac{1}{4k}.
    \]
\end{lemma}

We now embark on the proof, which contains a few steps below.

\paragraph{Step 1: an expression for the training-conditional coverage gap.}
Under the distribution $\gD$ with parameters $V_{1:k}$ (see \eqref{eq:construct-fx-V}), we can derive
%
\begin{align*}
    \PB\big(Y_{n+1}\in \gC\,|\, Y_{1:n},U\big) &= \ssum{j}{1}{k}\int \frac{1+\epsilon V_j}{1 + \epsilon \overline{V}}\mathbbm{1}\{y \in \gC\cap I_j\} \rd y\\
    &= \ssum{j}{1}{k}\frac{1+\epsilon V_j}{1 + \epsilon \overline{V}}a_j
    = \ssum{j}{1}{k}a_j + \frac{\epsilon}{1 + \epsilon \overline{V}}\ssum{j}{1}{k}(V_j - \overline{V})a_j\\
    &\overset{\mathrm{(a)}}{=} \mathsf{Leb}(\gC) + \frac{\epsilon}{1 + \epsilon \overline{V}} \ssum{j}{1}{k} (V_j - \overline{V}) a_j
    = \mathsf{Leb}(\gC) + \frac{\epsilon}{1 + \epsilon \overline{V}} \ssum{j}{1}{k}V_j \widetilde{a}_j,
\end{align*}
where (a) follows since the sets $\{\gC \cap I_j\}_{j=1}^k$ are mutually disjoint and the last equality results from
\[
\ssum{j}{1}{k}\overline{V}a_j = k\overline{V}\frac{1}{k}\ssum{j}{1}{k}a_j = \ssum{j}{1}{k}V_j \overline{a}.
\]
Note that each distribution $\gD$ in $\gI$ (cf.~\eqref{eq:defn-I-dist-class-each-V}) is uniquely determined by the sequence $V_{1:k}\in \{\pm 1\}^k$. 
To make the subsequent analysis clearer and more concise, 
we define
$$l(V_{1:k}, \gC) \coloneqq \abs{\PB\big(Y_{n+1} \in \gC\,|\, Y_{1:n},U) - (1 - \alpha)},$$
which, according to the above calculation, can be written as
\begin{equation}\label{eq:def_l_VC}
    l(V_{1:k}, \gC) = \bigg|\mathsf{Leb}(\gC) - (1-\alpha)
     + \frac{\epsilon}{1 + \epsilon \overline{V}}\ssum{j}{1}{k}V_j \widetilde{a}_j
    \bigg|.
\end{equation}

As can be easily seen, averaging over $\gD\in\gI$ on the left-hand side of \eqref{eq:coverage_gap_lower_bd} is equivalent to taking expectation with respect to $V_{1:k}$ drawn from the uniform distribution $\gD_{V}^k \coloneqq \mathsf{Unif}\big(\{\pm 1\}^k\big)$, namely,
\begin{align}
\frac{1}{\abs{\gI}}\sum\limits_{\gD\in \gI} \mathop{\EB}\limits_{Y_{1:n}\sim \gD^n,\, U \sim p_U}
&\Big[\big|\,\PB\big(Y_{n+1} \in \gC(Y_{1:n}, U)\,|\, Y_{1:n},U\big) - (1-\alpha)\,\big|\Big]\notag\\
&=
\mathop{\EB}\limits_{V_{1:k}\sim \gD_{V}^k}\Big[\mathop{\EB}\limits_{\gD^n \times p_U}\big[\abs{\PB\big(Y_{n+1} \in \gC\,|\, Y_{1:n},U\big) - (1 - \alpha)}\big]\Big].
\label{eq:equiv-E-Rademaker-123}
\end{align}
By fully expanding the right-hand side of \eqref{eq:equiv-E-Rademaker-123} w.r.t.~the generative process of $U,~V_{1:k}$ and $Y_{1:n}$, we obtain
\begin{align}
    \mathop{\EB}\limits_{V_{1:k}\sim \gD_V^k}\Big[\mathop{\EB}\limits_{\gD^n \times p_U}\big[\big|\,\PB\big(Y_{n+1} \in \gC&\,|\, Y_{1:n},U\big) - (1 - \alpha)\,\big|\big]\Big]\notag\\
    &= \frac{1}{2^k}\sum\limits_{V_{1:k}}\Bigg\{\int_{[0,1]^n\times  \gU}l(V_{1:k}, \gC)\bigg(\prod\limits_{i=1}^n f(y_i\mid V_{1:k})\bigg)p_U(u)\rd y_{1:n}\rd u\Bigg\},\label{eq:all_prob_extention}
\end{align}
where $\mathcal{U}$ is the support of $U$, and the summation  ranges over all $2^k$ choices of  $V_{1:k}$ in $\{\pm 1\}^k$. 
Denote by 
\begin{align}
N_j(y_{1:n}) = \abs{\{i \mid i\in[n], y_i \in I_j\}}
\end{align}
the number of responses that fall in the interval $I_j$; when there is no ambiguity, we abbreviate $N_j(y_{1:n})$ as $N_j$. 
This allows us to express the joint density as
\[
\prod_{i=1}^n f(y_i \mid V_{1:k})
= \prod_{j=1}^k \left(\frac{1 + \epsilon V_j}{1 + \epsilon \overline{V}}\right)^{N_j}
= \frac{\prod_{j=1}^k (1 + \epsilon V_j)^{N_j}}{\big(1 + \epsilon \overline{V}\big)^n},
\]
which combined with~\eqref{eq:all_prob_extention}  yields an expression for the expected training-conditional coverage gap: 
\begin{align}
    &\mathop{\EB}\limits_{V_{1:k}\sim \gD_V^k}\Big[\mathop{\EB}\limits_{\gD^n \times p_U}\big[\big|\PB(Y_{n+1} \in \gC \mid Y_{1:n}, U) - (1 - \alpha)\big|\big]\Big] \notag\\
    &\qquad= \sum\limits_{V_{1:k}}
    \Bigg\{\int\limits_{ [0,1]^n\times \gU} \underset{\eqqcolon\, p(V_{1:k}, y_{1:n}, u)}{\underbrace{\frac{p_U(u)\prod_{j=1}^k(1+\epsilon V_j)^{N_j(y_{1:n})}}{2^k (1+ \epsilon \overline{V})^n} }}l(V_{1:k}, \gC)
    \rd y_{1:n} \rd u\Bigg\}
    \label{eq:cvg_err_to_lVC}.
\end{align}
Here, the term $p(V_{1:k}, y_{1:n}, u)$ represents the joint density of the random tuple $(V_{1:k}, Y_{1:n}, U)$.

\paragraph{Step 2: a lower bound on the coverage gap using auxiliary conditional distributions.}
Now, consider the conditional distribution of $V_{1:k}$ given $y_{1:n}$ and $u$, denoted by $p(V_{1:k} \mid y_{1:n}, u)$. 
%
As it turns out, one can construct a set of auxiliary conditional densities 
$q_j(\cdot \mid y_{1:n}, u)$ for $j = 1, \ldots, k$, 
whose product provides a good approximation to $p(V_{1:k} \mid y_{1:n}, u)$. This is formalized in the following lemma, whose proof is given in \Cref{sec:prf:lem:factor_approx_p}.
\begin{lemma}\label{lem:factor_approx_p}
For any $y_{1:n}$ and $u$, let $p(y_{1:n},u)\coloneqq \sum_{v_{1:k}} p(v_{1:k}, y_{1:n}, u)$.
There exists a collection of conditional distributions 
$\{q_j(V_j \mid y_{1:n}, u)\}_{j=1}^k$ such that
\begin{equation}\label{eq:factor_approx_p_lem}
\begin{aligned}
\mathop{\EB}\limits_{V_{1:k}\sim \gD_V^k}\Big[\mathop{\EB}\limits_{\gD^n \times p_U}\big[\big|\,\PB\big(Y_{n+1} \in \gC\,|\, Y_{1:n},U\big) &- (1 - \alpha)\,\big|\big]\Big] \\
&\ge \frac{1}{3} \int\limits_{[0,1]^n\times \gU} 
\gL_Q(y_{1:n},u)p(y_{1:n}, u)\rd y_{1:n} \rd u 
- \frac{\alpha^{6}}{4n^3 k^3}
\end{aligned}
\end{equation}
as long as $\epsilon \le  \sqrt{\frac{k}{12n\log(2nk/\alpha^2)}}$, where
\begin{align}
\gL_Q(y_{1:n}, u) \coloneqq 
\sum_{V_{1:k}} 
\bigg(\prod_{j=1}^k q_j(V_j \mid y_{1:n}, u)\bigg)\,
l(V_{1:k}, \gC).
\label{eq:defn-Lq-y-u}
\end{align}
\end{lemma}
As can be seen, each summand in $\gL_Q(y_{1:n}, u)$ (cf.~\eqref{eq:defn-Lq-y-u}) involves the product distribution based on $\{q_j(\cdot\mid y_{1:n},u)\}$. 
 
\paragraph{Step 3: a decomposition of $\gL_Q(y_{1:n}, u)$.} 
With Lemma~\ref{lem:factor_approx_p}, we now turn to the analysis of the quantity $\gL_Q(y_{1:n}, u)$ (cf.~\eqref{eq:defn-Lq-y-u}).
Given $(y_{1:n}, u)$, define $\gD_Q$ as the product distribution
\begin{align}
\gD_Q(V_{1:k}) \coloneqq \prod_{j=1}^k q_j(V_j \mid y_{1:n}, u),
\label{eq:defn-DQ-product}
\end{align}
and draw an independent copy 
$V_{1:k}^\prime$ of $V_{1:k}$  from the same distribution 
$\gD_Q$. 
Then, invoking \eqref{eq:def_l_VC} and the triangle inequality, we can derive
\begin{align*}
    \gL_Q(y_{1:n}, u) &= \sum_{V_{1:k}} 
\Big(\prod_{j=1}^k q_j(V_j \mid y_{1:n}, u)\Big)\,
l(V_{1:k}, \gC) = \EB_{\gD_Q}\big[l(V_{1:k}, \gC)\big]\\
&= \frac{1}{2}\EB_{\gD_Q}\big[l(V_{1:k}, \gC)\big] + \frac{1}{2}\EB_{\gD_Q}\big[l(V_{1:k}^\prime, \gC)\big]\\
&\overset{\eqref{eq:def_l_VC}}{=}\frac{1}{2}\EB_{\gD_Q^2}\left[
\abs{\mathsf{Leb}(\gC) - (1-\alpha)
     + \frac{\epsilon}{1 + \epsilon \overline{V}}\inner{V_{1:k}}{\widetilde{a}_{1:k}}
    } + \abs{\mathsf{Leb}(\gC) - (1-\alpha)
     + \frac{\epsilon}{1 + \epsilon \overline{V}^\prime}\inner{V_{1:k}^\prime}{\widetilde{a}_{1:k}}
    }
\right]\\
&\ge \frac{1}{2}\EB_{\gD_Q^2}\left[
\abs{\left(\frac{\epsilon}{1+\epsilon \overline{V}}\inner{V_{1:k}}{\widetilde{a}_{1:k}} - \frac{\epsilon}{1 + \epsilon \overline{V}}\inner{V_{1:k}^\prime}{\widetilde{a}_{1:k}}\right)
+ \left(\frac{\epsilon}{1+\epsilon \overline{V}} - \frac{\epsilon}{1+\epsilon \overline{V}^\prime}\right)\inner{V_{1:k}^\prime}{\widetilde{a}_{1:k}}
}
\right]\\
&\ge \frac{1}{2}\EB_{\gD_Q^2}\left[\frac{\epsilon}{1+\epsilon \overline{V}}\big|\inner{V_{1:k} - V_{1:k}^\prime}{\widetilde{a}_{1:k}}\big|\right] 
- \frac{1}{2}\EB_{\gD_Q^2}\left[\abs{
\left(\frac{\epsilon}{1+\epsilon \overline{V}} - \frac{\epsilon}{1+\epsilon \overline{V}^\prime}\right)\inner{V_{1:k}^\prime}{\widetilde{a}_{1:k}}
}\right],
\end{align*}
from which it follows that
\begin{align}
\EB_{Y_{1:n}, U} \big[\gL_Q(Y_{1:n}, U) \big]  &\ge \underbrace{\frac{1}{2}\EB_{Y_{1:n}, U}\left[\EB_{\gD_Q^2}\left[\frac{\epsilon}{1+ \epsilon \overline{V}}\big|\inner{V_{1:k} - V_{1:k}^\prime}{\widetilde{a}_{1:k}}\big|\right]\right]}_{\eqqcolon\, \gL_1} \notag\\
&\quad - \frac{1}{2}\underbrace{\EB_{Y_{1:n},U}\left[
\EB_{\gD_Q^2}\left[\abs{
\left(\frac{\epsilon}{1+\epsilon\overline{V}} - \frac{\epsilon}{1+\epsilon\overline{V}^\prime}\right)\inner{V_{1:k}^\prime}{\widetilde{a}_{1:k}}
}\right]
\right]}_{\eqqcolon\, \gL_2}.
\label{eq:E-LQ-lb-L1-L2}
\end{align}
This leaves us two terms to control. 

\paragraph{Step 4: lower bounds on $\mathcal{L}_1$ and $\mathcal{L}_2$.}
Next, we would like to control the two terms on the right-hand side of \eqref{eq:E-LQ-lb-L1-L2} separately. 
Before continuing, 
we introduce some additional notation: for any $j = 1, 2, \ldots, k$, define
\begin{align}
\widetilde{V}_j \coloneqq \frac{1}{2}(V_j - V_j^\prime), 
\qquad 
\check{V} \coloneqq \frac{1}{2}(\overline{V} - \overline{V}^\prime), 
\qquad 
\widetilde{V}_{1:k} \coloneqq \frac{1}{2}(V_{1:k} - V_{1:k}^\prime).
\label{eq:defn-Vtilde-check-j}
\end{align}
It is straightforward to see that since $V_j, V_j^\prime \in \{\pm 1\}$, 
we have $\widetilde{V}_j \in \{-1, 0, 1\}$ for all $j = 1, 2, \ldots, k$.

\begin{itemize}

\item 
With regards to $\gL_1$, recognizing that $\abs{\epsilon \overline{V}} \le \abs{\epsilon} \le 1/4$, one has
\begin{align}
\gL_1 \ge \frac{4\epsilon}{3}\EB_{Y_{1:n},U} \left[\EB_{\gD_Q^2}\left[\big|\big\langle \widetilde{V}_{1:k},\widetilde{a}_{1:k} \big\rangle \big| \right]\right],
\label{eq:gL1-lower-bound-123}
\end{align}
the right-hand side of which is further controlled by the following lemma. The proof can be found in \Cref{sec:prf:lem:info_lower_bound}. 
%
\begin{lemma}\label{lem:info_lower_bound}
Let $k \ge 64/\alpha$ and $\epsilon \le \min\left\{\frac{\alpha^{5/2}}{200}, \frac{1}{64}\sqrt{\frac{\alpha k}{ n \log(2nk/\alpha^2)}}\right\}$. For any admissible algorithm whose resulting prediction set $\gC$ is the union of at most $k$ intervals, we have
\[
\EB_{Y_{1:n},U} \left[\EB_{\gD_Q^2}\left[\big|\big\langle\widetilde{V}_{1:k},\widetilde{a}_{1:k}\big\rangle\big|\right]\right] \ge \left(\frac{\sigma_\pi^2}{2} - \frac{{\alpha}}{16}\right)^{\frac{5}{2}}\frac{2}{\sqrt{k}},
\]
where we define
\begin{align}
\sigma_\pi^2 \coloneqq  {\EB_{Y_{1:n},U}\Bigg[\bigg(\mathsf{Leb}(\gC)\big(1 - \mathsf{Leb}(\gC)\big) - \frac{2K}{k}\bigg)_+\Bigg]}.
\label{eq:defn-sigma-pi-square}
\end{align}
\end{lemma}
Consequently, taking \eqref{eq:gL1-lower-bound-123} and Lemma~\ref{lem:info_lower_bound}  collectively reveals that
\begin{equation}\label{eq:gL_1_lower_bd}
\gL_1 \ge \frac{8\epsilon}{3\sqrt{k}}\left(\frac{\sigma_\pi^2}{2} - \frac{{\alpha}}{16}\right)^{\frac{5}{2}}.
\end{equation}

\item 
When it comes to $\gL_2$, we make the following observation:
\begin{equation}\label{eq:gL_2_upper_bd}
\begin{aligned}
    \gL_2 &= \EB_{Y_{1:n},U}\left[
    \EB_{\gD_Q^2}\left[
    \abs{
    \left(\frac{\epsilon}{1+\epsilon \overline{V}} - \frac{\epsilon}{1+\epsilon \overline{V}^\prime}\right)\inner{V_{1:k}^\prime}{\widetilde{a}_{1:k}}
    }\right]\right]\\
    &\overset{\mathrm{(a)}}{\le} \EB_{Y_{1:n}, U}\left[\EB_{\gD_Q^2}\left[
    \abs{\frac{\epsilon}{1+\epsilon \overline{V}} - \frac{\epsilon}{1+\epsilon \overline{V}^\prime}}
    \right]\right]
    \overset{\mathrm{(b)}}{\le} 4\epsilon^2\EB_{Y_{1:n},U}\left[\EB_{\gD_Q^2}\left[
    \big|\overline{V}^\prime - \overline{V}\big|
    \right]\right]\\
    &= 8\epsilon^2 \EB_{Y_{1:n},U}\left[\EB_{\gD_Q^2}\left[
    \big|\check{V}\big|
    \right]\right]
    \overset{\mathrm{(c)}}{\le} 8\epsilon^2 \EB_{Y_{1:n},U}\Bigg[
    \bigg(\EB_{\gD_Q^2}\bigg[
    \frac{1}{k^2}\sum_{j=1}^k \widetilde{V}_j^2
    \bigg]\bigg)^{1/2}\Bigg] 
    \overset{\mathrm{(d)}}{\le} \frac{8\epsilon^2}{\sqrt{k}}.
\end{aligned}
\end{equation}
Here, (a) is valid since 
$\abs{\inner{V_{1:k}^\prime}{\widetilde{a}_{1:k}}} 
\le \sum_{j=1}^k \abs{\widetilde{a}_j} \le \mathsf{Leb}([0,1])= 1$ (see \eqref{eq:defn-aj-Y-U}); 
(b) follows from the fact that 
$\min\{1 + \epsilon \overline{V},\, 1 + \epsilon \overline{V}^\prime\} 
\ge 1 - \epsilon \ge 1/2$;  (c) applies Jensen’s inequality and uses the properties that 
$\{\widetilde{V}_j\}$ are independent zero-mean random variables; 
and (d) arises from the inequality  $\widetilde{V}_j^2 \le 1$. 
\end{itemize}
Combine Eqns.~\eqref{eq:gL_1_lower_bd} and \eqref{eq:gL_2_upper_bd} with \eqref{eq:E-LQ-lb-L1-L2} to yield
\begin{equation}\label{eq:E-LQ-lb-L1-L2_fin}
\begin{aligned}
    \EB_{Y_{1:n},U}\left[\gL_Q(Y_{1:n},U)\right] &\ge \frac{1}{2}\gL_1 - \frac{1}{2}\gL_2\\
    &\ge \frac{4\epsilon}{3\sqrt{k}} \left(\frac{\sigma_\pi^2}{2} - \frac{{\alpha}}{16}\right)^{\frac{5}{2}} - \frac{4\epsilon^2}{\sqrt{k}} = \frac{4\epsilon}{3\sqrt{k}}\left(\left(\frac{\sigma_\pi^2}{2} - \frac{{\alpha}}{16}\right)^{\frac{5}{2}} - 3\epsilon\right).
\end{aligned}
\end{equation}

\paragraph{Step 5: putting all pieces together.} 
%
To finish up, we divide into two cases and analyze them separately. 
\begin{itemize}
\item 
If $\PB\left(\abs{\mathsf{Leb} (\gC) - (1-\alpha)} > \alpha/8\right)\ge 1/4$, then one has
\begin{align*}
    &\EB_{V_{1:n}}\bigg[\mathop{\EB}\limits_{\gD^n\times p_U}\left[\abs{\PB\big(Y_{n+1}\in \gC \mid Y_{1:n},U\big) - (1-\alpha)}\right]\bigg] 
    \ge \EB\left[\big|\mathsf{Leb}(C) - (1-\alpha)\big| - \abs{
    \frac{\epsilon}{1+\epsilon\overline{V}}\inner{V_{1:k}}{\widetilde{a}_{1:k}}
    }\right]\\
    &\qquad \ge \frac{\alpha}{32} - 2\epsilon\EB\bigg[\bigg|\,\ssum{j}{1}{k}V_j\widetilde{a}_j\,\bigg|\bigg]
    \ge \frac{\alpha}{32} - 2\epsilon \EB\Bigg[\sqrt{\Big(\sum_{j=1}^k V_j^2\Big)\Big(\sum_{j=1}^k \widetilde{a}_j^2\Big)}\Bigg]\\
    &\qquad\overset{\mathrm{(e)}}{\ge} \frac{\alpha}{32}
    - \epsilon
    \overset{\mathrm{(f)}}{\ge} \frac{\alpha}{32} - \frac{\alpha}{64} = \frac{\alpha}{64},
\end{align*}
where (e) invokes \Cref{lem:range_of_sum_tilde_a_j2}, and (f) is due to our choice of $\epsilon$. 

\item If instead $\PB\left(\abs{\mathsf{Leb} (\gC) - (1-\alpha)} > \alpha/8\right)< 1/4$, then it holds that
\begin{equation}\label{eq:typical_gC_prob}
\PB\left(
1 - \frac{9\alpha}{8}\le \mathsf{Leb}(\gC) \le 1 - \frac{7\alpha}{8}
\right) \ge \frac{3}{4}.
\end{equation}
We can then lower bound $\sigma_{\pi}^2$ (cf.~\eqref{eq:defn-sigma-pi-square}) by 
\begin{align*}
    \sigma_\pi^2 &= \EB\left[
    \bigg(\mathsf{Leb}(\gC)\big(1 - \mathsf{Leb}(\gC)\big) - \frac{2K}{k}\bigg)_+
    \right]\\
    &\ge \EB\left[
    \bigg(\mathsf{Leb}(\gC)\big(1 - \mathsf{Leb}(\gC)\big) - \frac{2K}{k}\bigg)_+\mathbbm{1}\left\{
    1 - \frac{9\alpha}{8} \le \mathsf{Leb}(\gC) \le 1- \frac{7\alpha}{8}
    \right\}\right]\\
    &\ge \EB\bigg[\bigg(\Big(1 - \frac{9\alpha}{8}\Big)\frac{7\alpha}{8} - \frac{2K}{k}\bigg)_+ \mathbbm{1}\left\{
    1 - \frac{9\alpha}{8} \le \mathsf{Leb}(\gC) \le 1- \frac{7\alpha}{8}
    \right\}\bigg]\\
    &\overset{\mathrm{(g)}}{\ge} \Big(\frac{49\alpha}{128} - \frac{2K}{k}\Big)_+ \PB\left(1 - \frac{9\alpha}{8} \le \mathsf{Leb}(\gC) \le 1- \frac{7\alpha}{8}\right)  \ge \frac{3\alpha}{8}\times \frac{3}{4} \ge \frac{9\alpha}{32}.
\end{align*}
Here, (g) holds when $\alpha \le 1/2$, and the last inequality follows by combining $k \ge \frac{256 K}{\alpha}$ and \eqref{eq:typical_gC_prob}.
%
%
%
%
Taking this together with \eqref{eq:E-LQ-lb-L1-L2_fin} and $\epsilon \le \frac{\alpha^{5/2}}{200}$, we arrive at 
\begin{equation}\label{eq:EB_gL_Q_lower_bd}
    \EB_{Y_{1:n},U}\big[\gL_Q(Y_{1:n},U)\big]
    \ge \frac{4\epsilon}{3\sqrt{k}}\left(\left(\frac{{9\alpha}}{32} - \frac{{\alpha}}{32}\right)^{\frac{5}{2}} - \frac{\alpha^{\frac{5}{2}}}{64}\right) \ge \frac{\alpha^{\frac{5}{2}}\epsilon}{48\sqrt{k}}.
\end{equation}

Substituting Eqn.~\eqref{eq:EB_gL_Q_lower_bd} into \eqref{eq:factor_approx_p_lem}, we obtain
\[
\mathop{\EB}\limits_{V_{1:n}}\left[\mathop{\EB}\limits_{\gD^n\times p_U}\Big[
\big|\PB(Y_{n+1}\in \gC \mid Y_{1:n},U) - (1-\alpha)\big|
\right]\Big] \ge \frac{1}{3}\mathop{\EB}\limits_{Y_{1:n},U}[\gL_Q(Y_{1:n},U)] - \frac{\alpha^6}{4n^3k^3} \ge 
\frac{\alpha^{\frac{5}{2}}\epsilon}{144\sqrt{k}} - \frac{\alpha^{6}}{4n^3k^3}.
\]
\end{itemize}

%

\noindent 
The above two cases taken collectively conclude the proof of Lemma~\ref{lem:coverage_gap_lower_bd}.

\subsubsection{Proof of \Cref{lem:range_of_sum_tilde_a_j2}}\label{sec:prf:lem:range_of_sum_tilde_a_j2}
\paragraph{Upper bound.}
According to the definition of $a_j$ and $\overline{a}$ (see \eqref{eq:defn-aj-Y-U}), it is readily seen that
\begin{align*}
    \ssum{j}{1}{k}\widetilde{a}_j^2 &= \ssum{j}{1}{k}\big(a_j - \overline{a}\big)^2 = \ssum{j}{1}{k} a_j^2 - k\overline{a}^2 \le \frac{1}{k}\ssum{j}{1}{k}a_j - \frac{\mathsf{Leb}(\gC)^2}{k} = \frac{\mathsf{Leb}(\gC)\big(1-\mathsf{Leb}(\gC)\big)}{k} \leq \frac{1}{4k},
\end{align*}
where the first inequality follows since $a_j\leq \mathsf{Leb}(I_j)=1/k$ and $\overline{a}=\mathsf{Leb}(\mathcal{C})/k$, and the last inequality comes from the AM-GM inequality.   

\paragraph{Lower bound.}
To establish the lower bound, we exploit the structural property of the prediction set $\gC$ (i.e., $\gC$ is the union of at most $K$ intervals).  
Consider the following conditions:  
\begin{itemize}
\item suppose  $l$  intervals from $\{I_j\}_{j=1}^k$ (without loss of generality, $\{I_j\}_{j=1}^l$) are completely contained in $\gC$; 
\item at most $2K$ of the $\{I_{j}\}_{j=1}^k$ (without loss of generality, $\{I_{j_i}\}_{i=1}^{2K}$) are only partially covered by $\gC$, which always holds due to the structural property of $\gC$.
\end{itemize}
Since the intervals $\{I_j\}_{j=1}^k$ are mutually disjoint, the total Lebesgue measure of these at most $l+2K$ intervals must be at least $\mathsf{Leb}(\gC)$. This leads to
\[
\frac{l + 2K}{k} = \ssum{j}{1}{l}\big| I_j\big| + \ssum{i}{1}{2K}\big| I_{j_i}\big| \ge \mathsf{Leb}(\gC),
\]
thus indicating that
\begin{equation}\label{eq:l_ge_muC-2K}
l \ge \big(\mathsf{Leb}(\gC)k - 2K\big)_+.
\end{equation}

Now, let us look at $\{\widetilde{a}_j\}_{j=1}^k$.  
For each $1 \le j \le l$, we have $a_j = \mathsf{Leb}(\gC \cap I_j) =1/k$ and $k\overline{a} = \sum_{j=1}^k a_j = \mathsf{Leb}(\gC)$, allowing one to derive
\begin{align*}
\sum_{j=1}^k \widetilde{a}_j^2 
&= \sum_{j=1}^k \big(a_j - \overline{a}\big)^2 = \sum_{j=1}^k a_j^2 - k\overline{a}^2= \sum_{j=1}^k a_j^2 - \frac{\mathsf{Leb}(\gC)^2}{k}\\
&\ge \sum_{j=1}^l a_j^2 - \frac{\mathsf{Leb}(\gC)^2}{k}
= \frac{l}{k^2} - \frac{\mathsf{Leb}(\gC)^2}{k}\\
&\overset{\eqref{eq:l_ge_muC-2K}}{\ge} \frac{1}{k}\bigg\{\Big(\mathsf{Leb}(\gC) - \frac{2K}{k}\Big)_+ -\mathsf{Leb}(\gC)^2\bigg\}
\ge \frac{1}{k}\bigg\{\mathsf{Leb}(\gC)\big(1-\mathsf{Leb}(\gC)\big) - \frac{2K}{k}\bigg\}.
\end{align*}
This combined with the trivial fact $\sum_{j=1}^k \widetilde{a}_j^2 \ge 0$ establishes the advertised lower bound on $\sum_{j=1}^k \widetilde{a}_j^2$.  
%

\subsubsection{Proof of Lemma~\ref{lem:factor_approx_p}}\label{sec:prf:lem:factor_approx_p}
By virtue of Bayes's rule, the conditional density $p(V_{1:k}\mid y_{1:n}, u)$ admits the following expression:
\begin{equation}\label{eq:factor_approx_p1}
    \begin{aligned}
        p(V_{1:k}\mid y_{1:n}, u) &= \frac{p(V_{1:k}, y_{1:n}, u)}{\sum_{v_{1:k}}p(v_{1:k},y_{1:n}, u)} = \frac{\left(p_U(u)\prod_{j=1}^k(1+\epsilon V_j)^{N_j}\right)\big/2^k(1+\epsilon \overline{V})^n
        }{
        \sum_{v_{1:k}}
        \left(p_U(u)\prod_{j=1}^k(1+\epsilon v_j)^{N_j}\right)\big/2^k(1+\epsilon \overline{v})^n
        }\\
        &= \frac{\left(
    \prod_{j=1}^k(1+\epsilon V_j)^{N_j}
    \right)\big/ (1 + \epsilon \overline{V})^n
    }{\sum_{v_{1:k}}\left[\left(
    \prod_{j=1}^k(1+\epsilon v_j)^{N_j}
    \right)\big/ (1 + \epsilon \overline{v})^n\right]}.
    \end{aligned}
\end{equation}
For any given $V_{1:k}$, we observe that
\begin{equation}\label{eq:factor_approx_p2}
\begin{aligned}
    \bigg(\prod\limits_{j=1}^k (1+ \epsilon V_j)^{N_j}\bigg)\Big/\left(1 + \epsilon \overline{V}\right)^n &
    = \bigg(\prod\limits_{j=1}^k (1 + \epsilon V_j)^{N_j}\bigg) e^{-\epsilon \overline{V}n}\cdot \frac{e^{\epsilon \overline{V}n}}{\big(1 + \epsilon \overline{V}\big)^n}\\
    &\ge \prod\limits_{j=1}^k (1 + \epsilon V_j)^{N_j} e^{-\epsilon V_j\frac{n}{k}} \eqqcolon \prod\limits_{j=1}^k q_j(V_j, y_{1:n}),
\end{aligned}
\end{equation}
where the last line follows from the elementary inequality $(e^x)^n \ge (1+ x)^n$ for all $x\in  [-1,1]$, and we define $$q_j(V_j, y_{1:n}) \coloneqq (1 + \epsilon V_j)^{N_j} e^{-\epsilon V_j\frac{n}{k}}.$$
Further,  define
\begin{subequations}\label{eq:def_q_margin_cond}
\begin{align}
    q_j(y_{1:n}) &\coloneqq \sum_{V_j} (1 + \epsilon V_j)^{N_j} e^{-\epsilon V_j \frac{n}{k}},\quad  &&j = 1, \ldots, k;\label{eq:def_q_j_margin}\\
    q(y_{1:n}) &\coloneqq \sum_{V_{1:k}} \prod\limits_{j=1}^k\left[
    (1+\epsilon V_j)^{N_j}e^{-\epsilon V_j \frac{n}{k}}
    \right] = \prod\limits_{j=1}^k q_j(y_{1:n});\label{eq:def_q_margin}\\
    q_j(V_j\mid y_{1:n}) &\coloneqq \frac{q_j(V_j, y_{1:n})}{q_j(y_{1:n})},\qquad &&j = 1,\ldots, k \label{eq:def_q_j_cond}\\
    q(V_{1:k}\mid y_{1:n}) &\coloneqq \prod_{j=1}^k q_j(V_j\,|\, y_{1:n})\label{eq:def_q1k_cond}.
\end{align}
\end{subequations}
Combining the above analysis and notation with a little algebra, we reach
\begin{align}
    \EB_{V_{1:k}}&\bigg[
    \EB_{\gD^n\times p_U}\Big[\big|\,
    \PB(Y_{n+1} \in \gC\,|\, Y_{1:n},U) - (1 - \alpha)
    \,\big|\Big]
    \bigg]\notag\\
    &\overset{\eqref{eq:cvg_err_to_lVC}}{=} \int\limits_{[0,1]^n\times \gU}\bigg\{\sum\limits_{V_{1:k}}p(V_{1:k}\mid y_{1:n}, u) l(V_{1:k}, \gC)\bigg\}p(y_{1:n}, u)\rd y_{1:n}\rd u\notag\\
    &\overset{\eqref{eq:factor_approx_p1}}{=} \int\limits_{[0,1]^n\times \gU} \bigg\{\sum\limits_{V_{1:k}}\frac{\Big(\prod_{j=1}^k(1+\epsilon V_j)^{N_j}\Big)\big(1+\epsilon \overline{V}\big)^{-n}}{\sum_{v_{1:k}}\Big(\prod_{j=1}^k 
    (1+\epsilon v_j)^{N_j}\Big) (1 + \overline{v})^{-n}
    } l(V_{1:k}, \gC)\bigg\}p(y_{1:n},u)\rd y_{1:n}\rd u\notag\\
    &\overset{\eqref{eq:factor_approx_p2}}{\ge} 
    \int\limits_{[0,1]^n\times \gU} \left\{
    \sum\limits_{V_{1:k}}\frac{
    \prod_{j=1}^k q_j(V_j, y_{1:n})
    }{\sum_{v_{1:k}}\Big(\prod_{j=1}^k (1 + \epsilon v_j)^{N_j}\Big)(1+\epsilon \overline{v})^{-n}} l(V_{1:k}, \gC)
    \right\}p(y_{1:n}, u ) \rd y_{1:n} \rd u \label{eq:factor_approx_p3}\\
    &\overset{\eqref{eq:def_q_margin_cond}}{=}
    \int\limits_{[0,1]^n\times \gU} \underbrace{\left(\frac{
    q(y_{1:n})
    }{
    \sum_{v_{1:k}}\Big(\prod_{j=1}^k (1 {+} \epsilon v_j)^{N_j}\Big)(1{+}\epsilon \overline{v})^{-n}
    }\right)}_{\eqqcolon\, \gG(y_{1:n})}\left\{
    \sum\limits_{V_{1:k}}\prod\limits_{j=1}^k q_j(V_j\mid y_{1:n}) l(V_{1:k}, \gC)
    \right\} p(y_{1:n}, u ) \rd y_{1:n} \rd u.\notag
\end{align}


The next step is to control $\gG(y_{1:n})$. Specifically, by expanding $q(y_{1:n})$ and applying Eqn.~\eqref{eq:factor_approx_p1}, we can establish the following inequality:
\begin{equation}\label{eq:factor_approx_p4}
\begin{aligned}
    \gG(y_{1:n}) &\overset{\eqref{eq:def_q_margin}}{=} \frac{
    \sum_{V_{1:k}} \Big(\prod_{j=1}^k (1+\epsilon V_j)^{N_j}\Big)\big(1+\epsilon \overline{V}\big)^{-n}\frac{(1+\epsilon \overline{V})^n}{e^{\epsilon \overline{V}n}}
    }{
    \sum_{v_{1:k}}\Big(\prod_{j=1}^k (1+\epsilon v_j)^{N_j}\Big)\big(1+\epsilon \overline{v}\big)^{-n}
    }
    \overset{\eqref{eq:factor_approx_p1}}{=} \sum\limits_{V_{1:k}} p(V_{1:k}\mid y_{1:n})\frac{\big(1+\epsilon \overline{V}\big)^n}{e^{\epsilon \overline{V}n}}\\
    &\ge \sum\limits_{V_{1:k}} p(V_{1:k}\mid y_{1:n})\frac{\big(1+\epsilon \overline{V}\big)^n}{e^{\epsilon \overline{V}n}}\mathbbm{1}\left\{\abs{\overline{V}} \le \sqrt{6k^{-1}\log (2nk/\alpha^2)}\right\}\\
    &\overset{\mathrm{(a)}}{\ge}
    \frac{1}{3}\sum\limits_{V_{1:k}}p(V_{1:k}\mid y_{1:n})\mathbbm{1}\left\{
    \abs{\overline{V}} \le \sqrt{6k^{-1}\log (2nk/\alpha^2)}\right\}\\
    &= \frac{1}{3}\left(1 - \PB\left(\overline{V} > \sqrt{\frac{6\log (2nk/\alpha^2)}{k}}~\Bigg|~ Y_{1:n} = y_{1:n}\right)\right).
\end{aligned}
\end{equation}
To justify why (a) holds, note that  
\[\epsilon\abs{\overline{V}} \le \sqrt{\frac{k}{12n\log (2nk/\alpha^2)}} \sqrt{\frac{6\log (2nk/\alpha^2)}{k}} = \sqrt{\frac{1}{2n}}\]
holds under the condition that  $\abs{\overline{V}} \le \sqrt{2k^{-1}\log(nk/\alpha)}$, 
which combined with the elementary fact $1+x+x^2 \ge e^x$ $(x\in[-1,1])$ gives
\[
\frac{e^{\epsilon \overline{V} n }}{\big(1 + \epsilon \overline{V}\big)^n}
\le \left(\frac{1 + \epsilon \overline{V} + (\epsilon \overline{V})^2}{1 + \epsilon \overline{V}}\right)^n \le \left(1 + 2(\epsilon\overline{V})^2\right)^n \le \left(1 + \frac{1}{n}\right)^n \le e < 3.
\]

As a consequence, substituting Eqn.~\eqref{eq:factor_approx_p4} into Eqn.~\eqref{eq:factor_approx_p3} yields
\begin{align}
    &\EB_{V_{1:k}}\bigg[\EB_{\gD^n \times p_U}\Big[\abs{
    \PB\big(Y_{n+1}\in \gC\,|\, Y_{1:n},U\big) - (1 - \alpha)
    }\Big]\bigg] \notag\\
    &\quad \ge \frac{1}{3}\int\limits_{[0,1]^n\times \gU}\left(1 - \PB\left(\overline{V} > \sqrt{\frac{6\log (2nk/\alpha^2)}{k}}~\Bigg|~  y_{1:n}\right)\right)\left\{
    \sum\limits_{V_{1:k}}\prod\limits_{j=1}^k q_j(V_j\mid y_{1:n}) l(V_{1:k}, \gC)
    \right\}p(y_{1:n}, u )\rd y_{1:n}\rd u \notag\\
    &\quad \overset{\mathrm{(b)}}{\ge} \frac{1}{3}\int\limits_{[0,1]^n\times \gU}  \left\{\sum\limits_{V_{1:k}}\prod\limits_{j=1}^k q_j(V_j\mid y_{1:n}) l (V_{1:k}, \gC)\right\}p(y_{1:n}, u)\rd y_{1:n}\rd u \notag\\
    &\quad \hspace{16em}
    - \frac{1}{3}\int\limits_{[0,1]^n\times \gU}\PB\left(\abs{\overline{V}} > \sqrt{\frac{6\log(2nk/\alpha^2)}{k}}\Big|~ y_{1:n}\right)p(y_{1:n},u) \rd y_{1:n}\rd u \notag\\
    &\quad = \frac{1}{3}\int\limits_{[0,1]^n\times \gU} \gL_Q(y_{1:n})p(y_{1:n}, u)\rd y_{1:n} \rd u
    - \frac{1}{3}\PB\left(\abs{\overline{V}} > \sqrt{\frac{6\log (2nk/\alpha^2)}{k}}\right),
    \label{eq:E-gap-lb-12345}
\end{align}
where (b) relies on the fact that
\[
\sum\limits_{V_{1:k}}\prod\limits_{j=1}^k q_j(V_j\mid y_{1:n}) l(V_{1:k}, \gC) \overset{\eqref{eq:def_l_VC}}{=} \mathop{\EB}\limits_{V_{1:k}\sim q(V_{1:k}\,|\,y_{1:n})}\Big[\Big|\,\PB\big(Y_{n+1}\in \gC\,|\, Y_{1:n},U\big) - (1-\alpha)\,\Big|\Big] \le 1.
\]

 Finally, recalling that $V_{1:k}$ are $k$ independent Rademacher random variables, we can apply Hoeffding's inequality to yield
\[
\PB\left(\abs{\overline{V}} > \sqrt{\frac{6\log (2nk/\alpha^2)}{k}}\right) \le 2\exp\left\{-\frac{6k\log \frac{2nk}{\alpha^2}}{2k}\right\} \le \frac{\alpha^6}{4n^3k^3}.
\]
Substitution into the above bound \eqref{eq:E-gap-lb-12345} concludes the proof of Lemma~\ref{lem:factor_approx_p}.

\subsubsection{Proof of Lemma~\ref{lem:info_lower_bound}}\label{sec:prf:lem:info_lower_bound}
As discussed earlier, we have $\widetilde{V}_j \in \{-1, 0, 1\}$ (see \eqref{eq:defn-Vtilde-check-j}).  
Further,  define  
\[
\delta_j \coloneqq \mathbbm{1}\{\widetilde{V}_j \neq 0\},
\qquad \text{and}\qquad
\xi_j \coloneqq 
\begin{cases}
\mathrm{sgn}\big(\widetilde{V}_j\big), & \text{if } \delta_j = 1,\\[4pt]
\zeta_j, & \text{if } \delta_j = 0,
\end{cases}
\]
where $\{\zeta_j\}_{j=1}^k$ are $k$ independent Rademacher random variables, generated 
independently of $(\delta_j)_{j=1}^k$, $Y_{1:n}$ and $U$.  
It is then straightforward to verify that
\(
\widetilde{V}_j = \xi_j \delta_j.
\)
Additionally, it is easy to see that for any given $y_{1:n}$ and $u$, one has
\[
\PB(\xi_j = 1, \delta_j = 1 \mid y_{1:n}, u) 
= \PB(V_j > V_j^\prime \mid y_{1:n}, u) 
= \PB(V_j < V_j^\prime \mid y_{1:n}, u) 
= \PB(\xi_j = -1, \delta_j = 1 \mid y_{1:n}, u),
\]
where the second identity follows from the fact that $V_j$ and $V_j^\prime$ are independently and identically distributed.  
This implies that, conditional on $\delta_j = 1$ and $(y_{1:n}, u)$, the variable $\xi_j$ is a Rademacher random variable.  
Similarly, when $\delta_j = 0$ and when $(y_{1:n}, u)$ are given, we have $\xi_j = \zeta_j$, which is also Rademacher.  

Moreover, since $\{V_j\}_{j=1}^k$ and $\{V_j^\prime\}_{j=1}^k$ are independent within each sequence 
(with $V_{1:k} \sim \prod_{j=1}^k q_j(V_j \mid y_{1:n}, u)$), 
it follows that $\{\xi_j\}_{j=1}^k$ are independent Rademacher random variables, which are also 
independent of $\big(\{\delta_j\}_{j=1}^k, Y_{1:n}, U\big)$. Therefore, for any fixed $\big(\{\delta_j\}_{j=1}^k, y_{1:n}, u\big)$, applying Lemma~\ref{lem:khintchine_p1} to $\{\xi_j\}_{j=1}^k$ yields
\begin{equation}\label{eq:info_lower_bound_1}
\begin{aligned}
    \EB\Big[ \big| \big\langle\widetilde{V}_{1:k}, \widetilde{a}_{1:k}\big\rangle \big|\Big] &= \EB\Bigg[\bigg|\ssum{j}{1}{k}\xi_j \delta_j \widetilde{a}_j\bigg|\Bigg]
    = \EB\Bigg[\EB\Bigg[\bigg|
    \ssum{j}{1}{k}\xi_j \delta_j \widetilde{a}_j
    \bigg|~\Big|~ \{\delta_j\}_{j=1}^k,~\{\widetilde{a}_j\}_{j=1}^k\Bigg]\Bigg]\\
    &\ge \frac{1}{\sqrt{2}}\EB \Bigg[\bigg(\ssum{j}{1}{k}\widetilde{a}_j^2 \delta_j^2\bigg)^{\frac{1}{2}} \Bigg]= \frac{1}{\sqrt{2}}\EB\Bigg[\bigg(\ssum{j}{1}{k}\widetilde{a}_j^2 \delta_j\bigg)^{\frac{1}{2}}\Bigg],
\end{aligned}
\end{equation}
where the last equality holds since $\delta_j^2 = \delta_j$.

To continue, let us first examine 
\(
\EB\left[\sum_{j=1}^k \widetilde{a}_j^2 \delta_j\right].
\)
%
Note that, given $(y_{1:n}, u)$, under the distribution 
$\prod_{j=1}^k q_j(V_j \mid y_{1:n}, u)$ one has
\begin{align*}
\EB[\delta_j\mid y_{1:n}, u] &= \PB(V_j \neq V_j^\prime \mid y_{1:n}, u) = 2q_j(1\mid y_{1:n}, u)q_j(0\mid y_{1:n}, u)\\
&= \frac{1}{2}\left(1 - (2q_j(1\mid y_{1:n}, u) - 1)^2\right) \eqqcolon \frac{\left(1 - (2q_j - 1)^2\right)}{2},
\end{align*}
where we define 
\begin{equation}
q_j \coloneqq q_j(1\mid y_{1:n},u) = \frac{(1+\epsilon)^{N_j}e^{-\epsilon\frac{n}{k}}}{\sum_{V = \pm 1}(1+ \epsilon V)^{N_j}e^{-\epsilon\frac{n}{k}V}}.
\label{eq:shorthand-qj-1357}
\end{equation}
From this, we can derive that
\begin{equation}\label{eq:info_lower_bound_1-5}
\EB\Bigg[\ssum{j}{1}{k}\widetilde{a}_j^2 \delta_j\Bigg] = \frac{1}{2}\EB\Bigg[\ssum{j}{1}{k}\widetilde{a}_j^2\Bigg] - \frac{1}{2}\EB\Bigg[\ssum{j}{1}{k}\widetilde{a}_j^2 (2q_j - 1)^2\Bigg].
\end{equation}
As for the first term on the right-hand side of \eqref{eq:info_lower_bound_1-5}, applying \Cref{lem:range_of_sum_tilde_a_j2} and then taking expectation over $Y_{1:n}$ and $U$ yield
\begin{equation}\label{eq:lower_bd_EB_sum_a2_delta5}
    \frac{1}{2}\EB\bigg[\sum_{j=1}^k \widetilde{a}_j^2\bigg] \ge \frac{1}{2k}\EB\bigg[\Big(\mathsf{Leb}(\gC)\big(1-\mathsf{Leb}(\gC)\big) - \frac{2K}{k}\Big)_+\bigg]
\end{equation}
Also, the second term on the right-hand side of \eqref{eq:info_lower_bound_1-5} satisfies
\begin{align}
    \frac{1}{2}\EB\Bigg[\ssum{j}{1}{k}\widetilde{a}_j^2(2q_j - 1)^2\Bigg] \le \frac{1}{2k^2}\EB\left[\ssum{j}{1}{k} (2q_j - 1)^2\right].
    \label{eq:info_lower_bound_1-5-ab}
\end{align}
To control $\EB\big[\sum_{j=1}^{k} (2q_j - 1)^2\big]$, we resort to the following lemma, with the proof given in \Cref{sec:proof:lem:bd_2q-1^2}.
\begin{lemma}\label{lem:bd_2q-1^2}
Recall that $q_j$ is defined in \eqref{eq:shorthand-qj-1357}. For any admissible algorithm whose resulting prediction set $\gC$ is the union of at most $k$ intervals, we have 
\[
    \EB\Bigg[\ssum{j}{1}{k}(2q_j - 1)^2\Bigg] \le 321n\epsilon^2 \log \Big(\frac{2nk}{\alpha^2}\Big) + \frac{\alpha^6}{32n^3k^3}.
\]
\end{lemma}
Taking Eqns.~\eqref{eq:info_lower_bound_1-5}, \eqref{eq:lower_bd_EB_sum_a2_delta5}, \eqref{eq:info_lower_bound_1-5-ab}  and Lemma~\ref{lem:bd_2q-1^2} together yields
\begin{align}
    \EB\Bigg[\ssum{j}{1}{k}\widetilde{a}_j^2 \delta_j\Bigg] &\ge
    \frac{1}{2k}\EB\bigg[\Big(\mathsf{Leb}(\gC)\big(1-\mathsf{Leb}(\gC)\big) - \frac{2K}{k}\Big)_+\bigg]
    - \frac{1}{k^2}\left(161n\epsilon^2\log \Big(\frac{2nk}{\alpha^2}\Big) + \frac{\alpha^6}{64n^3k^3}\right)\notag\\
    &\ge \frac{\sigma_\pi^2}{2k} - \frac{\alpha}{16 k},\label{eq:lower_bd_EB_sum_a2_delta}
\end{align}
where the last line holds as long as $k \ge \frac{256K}{\alpha}$ and $\epsilon \le \min\left\{\frac{\alpha^{5/2}}{200}, \frac{1}{64}\sqrt{\frac{\alpha k}{ n \log(2nk/\alpha^2)}}\right\}$.


To finish up, denoting  $\EB \left[\sum_{j=1}^{k}\widetilde{a}_j^2\delta_j\right]$ as $\lambda$, and applying Lemma~\ref{lem:paley-zygmund} with $\theta = 1/2$, we arrive at
\begin{align*}
    \EB\left[\sqrt{\ssum{j}{1}{k}\widetilde{a}_j^2 \delta_j}\,\right] &\ge \frac{1}{\sqrt{2}}\sqrt{\lambda}\PB\Bigg(\ssum{j}{1}{k}\widetilde{a}_j^2 \delta \ge \frac{\lambda}{2}\Bigg)
    \ge \frac{1}{4\sqrt{2}}\sqrt{\lambda}\frac{\lambda^2}{\EB\left[\big(\sum_{j=1}^k \widetilde{a}_j^2 \delta_j\big)^2\right]}\\
    &\ge 
    \frac{1}{4\sqrt{2}}\frac{\lambda^{5/2}}{\EB\Big[\big(\sum_{j=1}^k \widetilde{a}_j^2\big)^2\Big]} \overset{\text{Lemma }\ref{lem:range_of_sum_tilde_a_j2}}{\ge} \frac{k^2}{4\sqrt{2}}\frac{\lambda^{5/2}}{\EB\Big[\mathsf{Leb}(\gC)^2\big(1 - \mathsf{Leb}(\gC)\big)^2\Big]}
    \\
    &\ge 2\sqrt{2}k^2\lambda^{5/2} \overset{\eqref{eq:lower_bd_EB_sum_a2_delta}}{\ge} \frac{2\sqrt{2}}{ \sqrt{k}}\left(\frac{\sigma_\pi^2}{2} - \frac{\alpha}{16}\right)^{\frac{5}{2}}.
\end{align*}
Combining this with Eqn.~\eqref{eq:info_lower_bound_1} establishes the advertised result of the lemma.
\subsubsection{Proof of Lemma~\ref{lem:bd_2q-1^2}}
\label{sec:proof:lem:bd_2q-1^2}
To begin with, the TV distance between two Bernoulli distributions $\mathrm{Ber}(q_j)$ and $ \mathrm{Ber}(1/2)$ obeys
\[
\mathsf{TV}\big(\mathrm{Ber}(q_j),~ \mathrm{Ber}(1/2)\big) = \abs{q_j - \frac{1}{2}} + \abs{(1 - q_j) - \frac{1}{2}} 
= \abs{2q_j - 1},
\]
which combined with Pinsker's inequality \citep[Lemm~2.5]{tsybakov2009introduction} yields
\begin{equation}\label{eq:lower_bd_EB_sum_a2_delta_1}
    (2q_j - 1)^2 = \big(\mathsf{TV}\big(\mathrm{Ber}(q_j),~ \mathrm{Ber}(1/2)\big)\big)^2 \le \frac{1}{2}\mathsf{KL}\left(\mathrm{Ber}(q_j) \,\big\|\, \mathrm{Ber}(0.5)\right).
\end{equation}
Note that the distribution 
$q(V_{1:k} \mid y_{1:n}) = \prod_{j=1}^k q_j(V_j \mid y_{1:n})$ 
factorizes. 
Therefore, recalling the definition of $q_j$ in \eqref{eq:shorthand-qj-1357}, one can apply the chain rule of the KL divergence to derive 
\begin{align}
    2\EB\Bigg[\ssum{j}{1}{k}(2q_j - 1)^2\Bigg] &\le 
    \EB_{Y_{1:n}} \Bigg[\ssum{j}{1}{k}\mathsf{KL}\left(\mathrm{Ber}(q_j) \,\big\|\, \mathrm{Ber}(0.5)\right)\Bigg]\notag\\
    &= \EB_{Y_{1:n}}\Bigg[
    \ssum{j}{1}{k}\bigg(\mathop{\EB}_{V_j\sim q_j(\cdot|Y_{1:n})}\Big[ \log \big(2q_j(V_j\,|\, Y_{1:n})\big)\Big]\bigg)\Bigg] 
   \notag \\
    &= \EB_{Y_{1:n}}\Bigg[\mathop{\EB}_{V_{1:k}\sim q(\cdot|Y_{1:n})}\bigg[\log\bigg(2^k\prod\limits_{j=1}^k q_j(V_j\,|\, Y_{1:n})\bigg)\bigg]\Bigg]\notag\\ &=\EB_{Y_{1:n}}\Bigg[\sum\limits_{V_{1:k}}q(V_{1:k}\,|\, Y_{1:n})\log\big(2^k q(V_{1:k}\,|\, Y_{1:n})\big)\Bigg]\notag.
    \end{align}
    Recognizing that $q(V_{1:k}\,|\, Y_{1:n})$ is an approximation of $p(V_{1:k}\,|\, Y_{1:n},U)$ (recall the arguments in \eqref{eq:factor_approx_p1}-\eqref{eq:def_q_margin_cond}), we consider the following decomposition by incorporating the term $p(V_{1:k}\,|\, Y_{1:n},U)$ into the preceding identity: 
    \begin{align}
    2\EB\Bigg[\ssum{j}{1}{k}(2q_j - 1)^2\Bigg]&= \EB_{Y_{1:n},U}\Bigg[\sum\limits_{V_{1:k}}p(V_{1:k}\mid Y_{1:n}, U)\frac{q(V_{1:k}\,|\, Y_{1:n})}{p(V_{1:k}\mid Y_{1:n},U)}\log \bigg(2^k q(V_{1:k}\,|\, Y_{1:n})\bigg)\Bigg]\notag\\
    &= \EB_{Y_{1:n},U}\left[\sum\limits_{V_{1:k}}\frac{q(V_{1:k}\mid Y_{1:n})}{p(V_{1:k}\mid Y_{1:n}, U)}\log \big(2^k p(V_{1:k}\mid Y_{1:n},U)\big)p(V_{1:k}\mid Y_{1:n}, U)\right]\notag\\
    &\quad   + \EB_{Y_{1:n},U}\left[\sum\limits_{V_{1:k}}\frac{q(V_{1:k}\mid Y_{1:n})}{p(V_{1:k}\mid Y_{1:n}, U)}\log \bigg(\frac{q(V_{1:k}\mid Y_{1:n})}{p(V_{1:k}\mid Y_{1:n}, U)}\bigg)p(V_{1:k}\mid Y_{1:n}, U)\right].\label{eq:tv_pinsker_to_kl_1}
\end{align}

Next, we examine the ratio term $\frac{q(V_{1:k}\mid Y_{1:n})}{p(V_{1:k}\mid Y_{1:n}, U)}$, which appears multiple times in \eqref{eq:tv_pinsker_to_kl_1}.  From the elementary fact $1 + x \le e^x$ as well as the definitions \eqref{eq:factor_approx_p1} and \eqref{eq:def_q_margin_cond},  we can demonstrate that
\begin{align}
    &\frac{q(V_{1:k}\mid y_{1:n})}{p(V_{1:k}\mid y_{1:n}, u)} 
    = \frac{\prod_{j=1}^k\Big( (1+\epsilon V_j)^{N_j}e^{-\epsilon V_j \frac{n}{k}}\Big)}{\sum_{v_{1:k}}\prod_{j=1}^k\Big((1+\epsilon v_j)^{N_j}e^{-\epsilon v_j\frac{n}{k}}\Big)}
    \cdot \frac{\sum_{v_{1:k}}\Big(\prod_{j=1}^k(1+\epsilon v_j)^{N_j}\Big)(1+\epsilon \overline{v})^{-n}}
    {\Big(\prod_{j=1}^k (1+\epsilon V_j)^{N_j}\Big)(1+\epsilon \overline{V})^{-n}} \notag\\
    &\hspace{2em}= \frac{(1+\epsilon\overline{V})^n}{e^{\epsilon\overline{V}n}}\cdot\frac{\sum_{v_{1:k}}\Big(\prod_{j=1}^k(1+\epsilon v_j)^{N_j}\Big)(1+\epsilon \overline{v})^{-n}}{\sum_{v_{1:k}}\prod_{j=1}^k\Big((1+\epsilon v_j)^{N_j}e^{-\epsilon v_j\frac{n}{k}}\Big)} 
    = \frac{(1+\epsilon\overline{V})^n}{e^{\epsilon\overline{V}n}}\bigg/\left(\sum\limits_{v_{1:k}}p(v_{1:k}\mid y_{1:n},u) \frac{(1+\epsilon \overline{v})^n}{e^{\epsilon\overline{v}n}}\right) \notag\\
    &\hspace{2em}= \frac{F_n(\epsilon\overline{V})}{\EB_{V_{1:k}\sim p(\cdot| y_{1:n},u)}\big[ F_n(\epsilon\overline{V})\big]}
    = \frac{F_n(\epsilon \overline{V})}{F(y_{1:n},u)}
    \leq \frac{1}{F(y_{1:n},u)},
    \label{eq:q-p-ratio-equiv-expression}
\end{align}
where we define
\begin{equation}
F_n(\epsilon \overline{V}) \coloneqq \frac{(1+\epsilon \overline{V})^n}{e^{\epsilon\overline{V}n}} \qquad \text{and}\qquad  F(y_{1:n},u)\coloneqq \EB_{V_{1:k}\sim p(\cdot| y_{1:n},u)}\big[F_n(\epsilon \overline{V})\big].
\label{eq:defn-Fn-F-lower-bound-full-conformal}
\end{equation}
Further, given that the choice of $\epsilon$ guarantees that $\abs{\epsilon\overline{V}} < 1/2$, we can invoke the elementary inequality $1+x\leq e^x\leq 1+x+x^2$ ($|x|< 1/2$) to derive
\begin{equation}\label{eq:F_n_inverse_upper_bd}
1 \le \frac{1}{F_n(\epsilon \overline{V})} = \frac{e^{\epsilon \overline{V}n}}{\big(1 + \epsilon \overline{V}\big)^n} \le \left(\frac{1 + \epsilon\overline{V} + \epsilon^2 \overline{V}^2}{1+\epsilon \overline{V}}\right)^n \le \left(1 + 2\epsilon^2 \overline{V}^2\right)^n \le \exp\left\{2\big(\epsilon \overline{V}\big)^2 n\right\}.
\end{equation}
Substituting \eqref{eq:q-p-ratio-equiv-expression} and \eqref{eq:F_n_inverse_upper_bd} into \eqref{eq:tv_pinsker_to_kl_1} leads to
\begin{align}
    2\EB\Bigg[\ssum{j}{1}{k}(2q_j - 1)^2\Bigg] &= \EB_{Y_{1:n},U}\left[\sum\limits_{V_{1:k}}\frac{F_n(\epsilon\overline{V})}{F(Y_{1:n},U)}\log \big(2^k p(V_{1:k}\mid Y_{1:n},U)\big)p(V_{1:k}\mid Y_{1:n}, U)\right]\notag\\
    &\quad + \EB_{Y_{1:n},U}\left[\sum\limits_{V_{1:k}}\frac{F_n(\epsilon\overline{V})}{F(Y_{1:n},U)}\log \bigg(\frac{F_n(\epsilon\overline{V})}{F(Y_{1:n},U)}\bigg)p(V_{1:k}\mid Y_{1:n}, U)\right]\notag\\ &= \underset{\eqqcolon\, \mathcal{S}_0}{\underbrace{ \EB_{Y_{1:n},U}\bigg[\mathop{\EB}_{V_{1:k}\sim p(\cdot| Y_{1:n},U)}
    \Big[\log \big(2^k p(V_{1:k}\,|\, Y_{1:n},U)\big)\Big]
    \bigg] }}
    \notag\\
    &\quad +\underbrace{\EB_{Y_{1:n}, U}\left[\mathop{\EB}_{V_{1:k}\sim p(\cdot| Y_{1:n},U)}\left[\bigg(\frac{F_n(\epsilon\overline{V})}{F(Y_{1:n},U)}-1\bigg)\log \big(2^k p(V_{1:k}\mid Y_{1:n},U)\big)\right]\right]}_{\eqqcolon \,\gS_1}\notag\\
    &\quad + \underbrace{\EB_{Y_{1:n},U}\Bigg[\mathop{\EB}_{V_{1:k}\sim p(\cdot| Y_{1:n},U)}\bigg[\frac{F_n(\epsilon\overline{V})}{F(Y_{1:n},U)}\log \bigg(\frac{F_n(\epsilon\overline{V})}{F(Y_{1:n},U)}\bigg)\bigg]\Bigg]}_{\eqqcolon \,\gS_2},
    \label{eq:E-2q-1-UB-S12}
\end{align}
leaving us with three terms to cope with. 

\paragraph{Bounding the term $\mathcal{S}_0$.}
Towards this end, we first define 
$I(V_{1:k};Y_{1:n})$ to be the mutual information between $V_{1:k}$ and $Y_{1:n}$, namely,
\begin{equation}
I(V_{1:k};Y_{1:n})
\coloneqq \sum_{V_{1:k}}\int_{[0,1]^n}
\log\!\left(\frac{p(V_{1:k},y_{1:n})}{2^{-k}p(y_{1:n})}\right)\,
p(V_{1:k},y_{1:n})\,\rd y_{1:n},
\label{eq:defn-mutual-info-I-VY}
\end{equation}
where 
\begin{equation}
p(V_{1:k},y_{1:n}) \coloneqq \int_{\gU} p(V_{1:k},y_{1:n},u)\,\rd u
\qquad \text{and}\qquad
p(y_{1:n}) \coloneqq \sum_{V_{1:k}} p(V_{1:k},y_{1:n}).
\end{equation}
As it turns out, the term $\mathcal{S}_0$ (cf.~\eqref{eq:E-2q-1-UB-S12}) is equivalent to this mutual information quantity since
\begin{align*}
    \mathcal{S}_0&=\EB_{Y_{1:n},U}\bigg[\mathop{\EB}_{V_{1:k}\sim p(\cdot| Y_{1:n},U)}
    \Big[\log \big(2^k p(V_{1:k}\,|\, Y_{1:n},U)\big)\Big]
    \bigg] 
    \\
    &=
    \sum_{V_{1:k}}\int_{[0,1]^n\times \gU}p(V_{1:k},y_{1:n})p_U(u)\log \Big(\frac{p(V_{1:k},y_{1:n})p_U(u)}{2^{-k}p(y_{1:n})p_U(u)}\Big)\rd y_{1:n}\rd u\\
    &= \sum_{V_{1:k}}\int_{[0,1]^n}p(V_{1:k},y_{1:n})\log \Big(\frac{p(V_{1:k},y_{1:n})}{2^{-k}p(y_{1:n})}\Big)\rd y_{1:n} = I(V_{1:k}; Y_{1:n}),
\end{align*}
where the second line holds since $U$ is independent of $(V_{1:k},Y_{1:n})$ (as it only affects the construction of the prediction set $\gC$).

According to the data-generating mechanism and the chain rule of the KL divergence, one has
\begin{align}
I(V_{1:k}, Y_{1:n}) &= \sum\limits_{V_{1:k}} \int_{[0,1]^n}  p(V_{1:k})p(y_{1:n}\mid V_{1:k})\log \frac{p(V_{1:k})p(y_{1:n}\mid V_{1:k})}{p(V_{1:k})p(y_{1:n})} \rd y_{1:n}\notag\\
&= \ssum{i}{1}{n}\sum\limits_{V_{1:k}} \int_{[0,1]}\frac{p(y_i\mid V_{1:k})}{2^k}\log \frac{p(y_i\mid V_{1:k})}{p(y_i)} \rd y_i\notag\\
&\overset{\mathrm{(a)}}{\le} \ssum{i}{1}{n}\sum_{V_{1:k}}\frac{1}{2^k}\int_{[0,1]}\left(p(y_i\mid V_{1:k}) - p(y_i) + p(y_i\mid V_{1:k})\left(\frac{p(y_i)}{p(y_{i}\mid V_{1:k})} - 1\right)^2\right) \rd y_i \notag\\
& \overset{\mathrm{(b)}}{=} \ssum{i}{1}{n}\sum\limits_{V_{1:k}}\int_{[0,1]}\frac{p(y_i\mid V_{1:k})}{2^k}\left(\frac{p(y_i)}{p(y_{i}\mid V_{1:k})} - 1\right)^2 \rd y_i 
\overset{\mathrm{(c)}}{\le} 36n\epsilon^2,\label{eq:bd_mutual_info}
\end{align}
where (b) follows because both $p(y_i)$ and $p(y_i \mid V_{1:k})$ are density functions;
(a) and (c) are proven below.
\begin{itemize}
\item 
To justify (c), suppose that $y_i \in I_j$ (defined in \eqref{eq:def_I_j}). Denoting by $V_{1:k}'$ an independent copy of $V_{1:k}$, one can derive
\begin{align*}
\left(\frac{p(y_i)}{p(y_{i}\mid V_{1:k})} - 1\right)^2 &=
    \left(\frac{\EB_{V_{1:k}^\prime}[p(y_i\mid V_{1:k}^\prime)]}{p(y_i\mid V_{1:k})}-1\right)^2\\ &\le \EB_{V_{1:k}^\prime}\left[\left(
    \frac{p(y_i\mid V_{1:k}^\prime)}{p(y_i\mid V_{1:k})} - 1
    \right)^2\right]
    = \EB_{V_{1:k}^\prime}\left[\left(
    \frac{(1+\epsilon V_j^\prime)(1+\epsilon \overline{V})}{(1+\epsilon V_j)(1+\epsilon\overline{V}^\prime)}-1
    \right)^2\right]\\
    &= \EB_{V_{1:k}^\prime}\left[\left(
    \frac{
    \epsilon(V_j^\prime + \overline{V} - V_j - \overline{V}^\prime) + \epsilon^2(V_j^\prime \overline{V} - V_j \overline{V}^\prime)
    }{(1+\epsilon V_j)(1+\epsilon \overline{V}^\prime)}
    \right)^2\right] \le 36\epsilon^2,
\end{align*}
with the proviso that $\epsilon \le 1/6$. This also implies that
\begin{align}
\abs{\frac{p(y_i)}{p(y_i\,|\, V_{1:k})}-1} \le 1
\qquad \text{as long as }\epsilon \leq \frac{1}{6}.
\label{eq:ratio-magnitude-1-eps}
\end{align}

\item Regarding inequality (a), we invoke the elementary fact $\log(1 + x) \ge x - x^2$ for $x \in [-1, 1]$ along with \eqref{eq:ratio-magnitude-1-eps} to reach
\begin{align*}
p(y_{i}\mid V_{1:k}) \log \frac{p(y_i\mid V_{1:k})}{p(y_i)} &= -
p(y_{i}\mid V_{1:k}) \log \frac{p(y_i)} {p(y_i\mid V_{1:k})}\\
&\le -p(y_i\mid V_{1:k}) \left(\frac{p(y_i)}{p(y_i\mid V_{1:k})} - 1 - \left(\frac{p(y_i)}{p(y_i\mid V_{1:k})} - 1\right)^2\right)\\
&=   p(y_i\mid V_{1:k}) - p(y_i) + p(y_i \mid V_{1:k})\left(\frac{p(y_i)}{p(y_i\mid V_{1:k})} - 1\right)^2.
\end{align*}

\end{itemize}

\paragraph{Bounding the term $\gS_1$.}
Based on the definition (cf.~\eqref{eq:E-2q-1-UB-S12}), we can write
\begin{align}
    \gS_1 = \EB_{Y_{1:n},U}\bigg[\underbrace{\mathop{\EB}_{V_{1:k}\sim p(\cdot| Y_{1:n},U)}\bigg[\Big(
    F_n(\epsilon\overline{V}) - F(Y_{1:n},U)
    \Big) \log \big(2^k p(V_{1:k}\,|\, Y_{1:n},U)\big)
    \bigg]}_{\eqqcolon \, \gS_1^{\mathsf{num}}(Y_{1:n}, U)}
    \Big/F(Y_{1:n},U)\bigg].
    \label{eq:defn-S1-decompose-num}
\end{align}

We begin by examining $\gS_1^{\mathsf{num}}(Y_{1:n}, U)$.
For notational simplicity, set
\begin{equation}\label{eq:def_sigma_tail_thresh}
\sigma \coloneqq \sqrt{16k^{-1}\log(2nk/\alpha^2)}.
\end{equation}
Consider any given $y_{1:n},u$, and any two realizations $v_{1:k},~ v_{1:k}'$. If $|\overline{v}|\vee |\overline{v}'| \le \sigma$, then by the mean value theorem, there exists a real number $\zeta$ between $\epsilon\overline{v}$ and $\epsilon \overline{v}'$ (which means $|\zeta| \le \epsilon \big(|\overline{v}|\vee |\overline{v}'|\big) \le \epsilon \sigma$) satisfying
\begin{align}
    \Big|\,F_n(\epsilon \overline{v}) - F_n(\epsilon \overline{v}')\,\Big| \le \big|\, F_n'(\zeta) \,\big|\cdot \big|\,\epsilon \overline{v} - \epsilon \overline{v}'\,\big| \le 2\epsilon n\sigma |\zeta| e^{-\zeta} \Big(\frac{1+\zeta}{e^\zeta}\Big)^{n-1} \le 6\epsilon^2 \sigma^2 n \label{eq:lagrangian_mean_of_F_n}, 
\end{align}
where the choice of $\epsilon$ ensures $\epsilon \sigma <1$
and we have used the definition \eqref{eq:defn-Fn-F-lower-bound-full-conformal} of $F_n(\cdot)$.
Splitting the expectation in the definition of $\gS_1^{\mathsf{num}}(Y_{1:n}, U)$ into two parts 
based on whether or not $|\overline{V}|\vee |\overline{V}'| \le \sigma$, we can derive
\begin{align}
    &\gS_1^{\mathsf{num}}(Y_{1:n}, U)=\mathop{\EB}_{V_{1:k}\sim p(\cdot| Y_{1:n},U)} \bigg[
    \Big(F_n(\epsilon\overline{V}) - \mathop{\EB}_{V'_{1:k}\sim p(\cdot| Y_{1:n},U)}\big[F_n\big(\epsilon\overline{V}'\big) \big] \Big) \log \big(2^k p(V_{1:k}\mid Y_{1:n}, U)\big)
    \bigg]\notag\\
    &\qquad 
    \le \mathop{\EB}_{V_{1:k},V'_{1:k}\sim p(\cdot| Y_{1:n},U)}\bigg[
    \Big|F_n(\epsilon\overline{V}) - F_n\big(\epsilon\overline{V}'\big)\Big| \log \Big(\frac{1}{p(V_{1:k}\mid Y_{1:n}, U)}\Big)
    \bigg]\notag\\
    &\qquad \overset{\eqref{eq:lagrangian_mean_of_F_n}}{\le}
    6\epsilon^2\sigma^2 n \mathop{\EB}_{V_{1:k}\sim p(\cdot| Y_{1:n},U)} \bigg[\log \frac{1}{p(V_{1:k}\mid Y_{1:n},U)}\bigg]\notag\\
    &\qquad \hspace{2em} + \mathop{\EB}_{V_{1:k},V'_{1:k}\sim p(\cdot| Y_{1:n},U)} \bigg[
    \Big|F_n(\epsilon\overline{V}) - F_n\big(\epsilon\overline{V}'\big)\Big|\mathbbm{1}\{|\overline{V}|\vee |\overline{V}'| > \sigma\} \log \Big(\frac{1}{p(V_{1:k}\mid Y_{1:n}, U)}\Big)
    \bigg], \label{eq:gS_1_numerator}
\end{align}
where the second line follows from Jensen's inequality and the following:
\[
\mathop{\EB}_{V_{1:k},V'_{1:k}\sim p(\cdot| Y_{1:n},U)}\bigg[
\Big(F_n(\epsilon\overline{V}) - F_n\big(\epsilon\overline{V}'\big)\Big) \log \big(2^k\big)
\bigg]=0.
\]
\begin{itemize}
\item Regarding the first term on the right-hand side of \eqref{eq:gS_1_numerator}, it is observed that for any given $y_{1:n}$ and $u$, 
\begin{align*}
    \mathop{\EB}_{V_{1:k}\sim p(\cdot| y_{1:n},u)} \bigg[\log \frac{1}{p(V_{1:k}\mid y_{1:n},u)}\bigg] &= \log 2^k - \mathop{\EB}_{V_{1:k}\sim p(\cdot| y_{1:n},u)} \big[ \log\big( 2^k p(V_{1:k}\mid y_{1:n},u) \big)\big]\\
    &= k\log 2 - \mathsf{KL}\big(p(V_{1:k}\mid y_{1:n},u) \,\|\, p(V_{1:k})\big) \le k\log 2.
\end{align*}
Plug this into the first term on the right-hand side of \eqref{eq:gS_1_numerator} and use \eqref{eq:def_sigma_tail_thresh}  to yield
\begin{align}\label{eq:1st_term_bd_of_gS_1}
    6\epsilon^2\sigma^2 n \mathop{\EB}_{V_{1:k}\sim p(\cdot| Y_{1:n},U)} \bigg[\log \frac{1}{p(V_{1:k}\mid Y_{1:n},U)}\bigg] \le 96 \epsilon^2 n\log \frac{2nk}{\alpha^2}.
\end{align}

\item 
As for the second term on the right-hand side of \eqref{eq:gS_1_numerator}, it can be derived that
\begin{align}
    &\mathop{\EB}_{V_{1:k},V'_{1:k}\sim p(\cdot| Y_{1:n},U)}\bigg[
    \Big|F_n(\epsilon\overline{V}) - F_n\big(\epsilon\overline{V}'\big)\Big|\mathbbm{1}\{|\overline{V}|\vee |\overline{V}'| > \sigma\} \log \Big(\frac{1}{p(V_{1:k}\mid Y_{1:n}, U)}\Big)
    \bigg]\notag\\
    &\overset{\eqref{eq:F_n_inverse_upper_bd}}{\le} \mathop{\EB}_{V_{1:k},V'_{1:k}\sim p(\cdot| Y_{1:n},U)}\bigg[\mathbbm{1}\Big\{\big|\overline{V}\big|\vee \big|\overline{V}'\big| > \sigma\Big\}\log \Big(\frac{1}{p(V_{1:k}\mid Y_{1:n}, U)}\Big)\bigg]\notag\\ 
    &\overset{\mathrm{(a)}}{\le} 4\epsilon n \PB\Big(
    \big|\overline{V}\big|\vee \big|\overline{V}'\big| > \sigma \,\big|\, Y_{1:n},U
    \Big).\label{eq:2nd_term_bd_of_gS_1}
\end{align}
To justify step (a) of \eqref{eq:2nd_term_bd_of_gS_1}, consider any two different realization $v_{1:k}$ and $v_{1:k}'$, which obey
\[
\frac{p(v_{1:k}'\mid y_{1:n}, u)}{p(v_{1:k}\mid y_{1:n}, u)} = \frac{p(v_{1:k}')p(y_{1:n}\mid v_{1:k}')p_U(u)}{p(v_{1:k})p(y_{1:n}\mid v_{1:k})p_U(u)} = \frac{\prod_{j=1}^k (1+\epsilon v_j')^{N_j}}{\prod_{j=1}^k (1+\epsilon v_j)^{N_j}} \le e^{4\epsilon n},
\]
\[
\Longrightarrow \qquad 
\log \Big(\frac{1}{p(V_{1:k}\mid y_{1:n}, u)}\Big) 
= \log \bigg(
\mathop{\EB}_{V_{1:k}'\mid y_{1:n},u}\Big[\frac{p(V_{1:k}'\mid y_{1:n}, u)}{p(V_{1:k}\mid y_{1:n}, u)}\Big]
\bigg) \le 4\epsilon n.
\]
\end{itemize}
Consequently, by substituting  \eqref{eq:2nd_term_bd_of_gS_1} and \eqref{eq:1st_term_bd_of_gS_1} into \eqref{eq:gS_1_numerator}, we arrive at
\begin{align}
    \gS_1^{\mathsf{num}}(Y_{1:n}, U) &\le 96\epsilon^2 n \log \frac{2nk}{\alpha^2} + 8\epsilon n \PB\Big(|\overline{V}| > \sigma\,\Big|\, Y_{1:n},U\Big). \label{eq:gS_1_numerator_fin}
\end{align}

Now, we switch attention to the term $\frac{1}{F(Y_{1:n},U)}$.
Applying Jensen’s inequality  yields
\begin{align}
   & \frac{1}{\mathop{\EB}_{V_{1:k}\sim p(\cdot| Y_{1:n},U)}\big[F_n(\epsilon \overline{V})\big]} 
    \le 
    \mathop{\EB}_{V_{1:k}\sim p(\cdot| Y_{1:n},U)}\left[\frac{e^{\epsilon \overline{V}n}}{\big(1+\epsilon \overline{V}\big)^n}\right]\notag\\
    &\qquad = \mathop{\EB}_{V_{1:k}\sim p(\cdot| Y_{1:n},U)}\left[\frac{e^{\epsilon \overline{V}n}}{\big(1+\epsilon \overline{V}\big)^n}\mathbbm{1}\left\{\abs{\overline{V}}\le \sigma \right\}\right]
    + 
    \mathop{\EB}_{V_{1:k}\sim p(\cdot| Y_{1:n},U)}\left[\frac{e^{\epsilon \overline{V}n}}{\big(1+\epsilon \overline{V}\big)^n}\mathbbm{1}\left\{\abs{\overline{V}}> \sigma\right\}\right]\notag\\
    &\qquad \overset{\eqref{eq:F_n_inverse_upper_bd}}{\le} 
    \mathop{\EB}_{V_{1:k}\sim p(\cdot| Y_{1:n},U)}\left[e^{2\big(\epsilon \overline{V}\big)^2 n}\mathbbm{1}\left\{\abs{\overline{V}}\le \sigma\right\}\right]+ 
     \mathop{\EB}_{V_{1:k}\sim p(\cdot| Y_{1:n},U)}\left[e^{2\big(\epsilon \overline{V}\big)^2 n}\mathbbm{1}\left\{\abs{\overline{V}} > \sigma\right\}\right]
    \notag\\
    &\qquad \le 3 + \mathop{\EB}_{V_{1:k}\sim p(\cdot| Y_{1:n},U)}\left[\exp\left\{2\big(\epsilon \overline{V}\big)^2 n\right\}\mathbbm{1}\left\{\abs{\overline{V}}> \sigma\right\}\right],\label{eq:gS_1_denominator_bd}
\end{align}
where the last line follows since
\(
2\big(\epsilon \overline{V}\big)^2n \le \frac{2k}{32n}\cdot \frac{16n}{k} =1
\)
provided that $\epsilon \le \min\left\{\frac{\alpha^{5/2}}{128}, \frac{1}{64}\sqrt{\frac{\alpha k}{n\log(2nk/\alpha^2)}}\right\}$.

We are now ready to bound $\gS_1$. It is readily seen from \eqref{eq:defn-S1-decompose-num} that
\begin{align}
    \gS_1 &\le \EB_{Y_{1:n},U}\left[\frac{\gS_1^{\mathsf{num}}(Y_{1:n},U)}{F(Y_{1:n},U)}\right] \overset{\eqref{eq:gS_1_numerator_fin}}{\le} \EB_{Y_{1:n},U}\bigg[
    \frac{96 \epsilon^2 n \log \Big(\frac{2nk}{\alpha^2}\Big) + 8\epsilon n \PB \Big(|\overline{V}|> \sigma \mid Y_{1:n},U\Big)
    }
    {F(Y_{1:n},U)}
    \bigg]\notag\\
    &\overset{\eqref{eq:gS_1_denominator_bd}}{\le}
    \mathop{\EB}_{Y_{1:n},U}\Bigg[
    \bigg(96 \epsilon^2 n \log \Big(\frac{2nk}{\alpha^2}\Big) + 8\epsilon n \PB \Big(|\overline{V}|> \sigma \mid Y_{1:n},U\Big)\bigg)
    \bigg(3 + \mathop{\EB}_{V_{1:k}\sim p(\cdot| Y_{1:n},U)}\Big[e^{2(\epsilon\overline{V})^2n}\mathbbm{1}\{|\overline{V}| > \sigma\}\Big]\bigg)
    \Bigg]\notag\\
    &\le 100 \epsilon^2 n \log \Big(\frac{2nk}{\alpha^2}\Big)\bigg(
    3 + \EB_{V_{1:k}}\Big[e^{2(\epsilon\overline{V})^2n}\mathbbm{1}\{|\overline{V}| > \sigma\}\Big]
    \bigg) + 24\epsilon n \PB\big(|\overline{V}|> \sigma\big)\notag\\
    &\hspace{2em}+ 8\epsilon n \mathop{\EB}_{Y_{1:n},U}\Bigg[
    \PB \Big(|\overline{V}|> \sigma \mid Y_{1:n},U\Big)
    \mathop{\EB}_{V_{1:k}\sim p(\cdot| Y_{1:n},U)}\Big[e^{2(\epsilon\overline{V})^2n}\mathbbm{1}\{|\overline{V}| > \sigma\}\Big]
    \Bigg]. \label{eq:gS_1_bound_first}
\end{align}
Note that Hoeffding's inequality tells us that
%
\begin{align}
    \PB_{V_{1:k}}\big(|\overline{V}|> \sigma\big) &= \PB_{V_{1:k}}\bigg(|\overline{V}|> \sqrt{\frac{16\log(2nk/\alpha^2)}{k}}\bigg) \le \frac{\alpha^{16}}{128 n^8k^8}, \label{eq:V_larger_sqrt_k_1}
\end{align}
we also have for any $\Delta \ge 0$ that 
\begin{align}
    \EB_{V_{1:k}}\left[e^{\frac{k}{4} \overline{V}^2}\mathbbm{1}\left\{e^{\frac{k}{4} \overline{V}^2} > e^{\frac{k}{4}\Delta^2}\right\}\right] &\overset{\mathrm{(a)}}{=} e^{\frac{k}{4}\Delta^2}\PB_{V_{1:k}}\left(e^{\frac{k}{4} \overline{V}^2} > e^{\frac{k}{4}\Delta^2}\right) + \int_{e^{\frac{k}{4}\Delta^2}}^\infty\PB_{V_{1:k}}\left(e^{\frac{k}{4} \overline{V}^2} > y\right)\rd y\notag\\
    &= e^{\frac{k\Delta^2}{4}}\PB_{V_{1:k}}\Big(|\overline{V}|> \Delta\Big)
    + \int_{e^{\frac{k\Delta^2}{4}}}^\infty \PB\bigg(
    |\overline{V}| > \sqrt{\frac{4\log y}{k}}
    \bigg)\rd y\notag\\
    &\le e^{-\frac{k\Delta^2}{4}} + \int_{e^{\frac{k\Delta^2}{4}}}^\infty 2e^{-2\log y}\rd y
    \le e^{-\frac{k\Delta^2}{4}} - 
    2y^{-1}\Big|_{e^{\frac{k\Delta^2}{4}}}^\infty = 3e^{-\frac{k\Delta^2}{4}}, \label{eq:V_larger_sqrt_k_2}
\end{align}
where (a) follows from Fubini's formula, namely, for any  non-negative random variable $X$ with CDF $F_X$ and any $x \ge 0$,
\begin{align*}
\EB [X\mathbbm{1}\{X > x_0\}]  &= \int_{x_0}^\infty xF_X(\rd x) = \int_{x_0}^\infty \bigg(\int_0^x \rd t\bigg) F_X(\rd x) 
= \int_0^\infty \bigg(\int_{x_0\vee t}^\infty F_X(\rd x)\bigg)\rd t\\
&= \int_0^{x_0}\PB(X > x_0)\rd t + \int_{x_0}^\infty \PB(X>t)\rd t = x_0\PB(X> x_0) + \int_{x_0}^\infty \PB(X>t)\rd t.
\end{align*}
Further, when $\epsilon \le \frac{1}{64}\sqrt{\frac{\alpha k}{n\log(2nk/\alpha^2)}}$, one has
\[
2\epsilon^2 \overline{V}^2 n \le \frac{k}{128n}\overline{V}^2 n \le \frac{k}{4}\overline{V}^2.
\]
Combining this with \eqref{eq:V_larger_sqrt_k_1} and \eqref{eq:V_larger_sqrt_k_2} (with $\Delta$ set to $\sigma$), and substituting the resulting bounds into \eqref{eq:gS_1_bound_first}, yields
\begin{equation}\label{eq:lower_bd_EB_sum_a2_delta3}
\begin{aligned}
\gS_1 &\le 100\epsilon^2n \log \Big(\frac{2nk}{\alpha^2}\Big)\Big(3 + 3e^{-\frac{k\sigma^2}{4}}\Big)
+ \frac{\alpha^{16}}{4n^7k^8} + 8\epsilon n \EB_{V_{1:k}}\Big[e^{\frac{k}{4}\overline{V}^2}\mathbbm{1}\{|\overline{V}|> \sigma \}\Big]\\
&\overset{\eqref{eq:V_larger_sqrt_k_2}}{\le} 600\epsilon^2n\log \Big(\frac{2nk}{\alpha^2}\Big) + \frac{\alpha^{16}}{4n^7k^8} + 24\epsilon n e^{-\frac{k\sigma^2}{4}}\\
&\overset{\eqref{eq:def_sigma_tail_thresh}}{\le} 600\epsilon^2n\log \Big(\frac{2nk}{\alpha^2}\Big) + \frac{\alpha^{16}}{4n^7k^8} + \frac{\alpha^8}{4n^3k^4}.
\end{aligned}
\end{equation}

\paragraph{Bounding the term $\gS_2$.}
We first rewrite $\gS_2$ (cf.~\eqref{eq:E-2q-1-UB-S12}) slightly by using \eqref{eq:F_n_inverse_upper_bd} as follows:
\begin{align*}
    \gS_2 = \EB_{Y_{1:n},U}\Bigg[\mathop{\EB}_{V_{1:k}\sim p(\cdot| Y_{1:n},U)}&\bigg[\frac{F_n(\epsilon\overline{V})}{F(Y_{1:n},U)}\log \bigg(\frac{F_n(\epsilon\overline{V})}{F(Y_{1:n},U)}\bigg)\bigg]\Bigg]\\
    &\overset{\eqref{eq:F_n_inverse_upper_bd}}{\le} 
    \EB_{Y_{1:n},U}\Bigg[\mathop{\EB}_{V_{1:k}\sim p(\cdot| Y_{1:n},U)}\bigg[\frac{F_n(\epsilon\overline{V})}{F(Y_{1:n},U)}\log \bigg(\frac{1}{F(Y_{1:n},U)}\bigg)\bigg]\Bigg]\\
    &\le
    \EB_{Y_{1:n},U}\Bigg[\mathop{\EB}_{V_{1:k}\sim p(\cdot| Y_{1:n},U)}\bigg[\frac{1}{F(Y_{1:n},U)}\log \bigg(\frac{1}{F(Y_{1:n},U)}\bigg)\bigg]\Bigg],
\end{align*}
where the last line holds since $\log \Big(\frac{1}{F(Y_{1:n},U)}\Big) \ge 0$.
It can be easily verified that the function $\frac{1}{x}\log \frac{1}{x}$ is convex when $0< x \le 1$. Further, since $F_n(\epsilon \overline{v}) = \frac{(1+\epsilon\overline{v})^n}{e^{\epsilon\overline{v}n}} \le 1$, one can invoke Jensen's inequality to obtain
\begin{align*}
    \gS_2 &\le  \EB_{Y_{1:n},U}\left[\mathop{\EB}_{V_{1:k}\sim p(\cdot| Y_{1:n},U)}\bigg[\frac{1}{F(Y_{1:n},U)}\log \bigg(\frac{1}{F(Y_{1:n},U)}\bigg)\bigg]\right]\\
    &= \EB_{Y_{1:n},U}\left[\frac{1}{F(Y_{1:n},U)}\log \bigg(\frac{1}{F(Y_{1:n},U)}\bigg)\right] \le \EB_{Y_{1:n},U}\left[\mathop{\EB}_{V_{1:k}\sim p(\cdot| Y_{1:n},U)} \left[\frac{1}{F_n(\epsilon\overline{V})}\log \frac{1}{F_n(\epsilon\overline{V})} \right]\right]\\
    &\le \EB_{V_{1:k}}\left[ \left(2\big(\epsilon \overline{V}\big)^2 n\right)\exp\left\{2\big(\epsilon \overline{V}\big)^2 n \right\}\right],
\end{align*}
where the last inequality follows from \eqref{eq:F_n_inverse_upper_bd}, together with the fact that the function
$x\mapsto x\log x$ is monotonically increasing on $[1,\infty)$.
By letting $\Delta$ in \eqref{eq:V_larger_sqrt_k_2} be $0$, we arrive at
\begin{equation}\label{eq:lower_bd_EB_sum_a2_delta4}
    \gS_2 \le 2\epsilon^2 n\EB_{V_{1:k}}\Big[e^{\frac{k}{4}\overline{V}^2}\Big] \le 6\epsilon^2 n.
\end{equation}

\paragraph{Putting all this together.}
Finally, combining~\eqref{eq:bd_mutual_info},~\eqref{eq:lower_bd_EB_sum_a2_delta3} and~\eqref{eq:lower_bd_EB_sum_a2_delta4} with \eqref{eq:E-2q-1-UB-S12} yields
\[
2\EB\bigg[\ssum{j}{1}{k}(2q_j - 1)^2 \bigg] \le \gS_0 + \gS_1 + \gS_2 \le 642 n\epsilon^2\log \Big(\frac{2nk}{\alpha^2}\Big) + \frac{\alpha^6}{16n^3k^3},
\]
thereby concluding the proof of Lemma~\ref{lem:bd_2q-1^2}.

\subsection{Proof of \Cref{prop:pathological_cover} and \Cref{cor:offline_tc_lb}}
\subsubsection{Proof of \Cref{prop:pathological_cover}}
Denote $p_\gD(\cdot)$ as the density function of the distribution $\gD$. According to the definition of Riemann-integrability, for every $\varepsilon > 0$, the following holds for large enough $n$:
\begin{align*}
\bigg|\, \PB(Y\in \gC_n) - (1-\alpha)\,\bigg|&\le
\bigg|\,\PB(Y\in \gC_n) - (1-\alpha)\sum_{i=0}^{n-1}p_{\gD}(i/n)\,\bigg| + (1-\alpha)\bigg|\,\sum_{i=0}^{n-1} \int_{i/n}^{(i + 1)/n} \big(p_{\gD}(i/n) - p_{\gD}(y)\big) \rd y\,\bigg|\\
&\le \bigg|\, \sum_{i=0}^{n-1} \PB\Big(Y\in \big[i/n, \big(i + (1-\alpha)\big)\big/n]\Big)-(1-\alpha)\sum_{i=0}^{n-1}p_{\gD}(i/n) \,\bigg| + \varepsilon  \\
&= \bigg|\, \sum_{i=0}^{n-1}\int_{i/n}^{(i+(1-\alpha))/ n} \big(p_\gD(y) - p_{\gD}(i/n)\big) \rd y \,\bigg| + \varepsilon \le (2-\alpha)\varepsilon.
\end{align*}
This immediately establishes the result of \Cref{prop:pathological_cover}.
\subsubsection{Proof of \Cref{cor:offline_tc_lb}}
The claim follows immediately by invoking \Cref{lem:coverage_gap_lower_bd} with
\[
k \coloneqq \frac{256K}{\alpha}
\quad\text{and}\quad
\epsilon \coloneqq \min\Biggl\{\frac{\alpha^{5/2}}{200},\ \frac{1}{64}\sqrt{\frac{\alpha k}{n \log\!\bigl(2nk/\alpha^2\bigr)}}\Biggr\}.
\]

\section{Examples of stable learning algorithms}\label{sec:exp_for_full_conf}
To illustrate the applicability of Assumption~\ref{ass:fair_alg}, this section verifies it for several learning algorithms commonly used in statistical applications.
In particular, Section~\ref{sec:stochastic_scvx_opt} describes how to incorporate stochastic optimization methods into our online conformal framework: the fitted model can be updated incrementally using the newly arrived data at each time step, without retraining from scratch.
\subsection{Constrained M-estimation}
\label{sec:constrained-M-estimation}
We begin with the classical constrained M-estimator \citep{van2000asymptotic} that minimizes the empirical loss:
\[
\widehat{\vartheta}_n 
= \argmin_{\vartheta \in \gC} \widehat L_n(\vartheta)
\coloneqq \argmin_{\vartheta \in \gC} \frac{1}{n}\sum_{i=1}^n \ell(\vartheta; Z_i),
\]
where $\{Z_i\}\subset \mathbb{R}^{d_Z}$ denote $n$ independent data samples with $Z_i$ drawn from the distribution $\mathcal{D}_i$,  
\(\gC \subset \R^d\) represents a closed convex constraint set, and \(\ell(\cdot; z)\) is a loss metric assumed to be differentiable in \(\vartheta\) for every \(z\).
We also introduce the population
risk
\[
L(\vartheta) \coloneqq \frac{1}{n}\ssum{i}{1}{n} \mathop{\E}\limits_{Z_i\sim \gD_i}\big[\ell(\vartheta;Z_i)\big].
\]
Note that we allow for a non-identically distributed sequence \(\{Z_i\}_{i=1}^n\), which is compatible
with the drifting environments considered in the present paper.

We impose the following standard assumptions ensuring well-posedness and curvature of the loss functions.

\begin{assumption}\label{ass:M-estimator_reg_ass}
Suppose that the loss functions satisfy the following properties:
\begin{enumerate}
\item \(L\) is \(\mu\)-strongly convex on \(\gC\) for some \(\mu>0\), in the sense that
\[
\nabla^2 L(\vartheta) \succeq \mu I_d \qquad \text{for all } \vartheta \in \gC.
\]
\item For any \(\vartheta\in \gC\), \(\norm{\vartheta}_2 \le D_\gC\).
\item For any \(\vartheta, \vartheta'\in \gC\) and  \(Z\), \(\norm{\nabla_{\vartheta} \ell(\vartheta; Z)} \le \beta_L\) and \(\norm{\nabla_{\vartheta}^2 \ell(\vartheta; Z) - \nabla_{\vartheta}^2 \ell(\vartheta'; Z)}_{\infty,\infty} \le \beta_s\norm{\vartheta - \vartheta'}_2\). Here $\norm{A}_{\infty,\infty} \coloneqq \max\limits_{i,j\in [d]}\abs{A_{ij}}$.
\end{enumerate}
\end{assumption}
With Assumption~\ref{ass:M-estimator_reg_ass} in place, we obtain the following stability guarantees for the
constrained M-estimator.
\begin{proposition}\label{prop:m_esti_bd_diff}
Suppose that Assumption~\ref{ass:M-estimator_reg_ass} holds. For any $n \ge \frac{32d\beta_s^2D_\gC^2}{\mu^2}\log\big(\frac{12\beta_s D_\gC}{\mu}\big)$, We can find a typical set \(\gE\) in the \(n\)-sample
space \(\RB^{d_Z\times n}\) obeying \(\PB(\gE^c) \le d^2 \exp\big\{-\frac{\mu^2 n}{32\beta_s^2 D_\gC^2} \big\}\) such that the following holds: 
consider two adjacent datasets in \(\gE\), \(\{Z_i\}_{i=1}^n\) and \(\{Z_i'\}_{i=1}^n\), that differ only in the last
coordinate, i.e., \(Z_i'=Z_i\) for \(i=1,\dots,n-1\) and \(Z_n'\neq Z_n\), 
and denote by  \(\widehat{\vartheta}_n\) and \(\widehat{\vartheta}_n'\) the corresponding constrained M-estimates
\[
\widehat{\vartheta}_n \coloneqq \argmin_{\vartheta \in \gC}\left\{\sum_{i=1}^n \ell(\vartheta;Z_i)\right\},
\quad
\widehat{\vartheta}_n' 
= \argmin_{\vartheta \in \gC} \left\{\sum_{i=1}^{n-1} \ell(\vartheta;Z_i) + \ell(\vartheta;Z_n')\right\},
\]
then one has
\[
\big\|\widehat{\vartheta}_n - \widehat{\vartheta}_n^\prime\big\|_2 \le \frac{4\beta_L}{\mu n}.
\]
\end{proposition}

The proof of Proposition~\ref{prop:m_esti_bd_diff} is deferred to \Cref{sec:prf:prop:m_esti_bd_diff}. With the above proposition in hand, consider a prediction model $\mu(\cdot\mid \vartheta)$. Assume that, for every input $x$, the mapping $\vartheta \rightarrow \mu(x\mid \vartheta)$ is $L_0$-Lipschitz. Then, by \Cref{prop:m_esti_bd_diff}, for any two neighboring data sequences $\{Z_n\}$ and $\{Z_n'\}$ that both lie in $\gE$ (which happens with high probability), the corresponding fitted models $\mu(\cdot\mid \widehat{\vartheta})$ and $\mu(\cdot\mid \widehat{\vartheta}')$ satisfy:
\[
\big|\mu(x\mid \widehat{\vartheta}) - \mu(x\mid \widehat{\vartheta}')\big| \le L_0 \big\|\widehat{\vartheta} - \widehat{\vartheta}'\big\|_2 \le \frac{4\beta_L L_0}{\mu n},\quad \forall x\in \gX,
\]
thereby validating Assumption~\ref{ass:fair_alg} for this setting with $L_2=\frac{4\beta_LL_0}{\mu}$.

\subsection{Linear stochastic approximation}
\label{sec:linear-stoc-approx}

Next, we verify stability for an important class of \emph{online} learning methods based on stochastic approximation.
Unlike the previous example in \Cref{sec:constrained-M-estimation}, where the estimator is obtained via empirical risk minimization,
stochastic approximation updates the parameter incrementally as new data arrive \citep{bottou2018optimization}.  This makes it particularly
well-suited to our online conformal framework, as it avoids retraining from scratch while still enabling control over the sensitivity of the fitted model to individual observations.

To be concrete, consider a linear prediction model
\[
\widehat\mu(x \mid Z_{1:n}) = x^\top \vartheta_n ,
\]
where the parameter \(\vartheta_n \in \R^d\) is updated online using the data \( Z_{1:n}=\{(X_i,Y_i)\}_{i=1}^n\). At iteration
\(n\), the squared loss function is defined as
\[
\ell_n(\vartheta)
= (Y_n - X_n^\top \vartheta)^2
= \vartheta^\top X_n X_n^\top \vartheta
    - 2 Y_n X_n^\top \vartheta + Y_n^2 .
\]
Its gradient at \(\vartheta_n\) is given by
\[
\nabla \ell_n(\vartheta_n)
= \underset{\eqqcolon\, \widehat{A}_n}{\underbrace{2X_n X_n^\top}} \vartheta_n - \underset{\eqqcolon\, \widehat b_n}{\underbrace{2Y_n X_n}}
= \widehat A_n \vartheta_n - \widehat b_n .
\]
%
The linear stochastic approximation (LSA) recursion with a decaying stepsize \(\eta_n\) for iteration $n$ is
\begin{align}
\vartheta_{n+1}
&= \vartheta_n - \eta_n
   \bigl(\widehat A_n \vartheta_n - \widehat b_n\bigr)
 = \bigl(I - \eta_n\widehat A_n\bigr)\vartheta_n
    + \eta_n\widehat b_n = \bigl(I - \eta_nA_n - \eta_n\widetilde A_n\bigr)\vartheta_n
    + \eta_n\widehat b_n ,
\label{eq:lsa-recursion}
\end{align}
where we denote \(A_n \coloneqq \E[\widehat A_n]\) and \(\widetilde A_n \coloneqq \widehat A_n - A_n\).

To validate stability, we compare the LSA iterates generated from two adjacent data streams.  Consider two sequences \(\{(\widehat A_i,\widehat b_i)\}_{i=1}^n\) and \(\{(\widehat A_i',\widehat b_i')\}_{i=1}^n\) that differ
in exactly one index \(l\), i.e.,
\((\widehat A_i,\widehat b_i)=(\widehat A_i',\widehat b_i')\) for all \(i\neq l\), but
\((\widehat A_l,\widehat b_l)\neq(\widehat A_l',\widehat b_l')\).
We impose the following assumptions. 


\begin{assumption}\label{ass:lsa_stable_ass}
Suppose that there exist constants \(L,\mu >0\) and \(\widehat{\sigma},\sigma\ge 1\) satisfying the following properties:
\begin{enumerate}
\item \(A_i \succeq \mu I\) for all \(i\ge 1\);
\item \(\|\widehat{A}_i\|\le \widehat{\sigma},~ \|\widetilde{A}_i\| \le \sigma,~ \|\widehat{b}_i\|_2 \le L\) for all \(i\ge 1\).
\end{enumerate}
\end{assumption}

With Assumption~\ref{ass:lsa_stable_ass} in place, we establish stability of the terminal LSA
iterate.  In particular, changing a single observation in the data stream alters \(\vartheta_{n+1}\) by at most
\(O((\log^3 n)/n)\) with high probability.

\begin{proposition}[Bounded differences for LSA]\label{prop:lsa-bd}
Consider any fixed $\zeta \ge 1$, and suppose that
Assumption~\ref{ass:lsa_stable_ass} holds. 
Assume that the LSA recursion~\eqref{eq:lsa-recursion} adopts the stepsize \(\eta_n = \min\{1/\widehat{\sigma},\gamma_n/n\}\) at iteration $n$, where \(\gamma_n = C\log n\) for some constant \(C\ge\frac{2(\zeta+1)}{\mu}>0\). 
Then there exists a constant \(K>0\) (independent of \(n\)) such that, for any \(n\ge d^{\frac{1}{\zeta}}\)and any two adjacent datasets differing in a
single time index \(l\), the corresponding LSA iterates $\{\vartheta_n\}_{n=1}^\infty$ and $\{\vartheta_n^\prime\}_{n=1}^\infty$ satisfy
\[
\big\|\vartheta_{n+1} - \vartheta_{n+1}'\big\|_2
\le K \frac{\log^3 n}{n}
\]
with probability at least \(1 - n^{-\zeta}\).
\end{proposition}

The proof of Proposition~\ref{prop:lsa-bd} is postponed to \Cref{sec:prf:prop:lsa-bd}. 
We now return to verify Assumption~\ref{ass:fair_alg}. In this example, our prediction model takes the form
\(
\widehat{\mu}(x\mid Z_{1:n}) = x^\top \vartheta_n .
\)
Assume that the covariate $X$ is essentially bounded, i.e., $\norm{X}\le B_x$ almost surely. Then, for any two neighboring samples $Z_{1:n}$ and $Z_{1:n}'$ in a typical set $\gE$ with $\PB(\gE) \ge 1 - n^{-\zeta}$, the corresponding parameter estimates---denoted by $\vartheta_n$ and $\vartheta_n'$, respectively---obtained by LSA satisfy
\[
\big|~x^\top \vartheta_n - x^\top \vartheta_n'\big| \le \norm{x}_2 \norm{\vartheta_n - \vartheta_n'}_2 \le B_x K\frac{\log^3 n}{n}.
\]
Hence, this justifies Assumption~\ref{ass:fair_alg} with $L_2=B_xK\log^3 (m+1)$ (with $m$ the size of the training set used in this assumption).

\subsection{Stochastic strongly convex optimization}\label{sec:stochastic_scvx_opt}

%
%
%

We now turn to another case where the predictive model \(\widehat{\mu}(\cdot \mid \{Z_i\}_{i=1}^n)\) is trained in an adaptive manner. 
For parametric statistical models, a natural approach is to maintain a parameter vector and update it using an incremental
optimization rule.  Suppose that the model used at iteration \(\tau\) is
\(\widehat{\mu}_{\tau}(\cdot)=\widehat{\mu}(\cdot\mid \vartheta_\tau)\), where \(\vartheta_\tau\) is learned from the
data \(\{(X_i,Y_i)\}_{i=1}^{\tau-1}\) via stochastic optimization (or, more generally, an iterative online training
procedure).  Starting from \(\vartheta_{\tau-1}\), after observing the new data point \((X_{\tau-1},Y_{\tau-1})\) we
update
\begin{equation}\label{eq:vartheta_update}
\vartheta_\tau
= \vartheta_{\tau-1} - \eta_{\tau-1} f\big(\vartheta_{\tau-1};(X_{\tau-1},Y_{\tau-1})\big),
\end{equation}
where \(f(\vartheta;(X,Y))\) denotes the update direction and \(\eta_{\tau-1}\)  the stepsize. Let $\Theta$ denote the parameter domain and $\mathcal{Z}$ the domain of data points. 

To study stability for such adaptive methods, we impose several standard conditions on the update map $f$.

\begin{assumption}\label{ass:ssco_ass}
There exist constants \(0<\mu\leq L\) and \(B>0\) such that, for any \(\vartheta,\vartheta'\in\Theta\) and
\((x,y)\in\gZ\),
\begin{itemize}
\item Strong convexity: \(\inner{f(\vartheta;(x,y)) - f(\vartheta';(x,y))}{\vartheta-\vartheta'} \ge \mu\|\vartheta-\vartheta'\|_2^2\);
\item Smoothness: \(\|f(\vartheta;(x,y)) - f(\vartheta';(x,y))\|_2 \le L\|\vartheta-\vartheta'\|_2\);
\item Boundedness: \(\norm{f(\vartheta;(x,y))}_2 \le B\).
\end{itemize}
\end{assumption}

With Assumption~\ref{ass:ssco_ass} in place and suitably decaying stepsizes, a single-sample
perturbation has an \(O(1/n)\) effect on the parameter iterate, as asserted in the following lemma.
The proof is provided in \Cref{sec:prf:prop:ssco_bd}. 
\begin{proposition}\label{prop:ssco_bd}
Suppose that Assumption~\ref{ass:ssco_ass} holds. Consider the parameter sequence
\(\{\vartheta_n\}_{n=1}^\infty\) updated according to \eqref{eq:vartheta_update} with stepsize
\(\eta_n=\min\{\gamma/n,\,1/L\}\), where \(\gamma>3/\mu\). Then there exists a constant \(K>0\) (independent of \(n\))
such that, for any two adjacent datasets differing in a single time index \(l\), the corresponding iterates
\(\{\vartheta_n\}_{n=1}^\infty\) and \(\{\vartheta_n'\}_{n=1}^\infty\) satisfy
\[
\big\|\vartheta_n-\vartheta_n' \big\|_2 \le \frac{K}{n}.
\]
\end{proposition}

We now discuss how to incorporate the above adaptive updates into the online full conformal construction.  In particular, when forming the
augmented dataset \((X_{n,r,l},y)\) in the proposed full conformal procedure, we update the model parameter using the same one-step
rule.  For any pair \((x,y)\in \mathcal{X}\times \mathbb{R}\), define
\[
\vartheta_\tau^{(x,y)}
\coloneqq
\vartheta_{\tau-1}-\eta_{\tau-1} f\big(\vartheta_{\tau-1};(x,y)\big).
\]
Accordingly, in round $r$ of stage $n$, we redefine the residual score \(s_{i}(X_{n,r,l},y)\) from Eqn.~\eqref{eq:res_score-online} as follows:
\begin{subequations}
\label{eq:res_score-online-param}
\begin{align}
s_{i}^{(X,y)} & \coloneqq\Big|Y_{n,r-1,i}-\widehat{\mu}\big(X_{n,r-1,i}^{\mathsf{cal}}\,|\, \vartheta_{\tau_{n,r}}^{(X,y)}\big)\Big|,\qquad i=1,\ldots,T_{r-1},\label{eq:res_score-online-param-cal}\\
s_{\mathsf{test}}^{(X,y)} & \coloneqq\Big|y-\widehat{\mu}\big(X\,|\, \vartheta_{\tau_{n,r}}^{(X,y)}\big)\Big|.
\label{eq:res_score-online-param-test}
\end{align}
\end{subequations}
Equipped with~\eqref{eq:res_score-online-param}, we can then construct the prediction set according to~\eqref{eq:full_conformal_inf}. This leads to a modification of
Algorithm~\ref{alg:FOCID} 
that incorporates adaptive updates of the fitted model.

\subsection{Detailed proofs}

\subsubsection{Proof of Propostion~\ref{prop:m_esti_bd_diff}}\label{sec:prf:prop:m_esti_bd_diff}
Since $\gC$ is convex and $\ell(\cdot;z)$ is differentiable in the first argument, the standard optimality condition \citep{boyd2004convex} yields 
\begin{align}
\label{eq:foc-original}
\big\langle \nabla \widehat L_n(\widehat{\vartheta}_n), \vartheta - \widehat{\vartheta}_n \big\rangle &\ge 0
\quad \text{for all } \vartheta \in \gC, \\
\label{eq:foc-perturbed}
\big\langle \nabla \widehat L_n'(\widehat{\vartheta}_n'), \vartheta - \widehat{\vartheta}_n'\big\rangle &\ge 0
\quad \text{for all } \vartheta \in \gC.
\end{align}
In particular, taking $\vartheta = \widehat{\vartheta}_n'$ in \eqref{eq:foc-original} and $\vartheta = \widehat{\vartheta}_n$ in \eqref{eq:foc-perturbed}, adding these two inequalities, and setting $\Delta_n \coloneqq \widehat{\vartheta}_n' - \widehat{\vartheta}_n$, we obtain
\begin{equation}
\label{eq:basic-ineq-before-substitution}
\big\langle \nabla \widehat L_n(\widehat{\vartheta}_n) - \nabla \widehat L_n'(\widehat{\vartheta}_n'), \Delta_n \big\rangle \ge 0.
\end{equation}

By construction, the gradients of the two empirical loss functions satisfy
\[
\nabla \widehat L_n'(\vartheta)
= \nabla \widehat L_n(\vartheta) 
+ \frac{1}{n}\big(\nabla \ell(\vartheta;Z_n') - \nabla \ell(\vartheta;Z_n)\big), 
\]
which, when evaluated at $\vartheta=\widehat{\vartheta}_n'$, yields
 \[
\nabla \widehat L_n'(\widehat{\vartheta}_n')
= \nabla \widehat L_n(\widehat{\vartheta}_n')
+ \frac{1}{n}\big(\nabla \ell(\widehat{\vartheta}_n',Z_n') - \nabla \ell(\widehat{\vartheta}_n',Z_n)\big).
\]
Substituting this identity into \eqref{eq:basic-ineq-before-substitution} yields
\begin{equation*}
\big\langle \nabla \widehat L_n(\widehat{\vartheta}_n) - \nabla \widehat L_n(\widehat{\vartheta}_n'), \Delta_n \big\rangle 
- \frac{1}{n}\big\langle \nabla \ell(\widehat{\vartheta}_n';Z_n') - \nabla \ell(\widehat{\vartheta}_n';Z_n), \Delta_n\big\rangle 
\ge 0.
\end{equation*}
Rearranging terms, we are left with
\begin{equation}
\label{eq:grad-diff-vs-data-perturbation}
\big\langle \nabla \widehat L_n(\widehat{\vartheta}_n') - \nabla \widehat L_n(\widehat{\vartheta}_n), \Delta_n\big\rangle 
\le \frac{1}{n}\,
\big\langle \nabla \ell(\widehat{\vartheta}_n';Z_n) - \nabla \ell(\widehat{\vartheta}_n';Z_n'), \Delta_n\big\rangle .
\end{equation}

According to Assumption~\ref{ass:M-estimator_reg_ass}
there exists a quantity $\beta_L>0$ such that
\begin{equation}
\label{eq:gradient-bounded-diff}
\big\|\nabla \ell(\vartheta;z) - \nabla \ell(\vartheta;z')\big\|_2
\le 2\beta_L
\qquad\text{for all } \vartheta \in \gC, \; z,z'.
\end{equation}
With \eqref{eq:gradient-bounded-diff} in mind, we can bound the right-hand side of
\eqref{eq:grad-diff-vs-data-perturbation} using the Cauchy–Schwarz inequality:
\begin{align*}
\frac{1}{n}\,
\big|\big\langle\nabla \ell(\widehat{\vartheta}_n';Z_n') - \nabla \ell(\widehat{\vartheta}_n';Z_n),\Delta_n\big\rangle\big|
&\le \frac{1}{n}\,
\big\|\nabla \ell(\widehat{\vartheta}_n';Z_n') - \nabla \ell(\widehat{\vartheta}_n';Z_n)\big\|_2 \,\|\Delta_n\|_2 \\
&\le \frac{2\beta_L}{n}\,\|\Delta_n\|_2.
\end{align*}
Combine this with \eqref{eq:grad-diff-vs-data-perturbation} to arrive at the upper bound
\begin{equation}
\label{eq:grad-diff-upper-bound}
\big\langle\nabla \widehat L_n(\widehat{\vartheta}_n') - \nabla \widehat L_n(\widehat{\vartheta}_n), \Delta_n\big\rangle
\le \frac{2\beta_L}{n}\,\|\Delta_n\|_2.
\end{equation}

On the other hand, the left-hand side of \eqref{eq:grad-diff-upper-bound} can be expressed in terms of the empirical Hessian. Specifically, by the fundamental theorem of calculus for vector-valued functions,
\[
\nabla \widehat L_n(\widehat{\vartheta}_n') - \nabla \widehat L_n(\widehat{\vartheta}_n)
= \int_0^1 \nabla^2 \widehat L_n\big(\widehat{\vartheta}_n + t\Delta_n\big)\,\Delta_n \, \mathrm{d}t,
\]
and as a consequence, 
\[
\big\langle\nabla \widehat L_n(\widehat{\vartheta}_n') - \nabla \widehat L_n(\widehat{\vartheta}_n),\Delta_n\big\rangle
= \int_0^1 \Delta_n^\top \nabla^2 \widehat L_n\big(\widehat{\vartheta}_n + t\Delta_n\big)\Delta_n \, \mathrm{d}t.
\]
In particular, there exists some $\widetilde{\vartheta}$ lying within the line segment between $\widehat{\vartheta}_n$ and $\widehat{\vartheta}_n'$ such that
\begin{equation*}
\big\langle\nabla \widehat L_n(\widehat{\vartheta}_n') - \nabla \widehat L_n(\widehat{\vartheta}_n), \Delta_n\big\rangle
= \Delta_n^\top \nabla^2 \widehat L_n(\widetilde{\vartheta}) \Delta_n.
\end{equation*}

We now control the empirical Hessian $\nabla^2 \widehat L_n(\vartheta)$ uniformly over all $\vartheta \in \gC$. Denote $\gN\big(\frac{\mu}{4\beta_s}, \gC\big)$ as the $\frac{\mu}{4}$-cover of $\gC$. Then we have $|\gN\big(\frac{\mu}{4\beta_s}, \gC\big)|\le \Big(\frac{12\beta_sD_\gC}{\mu}\Big)^d$. According to Bernstein's inequality, we have
\begin{equation}
\label{eq:hessian-concentration}
\sup_{\vartheta \in \gN\big(\frac{\mu}{4\beta_s}, \gC\big)} 
\Big\{\big\|\nabla^2 \widehat L_n(\vartheta) - \nabla^2 L(\vartheta)\big\|_{\infty,\infty}\Big\}
\le \frac{\mu}{4}
\end{equation}
with probability at least $1-d^2 \exp\big\{-\frac{\mu^2 n}{32\beta_s^2 D_\gC^2} \big\}$, provided that $n \ge \frac{32d\beta_s^2D_\gC^2}{\mu^2}\log\big(\frac{12\beta_sD_\gC}{\mu}\big)$. Combining this with Assumption~\ref{ass:M-estimator_reg_ass} leads to
\begin{align*}
    \sup_{\vartheta \in \gC}\Big\{\big\|\nabla^2 \widehat L_n(\vartheta) - \nabla^2 L(\vartheta)\big\|_{\infty,\infty}\Big\} &\le \beta_s\sup_{\vartheta \in \gC}\inf_{\vartheta^* \in \gN\big(\frac{\mu}{4\beta_s}, \gC\big)}\Big\{\|\, \vartheta - \vartheta^*\,\|_2\Big\} \\
    &\quad + \sup_{\vartheta^* \in \gN\big(\frac{\mu}{4\beta_s}, \gC\big)}\Big\{\big\|\nabla^2 \widehat L_n(\vartheta) - \nabla^2 L(\vartheta)\big\|_{\infty,\infty}\Big\} \overset{\eqref{eq:hessian-concentration}}{\le} \beta_s \frac{\mu}{4\beta_s} + \frac{\mu}{4} \le \frac{\mu}{2}.
\end{align*}
This combined with the strong convexity assumption and \eqref{eq:hessian-concentration} tells us that: for all $\vartheta \in \gC$ and every unit vector $u\in\R^d$,
\[
u^\top \nabla^2 \widehat L_n(\vartheta) u
\ge u^\top \nabla^2 L(\vartheta) u
- \big\|\nabla^2 \widehat L_n(\vartheta) - \nabla^2 L(\vartheta)\big\|
\ge \mu - \frac{\mu}{2} = \frac{\mu}{2}.
\]
Putting all these pieces together, we arrive at
\[
\frac{\mu}{2} \|\Delta_n\|_2^2
\le \Delta_n^\top \nabla^2 \widehat L_n(\widetilde{\vartheta}) \Delta_n
= \inner{\nabla \widehat L_n(\widehat{\vartheta}_n') - \nabla \widehat L_n(\widehat{\vartheta}_n)}{\Delta_n}
\le \frac{2\beta_L}{n}\,\|\Delta_n\|_2.
\]
If $\Delta_n = 0$, the bound is trivial. Otherwise, cancelling $\|\Delta_n\|_2$ from both sides yields
\(
\|\Delta_n\| \le \frac{4\beta_L}{\mu n}.
\)

\subsubsection{Proof of Proposition~\ref{prop:lsa-bd}}\label{sec:prf:prop:lsa-bd}
Without loss of generality, we assume that $\vartheta_1 = 0$. Note that under the stepsize choice $\eta_n=\min\{1/\widehat{\sigma},\gamma_n/n\}$, it can be easily seen that: for any unit vector $u$, 
\[
1 = \norm{u}_2^2 \ge u^\top (I - \eta_n \widehat{A}_n)u \ge \norm{u}_2^2 - \eta_n \widehat{\sigma} = 1- \eta_n \widehat{\sigma}\ge 0.
\]
This implies that for any $n$,  $\|I - \eta_n \widehat{A}_n\|\leq 1$.
According to the LSA update rule~\eqref{eq:lsa-recursion}, we have
\begin{equation}\label{eq:conserv_bd_vartheta}
\begin{aligned}
\|\vartheta_n\|_2 &\le \big\|\big(I - \eta_{n-1}\widehat{A}_{n-1}\big)\vartheta_{n-1}\big\|_2 + \eta_{n-1} \big\|\widehat{b}_n\big\|_2\\
&\overset{\mathrm{(a)}}{\le} \norm{\vartheta_{n-1}}_2 + \frac{CL\log (n-1)}{n-1}
\le \ldots \le \ssum{i}{1}{n} \frac{CL \log i}{i} \le CL \log^2 n,
\end{aligned}
\end{equation}
where we have made use of Assumption~\ref{ass:lsa_stable_ass} and the choice that $\eta_n\leq \gamma_n/n$.

\paragraph{Stability at iteration $l$.}
 By construction,
\(
\vartheta_i = \vartheta_i' \) holds for all \( i = 1,\dots,l .
\)
At the perturbed iteration $l+1$, it holds that
\begin{align*}
\vartheta_{l+1} - \vartheta_{l+1}'
&= \Bigl[\vartheta_l - \eta_{l}
         \bigl(\widehat A_l \vartheta_l - \widehat b_l\bigr)\Bigr]
   - \Bigl[\vartheta_l' - \eta_{l}
         \bigl(\widehat A_l' \vartheta_l' - \widehat b_l'\bigr)\Bigr] = -\eta_{l}
   \Bigl[(\widehat A_l - \widehat A_l')\vartheta_l
         - (\widehat b_l - \widehat b_l')\Bigr] .
\end{align*}
Taking this together with  $\eta_l \le {\gamma_l}/{l}$, $\sigma \geq 1$, and Assumption~\ref{ass:lsa_stable_ass} leads to
\begin{equation}
\label{eq:Delta-l+1}
\begin{aligned}
\bigl\|\vartheta_{l+1} - \vartheta_{l+1}'\bigr\|_2
&\le \frac{\gamma_l }{l}\left(\big\|\widetilde{A}_l\big\| + \big\|\widetilde{A}_l'\big\|\right)\norm{\vartheta_l}_2 + \frac{2\gamma_l L}{l}\\
&\overset{\eqref{eq:conserv_bd_vartheta}}{\le} \frac{2\gamma_l}{l}\left(\sigma CL \log^2 l + L\right) \le \frac{4C\sigma L\log^3 l}{l}.
\end{aligned}
\end{equation}

\paragraph{Error propagation after iteration $l$.}
Let us define, for all $i\geq 1$, 
\(
\Delta_i \coloneqq \vartheta_i - \vartheta_i' .
\)
For every $i>l$, the two recursions share the same
$(\widehat A_i,\widehat b_i)$, so subtracting their updates yields
\[
\Delta_{i+1}
= \bigl(I - \eta_{i}\widehat A_i\bigr)\Delta_i .
\]
Iterating over $i=l,\dots,n$ gives
\begin{equation}
\label{eq:Delta-n+1-product}
\Delta_{n+1}
= \Gamma_{l,n}^\gamma \Delta_{l+1},
\qquad
\Gamma_{l,n}^\gamma
\coloneqq \prod_{i=l+1}^{n}
\bigl(I - \eta_{i}\widehat A_i\bigr),
\end{equation}
where the matrix product is ordered from $i=l+1$ (the rightmost) up to $n$ (the leftmost).
Combining~\eqref{eq:Delta-l+1} and~\eqref{eq:Delta-n+1-product} yields
\begin{equation}
\label{eq:Delta-n+1-preliminary}
\|\Delta_{n+1}\|_2
= \|\Gamma_{l,n}^\gamma \Delta_{l+1}\|_2
\le \|\Gamma_{l,n}^\gamma\| \,\|\Delta_{l+1}\|_2
\le \frac{4C\sigma L\log^3 l}{l}\,\|\Gamma_{l,n}^\gamma\|.
\end{equation}
Hence, it boils down to
controlling the operator norm of the random matrix product
$\Gamma_{l,n}^\gamma$.
\paragraph{Tools for analyzing products of random matrices.}
To bound $\|\Gamma_{l,n}^\gamma\|$, we follow the framework of
\citet{durmus2021tight,huang2022matrix}. For any matrix
$B\in\R^{d\times d}$, let $(\sigma_\ell(B))_{\ell=1}^d$ denote its
singular values, and for $p\ge1$ define the Schatten $p$-norm
$
\|B\|_p \coloneqq \bigl(\sum_{\ell=1}^d \sigma_\ell^p(B)\bigr)^{1/p}.
$
For $p,q\ge1$ and a random matrix $X$, define
$
\|X\|_{p,q} \coloneqq \bigl( \E [\|X\|_p^q] \bigr)^{1/q}.
$
We record the following two useful lemmas from \citet{durmus2021tight}.

\begin{lemma}[Proposition~2 in \citet{durmus2021tight}]
\label{lem:prod_rand_matrix_mmt_bd}
Let $\{Y_\ell: \ell\in\NB\}$ be an independent sequence. Assume that for each $\ell\in\NB$, there exist
$m_\ell\in(0,1)$ and $\sigma_\ell>0$ such that
\[
\|\E[Y_\ell]\|^2 \le 1 - m_\ell
\quad\text{and}\quad
\|Y_\ell - \E[Y_\ell]\| \le \sigma_\ell
\quad\text{almost surely}.
\]
Define $Z_n = \prod_{\ell=0}^n Y_\ell = Y_n Z_{n-1}$ for $n\ge1$, with an
arbitrary starting point $Z_0$. Then, for any $2\le q\le p$ and $n\ge1$, one has
\begin{equation*}
\|Z_n\|_{p,q}^2
\le  \prod_{\ell=1}^n
\bigl(1 - m_\ell + (p-1)\sigma_\ell^2\bigr)
\bigl\| Z_0  \bigr\|_{p,q}^2.
\end{equation*}
\end{lemma}

\begin{lemma}[Lemma~1 in \citet{durmus2021tight}]
\label{lem:prod_rand_matrix_conc_bd}
Let $A\in\R$, $B>0$, $C\ge1$, and $p_0,p_1\in\R$ with
$1\le p_0\le p_1<\infty$. Let $X$ be a real random variable satisfying,
for any $p\in[p_0,p_1]$,
\begin{equation}
\label{eq:X_p_mmt_bd}
\E[|X|^p] \le C \exp(-Ap + Bp^2).
\end{equation}
Then, for all $\delta\in(0,1]$, with probability at least $1-\delta$ one has
\begin{equation*}
|X| \le \exp\Bigl(
        -A + Bp_0
        + 2\sqrt{B \log(C/\delta)}
        + \frac{\log(C/\delta)}{p_1}
    \Bigr).
\end{equation*}
\end{lemma}

\paragraph{High-probability bound on $\|\Gamma_{l,n}^\gamma\|$.}
We intend to apply Lemma~\ref{lem:prod_rand_matrix_mmt_bd} to analyze the matrix product
\[
\Gamma_{j,n}^\gamma
\coloneqq \prod_{i=j}^n
\bigl(I - \eta_{i}A_i + \eta_{i}\widetilde A_i\bigr),
\qquad j\le n.
\]
Take $Y_i = I - \eta_{i}A_i + \eta_{i}\widetilde A_i$, and $Z_0=I_d$. Then $\E[Y_i] = I - \eta_{i}A_i$ and
$Y_i - \E[Y_i] = \eta_{i}\widetilde A_i$.  
With Assumption~\ref{ass:lsa_stable_ass} in place, by setting
\[
m_i = \mu \eta_i, \qquad
\sigma_i = \sigma \eta_i,
\]
we see that $\|\E[Y_i]\|^2 \le 1 - m_i$ and
$\|Y_i - \E[Y_i]\| \le \sigma_i$.
Lemma~\ref{lem:prod_rand_matrix_mmt_bd} then tells us that, for any $2\le q\le p$,
\begin{align}
\E\bigl[\|\Gamma_{j,n}^\gamma\|_p^q\bigr]^{1/q}
&= \|\Gamma_{j,n}^\gamma\|_{p,q} \le \prod_{i=j}^n
    \Bigl(1 - \mu \eta_i
            + (p-1)\sigma^2 \eta_i^2\Bigr)
    \|I_d\|_p \nonumber\\
&\le d^{1/p}
   \exp\left\{
       -\mu \sum_{i=j}^n \eta_i
       + (p-1)\sigma^2\sum_{i=j}^n \eta_i^2
   \right\},
\label{eq:Gamma-moment}
\end{align}
where the last step makes use of the elementary inequality $\log(1+x)\le x$.

Since 
$\|I - \eta_i\widehat A_i\|\leq 1$ for each $i$, we have
$\|\Gamma_{l,n}^\gamma\| \le 1$ for all $l\le n$.  
By introducing an index
\[
j \coloneqq \max\Bigl\{l,
\bigl\lceil 12C\sigma^2/\mu \bigr\rceil, \bigl\lceil 2C\widehat{\sigma}\log(C\widehat{\sigma})\bigr\rceil\Bigr\},
\]
we can easily check that $\eta_i = {\gamma_i}/{i}$ for every $i \ge j$.
Further, we can derive the following inequality:
\begin{equation}\label{eq:lsa_absorb_high_order}
\begin{aligned}
    \sigma^2 \ssum{i}{j}{n}\frac{\gamma_i^2}{i^2} &\le C^2\sigma^2\log^2 n\left(\frac{1}{j-1} - \frac{1}{n}\right) \le \frac{C\mu}{12}\log^2 n\\
    &\overset{\mathrm{(a)}}{\le} \frac{\mu}{6}\int_j^n \frac{C\log x}{x}\rd x \le \frac{\mu}{6} \ssum{i}{j}{n}\frac{C\log i}{i} = \frac{\mu}{6}\ssum{i}{j}{n}\frac{\gamma_i}{i}.
\end{aligned}
\end{equation}
Here, (a) holds whenever $n \ge 3j$. In the regime $n \le j$, the desired claim follows directly by combining \eqref{eq:Delta-n+1-preliminary} with the bound $\norm{\Gamma_{l,n}^\gamma}\le 1$.
Taking $p=q$ in~\eqref{eq:Gamma-moment}, we obtain
\begin{equation}
\label{eq:Gamma-moment-p}
\E\bigl[\|\Gamma_{j,n}^\gamma\|_p^p\bigr]
\le d \exp\left\{
   -p\mu \sum_{i=j}^n \frac{\gamma_i}{i}
   + p^2\sigma^2\sum_{i=j}^n \frac{\gamma_i^2}{i^2}
\right\}.
\end{equation}
Applying Lemma~\ref{lem:prod_rand_matrix_conc_bd} with
$p_0=2$, $p_1=\infty$, $C=d$,
\(
A = \mu\sum_{i=j}^n \frac{\gamma_i}{i},~
B = \sigma^2\sum_{i=j}^n \frac{\gamma_i^2}{i^2},
\)
and using the fact that~\eqref{eq:Gamma-moment-p} implies
\eqref{eq:X_p_mmt_bd} for all $p\ge2$, we can deduce that, with probability at
least $1-\delta/n$,
\begin{align*}
\|\Gamma_{j,n}^\gamma\|
&\le \exp\left\{
   -\mu\sum_{i=j}^n \frac{\gamma_i}{i}
   + 3\sigma^2\sum_{i=j}^n \frac{\gamma_i^2}{i^2}
   + \log\frac{dn}{\delta}
\right\}
\overset{\eqref{eq:lsa_absorb_high_order}}{\le} \exp\left\{
   -\frac{\mu}{2}\sum_{i=j}^n \frac{\gamma_i}{i}
   + \log\frac{dn}{\delta}
\right\}.
\end{align*}
Using the elementary fact $\sum_{i=j}^n \frac{\log i}{i} \ge \log^2 n - \log^2 j = \log(n/j)\log (nj)\ge \log(n/j)\log n$, we further obtain
\begin{equation}
\label{eq:Gamma-high-prob-final}
\|\Gamma_{j,n}^\gamma\|
\le \frac{dn}{\delta}\Bigl(\frac{j}{n}\Bigr)^{\frac{C\mu \log n}{2}}.
\end{equation}
Since $\|\Gamma_{l,n}^\gamma\| \le \|\Gamma_{j,n}^\gamma\|$ holds by
construction and $\|\Gamma_{l,n}^\gamma\|\le1$, combining
\eqref{eq:Gamma-high-prob-final} with the trivial upper bound 1 yields, for all
$l\le n$ and with probability at least $1-\delta$,
\begin{equation}
\label{eq:Gamma-min-bound}
\|\Gamma_{l,n}^\gamma\|
\le \min\left\{
    1,\,
    \frac{dn}{\delta}
    \left(\frac{\max\{l,\lceil 12C\sigma^2/\mu\rceil,\bigl\lceil 2C\widehat{\sigma}\log(C\widehat{\sigma})\bigr\rceil\}}{n}\right)^{\frac{C\mu \log n}{2}}
\right\},\quad l = 1, \ldots, n.
\end{equation}

\paragraph{Bounded-difference property of $\vartheta_{n+1}$.}
For ease of exposition, we introduce the following notation:
\[
C_0 \coloneqq \max\left\{\lceil 12C\sigma^2/\mu\rceil,\bigl\lceil 2C\widehat{\sigma}\log(C\widehat{\sigma})\bigr\rceil\right\}.
\]
Substituting~\eqref{eq:Gamma-min-bound} into
\eqref{eq:Delta-n+1-preliminary}, we arrive at
\begin{equation}\label{eq:lsa_Delta_bd_1}
\|\Delta_{n+1}\|_2
\le \frac{4C\sigma L\log^3 l}{l}
\min\left\{
    1,\,
    \frac{dn}{\delta}
    \Bigl(\frac{\max\{l,C_0\}}{n}\Bigr)^{\frac{C\mu \log n}{2}}
\right\}
\end{equation}
holds for any $l = 1,\ldots, n$.
Let $\kappa \coloneqq \sigma/\mu$ and choose
\(
\delta = n^{-\zeta},~
C = \frac{2(\zeta+1)}{\mu}
\)
.  For $l \le C_0$, use \eqref{eq:lsa_Delta_bd_1} to obtain
\begin{align*}
\norm{\Delta_{n+1}}_2 &\le 4C\sigma L\log^3 C_0 \min\left\{1, dn^{\zeta+1}\left(\frac{C_0}{n}\right)^{(\zeta + 1)\log n}\right\}\\
&\overset{\mathrm{(a)}}{\le}4C\sigma L\log^3 C_0 \left(\frac{d C_0^{(\zeta + 1)\log n}}{n^{(\zeta + 1)(\log n-1)}}\right)^{\frac{1}{(\zeta + 1)(\log n-1)}}
\overset{\mathrm{(b)}}{\le} 4C\sigma L\log^3 C_0 \frac{e C_0^2}{n}.
\end{align*}
Here,  (a) results from the fact that $a \le a^{\lambda}$ for any $a,\lambda \in [0, 1]$, and (b) holds since $d^{\frac{1}{(\zeta + 1)(\log n -1)}} \le e$ provided that $n^\zeta \ge d$. Thus, in this setting, \Cref{prop:lsa-bd} holds by taking $K = {8e(\zeta+1)\kappa LC_0^2\log^3 C_0}$.

For $l > C_0$, similarly we can derive
\begin{align*}
    \norm{\Delta_{n+1}}_2 &\le \frac{4C\sigma L \log^3 n}{l}\min\left\{1, dn^{\zeta+1}\left(\frac{l}{n}\right)^{(\zeta + 1)\log n}\right\}\\
    &\le \frac{4C\sigma L \log^3 n}{l}\left(\frac{d l^{(\zeta + 1)\log n}}{n^{(\zeta + 1)(\log n-1)}}\right)^{\frac{1}{(\zeta + 1)(\log n-1)}}\\
    &\le \frac{4eC\sigma L\log^3 n}{n}\cdot l^{\frac{(\zeta+1)\log n}{(\zeta+1)(\log n - 1)}-1}\le \frac{4eC\sigma L\log^3 n}{n}\cdot n^{\frac{1}{\log n -1}} \le \frac{16eC\sigma L\log^3 n}{n}.
\end{align*}
In this setting, \Cref{prop:lsa-bd} follows by taking $K = 16eC\sigma L$.

\subsubsection{Proof of Proposition~\ref{prop:ssco_bd}}\label{sec:prf:prop:ssco_bd}
We let $K\coloneqq \frac{2BL\gamma}{\mu}$ and prove this result by induction. Fix two adjacent data streams, and let
\(\{\vartheta_k\}_{k\ge 1}\) and \(\{\vartheta_k'\}_{k\ge 1}\) denote the iterates generated by
\eqref{eq:vartheta_update} from the same initialization. Assume for the moment that
\begin{align}\label{eq:induction-adaptive-eg}
\|\vartheta_n-\vartheta_n'\|_2 \le \min\left\{\frac{K}{n}, \frac{2B}{L}\right\}
\end{align}
holds for some
\(n\ge 1\), and 
we would like to bound \(\|\vartheta_{n+1}-\vartheta_{n+1}'\|_2\).

We divide into several cases according to the index at which the two data streams differ. First, consider the case where the two streams differ at the most recent observation, i.e.,
\((X_n,Y_n)\neq(X_n',Y_n')\) while \((X_i,Y_i)=(X_i',Y_i')\) for all \(i\le n-1\).
In this case, the iterates coincide up to time \(n\), hence \(\vartheta_n=\vartheta_n'\).
In view of the update rule \eqref{eq:vartheta_update},
\begin{align*}
\|\vartheta_{n+1}-\vartheta_{n+1}'\|_2
&=
\eta_n \Bigl\| f\bigl(\vartheta_n;(X_n,Y_n)\bigr) - f\bigl(\vartheta_n;(X_n',Y_n')\bigr)\Bigr\|_2\\
&{\le} \min\left\{\frac{2\gamma B}{n},~ \frac{2B}{L}\right\} \le \min\left\{\frac{K}{n+1}, \frac{2B}{L}\right\},
\end{align*}
where penultimate inequality arises from Assumption~\ref{ass:ssco_ass} and our stepsize choice, and the last inequality holds since
\[
\frac{2\gamma B}{n}\le \frac{4\gamma B}{n+1} \le \frac{2BL\gamma}{\mu (n+1)} = \frac{K}{n+1}.
\]

Next, consider the case where the two data streams differ at some index \(l\le n-1\).  Then the current update at
time \(n\) is computed from the same observation \(Z_n=(X_n,Y_n)\) in both streams.
Since \(\eta_n \le 1/L\),   \citet[Lemma~2]{leeleave} implies that the update map is nonexpansive, which, taken together with the induction hypothesis (\ref{eq:induction-adaptive-eg}), yields
\[
\|\vartheta_{n+1}-\vartheta_{n+1}'\|_2 \le \|\vartheta_n-\vartheta_n'\|_2
\le \min\Bigl\{\frac{K}{n},\,\frac{2B}{L}\Bigr\}.
\]
In addition, a direct expansion of the recursion yields
\begin{equation}\label{eq:sco_stability_eq1}
\begin{aligned}
\|\vartheta_{n+1}-\vartheta_{n+1}'\|_2^2
&=
\bigl\|\vartheta_n-\vartheta_n' - \eta_n\bigl(f(\vartheta_n;Z_n)-f(\vartheta_n';Z_n)\bigr)\bigr\|_2^2 \\
&
\le
\|\vartheta_n-\vartheta_n'\|_2^2
- 2\eta_n \inner{f(\vartheta_n;Z_n)-f(\vartheta_n';Z_n)}{\vartheta_n-\vartheta_n'}
+ L^2\eta_n^2 \|\vartheta_n-\vartheta_n'\|_2^2\\
&
\le
\left(1 - \mu\eta_n + L^2 \eta_n^2\right)\|\vartheta_n-\vartheta_n'\|_2^2 \le \frac{K^2\left(1 - 2\mu\eta_n + L^2 \eta_n^2\right)}{n^2},
\end{aligned}
\end{equation}
which follows from Assumption~\ref{ass:ssco_ass}. Further, it can be derived that
\begin{align*}
  \left(1 - 2\mu\eta_n + L^2 \eta_n^2\right)\frac{(n+1)^2}{n^2}
  \overset{\mathrm{(a)}}{\le} (1 - \mu\eta_n)(1 + 3/n) \overset{\mathrm{(b)}}{\le} 1,
\end{align*}
where (a) is valid provided that $n+1 \ge \frac{L^2 \gamma}{\mu}$ and (b) holds as long as $\gamma \ge {3}/{\mu}$. Substitution  into~\eqref{eq:sco_stability_eq1} yields
\[
\norm{\vartheta_{n+1} - \vartheta_{n+1}^\prime}_2^2 \le \frac{K^2}{(n+1)^2}.
\]
Moreover, if $n+1 < \frac{L^2 \gamma}{\mu}$, then it still holds that
\[
\norm{\vartheta_{n+1} - \vartheta   _{n+1}^\prime}_2 \le \frac{2B}{L} \le \frac{2BL^2 \gamma}{\mu L (n+1) } = \frac{K}{n+1}.
\]
We have thus concluded the proof of Proposition~\ref{prop:ssco_bd}.

\section{Auxiliary concentration inequalities}
This appendix collects several classical concentration inequalities that will be used repeatedly in our analysis.
Their proofs can be found in, e.g., \citet{boucheron2013concentration,vershynin2018high}.


\begin{lemma}[McDiarmid inequality]
\label{lem:mcdiarmid}
Let $X_1,\ldots,X_n$ be independent random variables taking values in measurable spaces $\mathcal{X}_1,\ldots,\mathcal{X}_n$, and let $f:\mathcal{X}_1\times\cdots\times\mathcal{X}_n\to\mathbb{R}$ be a measurable function satisfying the bounded differences condition: there exist constants $c_1,\ldots,c_n \ge 0$ such that for all $x_1,\ldots,x_n,x_i'\in\mathcal{X}_i$,
\[
\bigl|f(x_1,\ldots,x_i,\ldots,x_n) - f(x_1,\ldots,x_i',\ldots,x_n)\bigr| \le c_i.
\]
Then, for all $t>0$,
\[
\PB\!\left(\bigl|f(X_1,\ldots,X_n) - \EB \big[ f(X_1,\ldots,X_n) \big]\bigr| \ge t\right)
\;\le\;
2\exp\!\left(-\frac{2t^2}{\sum_{i=1}^n c_i^2}\right).
\]
\end{lemma}

\begin{lemma}[Khintchine inequality for $p=1$]
\label{lem:khintchine_p1}
Let $\{\varepsilon_i\}_{i=1}^n$ be independent Rademacher random variables, 
that is, $\PB(\varepsilon_i = 1) = \PB(\varepsilon_i = -1) = 1/2$. 
Then for any real coefficients $a_1, \ldots, a_n$, one has
\[
\EB\bigg[\bigg|\sum_{i=1}^n a_i \varepsilon_i\bigg|\bigg] \ge \frac{1}{\sqrt{2}} \Big(\sum_{i=1}^n a_i^2\Big)^{1/2}.
\]
\end{lemma}

\begin{lemma}[Paley–Zygmund inequality]
\label{lem:paley-zygmund}
Let $Z$ be a nonnegative random variable with $\EB[Z^2] < \infty$. 
Then, for any $\theta \in [0,1]$, the following inequality holds:
\[
\PB\big(Z \ge \theta\,\EB[Z]\big) 
\;\ge\; (1 - \theta)^2 \frac{(\EB[Z])^2}{\EB[Z^2]}.
\]
\end{lemma}

\begin{lemma}[Generalized DKW inequality]\label{lem:DKW_ineq}
Let $X_1,\ldots,X_n$ be independent random elements taking values in $\gX$.
Let $F:\gX\times\R\to[0,1]$ satisfy that, for every $u\in\gX$, the map $x\mapsto F(u,x)$ is
nondecreasing and right-continuous.
Define
\[
\widehat F_n(x)\coloneqq \frac1n\sum_{i=1}^n F(X_i,x),
\qquad
\overline F(x)\coloneqq \frac1n\sum_{i=1}^n \E\big[F(X_i,x)\big].
\]
Then for any $\delta\in(0,1)$, with probability at least $1-\delta$,
\[
\sup_{x\in\R}\bigl|\widehat F_n(x)-\overline F(x)\bigr|
\le \frac{4}{\sqrt n}+\sqrt{\frac{\log(1/\delta)}{2n}}.
\]
Specifically, when $\gX_i =\RB,~ i = 1,\ldots, n$ and $F(X_i, x) = \mathbbm{1}\{X_i \le x\}$ we have:
\[
\PB\left(\sup\limits_{x\in \RB}\abs{\frac{1}{n}\ssum{i}{1}{n}\bigl(
\mathbbm{1}\{X_i \le x\} - \PB(X_i \le x)
\bigr)} \le \frac{4}{\sqrt{n}} + \sqrt{\frac{\log(1/\delta)}{2n}}\right) \ge 1- \delta.
\]
\end{lemma}

\begin{proof}[Proof of \Cref{lem:DKW_ineq}]
The proof comprises the following steps.

\paragraph{Step 1: applying McDiarmid around the mean.}
Let
\[
g \coloneqq \sup_{x\in\R}\bigl|\widehat F_n(x)-\overline F(x)\bigr|.
\]
Fix any $i\in[n]$. Replace $X_i$ by an arbitrary $\widetilde X_i$, and define
\[
\widetilde F_n(x)\coloneqq \frac1n\Big(F(\widetilde X_i,x)+\sum_{j\ne i}F(X_j,x)\Big),
\qquad
\widetilde g\coloneqq \sup_{x\in\R}\bigl|\widetilde F_n(x)-\overline F(x)\bigr|.
\]
Since $0\le F\le 1$, for every $x\in\R$,
\[
\bigl|\widehat F_n(x)-\widetilde F_n(x)\bigr|
=\frac1n\bigl|F(X_i,x)-F(\widetilde X_i,x)\bigr|
\le \frac1n.
\]
Using \(\abs{\abs{u}-\abs{v}}\le \abs{u-v}\) and
\(\abs{\sup_x a(x)-\sup_x b(x)}\le \sup_x\abs{a(x)-b(x)}\), we obtain
\[
|g-\widetilde g|
\le \sup_{x\in\R}\bigl|\widehat F_n(x)-\widetilde F_n(x)\bigr|
\le \frac1n,
\]
so $g$ satisfies bounded differences with constants $c_i=1/n$.
By McDiarmid's inequality, for any $\delta\in(0,1)$, with probability at least $1-\delta$,
\begin{equation}\label{eq:gen-dkw-mcdiarmid}
g \le \E[g] + \sqrt{\frac{\log(1/\delta)}{2n}}.
\end{equation}

\paragraph{Step 2: controlling $\E[g]$ via symmetrization and a reduction to indicators.}
Introduce a ghost sample \(X_1',\ldots,X_n'\), independent of
\((X_1,\ldots,X_n)\) satisfying \(X_i'\stackrel{\mathrm{d}}{=}X_i\), and let \(\epsilon_1,\ldots,\epsilon_n\) be i.i.d.~Rademacher
signs, independent of everything else. Recognizing that \(\E [F(X_i,x)]=\E [F(X_i',x)]\), we can express
\[
\widehat{F}_n(x)-\overline F(x)
=
\E_{X'}\!\left[\frac1n\sum_{i=1}^n\bigl(F(X_i,x)-F(X_i',x)\bigr)\,\middle|\,X\right].
\]
By the Jensen inequality and the convexity of \(\sup\), we have
\[
\E\!\left[\sup_{x\in\R}\big|\widehat{F}_n(x)-\overline F(x)\big|\right]
\le
\E_{X,X'}\!\left[\sup_{x\in\R}\left|\frac1n\sum_{i=1}^n\bigl(F(X_i,x)-F(X_i',x)\bigr)\right|\right].
\]
Using exchangeability of \((X_i,X_i')\) to insert Rademacher signs and then invoking the triangle inequality gives
\[
\begin{aligned}
&\E_{X,X'}\!\left[\sup_{x}\left|\sum_{i=1}^n\bigl(F(X_i,x)-F(X_i',x)\bigr)\right|\right]
=\E_{X,X',\epsilon}\!\left[\sup_{x}\left|\sum_{i=1}^n\epsilon_i\bigl(F(X_i,x)-F(X_i',x)\bigr)\right|\right]\\
&\qquad\le
\E_{X,X',\epsilon}\!\left[\sup_{x}\left|\sum_{i=1}^n\epsilon_i F(X_i,x)\right|\right]
+
\E_{X,X',\epsilon}\!\left[\sup_{x}\left|\sum_{i=1}^n\epsilon_i F(X_i',x)\right|\right]
=2\,\E_{X,\epsilon}\!\left[\sup_{x}\left|\sum_{i=1}^n\epsilon_i F(X_i,x)\right|\right].
\end{aligned}
\]
Taking the above inequalities together yields

\begin{equation}\label{eq:gen-dkw-sym}
\E[g] = \E\!\left[\sup_{x\in\R}\abs{\widehat{F}_n(x)-\overline F(x)}\right]
\le \frac{2}{n}\E_{X,\epsilon}\Big[\sup_{x\in\R}\Big|\sum_{i=1}^n \epsilon_i F(X_i,x)\Big|\Big].
\end{equation}

We next bound the Rademacher term.
Since $0\le F(X_i,x)\le 1$, for each $i$ and $x$,
\[
F(X_i,x)=\int_0^1 \mathbbm{1}\{F(X_i,x)\ge t\}\,\rd t.
\]
Therefore, for any fixed $x\in\R$,
\[
\Big|\sum_{i=1}^n \epsilon_i F(X_i,x)\Big|
=\Big|\int_0^1 \sum_{i=1}^n \epsilon_i \mathbbm{1}\{F(X_i,x)\ge t\}\,\rd t\Big|
\le \int_0^1 \Big|\sum_{i=1}^n \epsilon_i \mathbbm{1}\{F(X_i,x)\ge t\}\Big|\,\rd t,
\]
and hence
\begin{equation}\label{eq:gen-dkw-layercake}
\sup_{x\in\R}\Big|\sum_{i=1}^n \epsilon_i F(X_i,x)\Big|
\le \int_0^1 \sup_{x\in\R}\Big|\sum_{i=1}^n \epsilon_i \mathbbm{1}\{F(X_i,x)\ge t\}\Big|\,\rd t.
\end{equation}

Fix $t\in[0,1]$. For any $u\in\gX$, define the threshold
\[
a(u,t)\coloneqq \inf\{x\in\R:\ F(u,x)\ge t\}\in[-\infty,+\infty].
\]
Since $x\mapsto F(u,x)$ is nondecreasing and right-continuous, we have
\[
\{x\in\R:\ F(u,x)\ge t\}=[a(u,t),\infty),
\]
and thus, for every $x\in\R$,
\begin{equation}\label{eq:gen-dkw-threshold}
\mathbbm{1}\{F(X_i,x)\ge t\}=\mathbbm{1}\{x\ge a(X_i,t)\}.
\end{equation}
Consequently,
\[
\sup_{x\in\R}\Big|\sum_{i=1}^n \epsilon_i \mathbbm{1}\{F(X_i,x)\ge t\}\Big|
=
\sup_{x\in\R}\Big|\sum_{i=1}^n \epsilon_i \mathbbm{1}\{x\ge a(X_i,t)\}\Big|.
\]
Let $a_{(1)}(t)\le\cdots\le a_{(n)}(t)$ be the order statistics of $\{a(X_i,t)\}_{i=1}^n$,
and let $\epsilon_{(1)}(t),\ldots,\epsilon_{(n)}(t)$ be the corresponding reordered signs.
Then the mapping $x\mapsto \sum_{i=1}^n \epsilon_i \mathbbm{1}\{x\ge a(X_i,t)\}$ is piecewise constant and,
as $x$ increases, it takes values $\sum_{j=1}^k \epsilon_{(j)}(t)$ for some $k\in\{0,1,\ldots,n\}$. Hence,
\[
\sup_{x\in\R}\Big|\sum_{i=1}^n \epsilon_i \mathbbm{1}\{x\ge a(X_i,t)\}\Big|
=
\max_{0\le k\le n}\Big|\sum_{j=1}^k \epsilon_{(j)}(t)\Big|.
\]
Since $(\epsilon_1,\ldots,\epsilon_n)$ are i.i.d. and independent of the \(X_i\)'s, conditionally on $X_{1:n}$ the reordered sequence
$(\epsilon_{(1)}(t),\ldots,\epsilon_{(n)}(t))$ has the same joint distribution as $(\epsilon_1,\ldots,\epsilon_n)$.
Therefore,
\[
\E_{\epsilon}\Big[\sup_{x\in\R}\Big|\sum_{i=1}^n \epsilon_i \mathbbm{1}\{F(X_i,x)\ge t\}\Big|\ \Big|\ X_{1:n}\Big]
=
\E_{\epsilon}\Big[\max_{0\le k\le n}\Big|\sum_{j=1}^k \epsilon_j\Big|\Big].
\]
Let $S_k\coloneqq \sum_{j=1}^k \epsilon_j$, $k=0,1,\ldots,n$. Then $(S_k)_{k=0}^n$ is a martingale, and
Doob's $L^2$ maximal inequality yields
\[
\E\Big[\max_{0\le k\le n} |S_k|^2\Big] \le 4\E\big[|S_n|^2\big] = 4n.
\]
By Cauchy--Schwarz,
\[
\E\Big[\max_{0\le k\le n} |S_k|\Big] \le 2\sqrt n.
\]
Combining this with \eqref{eq:gen-dkw-layercake} and applying Tonelli's theorem gives
\[
\E_{X,\epsilon}\Big[\sup_{x\in\R}\Big|\sum_{i=1}^n \epsilon_i F(X_i,x)\Big|\Big]
\le \int_0^1 2\sqrt n\,\rd t
=2\sqrt n.
\]
Substituting into \eqref{eq:gen-dkw-sym} yields
\begin{equation}\label{eq:gen-dkw-Eg}
\E[g]\le \frac{2}{n}\cdot 2\sqrt n=\frac{4}{\sqrt n}.
\end{equation}

\paragraph{Step 3: completing the proof.}
Substituting \eqref{eq:gen-dkw-Eg} into \eqref{eq:gen-dkw-mcdiarmid} completes the proof.
\end{proof}

\bibliographystyle{apalike}
\bibliography{reference}

\end{document}